\documentclass{article}

    \usepackage[preprint]{neurips_2025}

\usepackage{amsmath,amsfonts,bm}

\def\eqref#1{(\ref{#1})}
\def\1{\bm{1}}

\DeclareMathAlphabet{\mathsfit}{\encodingdefault}{\sfdefault}{m}{sl}
\SetMathAlphabet{\mathsfit}{bold}{\encodingdefault}{\sfdefault}{bx}{n}

\newcommand{\R}{\mathbb{R}}

\DeclareMathOperator*{\argmax}{arg\,max}

\usepackage[utf8]{inputenc} % allow utf-8 input
\usepackage[T1]{fontenc}    % use 8-bit T1 fonts
\usepackage{hyperref}       % hyperlinks
\usepackage{url}            % simple URL typesetting
\usepackage{booktabs}       % professional-quality tables
\usepackage{amsfonts}       % blackboard math symbols
\usepackage{nicefrac}       % compact symbols for 1/2, etc.
\usepackage{microtype}      % microtypography
\usepackage{xcolor}         % colors
\usepackage{subcaption}
\usepackage{graphicx}
\usepackage{algorithm}
\usepackage{algorithmic}

\usepackage{amsmath}
\usepackage{amssymb}
\usepackage{mathtools}
\usepackage{amsthm}
\usepackage{bbm}
\usepackage{enumitem}
\usepackage{xspace}

\usepackage{multibib}
\newcites{APX}{References}

\theoremstyle{plain}
\newtheorem{theorem}{Theorem}[section]

\newtheorem{lemma}[theorem]{Lemma}
\newtheorem{corollary}[theorem]{Corollary}
\newtheorem{example}[theorem]{Example}
\theoremstyle{definition}
\newtheorem{definition}[theorem]{Definition}
\newtheorem{assumption}[theorem]{Assumption}

\newcommand{\prth}[1]{\left(#1\right)}
\newcommand{\set}[1]{\left\{ #1 \right\}}
\newcommand{\brac}[1]{\left[ #1 \right]}

\newcommand{\abs}[1]{\left \vert #1 \right\vert}

\newcommand{\mc}{\mathcal}
\newcommand{\mbb}{\mathbb}

\newcommand{\hp}{{\hat p}}
\newcommand{\prob}[1]{\mathbb{P}\left(#1\right)}
\newcommand{\reprob}[1]{\mathbb{P}_*\left(#1\right)}
\newcommand{\expect}[1]{\mathbb{E}\left[#1\right]}
\newcommand{\reexpect}[1]{\mathbb{E}_*\left[#1\right]}
\newcommand{\Dkl}{D_{\operatorname{KL}}}

\newcommand{\move}{$\mathsf{MoVE}$\xspace}
\newcommand{\rove}{$\mathsf{ROVE}$\xspace}
\newcommand{\roves}{$\mathsf{ROVEs}$\xspace}

\title{Subsampled Ensemble Can Improve Generalization Tail Exponentially}

\author{%
  Huajie Qian\\
  DAMO Academy\\
  Alibaba Group\\
  Bellevue, WA 98004 \\
  \texttt{h.qian@alibaba-inc.com} \\
  \And
    Donghao Ying\\
  IEOR Department\\
  UC Berkeley\\
  Berkeley, CA 94720 \\
  \texttt{donghaoy@berkeley.edu} \\
  \And
    Henry Lam\\
  IEOR Department\\
  Columbia University\\
  New York, NY 10027 \\
  \texttt{henry.lam@columbia.edu} \\
  \And
  Wotao Yin\\
  DAMO Academy\\
  Alibaba Group\\
  Bellevue, WA 98004 \\
  \texttt{wotao.yin@alibaba-inc.com} \\
}

\begin{document}

\maketitle

\begin{abstract}
Ensemble learning is a popular technique to improve the accuracy of machine learning models. It traditionally hinges on the rationale that aggregating multiple weak models can lead to better models with lower variance and hence higher stability, especially for discontinuous base learners. In this paper, we provide a new perspective on ensembling. By selecting the most frequently generated model from the base learner when repeatedly applied to subsamples, we can attain exponentially decaying tails for the excess risk, even if the base learner suffers from slow (i.e., polynomial) decay rates. This tail enhancement power of ensembling applies to base learners that have reasonable predictive power to begin with and is stronger than variance reduction in the sense of exhibiting rate improvement. We demonstrate how our ensemble methods can substantially improve out-of-sample performances in a range of numerical examples involving heavy-tailed data or intrinsically slow rates.
\end{abstract}

% Motivation: application of bagging, which is a commonly used technique in machine learning, in stochastic optimization
\section{Introduction}\label{sec: introduction}
Ensemble learning \citep{dietterich2000ensemble,zhou2012ensemble} is a class of methods designed to improve the accuracy of machine learning models by combining multiple models, known as ``base learners'', through aggregation techniques such as averaging or majority voting.
In the existing literature, ensemble methods—most notably bagging \citep{breiman1996bagging} and boosting \citep{freund1996experiments}—are primarily justified based on their ability to reduce bias and variance or improve model stability.
This justification has been shown to be particularly relevant for certain U-statistics \citep{buja2006observations} and models with hard-thresholding rules, such as decision trees \citep{breiman2001random,drucker1995boosting}.

In contrast to this traditional understanding, we present a novel perspective on ensemble learning, demonstrating its capability to achieve a significantly stronger effect than variance reduction: By suitably selecting the best base learners trained on random subsamples, ensembling leads to exponentially decaying excess risk tails.
Specifically, for general stochastic optimization problems that suffer from a slow (polynomial) decay in excess risk tails, ensembling can reduce these tails to an exponential decay rate, a substantial improvement beyond the constant-factor gains typically exhibited by variance reduction.

% \paragraph{Main Results at a High Level.}
To detail our contribution, we consider the generic stochastic optimization problem
\begin{equation}\label{opt}
\min_{\theta\in\Theta}L(\theta):=\expect{l(\theta,z)},
\end{equation}
where $\theta\in \Theta$ is the decision variable, $z\in\mathcal{Z}$ is a random variable governed by some probability distribution, and $l(\cdot, \cdot)$ is the cost function. 
A dataset of $n$ i.i.d. samples $\{z_1, \ldots, z_n\}$ is drawn from the underlying distribution of $z$.
In the context of machine learning, $\theta$ corresponds to the model parameters, $\{z_1, \ldots, z_n\}$ represents the training data, $l$ denotes the loss function, and $L$ is the population-level expected loss.
More generally, \eqref{opt} also encompasses data-driven decision-making problems, i.e., the integration of data on $z$ into a downstream optimization task with overall cost function $l$ and prescriptive decision $\theta$. 
These problems are increasingly prevalent in various industrial applications \citep{kamble2020achieving, bertsimas2023data, ghosal2024unifying}.
For example, in supply chain network design, $\theta$ may represent the decision to open processing facilities, $z$ the uncertain supply and demand, and $l$ the total cost of processing and transportation.

Given the data, a learning algorithm can be used to map the data to an element in $\Theta$, yielding a trained model or decision. This encompasses a variety of methods, including machine learning training algorithms and data-driven approaches such as sample average approximation (SAA) \citep{shapiro2021lectures} and distributionally robust optimization (DRO) \citep{mohajerin2018data} in stochastic optimization. The theoretical framework and methodology proposed in this paper work for all learning algorithms that meet the formal performance criterion in our theorems.

% \paragraph{Main Results at a High Level.}
\textbf{Main Results at a High Level.$\quad$}
Let $\hat{\theta}$ be the output of a learning algorithm. We characterize its generalization performance through the tail probability bound on the excess risk $L(\hat{\theta}) - \min_{\theta \in \Theta} L(\theta)$, i.e., $\mathbb{P}(L(\hat{\theta})>\min_{\theta\in\Theta}L(\theta)+\delta)$ for some fixed $\delta>0$, where the probability is over both the data and training randomness.
A \textit{polynomially decaying generalization tail} refers to:
\begin{equation}\label{eq: polynomial decay}  \mbb{P}\Big(L(\hat{\theta})>\min_{\theta\in\Theta}L(\theta)+\delta\Big)\geq C_1n^{-\alpha},
\end{equation}
for some $\alpha>0$ and $C_1$ that depends on $\delta$. 
Such bounds are common under heavy-tailed data distributions \citep{kavnkova2015thin,jiang2020rates,jiang2021complexity} due to slow concentration, which frequently arises in machine learning applications such as large language models \citep{jalalzai2020heavy,zhang2020adaptive,cutkosky2021high}, finance \citep{mainik2015portfolio,gilli2006application}, and physics \citep{fortin2015applications,michel2007analysis}.
% \section{An Illustrative Example}
This can be illustrated with a simple linear program (LP):
\begin{example}[LP with a polynomial tail]\label{ex: linear_program}
Consider the stochastic LP $\min_{\theta\in [0,1]}\mathbb{E}[z \theta]$, i.e., $l(\theta, z)=z\theta$ and $\Theta=[0,1]$ in \eqref{opt}, and its SAA solution $\hat \theta\in \mathrm{argmin}_{\theta\in [0,1]}\sum_{i=1}^nz_i/n\cdot \theta$. 
Assume $z$ has a non-zero density everywhere and is symmetric with respect to its mean $\expect{z}=1$ (hence $L(\theta)=\theta$). 
Then we have $\mbb{P}(\hat \theta=1)\geq \mbb{P}(\sum_{i=1}^nz_i/n<0)\geq \mbb{P}(\sum_{i=1}^{n-1}z_i\leq n-1\text{ and } z_n<1-n)=\mbb{P}(\sum_{i=1}^{n-1}z_i\leq n-1)\mbb{P}( z_n<1-n)$, where the last equality uses the independence of $z_i$'s. By the symmetry of $z$, we have $\mbb{P}(\sum_{i=1}^{n-1}z_i\leq n-1)= 1/2$, thus for $\delta\in (0,1)$ the tail
% Therefore, for every $\delta<1$, the tail satisfies that
\begin{equation}\label{eq: lower bound for linear program example}
\mbb{P}\Big(L(\hat \theta)>\min_{\theta\in[0,1]}L(\theta)+\delta\Big)\geq \mbb{P}(\hat \theta=1)\geq \frac{1}{2} \mbb{P}( z<1-n).
\end{equation}
If $z$ has a polynomial tail, e.g., $\mbb{P}(z<1-n)=\Omega(n^{-\alpha})$ for some $\alpha>0$ where $\Omega(\cdot)$ contains some multiplicative constant, the generalization tail bound \eqref{eq: polynomial decay} applies. 
Appendix \ref{app: linear_regression} provides another example from linear regression.
\end{example}

% and are proved to be tight \citep{catoni2012challenging} for empirical risk minimization (ERM) \citep{vapnik1991principles}

As the key contribution of our work, we propose ensemble methods that significantly improve these bounds, achieving an \textit{exponentially decaying generalization tail}: 
\begin{equation}\label{eq: exponential decay}
\mbb{P}\Big(L(\hat{\theta})>\min_{\theta\in\Theta}L(\theta)+\delta\Big)\leq C_2\gamma^{n/k},
\end{equation}
% if the resampling Monte Carlo noise is made negligible, 
where $k$ is the subsample size and $\gamma<1$ depends on $k, \delta$ with $\gamma\to 0$ as $k\to\infty$.
By appropriately choosing $k$ at a slower rate in $n$, the decay becomes exponential.
This exponential acceleration is fundamentally different from the well-known variance reduction benefit of ensembling in two perspectives. 
First, variance reduction refers to the smaller variability in predictions from models trained on independent data sets, thus has a more direct impact on the expected regret than the tail decay rate. 
Second, variance reduction typically yields a constant-factor improvement (e.g., \citep{buhlmann2002analyzing} report a reduction factor of $3$), whereas we obtain an order-of-magnitude improvement.

% \paragraph{Main Intuition.} 
\textbf{Main Intuition.$\quad$}
Consider first the discrete space $\Theta$. 
Our ensemble method employs a majority-vote mechanism at the model level: The learning algorithm is repeatedly run on subsamples to generate multiple models, and the model appearing most frequently is selected as the output. This resembles the majority vote in ensemble methods for classification but the voting is on models instead of classes.
This process effectively estimates the mode of the sampling distribution of the learned model by subsampling, and thus is less susceptible to extreme data and training randomness that incurs the slow tail decay in \eqref{eq: polynomial decay}. This mode estimation can be formalized via a surrogate optimization problem over the same decision space $\Theta$ as \eqref{opt} that maximizes the probability of being output by the learning algorithm. The probability objective, as the expected value of a random indicator function, is uniformly bounded even if the original objective is heavy-tailed and hence admits exponential decay in the tail.
Consequently, base learners with high-quality mode models receive an exponential enhancement in their tail behavior.
To illustrate the main idea using Example \ref{ex: linear_program}, although $\mbb{P}(\hat \theta=1)$ can be substantial in a heavy-tailed setting, it holds that $\mbb{P}(\hat \theta=0)>\mbb{P}(\hat \theta=1)$, and thus the mode of $\hat \theta$, i.e., $0$, recovers the optimal solution.

% Together with a bootstrap argument that quantifies the closeness between the subsample and the full data,

For general problems with possibly continuous decision spaces, we replace the majority vote with a voting mechanism based on the likelihood of being $\epsilon$-optimal among all models when evaluated on random subsamples. 
This avoids the degeneracy of using a majority vote for continuous problems while retaining similar (in fact, even stronger) theoretical guarantees. 
For both discrete and continuous problems, our method fundamentally improves the tail behavior from \eqref{eq: polynomial decay} to \eqref{eq: exponential decay}.

% \paragraph{Organization.}
\textbf{Organization.$\quad$}
The rest of the paper is organized as follows. Section \ref{section: procedures} presents our methods and their finite-sample bounds. Section \ref{section: experiments} presents experimental results, Section \ref{sec: related work} discusses related work, and Section \ref{sec: conclusion} concludes the paper. 
Technical proofs and additional experiments can be found in the appendix.
\section{Methodology and Theoretical Guarantees}\label{section: procedures}
% This section presents our bagging frameworks and establishes finite-sample generalization tail bounds. 
We consider the generic learning algorithm in the form of
\begin{equation*}%\label{prob: data-driven approximation}
% \min_{\theta\in\Theta}\ \hat Z_n(\theta;z_1,\ldots,z_n),    
\mathcal{A}(z_1,\ldots,z_n;\omega):\mathcal{Z}^n\times \mathbf{\Omega} \to \Theta
\end{equation*}
that takes in the training data $(z_1,\ldots,z_n)\in \mathcal{Z}^n$ and outputs a model possibly under some algorithmic randomness $\omega\in \mathbf{\Omega}$ that is independent of the data. Examples of $\omega$ include gradient sampling in stochastic algorithms and feature/data subsampling in random forests. For convenience, we omit $\omega$ to write $\mathcal{A}(z_1,\ldots,z_n)$ when no confusion arises.

% Bagging for discrete problems
\subsection{A Basic Procedure}
We first introduce a procedure called $\mathsf{MoVE}$ that applies to discrete solution or model space $\Theta$. $\mathsf{MoVE}$, which is formally described in Algorithm \ref{bagging majority vote: set estimator}, repeatedly draws a total of $B$ subsamples from the data without replacement, learns a model via $\mathcal{A}$ on each subsample, and finally conducts a majority vote to output the most frequently subsampled model. Tie-breaking can be done arbitrarily.

\begin{algorithm}
% \caption{\textbf{E}nsembling via \textbf{M}aj\textbf{o}rity \textbf{V}ot\textbf{e} (EMOVE)}
\caption{\textbf{M}aj\textbf{o}rity \textbf{V}ote \textbf{E}nsembling ($\mathsf{MoVE}$)}
\label{bagging majority vote: set estimator}
\begin{algorithmic}[1]
\STATE {\textbf{Input:} A base learning algorithm $\mathcal{A}$, $n$ i.i.d. observations $\mathbf{z}_{1:n}=(z_1,\ldots,z_n)$, subsample size $k<n$, and ensemble size $B$.
% , and solver for \eqref{prob: data-driven approximation}.
}

% \vspace{1ex}

\FOR{$b=1$ to $B$ }
\STATE {Randomly sample $\mathbf{z}_k^b=(z_1^b,\ldots,z_k^b)$ uniformly from $\mathbf{z}_{1:n}$ without replacement, and obtain $\hat{\theta}_k^b=\mathcal{A}(z_1^b,\ldots,z_k^b)$.
}
% \STATE {Compute $\bar{\psi}^b_{AV}=\frac{1}{R}\sum_{r=1}^R\tilde{\psi}_r(\widehat{F}_1^b,\ldots,\widehat{F}_m^b)$}
\ENDFOR

% \vspace{1ex}

\STATE {\textbf{Output:} $\hat{\theta}_n\in\argmax_{\theta\in\Theta}\sum_{b=1}^B\mathbbm{1}(\theta=\hat{\theta}^b_k)$.}
\end{algorithmic}
\end{algorithm}

To understand $\mathsf{MoVE}$ from the lens of mode estimation, we consider an optimization associated with the base learner $\mathcal{A}$:
%\vspace{-2.8pt}
\begin{equation}\label{prob:maximizing selection probability}
    \max_{\theta\in\Theta}\ p_k(\theta):=\prob{\theta= \mathcal{A}(z_1,\ldots,z_k)},
    %\vspace{-2.8pt}
\end{equation}
which maximizes the probability of a model being output by the base learner on $k$ i.i.d. observations. Here the probability $\mathbb P$ is with respect to both the training data and the algorithmic randomness. If $B=\infty$, $\mathsf{MoVE}$ essentially maximizes an empirical approximation of \eqref{prob:maximizing selection probability}, i.e.
%\vspace{-2.8pt}
\begin{equation}\label{prob:maximizing selection probability empirical}
    \max_{\theta\in\Theta}\ \reprob{\theta= \mathcal{A}(z^*_1,\ldots,z^*_k)},
    %\vspace{-2.8pt}
\end{equation}
where $\prth{z^*_1,\ldots,z^*_k}$ is a uniform random subsample from $\prth{z_1,\ldots,z_n}$, and $\mathbb P_*$ denotes the probability with respect to the algorithmic randomness and the subsampling randomness conditioned on $\prth{z_1,\ldots,z_n}$. With a finite $B<\infty$, extra Monte Carlo noises are introduced, leading to the following maximization problem
%\vspace{-2.8pt}
\begin{equation}\label{prob:maximizing selection probability monte carlo}
    \max_{\theta\in\Theta}\ \frac{1}{B}\sum_{b=1}^B\mathbbm{1}(\theta=\mathcal{A}(z_1^b,\ldots,z_k^b)),
    %\vspace{-2.8pt}
\end{equation}
which gives exactly the output of $\mathsf{MoVE}$. In other words, $\mathsf{MoVE}$ is a \emph{bootstrap approximation} to the solution of \eqref{prob:maximizing selection probability}. 
The following result materializes the intuition explained in Section \ref{sec: introduction} on the conversion of the original potentially heavy-tailed problem \eqref{opt} into a probability maximization \eqref{prob:maximizing selection probability monte carlo} that possesses exponential bounds.

% We have the following finite-sample generalization bound for our Algorithm \ref{bagging majority vote: set estimator}:
\begin{theorem}[Informal bound for Algorithm \ref{bagging majority vote: set estimator}]\label{thm: general majority vote}
Consider discrete decision space $\Theta$. Let $\Theta^{\delta}:=\left\{\theta\in\Theta:L(\theta)\leq \min_{\theta'\in\Theta}L(\theta')+\delta\right\}$ be the set of $\delta$-optimal models and
%Let $p_k^{\max}:=\max_{\theta\in\Theta}p_k(\theta)$, 
%For any $\delta > 0$, denote $\mathcal{E}_{k,\delta}:=\mathbb P(L(\mathcal{A}(z_1,\ldots,z_k))>\min_{\theta\in\Theta}L(\theta)+\delta)$ as the excess risk tail of $\mathcal{A}$, and
\begin{equation*}%\label{SAA prob gap}
    \eta_{k,\delta}:=\max_{\theta\in\Theta}p_k(\theta) - \max_{\theta\in\Theta/\Theta^{\delta}}p_k(\theta),
\end{equation*}
where $p_k(\theta)$ is defined in \eqref{prob:maximizing selection probability} and $\max_{\theta\in\Theta\backslash\Theta^{\delta}}p_k(\theta)$ evaluates to $0$ if $\Theta\backslash\Theta^{\delta}$ is empty. Then, for every $k\leq n$ and $\delta\geq 0$ such that $\eta_{k,\delta}>0$, the solution output by \move satisfies that
%\vspace{-2.8pt}
\begin{equation}\label{finite-sample bound for general bagging solution_informal}
\begin{aligned}
\mbb{P}(L(\hat{\theta}_n) > \min_{\theta\in\Theta}L(\theta)+\delta) \leq \abs{\Theta}\Big[4\exp(-\Omega(n/k))+\exp(-\Omega(B))\Big],
%\vspace{-2.8pt}
\end{aligned}
\end{equation}
where $\abs{\Theta}$ denotes the cardinality of $\Theta$, and $\Omega(\cdot)$ contains multiplicative coefficients that depend on $\max_{\theta\in\Theta}p_k(\theta)$ and $\eta_{k,\delta}$. If $\eta_{k,\delta}>4/5$, \eqref{finite-sample bound for general bagging solution_informal} is further bounded by
\begin{equation}\label{eq: general finite-sample bound for large gap_informal}
\abs{\Theta}\left(3\min\left\{e^{-2/5}, C_1\max\set{1-\max_{\theta\in\Theta}p_k(\theta),\max_{\theta\in\Theta/\Theta^{\delta}}p_k(\theta)}\right\}^{\frac{n}{C_2k}} + e^{-\frac{B}{C_3}}\right),
\end{equation}
where $C_1,C_2,C_3>0$ are universal constants. 
\end{theorem}

The formal statement is deferred to Theorem \ref{thm: general majority vote_formal} in Appendix \ref{app: proof_theorem1_proposition1}. Theorem \ref{thm: general majority vote} states that the excess risk tail of $\mathsf{MoVE}$ decays exponentially in the ratio $n/k$ and ensemble size $B$. The bound \eqref{finite-sample bound for general bagging solution_informal} consists of two terms: 
The term $\exp(-\Omega(n/k))$ arises from the bootstrap approximation of \eqref{prob:maximizing selection probability} with \eqref{prob:maximizing selection probability empirical}, whereas the term $\exp(-\Omega(B))$ quantifies the Monte Carlo error in approximating \eqref{prob:maximizing selection probability empirical} with a finite $B$. The multiplier $\abs{\Theta}$ in the bound is avoidable, e.g., via a space reduction as in our next algorithm.

The quantity $\eta_{k,\delta}$ plays two roles. First, it quantifies how suboptimality in the surrogate problem \eqref{prob:maximizing selection probability} propagates to the original problem \eqref{opt} in that every $\eta_{k,\delta}$-optimal solution for \eqref{prob:maximizing selection probability} is $\delta$-optimal for \eqref{opt}. Second, $\eta_{k,\delta}>0$ simply means that the mode solution is $\delta$-optimal and hence $\eta_{k,\delta}$ directly quantifies the concentration of the base learner on near-optimal solutions. Therefore, a large $\eta_{k,\delta}$ signals the situation where the base learner already generalizes well. 
In this case, \eqref{finite-sample bound for general bagging solution_informal} reduces to \eqref{eq: general finite-sample bound for large gap_informal}. \eqref{eq: general finite-sample bound for large gap_informal} suggests that our approach does not hurt the performance of an already high-performing base learner as its generalization power is inherited through the $\max\set{1-\max_{\theta\in\Theta}p_k(\theta),\max_{\theta\in\Theta/\Theta^{\delta}}p_k(\theta)}$ term in the bound. See Appendix \ref{sec: additional technical results} for a more detailed discussion.

% EXPLAIN AND MAKE CONCISELY THE POINT THAT OUR BAGGING DOESN'T HURT EVEN IF THE BASE LEARNER GENERALIZES WELL ALREADY.

Theorem \ref{thm: general majority vote} also hints at the choice of hyperparameters $k$ and $B$. As long as $\eta_{k,\delta}>0$, our bound decays exponentially fast, and in this regime the bound \eqref{finite-sample bound for general bagging solution_informal} suggests that a smaller $k$ (consequently a larger ratio $n/k$) leads to thinner tails. However, like other subsampling-based ensemble methods (e.g., subagging \cite{buhlmann2002analyzing}), reducing the subsample size $k$ also enlarges the model bias. In experiments, we set $k=\max(10,n/200)$ for a balance between tail and bias performance. Regarding the choice of $B$, we observe from \eqref{finite-sample bound for general bagging solution_informal} that using a $B=\mc{O}(n/k)$ is sufficient to control the Monte Carlo error to a similar magnitude as the statistical error.
% the generalization sensitivity of the original problem \eqref{opt} with respect to the companion problem \eqref{prob:maximizing selection probability} in the sense that 

Applying Theorem \ref{thm: general majority vote} to Example \ref{ex: linear_program} gives an exponential tail as opposed to the slow decay in \eqref{eq: lower bound for linear program example}.
\begin{corollary}[Enhanced tail for Example \ref{ex: linear_program}]\label{cor: application of move to linear program example}
Consider the stochastic LP in Example \ref{ex: linear_program} and denote $q_k:=\mbb{P}(\sum_{i=1}^kz_i>0)$. We have $q_k>1/2$ by the symmetry of $z$. If \move is applied to Example \ref{ex: linear_program} with $\mathcal{A}$ being the SAA, we have $\max_{\theta\in\Theta}p_k(\theta)=q_k$, $\max_{\theta\in\Theta/\Theta^{\delta}}p_k(\theta)=1-q_k$ whenever $\delta<1$, and $\abs{\Theta}=2$. Consequently, $\eta_{k,\delta}=2q_k-1>0$ for every $k>0$ and $\delta<1$, ensuring the tail upper bound \eqref{finite-sample bound for general bagging solution_informal} holds. If $q_k>0.9$, we also have the tail upper bound $6\min\big\{e^{-2/5}, C_1(1-q_k)\big\}^{\frac{n}{C_2k}} + 2e^{-\frac{B}{C_3}}$ from \eqref{eq: general finite-sample bound for large gap_informal}.
\end{corollary}
The proof of Corollary \ref{cor: application of move to linear program example} can be found in Appendix \ref{sec: proof of linear program example}.

\subsection{A More General Procedure}\label{subsec:tow phase framework}
We next present a more general procedure called $\mathsf{ROVE}$ that applies to continuous space where the simple majority vote in \move can lead to degeneracy, i.e., all learned models appear exactly once in the pool. Moreover, this general procedure relaxes the dependence on $|\Theta|$ in the bound \eqref{finite-sample bound for general bagging solution_informal}.

\begin{algorithm}[!ht]
\caption{\textbf{R}etrieval and $\epsilon$-\textbf{O}ptimality \textbf{V}ote \textbf{E}nsembling ($\mathsf{ROVE}$ / $\mathsf{ROVEs}$)}
% \caption{$\epsilon$-\textbf{O}ptima\textbf{l}ity \textbf{V}ot\textbf{e} Ensembling (EVOLVE / EVOLVES)}
\label{bagging majority vote: two phase}
\begin{algorithmic}
\STATE {\textbf{Input:} A base learning algorithm $\mathcal{A}$, $n$ i.i.d. observations $\mathbf{z}_{1:n}=(z_1,\ldots,z_n)$, subsample size $k_1,k_2<n$ (if no split) or $n/2$ (if split), ensemble sizes $B_1$ and $B_2$.}
% , training algorithm in the form of \eqref{prob: data-driven approximation}, binary indicator of data splitting $s\in \{0,1\}$.}

\vspace{1ex}

\STATE {\textbf{Phase I: Model Candidate Retrieval}}

\FOR{$b=1$ to $B_1$ }
\STATE {Randomly sample $\mathbf{z}_{k_1}^b=(z_1^b,\ldots,z_{k_1}^b)$ uniformly from $\mathbf{z}_{1:n}$ (if no split) or $\mathbf{z}_{1:\lfloor\frac{n}{2}\rfloor}$ (if split) without replacement, and obtain $\hat{\theta}^b_{k_1}=\mathcal{A}(z_1^b,\ldots,z_{k_1}^b)$.
}
% \STATE {Compute $\bar{\psi}^b_{AV}=\frac{1}{R}\sum_{r=1}^R\tilde{\psi}_r(\widehat{F}_1^b,\ldots,\widehat{F}_m^b)$}
\ENDFOR

\STATE {Let $\mathcal{S}:=\{\hat{\theta}^b_{k_1}:b=1,\ldots,B_1\}$ be the set of all retrieved models.}

\vspace{1ex}

\STATE {\textbf{Phase II: $\epsilon$-Optimality Vote}}
\STATE {Choose $\epsilon\geq 0$ using the data $\mathbf{z}_{1:n}$ (if no split) or $\mathbf{z}_{1:\lfloor\frac{n}{2}\rfloor}$ (if split).}
\FOR{$b=1$ to $B_2$ }
\STATE {Randomly sample $\mathbf{z}_{k_2}^b=(z_1^b,\ldots,z_{k_2}^b)$ uniformly from $\mathbf{z}_{1:n}$ (if no split) or $\mathbf{z}_{\lfloor\frac{n}{2}\rfloor+1:n}$ (if split) without replacement, and calculate
$$\widehat{\Theta}^{\epsilon,b}_{k_2}:=\left\{\theta\in \mathcal{S}: \frac{1}{k_2}\sum_{i=1}^{k_2}l(\theta,z_i^b)\leq \min_{\theta'\in  \mathcal{S}}\frac{1}{k_2}\sum_{i=1}^{k_2}l(\theta',z_i^b)+\epsilon\right\}.$$
}
% \STATE {Compute $\bar{\psi}^b_{AV}=\frac{1}{R}\sum_{r=1}^R\tilde{\psi}_r(\widehat{F}_1^b,\ldots,\widehat{F}_m^b)$}
\ENDFOR

% \vspace{1ex}

\STATE {\textbf{Output:} $\hat{\theta}_n\in\argmax_{\theta\in \mathcal{S}}\sum_{b=1}^{B_2}\mathbbm{1}(\theta\in \widehat{\Theta}^{\epsilon,b}_{k_2})$.}
\end{algorithmic}
\end{algorithm}

$\mathsf{ROVE}$, displayed in Algorithm \ref{bagging majority vote: two phase}, proceeds initially the same as $\mathsf{MoVE}$ in repeatedly subsampling data and training the model using $\mathcal{A}$. However, in the aggregation step, instead of using a simple majority vote, $\mathsf{ROVE}$ outputs, among all the trained models, the one that has the highest likelihood of being $\epsilon$-optimal. This $\epsilon$-optimality avoids the degeneracy of the majority vote.
Moreover, since we have restricted our output to the collection of retrieved models, the corresponding likelihood maximization is readily doable by direct enumeration. In addition, it helps reduce competition for votes among best models, as each subsample can now vote for multiple candidates, ensuring a high vote count for each of the top models even when there are many of them. This makes \rove more effective than \move in the case of multiple (near) optima as our experiments will show. We have the following theoretical guarantees for Algorithm \ref{bagging majority vote: two phase}.

\begin{theorem}[Informal bound for Algorithm \ref{bagging majority vote: two phase}]\label{thm: finite-sample bound for multiple predictions two phase splitting}
Let $\mathcal{E}_{k,\delta}:=\mathbb P(L(\mathcal{A}(z_1,\ldots,z_k))>\min_{\theta\in\Theta}L(\theta)+\delta)$ be the excess risk tail of $\mathcal{A}$.
%Let $\mathcal{E}_{k,\delta}$ be the excess risk tail defined in Theorem \ref{thm: general majority vote}.
Consider Algorithm \ref{bagging majority vote: two phase} with data splitting, i.e., $\mathsf{ROVEs}$.
Let $T_k(\cdot):=\mathbb P(\sup_{\theta\in\Theta}\lvert (1/k)\sum_{i=1}^kl(\theta,z_i)-L(\theta)\rvert> \cdot)$ be the tail function of the maximum deviation of the empirical objective estimate.
Then, for every $\delta > 0$, under mild conditions on $\epsilon$ and $T_{k_2}(\cdot)$, it holds that 
\begin{equation}\label{eq: finite-sample bound for data splitting two phase under large pmax_informal}
\begin{aligned}
\mbb{P}\Big(L(\hat{\theta}_n) > \min_{\theta\in\Theta}L(\theta)+2\delta\Big)
&\leq B_1 \Big[3\exp(-\Omega(n/k_2)) + \exp(-\Omega(B_2))\Big] \\
&\quad+ \exp(-\Omega(n/k_1)) + \exp(-\Omega(B_1)),
\end{aligned}
\end{equation}
where $\Omega(\cdot)$ contains multiplicative coefficients depending on $T_{k_2}(\cdot)$, $\epsilon$, $\delta$ and $\mathcal{E}_{k_1,\delta}$.

Consider Algorithm \ref{bagging majority vote: two phase} without data splitting, i.e., $\mathsf{ROVE}$, and discrete space $\Theta$. Assume $\lim_{k\to\infty}T_k(\delta)= 0$ for all $\delta>0$. Then, for every fixed $\delta>0$, under mild conditions, it holds that $\lim_{n\to\infty}\mathbb P(L(\hat{\theta}_n) > \min_{\theta\in\Theta}L(\theta)+2\delta)\to 0$.
\end{theorem}

The formal statement can be found in Appendix \ref{sec: theories for multiple predictions}.
Theorem \ref{thm: finite-sample bound for multiple predictions two phase splitting} provides an exponential excess risk tail, regardless of discrete or continuous space. 
The terms in the square bracket of \eqref{eq: finite-sample bound for data splitting two phase under large pmax_informal} are inherited from the bound \eqref{eq: general finite-sample bound for large gap_informal} for $\mathsf{MoVE}$ with the majority vote replaced by $\epsilon$-optimality vote. 
In particular, the multiplier $\abs{\Theta}$ in \eqref{eq: general finite-sample bound for large gap_informal} is now replaced by $B_1$, the number of retrieved models from Phase I. 
The last two terms in \eqref{eq: finite-sample bound for data splitting two phase under large pmax_informal} bound the performance sacrifice due to the restriction to the retrieved models.

$\mathsf{ROVE}$ may be carried out with the data split between the two phases, where it is referred to as $\mathsf{ROVEs}$. Data splitting makes the procedure theoretically more tractable by avoiding inter-dependency between the phases but sacrifices some statistical power by halving the data size. Empirically we find it more effective not to split data.

The optimality threshold $\epsilon$ is allowed to be chosen in a data-driven way and the main goal guiding this choice is to distinguish models of different qualities. In other words, $\epsilon$ should be chosen to create enough variability in the likelihood of being $\epsilon$-optimal across models. In our experiments, we find it a good strategy to choose an $\epsilon$ that leads to a maximum likelihood around $1/2$.

% \paragraph{Technical Novelty.} 
\textbf{Technical Novelty.$\quad$}
Our main theoretical results, Theorems \ref{thm: general majority vote} and \ref{thm: finite-sample bound for multiple predictions two phase splitting}, are derived using several novel techniques. First, we develop a sharper concentration result for U-statistics with binary kernels, improving upon standard Bernstein-type inequalities (e.g., \cite{arcones1995bernstein,peel2010empirical}). This refinement ensures the correct order of the bound, particularly \eqref{eq: general finite-sample bound for large gap_informal}, which captures the convergence of both the bootstrap approximation and the base learner, offering insights into the robustness of our methods for fast-converging base learners. Second, we perform a sensitivity analysis on the regret for the original problem \eqref{opt} relative to the surrogate optimization \eqref{prob:maximizing selection probability}, translating the superior generalization in the surrogate problem into accelerated convergence for the original. Finally, to establish asymptotic consistency for Algorithm \ref{bagging majority vote: two phase} without data splitting, we develop a uniform law of large numbers (LLN) for the class of events of being $\epsilon$-optimal, using direct analysis of the second moment of the maximum deviation. Uniform LLNs are particularly challenging here because, unlike fixed function classes in standard settings, this function class dynamically changes with the subsample size $k_2$ as $n \to \infty$.

\section{Numerical Experiments}\label{section: experiments}
In this section, we numerically test Algorithm \ref{bagging majority vote: set estimator} (\move), Algorithm \ref{bagging majority vote: two phase} with (\roves) and without (\rove) data splitting in training machine learning models and solving stochastic programs. 
Due to space constraints, additional experimental results are provided in Appendix \ref{app: additional_experiment}. In particular, a comprehensive hyperparameter profiling of our algorithms is performed in Appendix \ref{subsec: exp_profile} to find empirically well-performing configurations for general use.
Unless specified otherwise, all our experiments use the recommended configuration summarized at the end of Appendix \ref{subsec: exp_profile}. 
All experiments are conducted on a personal computer, and Gurobi Optimizer is required for certain experiments on stochastic programs.
The code is available at: \href{https://github.com/mickeyhqian/VoteEnsemble}{https://github.com/mickeyhqian/VoteEnsemble}.

\subsection{Neural Networks and Trees for Regression}\label{subsec:neural network experiment}
\textbf{Setup.$\quad$} We consider regression problems with multilayer perceptrons (MLPs) and decision trees. 
Note that classification models are prevalently trained using the cross-entropy loss that is inherently less prone to heavy-tailed noises thanks to the presence of the logarithm.
For neural networks, the base learner splits the data into training (70\%) and validation (30\%), and uses Adam to minimize the mean squared error (MSE), with early stopping triggered when the validation improvement falls below 3\% between epochs. The architecture of the MLPs is provided in Appendix \ref{app: network architecture}. For trees, the base learner is a single regression decision tree with the MSE as the splitting criterion. Besides the base learner, we also compare with three popular tree ensembles, Random Forests (RF) \cite{breiman2001random,geurts2006extremely}, Gradient Boosted Decision Trees (GBDT) \cite{friedman2001greedy,friedman2002stochastic}, and XGBoost (XGB) \cite{chen2016xgboost}. RF are constructed with the same number of trees as our methods for a fair comparison, whereas GBDT and XGB are run with early stopping to avoid overfitting. \move is not included in this comparison as it's applicable to discrete problems only.

\textbf{Synthetic Data.$\quad$} Input-output pairs \((X, Y)\) are generated as $Y = (1/50)\cdot \sum_{j=1}^{50} \log(X_j + 1) + \varepsilon$, where each \(X_j\) is drawn independently from \(\mathrm{Unif}(0, 2 + 198(j-1)/49)\), and the noise \(\varepsilon\) is independent of \(X\) with zero mean. We consider both standard Gaussian noise and Pareto noise \(\varepsilon = \varepsilon_1 - \varepsilon_2\), where each \(\varepsilon_i \sim \mathrm{Pareto}(2.1)\). The out-of-sample performance is estimated on a common test set of one million samples. Each algorithm is repeatedly applied to $200$ independently generated datasets to assess the average and tail performance.

\textbf{Real Data.$\quad$} We use three datasets from the UCI Machine Learning Repository \citep{blake1998uci}: \textit{Bike Sharing} \citep{bike_sharing_275}, \textit{Superconductivity} \citep{superconductivty_data_464}, and \textit{Gas Turbine Emission} \citep{gas_turbine_co_and_nox_emission_data_set_551}. The data is standardized (zero mean, unit variance). To evaluate the tail probabilities of out-of-sample costs, we permute each dataset $100$ times, and each time use the first half for training and the second for testing. Results for three other datasets can be found in Figure \ref{fig: MLP L=4 public data plots appendix} in Appendix \ref{app: additional_experiment}.

\begin{figure*}[h]
\centering
\hspace{-13pt}
\begin{subfigure}{0.34\textwidth}
\includegraphics[width = \linewidth]{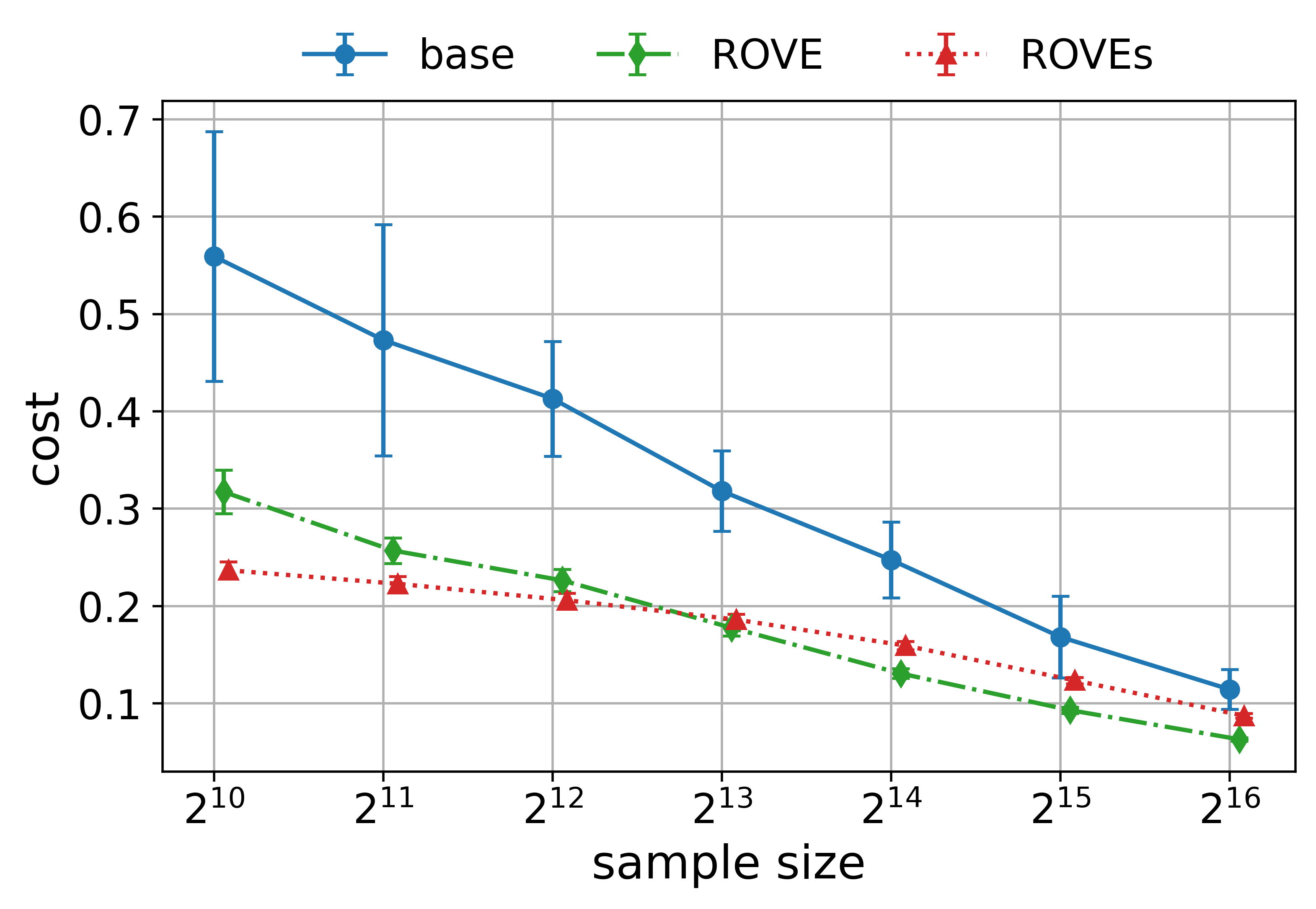}
\caption{Pareto noise, $H=4$.\label{subfig: MLP avg d=50 L=4}}
\end{subfigure}
\hspace{-5pt}
\begin{subfigure}{0.34\textwidth}
\includegraphics[width = \linewidth]{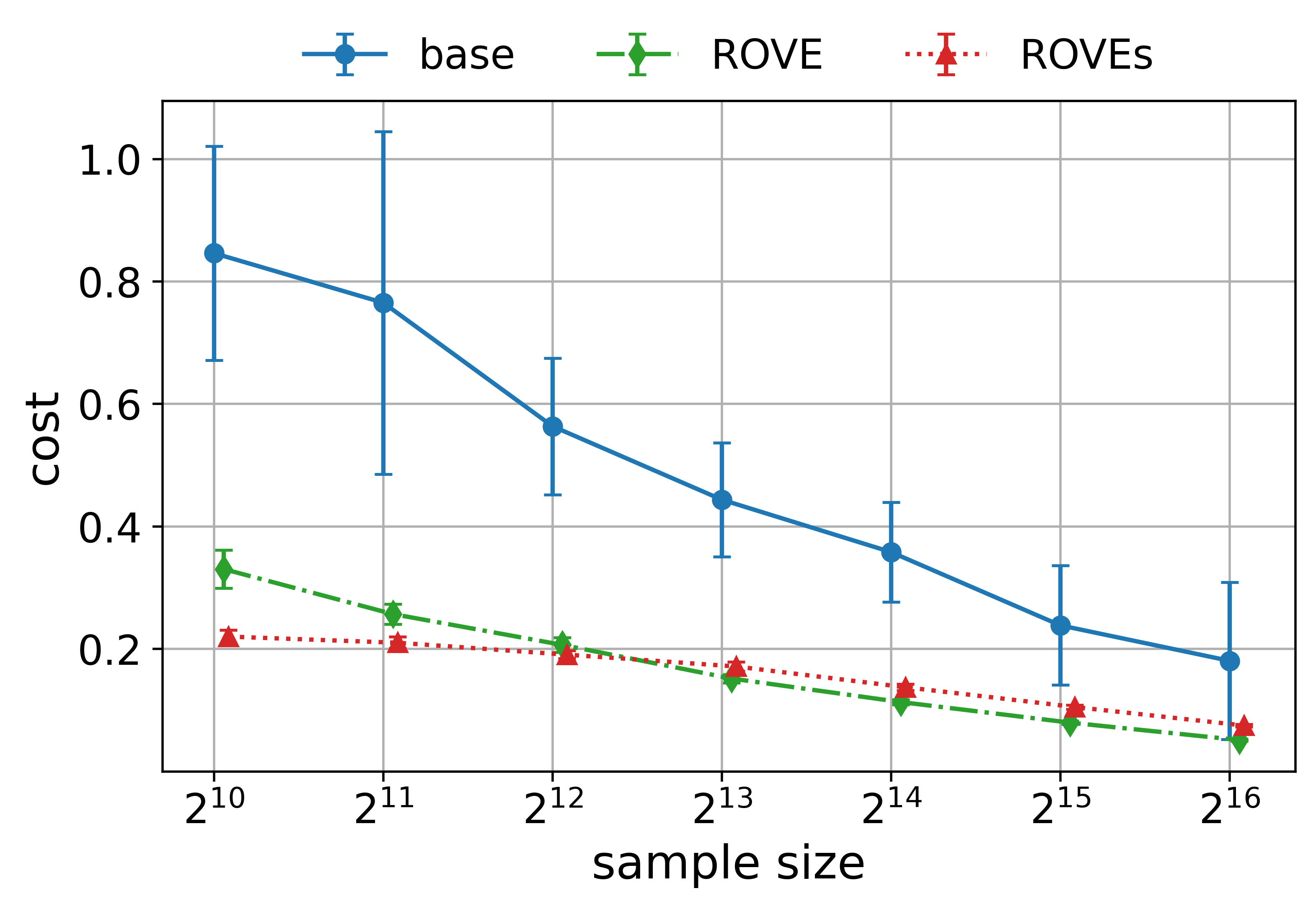}
\caption{Pareto noise, $H=8$.\label{subfig: MLP avg d=50 L=8}}
\end{subfigure}
\hspace{-5pt}
\begin{subfigure}{0.34\textwidth}
\includegraphics[width = \linewidth]{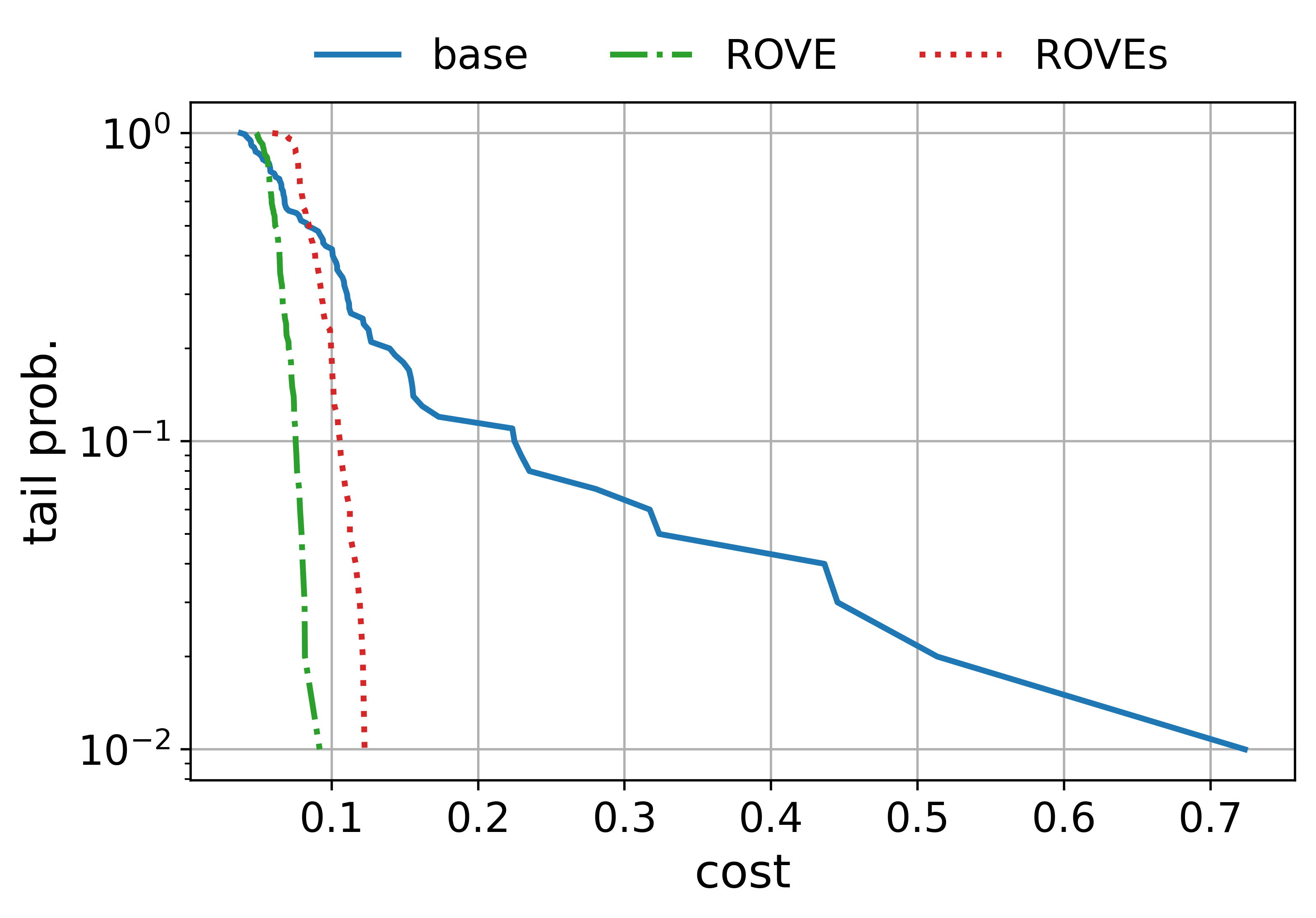}
\caption{Pareto noise, $H=4,n=2^{16}$.\label{subfig: MLP tail d=50 L=4 n=2^16}}
\end{subfigure}\\
\hspace{-13pt}
\begin{subfigure}{0.34\textwidth}
\includegraphics[width = \linewidth]{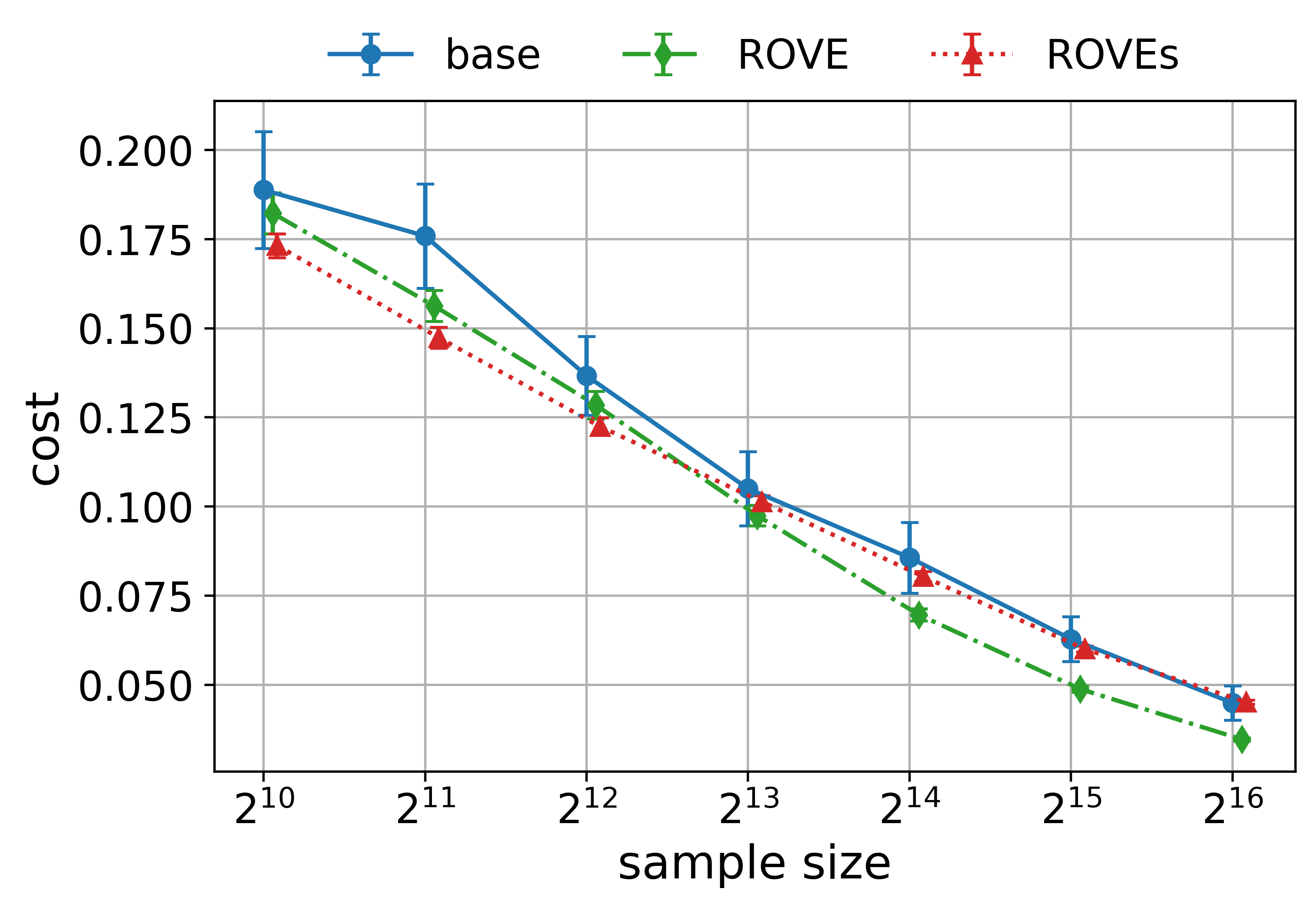}
\caption{Gaussian noise, $H=4$.\label{subfig: MLP avg d=50 L=4 light}}
\end{subfigure}
\hspace{-5pt}
\begin{subfigure}{0.34\textwidth}
\includegraphics[width = \linewidth]{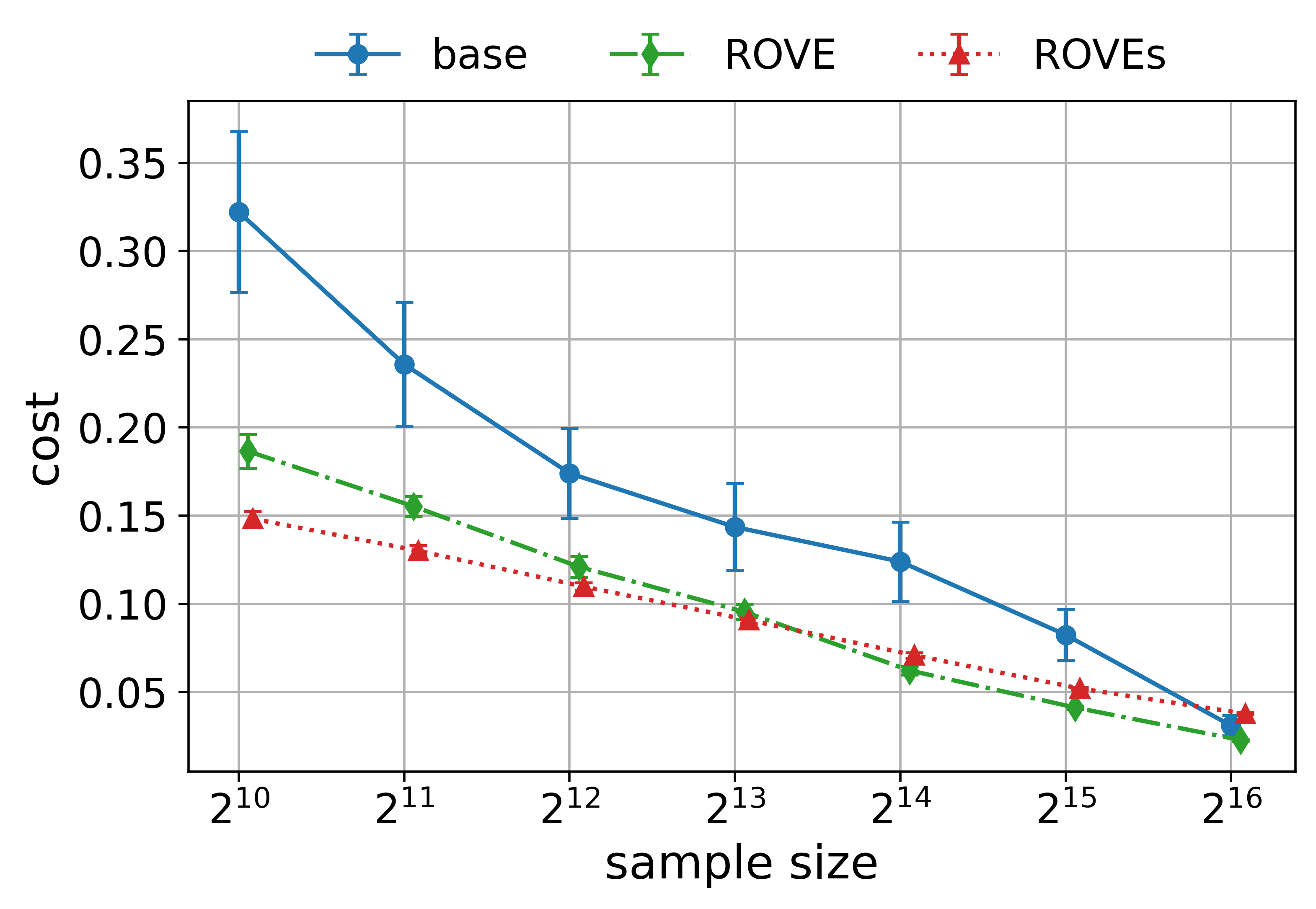}
\caption{Gaussian noise, $H=8$.\label{subfig: MLP avg d=50 L=8 light}}
\end{subfigure}
\hspace{-5pt}
\begin{subfigure}{0.34\textwidth}
\includegraphics[width = \linewidth]{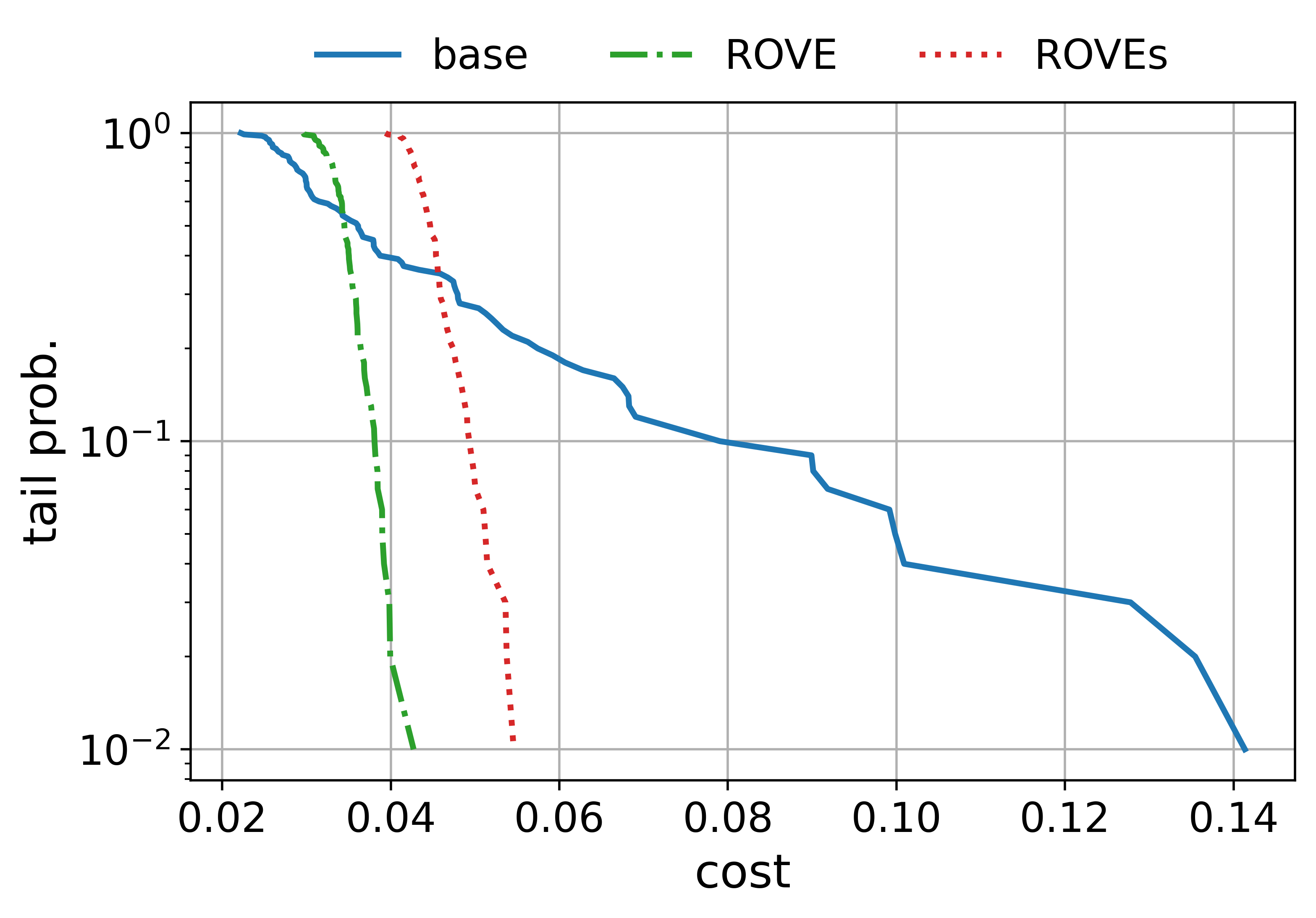}
\caption{Gaussian noise, $H=4,n=2^{16}$.\label{subfig: MLP tail d=50 L=4 n=2^16 light}}
\end{subfigure}
\caption{Results of neural networks. (a)(b)(d)(e): Expected out-of-sample costs (MSE) with $95\%$ confidence intervals under different noise distributions and varying numbers of hidden layers ($H$). (c) and (f): Tail probabilities of out-of-sample costs.
\label{fig: alg_comparison_MLP}}
\end{figure*}

\begin{figure*}[h]
\centering
\hspace{-8pt}
\begin{subfigure}{0.33\textwidth}
\includegraphics[width = \linewidth]{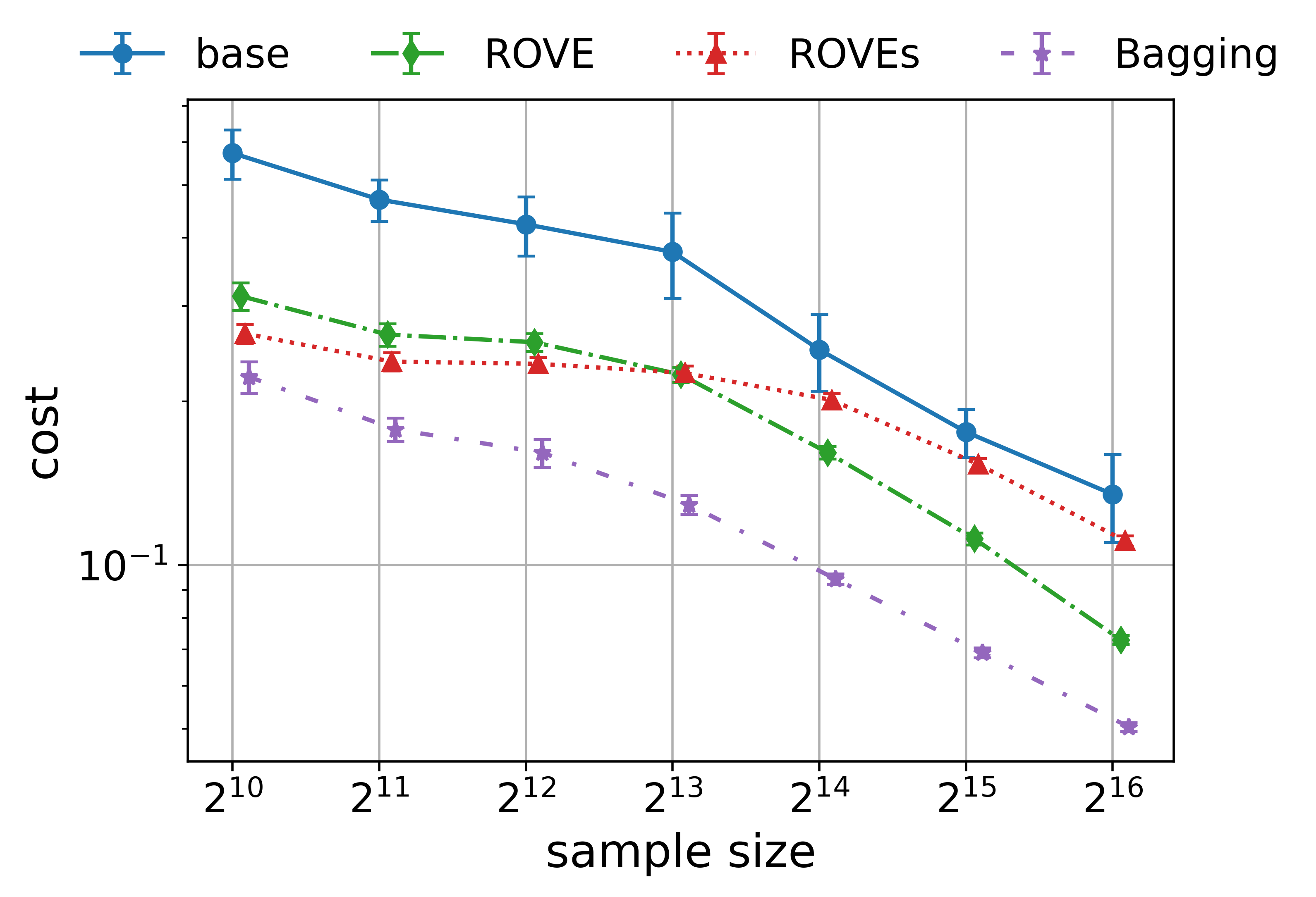}
\caption{Pareto shape $=2.1$.\label{subfig: comparison with bagging pareto2.1}}
\end{subfigure}
\begin{subfigure}{0.33\textwidth}
\includegraphics[width = \linewidth]{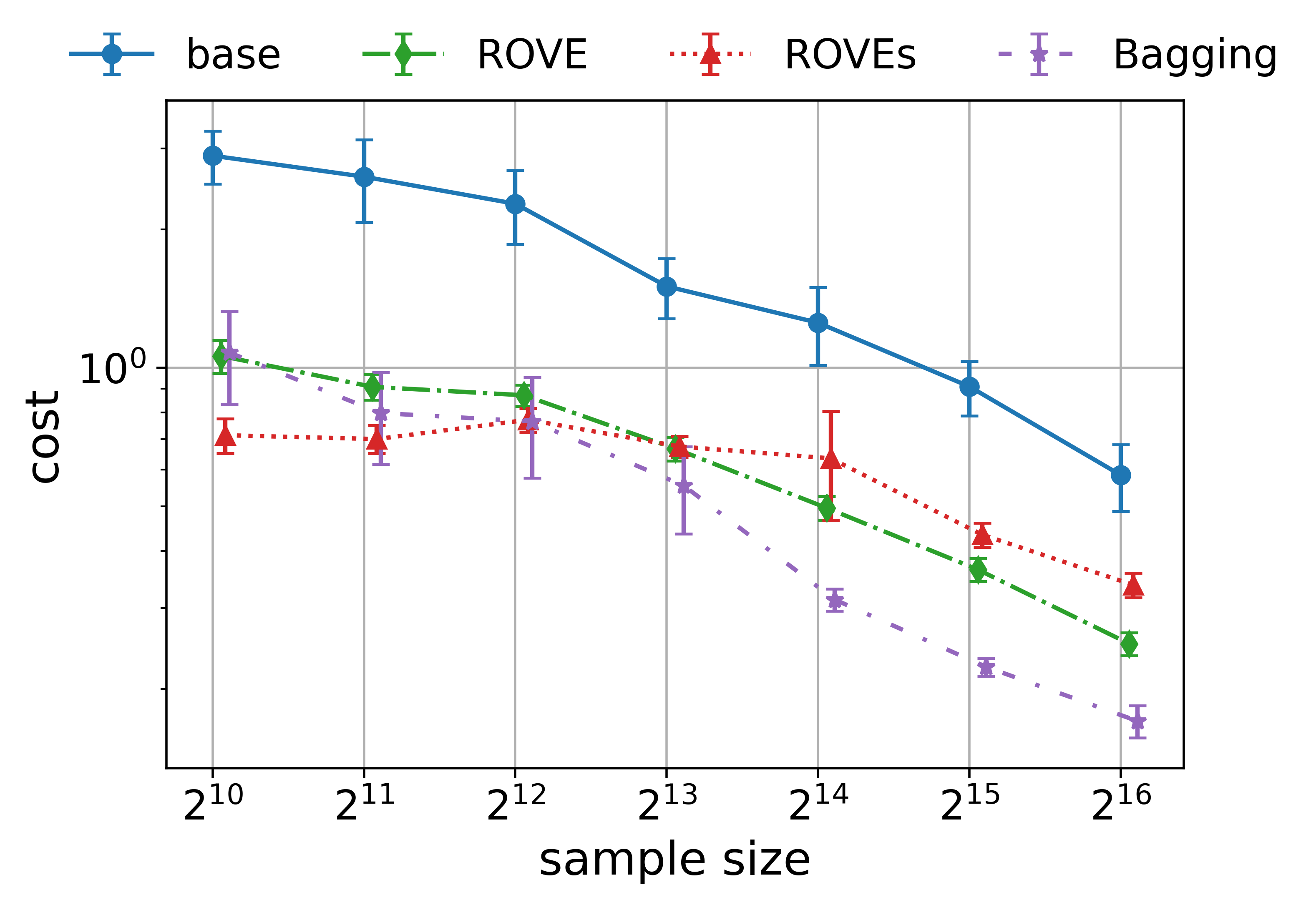}
\caption{Pareto shape $=1.5$.\label{subfig: comparison with bagging pareto1.5}}
\end{subfigure}
\begin{subfigure}{0.33\textwidth}
\includegraphics[width = \linewidth]{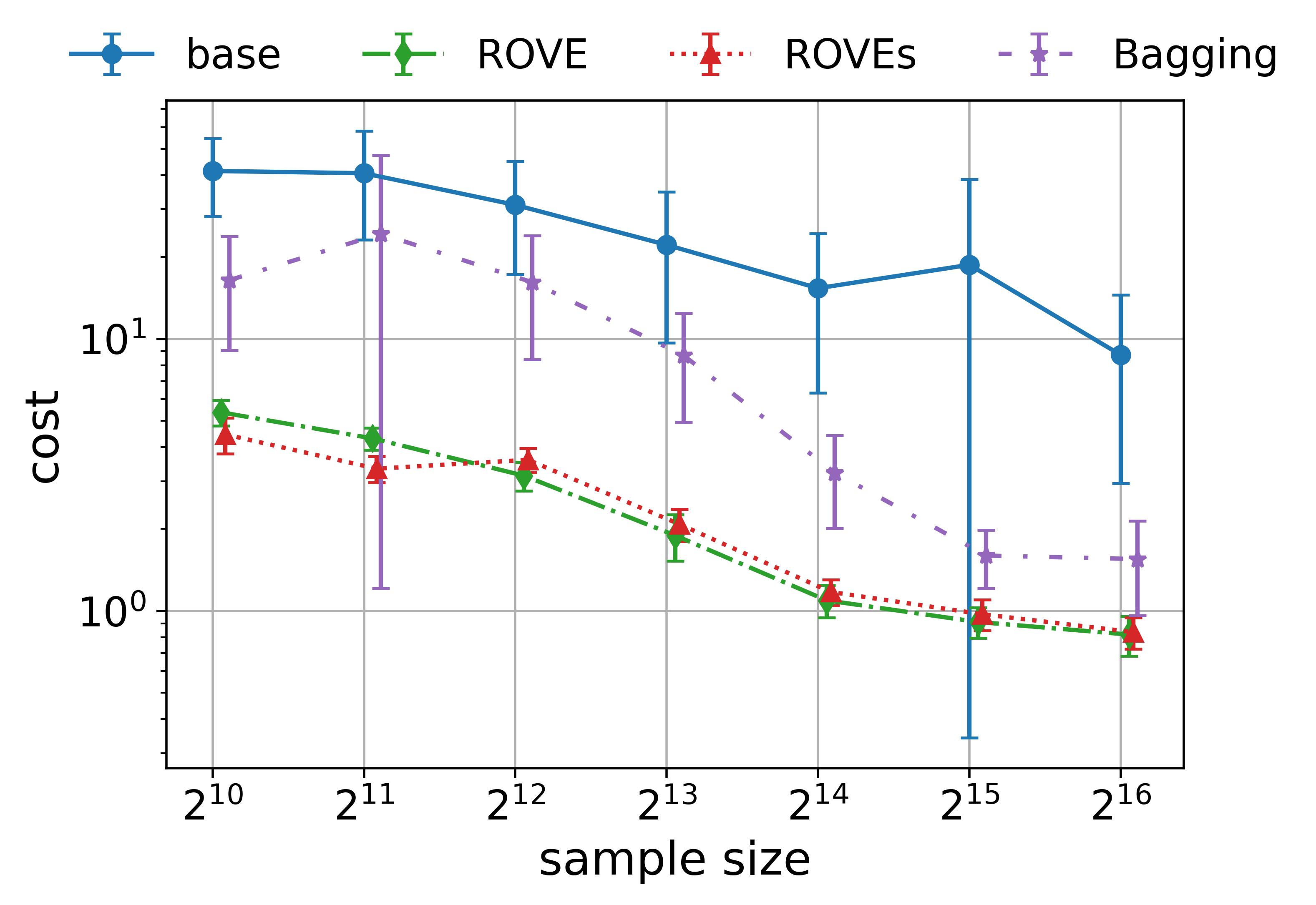}
\caption{Pareto shape $=1.1$.\label{subfig: comparison with bagging pareto1.1}}
\end{subfigure}\\
\hspace{-8pt}
\begin{subfigure}{0.33\textwidth}
\includegraphics[width = \linewidth]{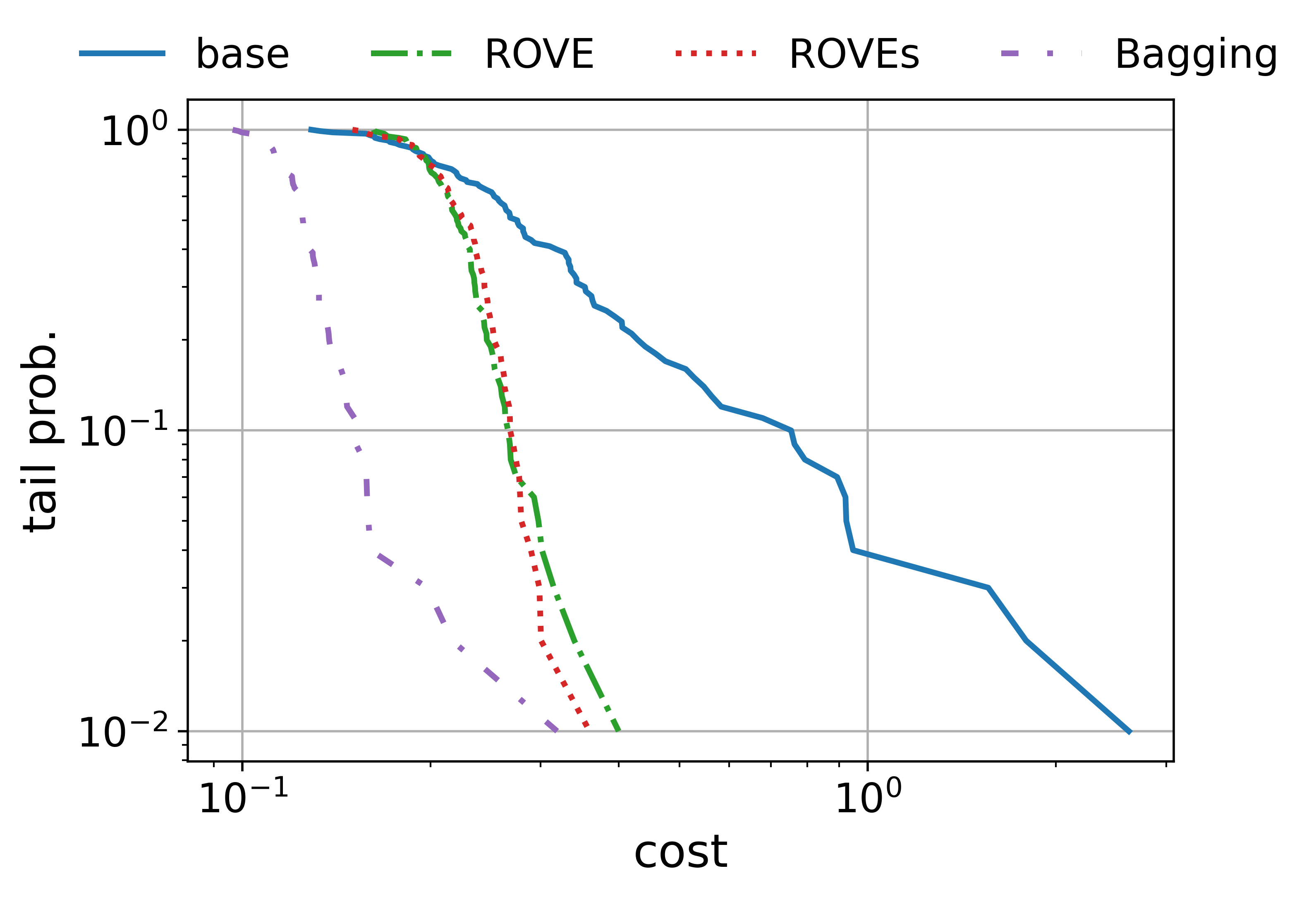}
\caption{Pareto shape $=2.1,n=2^{13}$.\label{subfig: tail comparison with bagging pareto2.1}}
\end{subfigure}
\begin{subfigure}{0.33\textwidth}
\includegraphics[width = \linewidth]{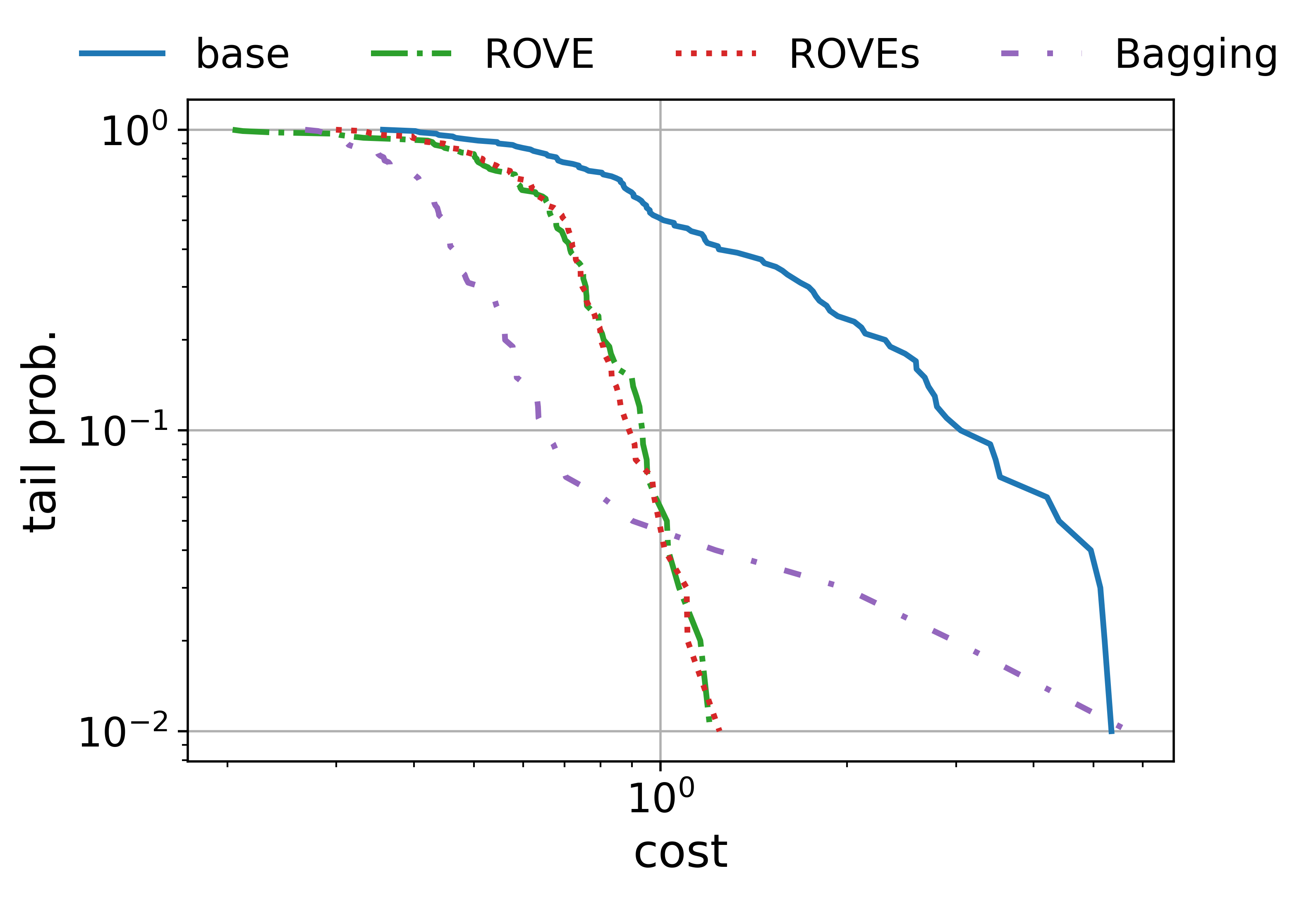}
\caption{Pareto shape $=1.5,n=2^{13}$.\label{subfig: tail comparison with bagging pareto1.5}}
\end{subfigure}
\begin{subfigure}{0.33\textwidth}
\includegraphics[width = \linewidth]{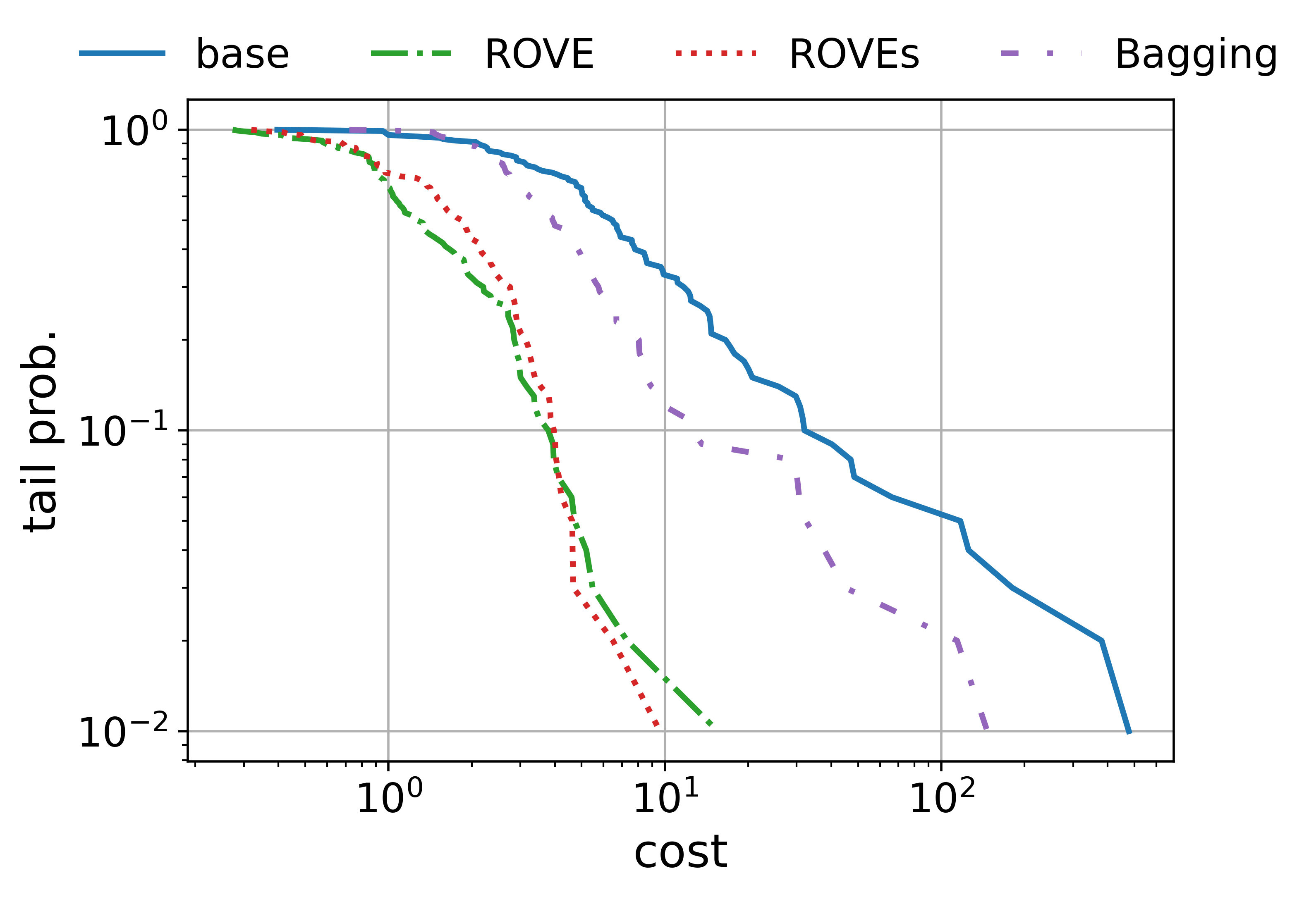}
\caption{Pareto shape $=1.1,n=2^{13}$.\label{subfig: tail comparison with bagging pareto1.1}}
\end{subfigure}
\caption{Comparison with bagging in terms of expected out-of-sample costs (MSE) with $95\%$ confidence intervals (a-c) or tail probabilities (d-f) under varying degrees of tail heaviness. Hyperparameters: $k_1 = \max(30, n/2), k_2 = \max(30, n/1000), B_1 = 50, B_2 = 1000$. \label{fig: MLP L=4 comparison with bagging}}
\end{figure*}

\begin{figure*}[h]
\centering
\hspace{-12pt}
\begin{subfigure}{0.33\textwidth}
\includegraphics[width = \linewidth]{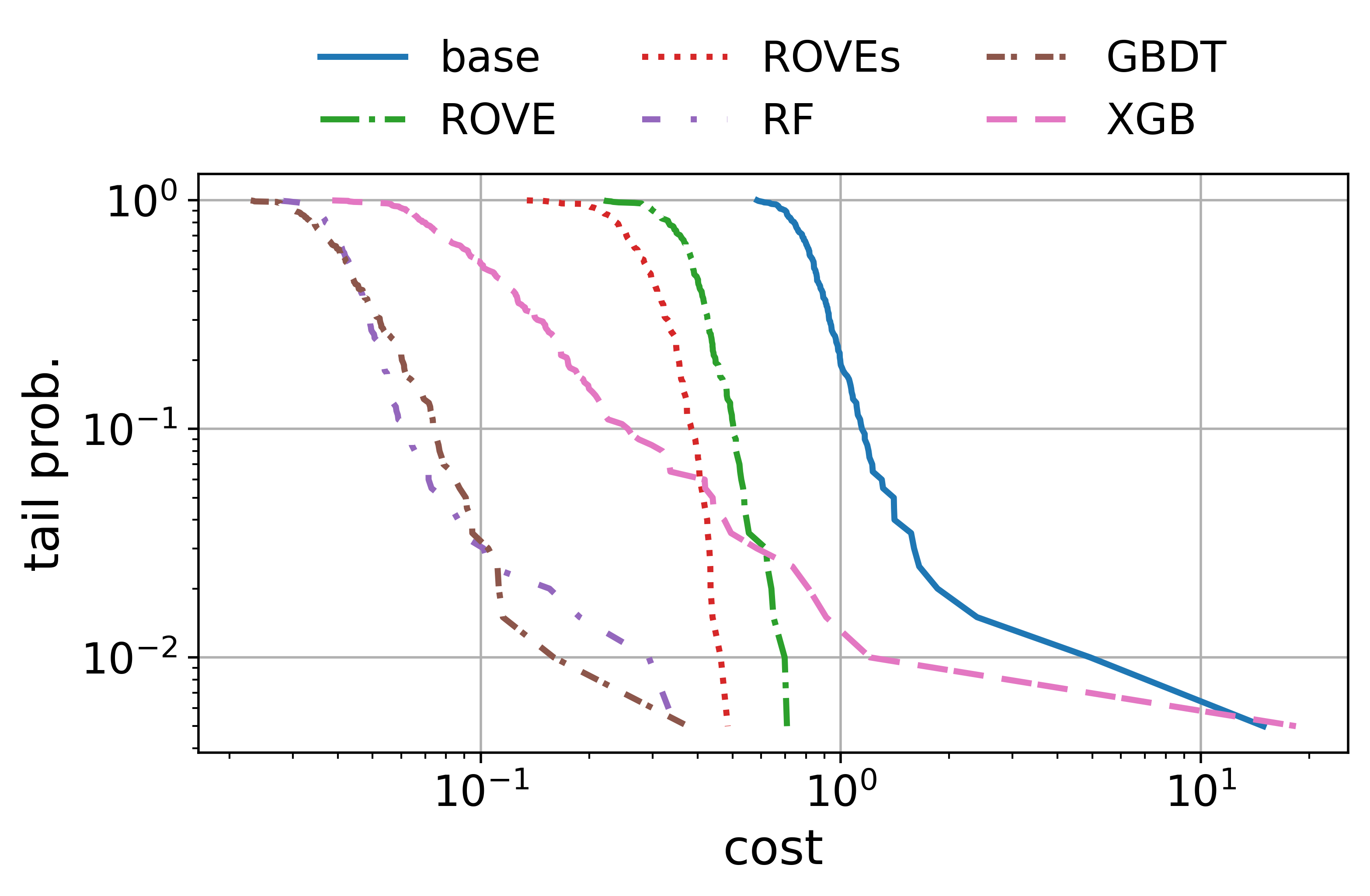}
\caption{Pareto shape = $2.5,n=2^{11}$.\label{subfig: tree methods tail shape 2.5}}
\end{subfigure}
\hspace{-6pt}
\begin{subfigure}{0.33\textwidth}
\includegraphics[width = \linewidth]{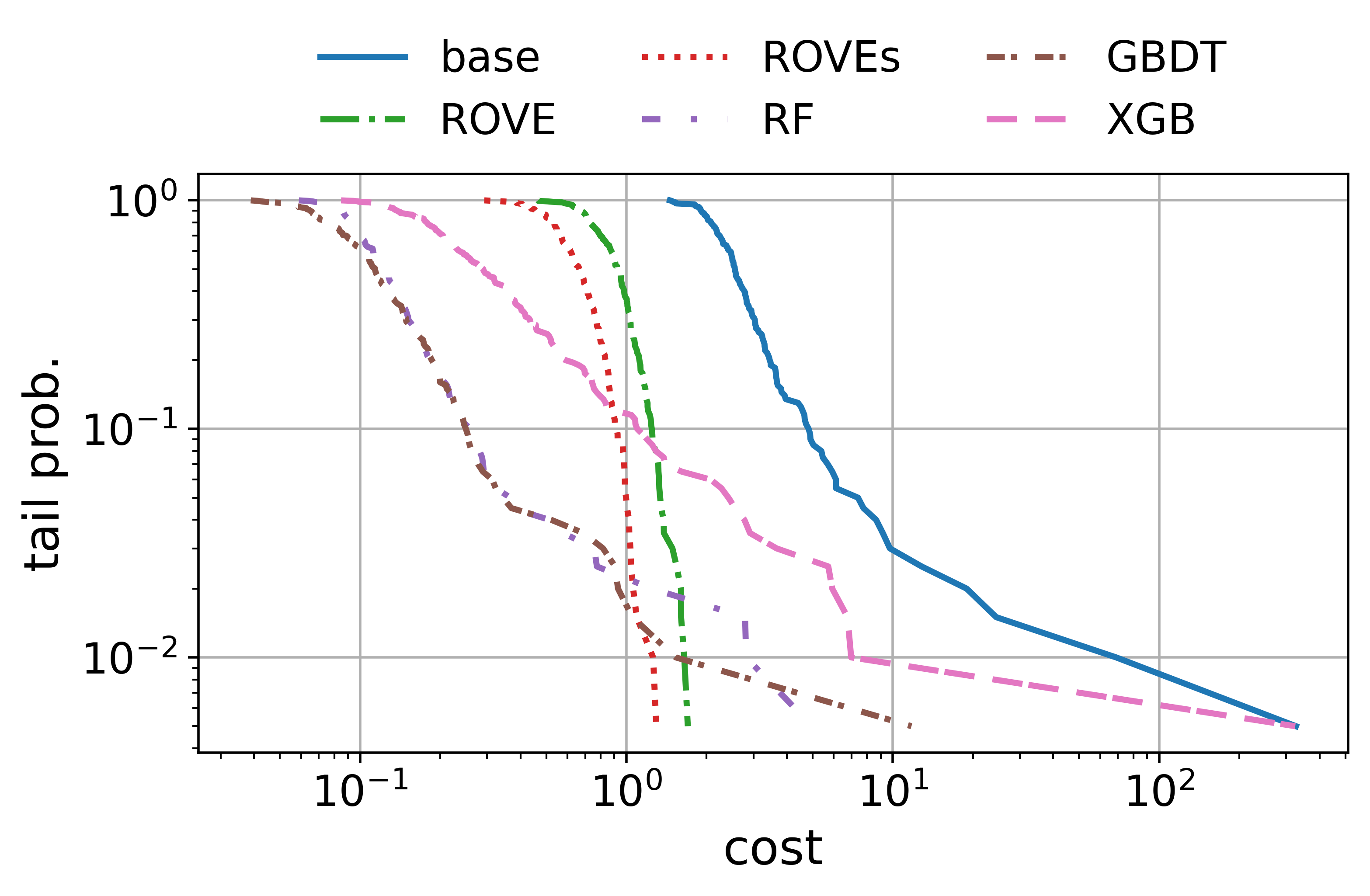}
\caption{Pareto shape = $2.0,n=2^{11}$.\label{subfig: tree methods tail shape 2.0}}
\end{subfigure}
\hspace{-6pt}
\begin{subfigure}{0.33\textwidth}
\includegraphics[width = \linewidth]{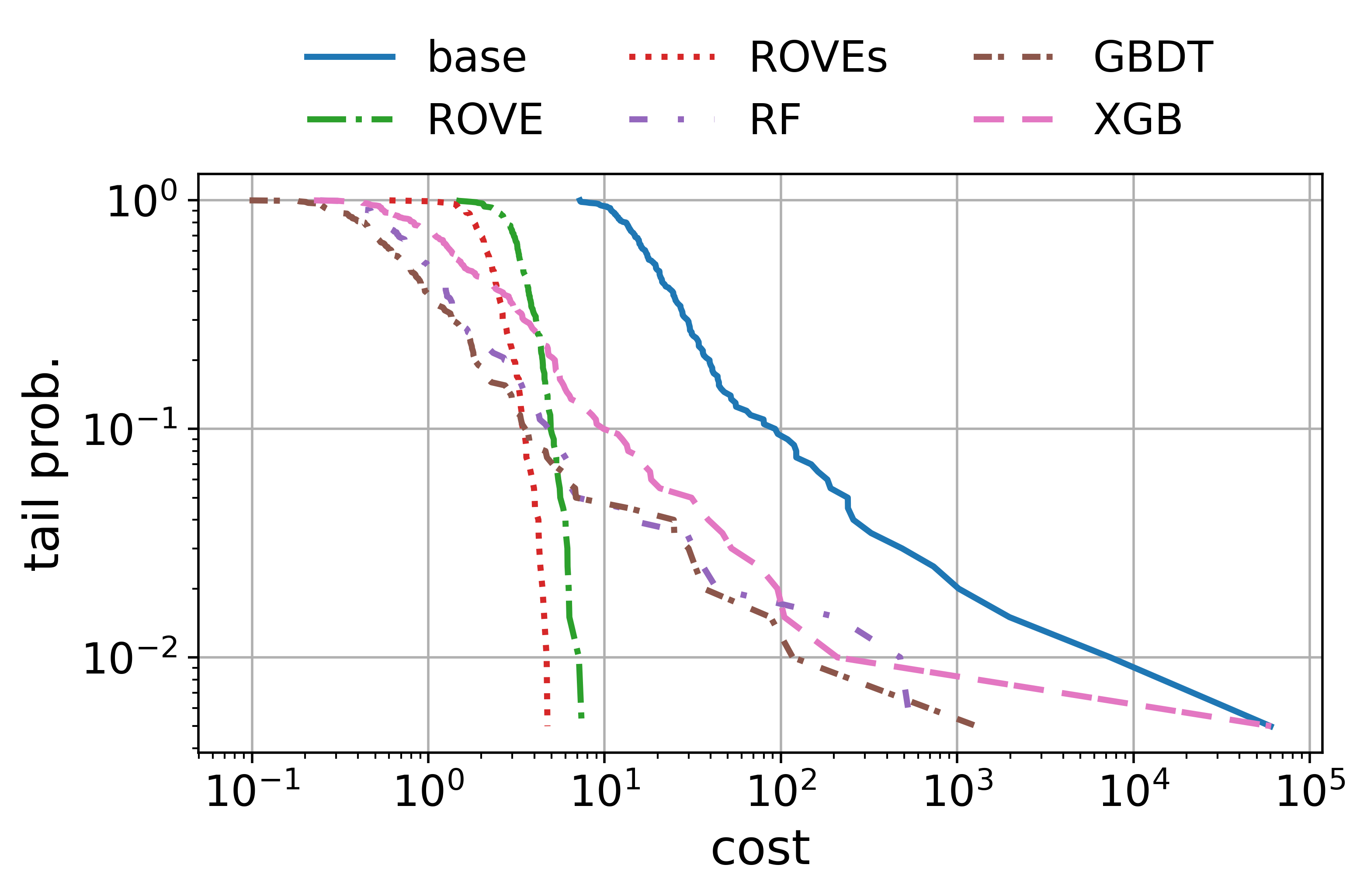}
\caption{Pareto shape = $1.5,n=2^{11}$.\label{subfig: tree methods tail shape 1.5}}
\end{subfigure}
\caption{Results of decision trees in terms of tail probabilities of out-of-sample costs (MSE). Hyperparameters: $k_1=\max(30,n/10),k_2=\max(30,n/200),B_1=50,B_2=200$.\label{fig: tree methods comparison main paper}}
\end{figure*}

\begin{figure*}[h]
\centering
\hspace{-12pt}
\begin{subfigure}{0.33\textwidth}
\includegraphics[width = \linewidth]{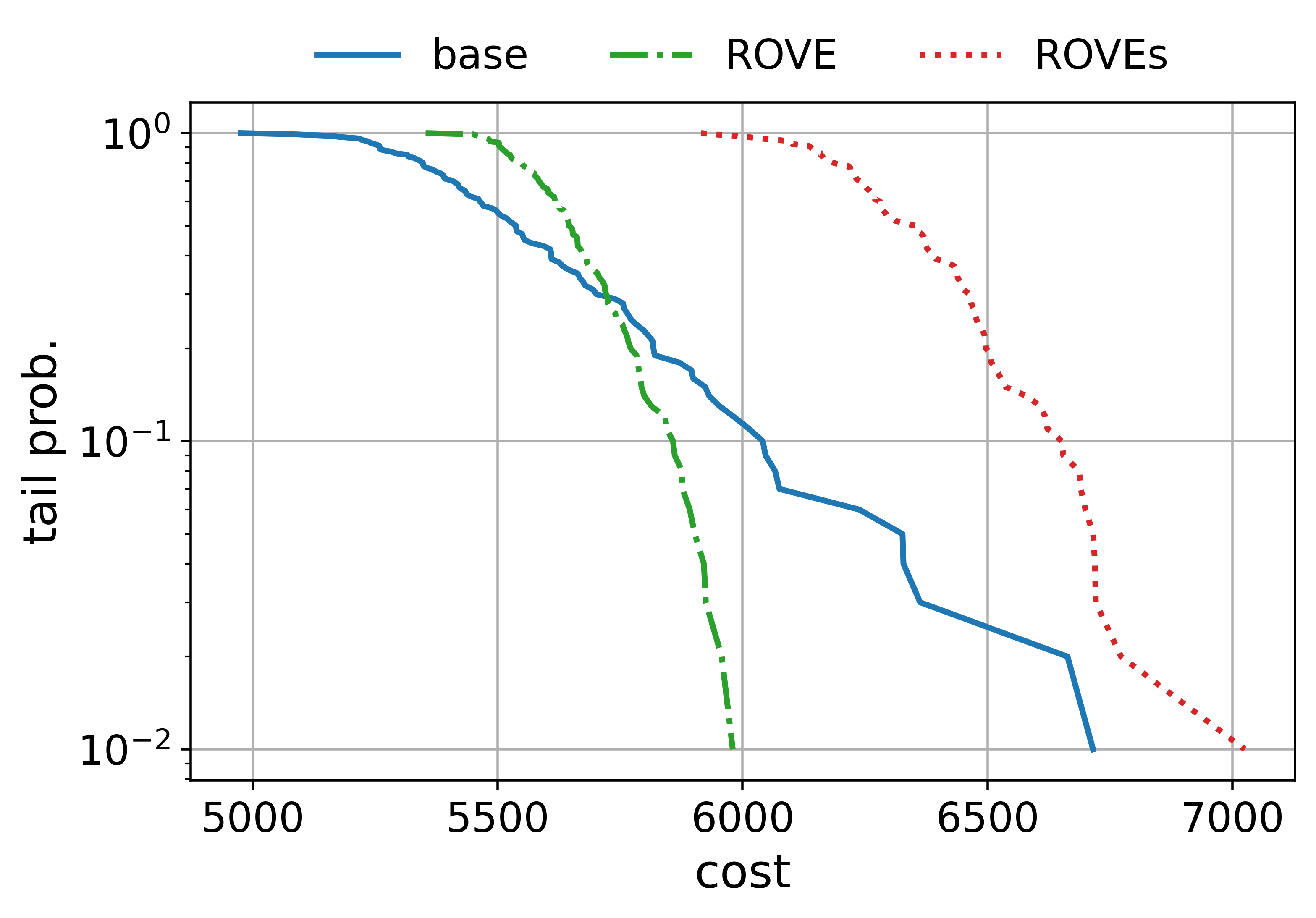}
\caption{\textit{Bike Sharing}.\label{subfig: bike sharing tail}}
\end{subfigure}
\hspace{-6pt}
\begin{subfigure}{0.33\textwidth}
\includegraphics[width = \linewidth]{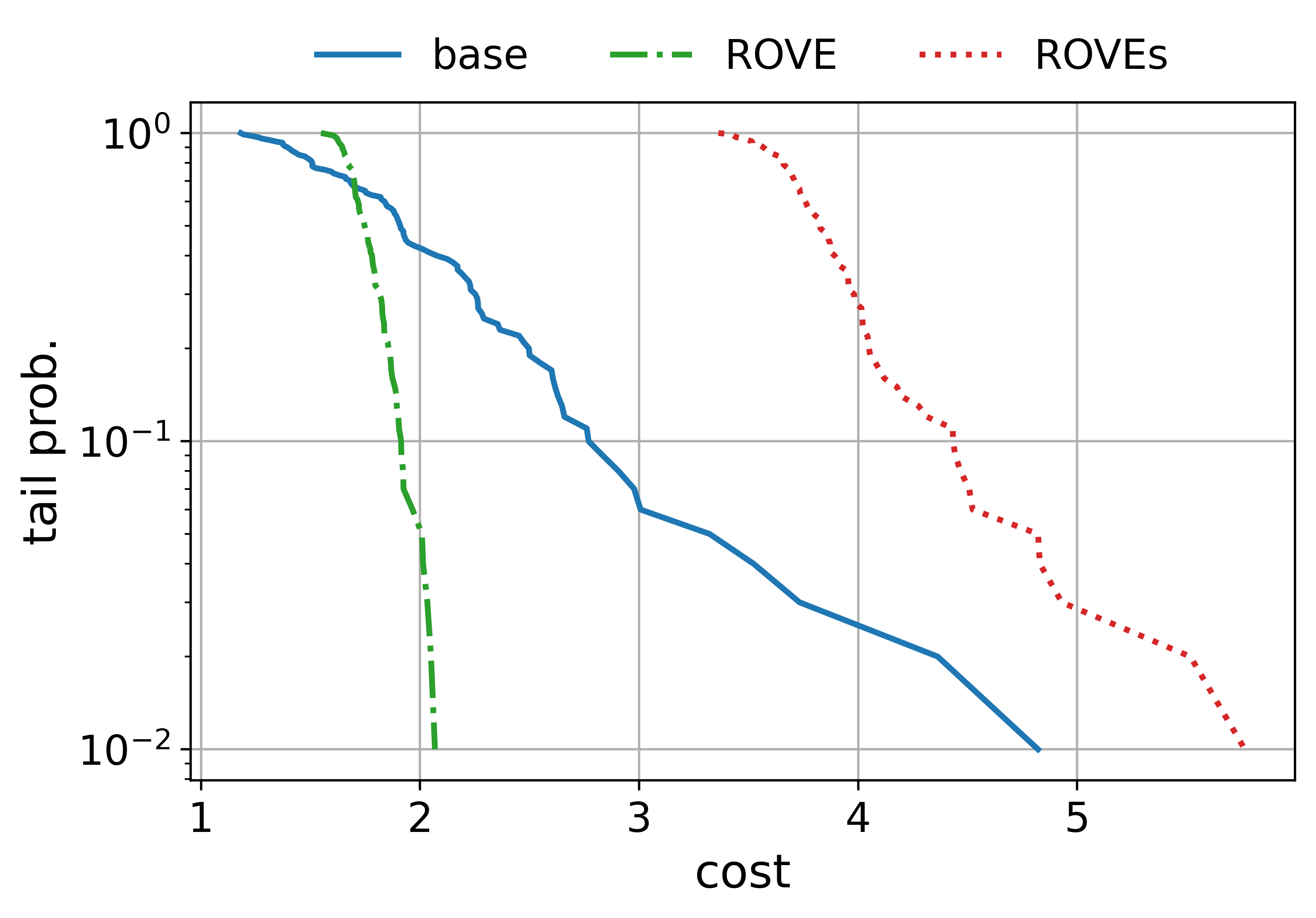}
\caption{\textit{Gas Turbine Emission}.\label{subfig: gas turbine tail}}
\end{subfigure}
\hspace{-6pt}
\begin{subfigure}{0.33\textwidth}
\includegraphics[width = \linewidth]{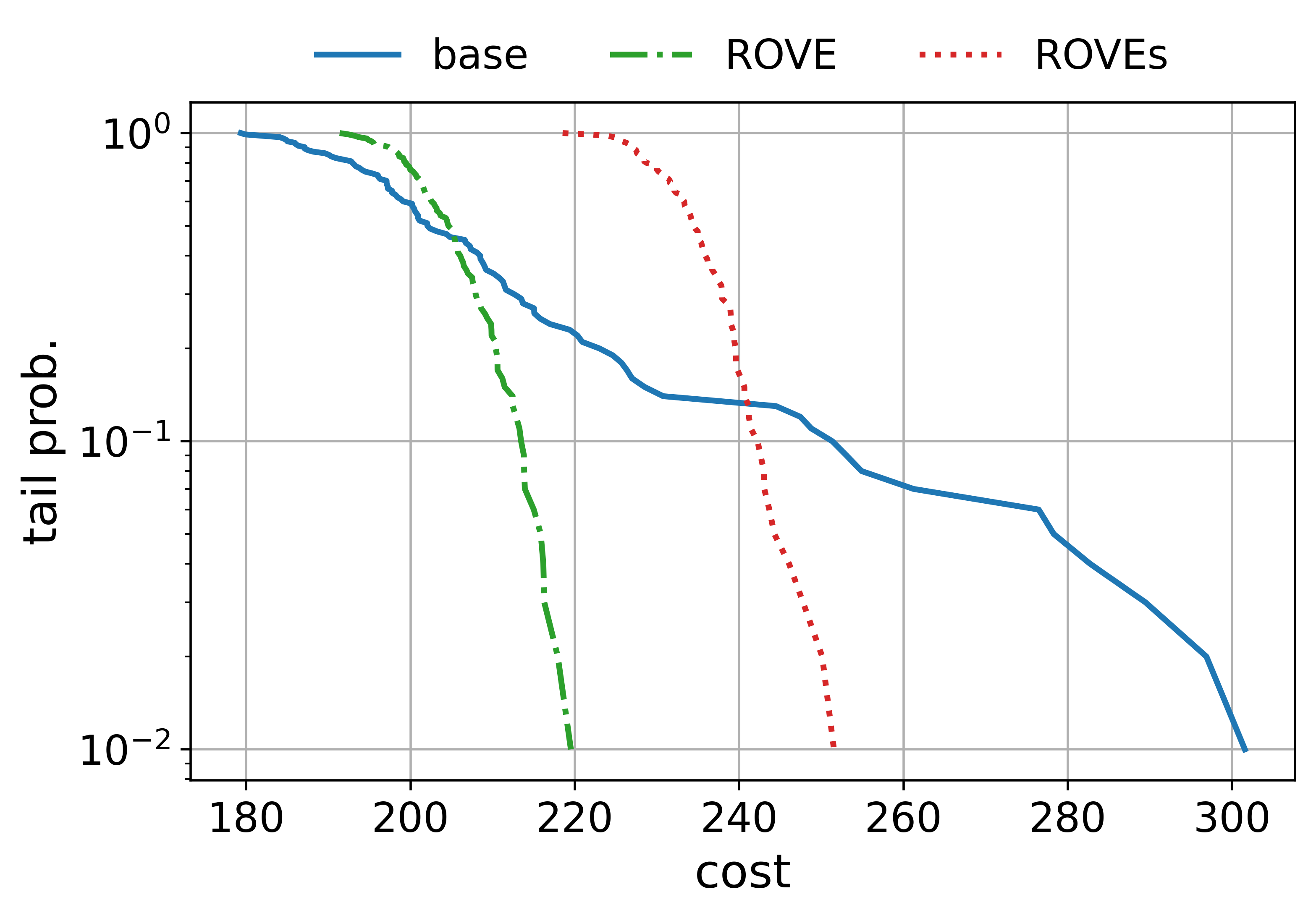}
\caption{\textit{Superconductivity}.\label{subfig: superconductivity tail}}
\end{subfigure}
\caption{Results of neural networks with $4$ hidden layers on three real datasets, in terms of tail probabilities of out-of-sample costs (MSE).\label{fig: MLP L=4 public data plots main paper}}
\end{figure*}

\textbf{Results.$\quad$}
As shown in Figure \ref{fig: alg_comparison_MLP}, in heavy-tailed noise settings (Figures 1a–1c), both \rove and \roves significantly outperform the base algorithm in terms of both expected out-of-sample MSE and tail performance under all sample sizes $n$. Notably, the performance improvement becomes more pronounced with deeper networks (\(H=8\)), indicating that the benefits of \rove and \roves are more apparent in models with higher expressiveness and lower bias.

In light-tailed settings (Figures 1d–1f), \rove and \roves show comparable expected out-of-sample performance to the base when \(H = 4\), but outperform the base as \(H\) increases. Additionally, \rove and \roves outperform the base in tail probabilities even when \(H = 4\). This indicates that \rove and \roves provide better generalization as the model complexity grows even for light-tailed problems. Similar results for MLPs with $2$ and $6$ hidden layers can be found in Appendix \ref{app: additional_figures}, where results on least squares regression and Ridge regression are also provided.

Figure \ref{fig: MLP L=4 comparison with bagging} shows a comparison with bagging that resembles our method most closely among existing ensemble methods as both involve repeated training on randomly drawn subsamples. We implement bagging, or subagging \citep{buhlmann2002analyzing} to be precise, on the MLP with $H=4$ hidden layers by averaging the predictions of the repeatedly trained MLPs. The same subsample size and ensemble size are used for our methods and bagging to ensure a fair comparison. Whether bagging or our method wins depends on the tail heaviness: \rove and \roves exhibit relatively inferior test performance when the noise has a shape of $2.1$, but outperform bagging as the tail of the noise gets heavier towards a shape of $1.1$.

Figure \ref{fig: tree methods comparison main paper} demonstrates a similar pattern for tree base learners: \rove and \roves outperform the base learner in all cases, and also outperform RF, GBDT, and XGB especially in high-end tails when the noise gets heavy-tailed with a Pareto shape of $1.5$. For not so heavy-tailed cases, RF, GBDT, and XGB may perform better. 

On real datasets (Figure \ref{fig: MLP L=4 public data plots main paper}), \rove exhibits much lighter tails compared to the base on three datasets, and similar tail behavior on the other three. \roves, however, underperforms the base in these real-world scenarios, potentially due to the data split that compromises its statistical power.

\subsection{Stochastic Programs}\label{subsec: alg_comparison}
% \paragraph{Setup.} 
\textbf{Setup.$\quad$}
We consider four discrete stochastic programs: resource allocation, supply chain network design, maximum weight matching, and stochastic linear programming, and continuous mean-variance portfolio optimization. All problems are designed to possess heavy-tailed uncertainties. The base learner for all the problems is the SAA. Details of the problems are deferred to Appendix \ref{app: prob_intro} and results with DRO being the base learner are provided in Appendix \ref{app: additional_figures} Figure \ref{fig: bagging_dro}.

% \paragraph{Results.} 
\textbf{Results.$\quad$}
Figure \ref{fig: alg_comparison_SAA} shows that our ensemble methods significantly outperform the base algorithm in all cases except for the linear program case (Figure \ref{subfig: saa_comparison_LP}). Notably, in the linear program case, \rove and \roves still outperform the base, demonstrating their robustness, while \move performs slightly worse than the base under small sample sizes. Comparing \rove and \roves, \rove consistently exhibits superior performance than \roves in all cases.

When there is a unique optimal solution, \move and \rove perform similarly, both generally better than \roves, as seen in Figures \ref{subfig: saa_comparison_sskp}-\ref{subfig: saa_comparison_matching}. However, in cases with multiple optima (Figure \ref{subfig: saa_comparison_LP}), the performance of \move deteriorates while \rove and \roves stay strong. This is in accordance with our discussion on the advantage of $\epsilon$-optimality vote in Section \ref{subsec:tow phase framework}. 
Additional results in Appendix \ref{app: additional_figures} shall further explain that optima multiplicity weakens the base learner for \move in the sense of decreasing the $\eta_{k,\delta}$ and hence inflating the tail bound in Theorem \ref{thm: general majority vote}.

The running time comparison in Figure \ref{subfig: runtime_network} shows that, despite requiring multiple runs on subsamples, our methods do not necessarily incur a higher computation cost compared to running base learner on the full sample, and can even be advantageous under large sample sizes. 
This is because, in problems like DRO \citep{ben2013robust, mohajerin2018data} and two-stage stochastic program, sample-based optimization often has a problem size that grows at least linearly with the sample size and induces a superlinearly growing computation cost. 
Subsampled optimizations, as our approach, are smaller and more manageable. In general, the computation efficiency of our method is ensured by the fact that no more than $\mc{O}(n/k)$ subsamples are needed as suggested by the theory and that training on subsamples can be easily parallelized.

% \paragraph{Recommended method.} 
\textbf{Recommended Method.$\quad$}
Among the three proposed ensemble methods, \rove is the preferred choice over \move and \roves for general use as it's applicable to both discrete and continuous problems and consistently delivers superior performance across all scenarios.

\begin{figure*}[h]
\centering
\hspace{-13pt}
\begin{subfigure}{0.34\textwidth}
\includegraphics[width = \linewidth]{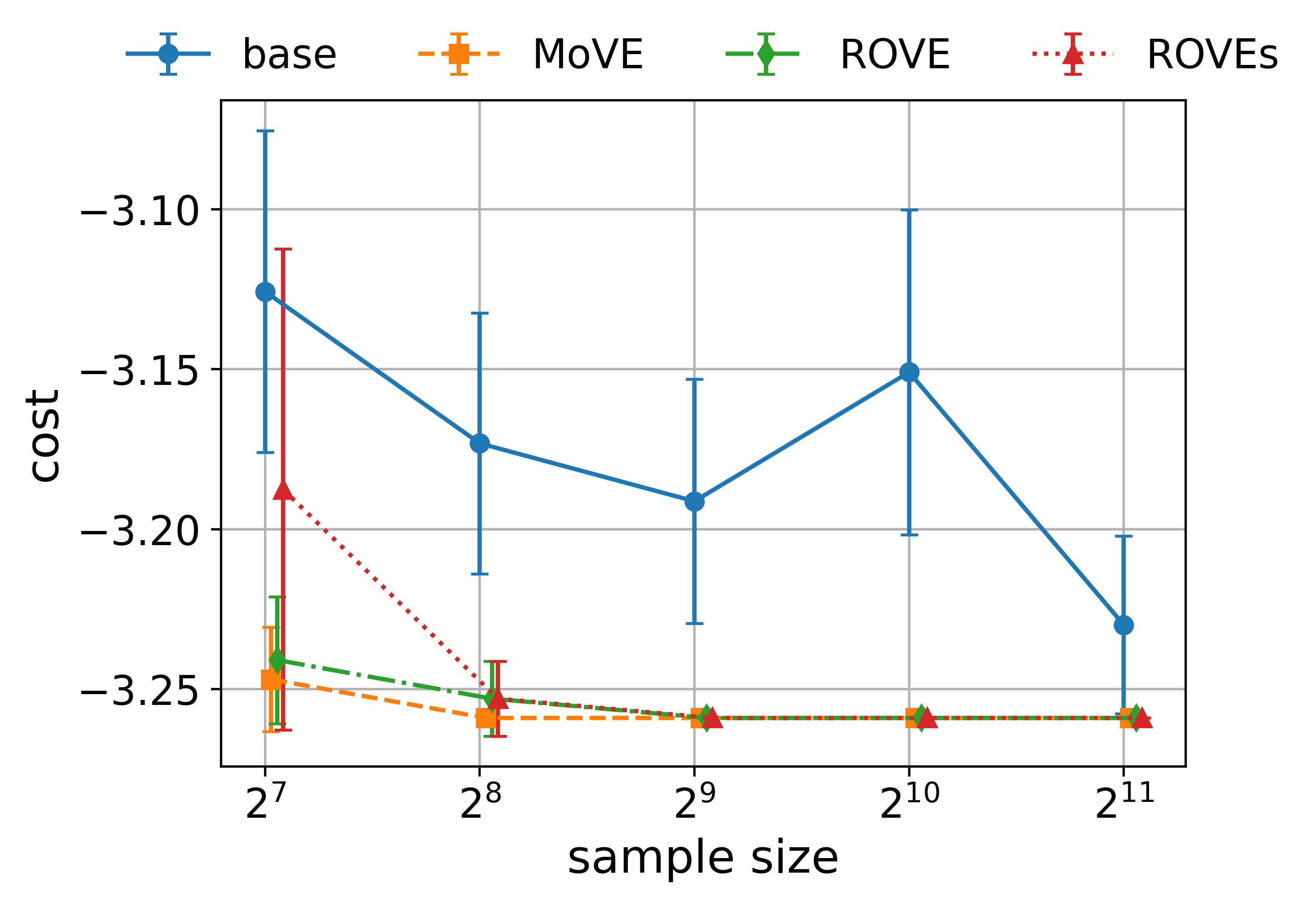}
\caption{Resource allocation.\label{subfig: saa_comparison_sskp}}
\end{subfigure}%
\hspace{-5pt}
\begin{subfigure}{0.338\textwidth}
\includegraphics[width = \linewidth]{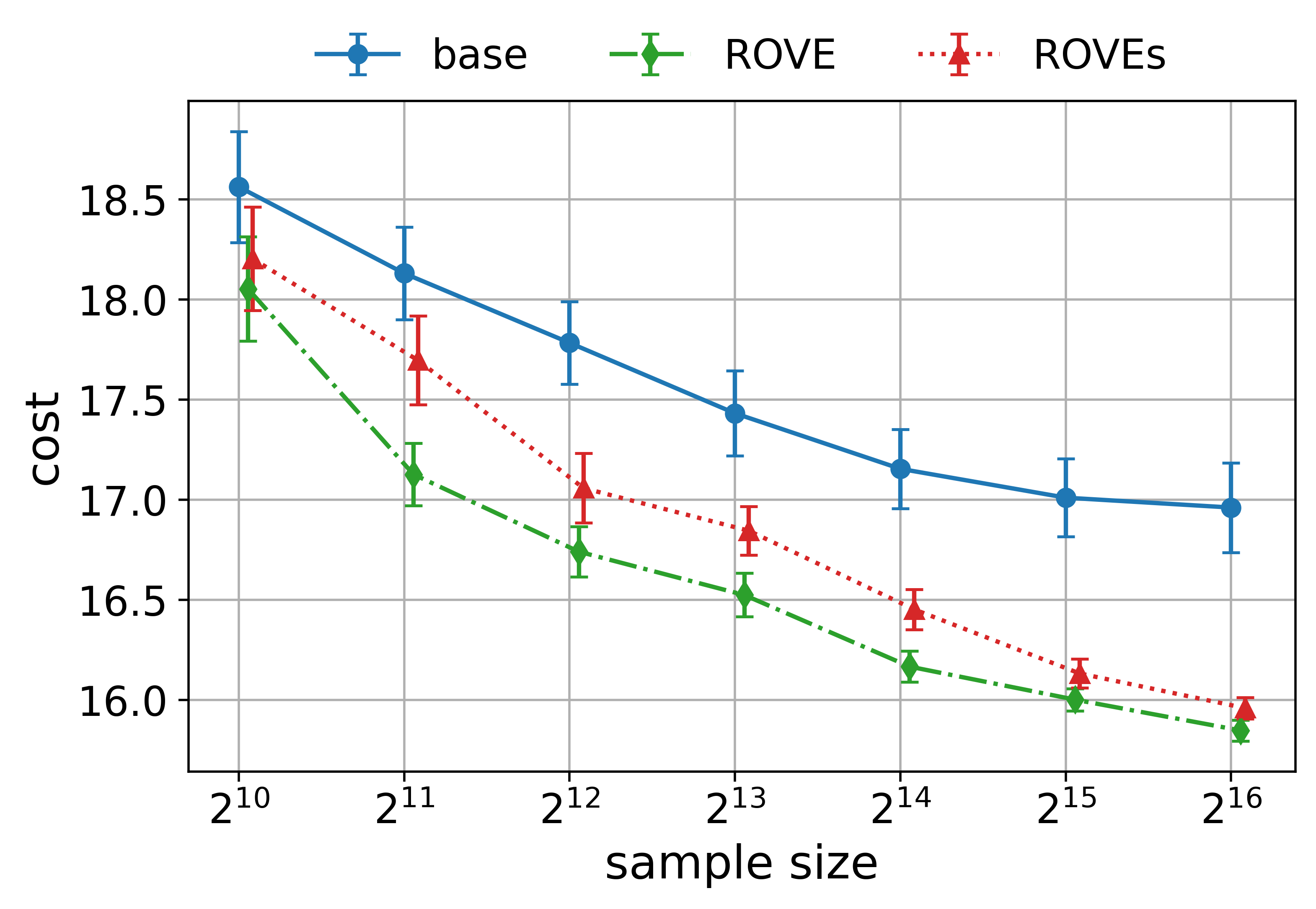}
\caption{Portfolio optimization.\label{subfig: saa_comparison_portfolio}}
\end{subfigure}
\hspace{-5pt}
\begin{subfigure}{0.345\textwidth}
\includegraphics[width = \linewidth]{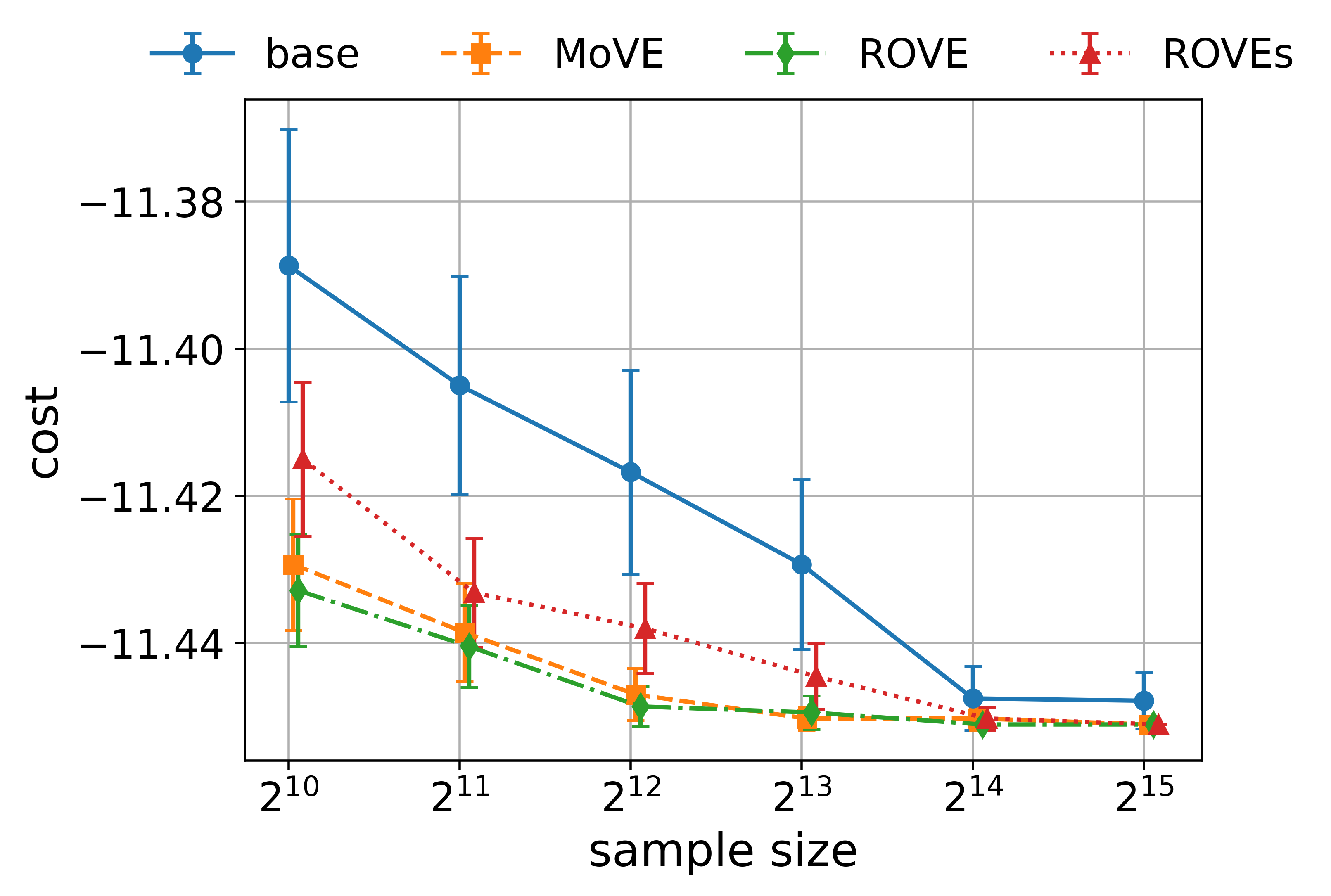}
\caption{Maximum weight matching.\label{subfig: saa_comparison_matching}}
\end{subfigure}\\
\hspace{-13pt}
\begin{subfigure}{0.34\textwidth}
\includegraphics[width = \linewidth]{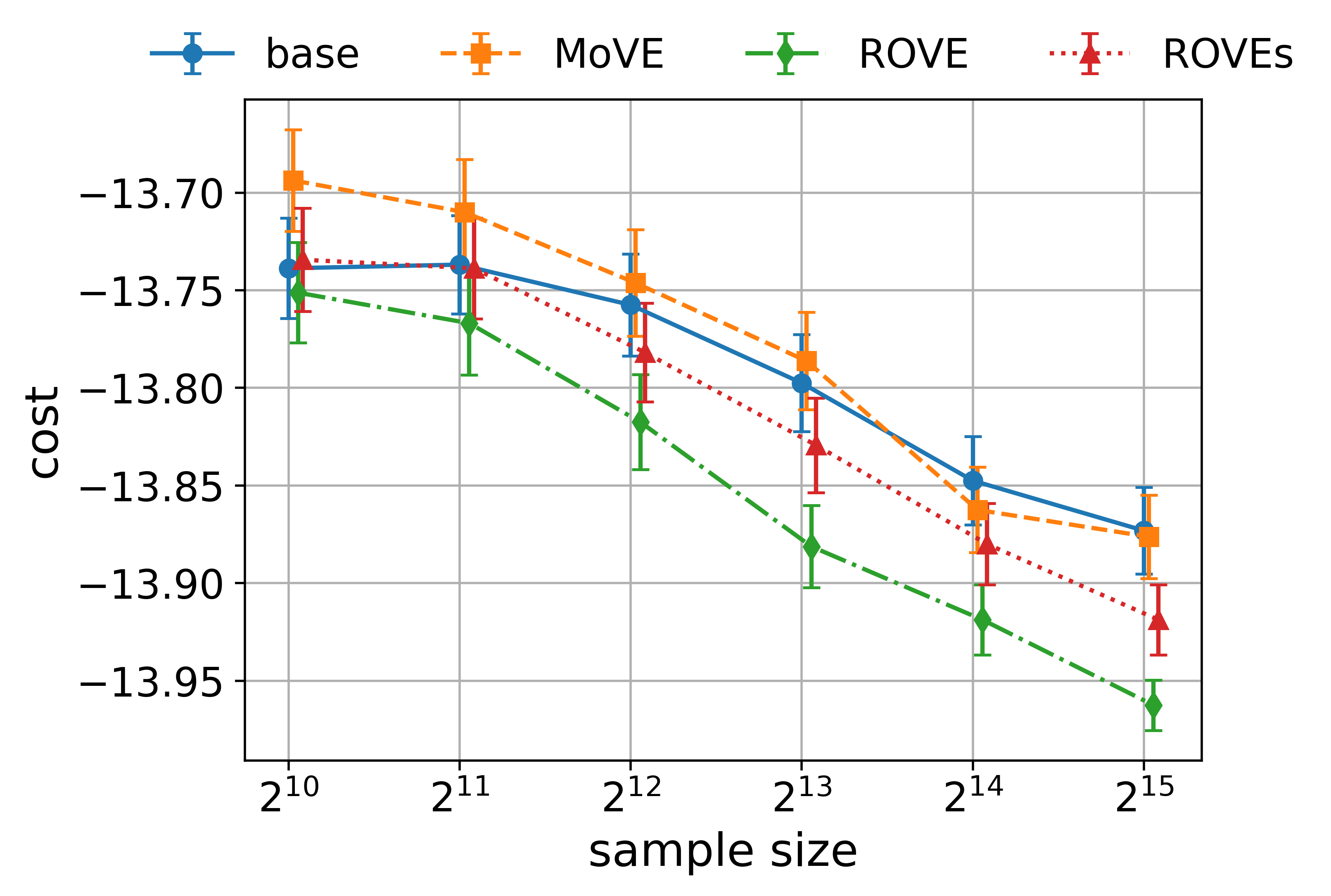}
\caption{Linear program (multiple optima).\label{subfig: saa_comparison_LP}}
\end{subfigure}
\hspace{-5pt}
\begin{subfigure}{0.33\textwidth}
\includegraphics[width = \linewidth]{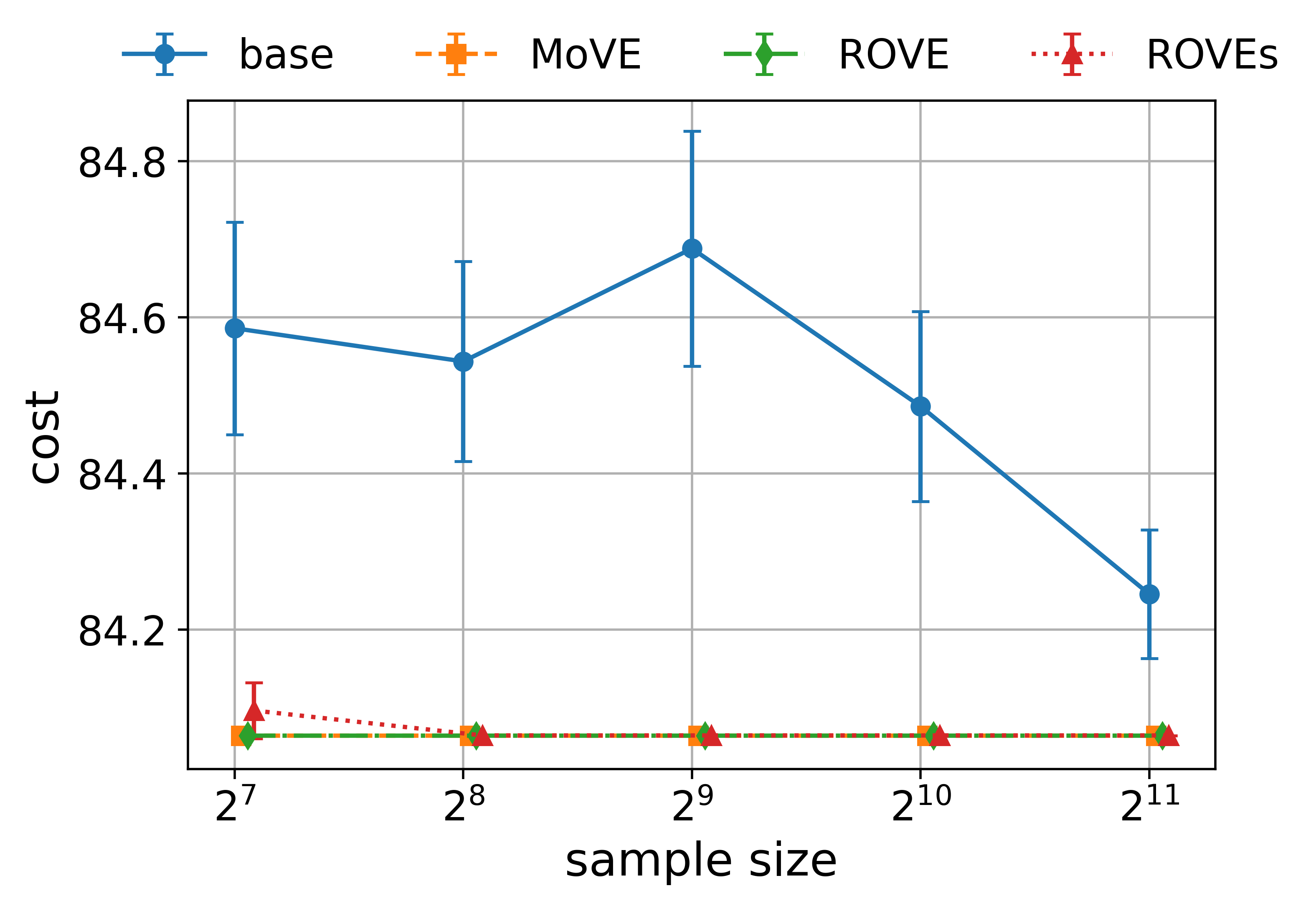}
\caption{Network design.\label{subfig: saa_comparison_network}}
\end{subfigure}%
\hspace{-5pt}
\begin{subfigure}{0.335\textwidth}
\includegraphics[width = \linewidth]{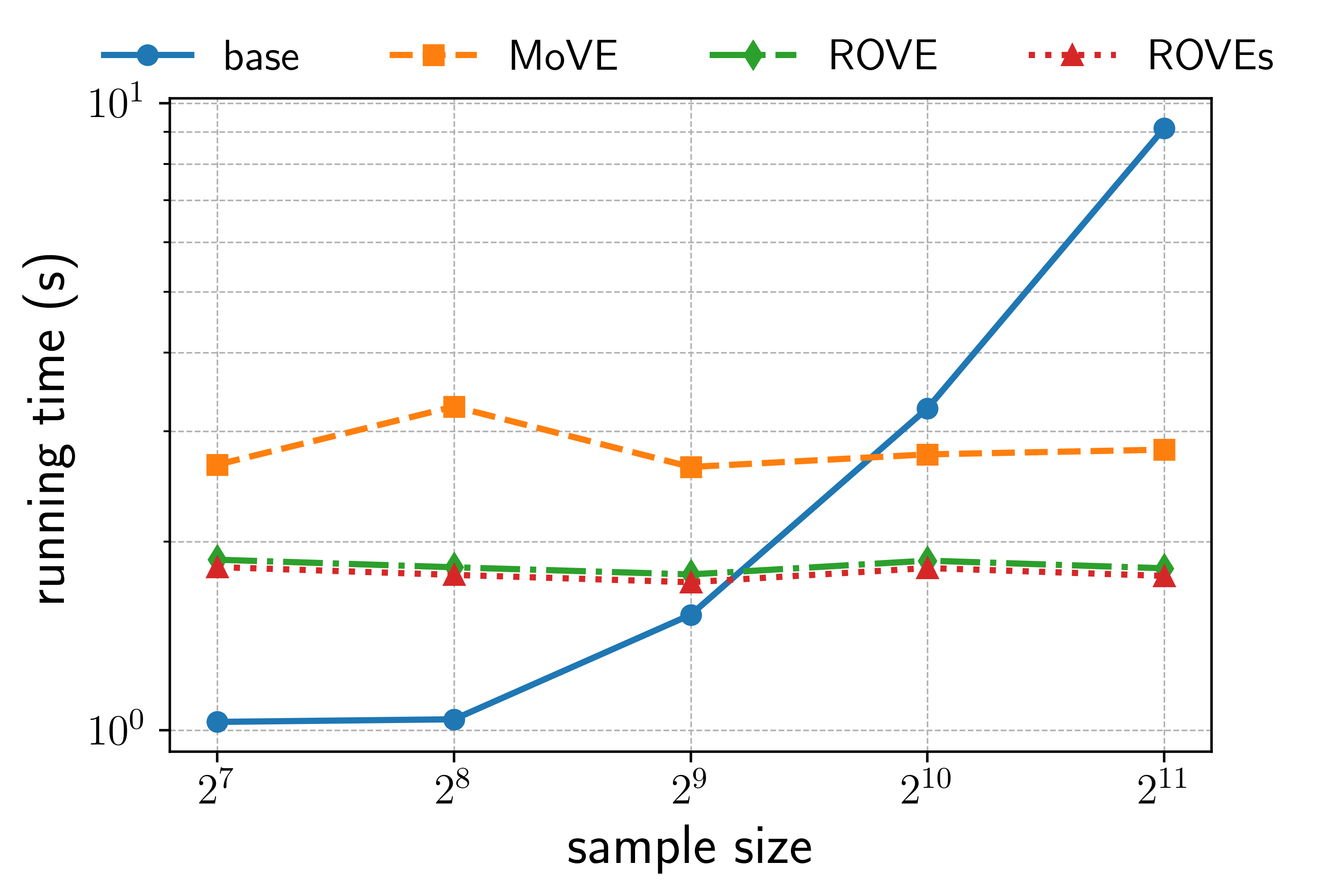}
\caption{Network design running time.\label{subfig: runtime_network}}
\end{subfigure}
\caption{Results for stochastic programs. (a)-(e): Expected out-of-sample costs with $95\%$ confidence intervals. 
(f): Running time comparison in the network design problem. 
\label{fig: alg_comparison_SAA}}
\end{figure*}

\section{Related Work}\label{sec: related work}
This work is closely connected to various topics in optimization and machine learning, and we only review the most relevant ones. 

\textbf{Ensemble Learning.$\quad$}
Ensemble learning \citep{dietterich2000ensemble,zhou2012ensemble,sagi2018ensemble} improves model performance by combining multiple weak learners into strong ones. Popular ensemble methods include bagging \citep{breiman1996bagging}, boosting \citep{freund1996experiments} and stacking \citep{wolpert1992stacked,dvzeroski2004combining}. Bagging enhances stability by training models on different bootstrap samples and combining their predictions through majority voting or averaging, effectively reducing variance, especially for unstable learners like decision trees that underpin random forests \citep{breiman2001random}. Subagging \citep{buhlmann2002analyzing} is a variant of bagging that constructs the ensemble from subsamples in place of bootstrap samples. Boosting is a sequential process where each subsequent model corrects its predecessors' errors, reducing both bias and variance \citep{ibragimov2019minimal,ghosal2020boosting}. Prominent boosting methods include AdaBoost \citep{freund2003efficient}, Stochastic Gradient Boosting (SGB) \citep{friedman2001greedy,friedman2002stochastic}, and Extreme Gradient Boosting (XGB) \citep{friedman2000additive} which differ in their approaches to weighting training data and hypotheses. Instead of using simple aggregation like weighted averaging or majority voting, stacking trains a model to combine base predictions to further improve performance. A key procedural difference of our approach from these methods is that we perform majority voting at the model rather than prediction level to select a single best model from the ensemble. That is, our approach outputs models in the same space as the base learner, whereas existing ensemble methods yield aggregated models outside the base space. This also means a constant inference cost for our output model with respect to the ensemble size, as opposed to linearly growing costs seen in existing ensemble methods. Methodologically, our approach operates by accelerating excess risk tail convergence in lieu of bias/variance reduction, and hence is particularly effective in settings with heavy-tailed noise.

\textbf{Optimization and Learning with Heavy Tails.$\quad$}
Optimization with heavy-tailed noises has garnered significant attention due to its relevance in traditional fields such as portfolio management \citep{mainik2015portfolio} and scheduling \citep{im2015stochastic}, as well as emerging domains like large language models \citep{brown2020language,achiam2023gpt}. Tail bounds of most existing algorithms are guaranteed to decay exponentially under sub-Gaussian or uniformly bounded costs but deteriorate to a slow polynomial decay under heavy-tailedness \citep{kavnkova2015thin,jiang2020rates,jiang2021complexity,oliveira2023sample}. 
For SAA or ERM, faster rates are possible under the small-ball \citep{mendelson2018learning,mendelson2015learning,roy2021empirical} or Bernstein's condition \citep{dinh2016fast} on the function class, while our approach is free from such conditions. 
Considerable effort has been made to mitigate the adverse effects of heavy-tailedness with robust procedures among which the geometric median \citep{minsker2015geometric}, or more generally, median-of-means (MOM) \citep{lugosi2019mean,lugosi2019sub} approach is most similar to ours. 
The basic idea there is to estimate a true mean by dividing the data into disjoint subsamples, computing an estimate on each, and then taking the median. \cite{lecue2019learning,lugosi2019risk,lecue2020robust,kwon2021mom} use MOM in estimating the expected cost and establish exponential tail bounds for the mean squared loss and convex function classes. \cite{hsu2016loss,hsu2014heavy} apply MOM directly on the solution level for continuous problems and require strong convexity from the cost to establish generalization bounds. Besides MOM, another approach estimates the expected cost via truncation \citep{catoni2012challenging} and allows heavy tails for linear regression \citep{audibert2011robust,zhang2018ell_1} or problems with uniformly bounded function classes \citep{brownlees2015empirical}, but is computationally intractable due to the truncation and thus more of theoretical interest. In contrast, our ensemble approach is a meta algorithm that provides exponential tails as long as the base learning algorithm possesses reasonble predictive performance as characterized in our Theorem \ref{thm: general majority vote}. Relatedly, various techniques such as gradient clipping \citep{cutkosky2021high,gorbunov2020stochastic} and MOM \citep{puchkin2024breaking} have been adopted in stochastic gradient descent (SGD) algorithms to handle heavy-tailed noises, but their focus is the faster convergence of SGD rather than generalization.

\textbf{Bagging for Stochastic Optimization.$\quad$}
Bagging has been adopted in stochastic optimization for various purposes. The most relevant line of works \cite{biggs2023constrained,perakis2021umotem,wang2021optimizing,biggs2023tightness} study mixed integer reformulations for stochastic optimization with bagging approximated objectives such as random forests and ensembles of neural networks with the ReLU activation. These works focus on computational tractability instead of generalization performance. 
\cite{anderson2020can} empirically evaluates several statistical techniques including bagging against the plain SAA and finds bagging advantageous for portfolio optimization problems. 
\cite{birge2023uses} investigates a batch mean approach for continuous optimization that creates subsamples by dividing the data set into non-overlapping batches instead of resampling and aggregates SAA solutions on the subsamples via averaging, which is empirically demonstrated to reduce solution errors for constrained and high-dimensional problems. Another related batch of works \cite{lam2018assessing,lam2018bounding,chen2024distributions,chen2023software,eichhorn2007stochastic} concern the use of bagging for constructing confidence bounds for generalization errors of data-driven solutions, but they do not attempt to improve generalization. Related to bagging, bootstrap has been utilized to quantify algorithmic uncertainties for randomized algorithms such as randomized least-squares algorithms \cite{lopes2018error}, randomized Newton methods \cite{chen2020estimating}, and stochastic gradient descent \cite{fang2018online,zhong2023online}, which is orthogonal to our focus on generalization performance.

\textbf{Machine Learning for Optimization.$\quad$}
Learning to optimize (L2O) studies the use of machine learning in accelerating existing or discovering novel optimization algorithms. Much effort has been in training models via supervised or reinforcement learning to make critical algorithmic decisions such as cut selection (e.g., \cite{deza2023machine,tang2020reinforcement}), search strategies (e.g., \cite{khalil2016learning,he2014learning,scavuzzo2022learning}), scaling \cite{berthold2021learning}, and primal heuristics \cite{shen2021learning} in mixed-integer optimization, or even directly generate high-quality solutions (e.g., neural combinatorial optimization pioneered by \cite{bello2016neural}). See \cite{chen2022learning,chen2024learning,bengio2021machine,zhang2023survey} for comprehensive surveys on L2O. This line of research is orthogonal to our goal, and L2O techniques can work as part of or directly serve as the base learning algorithm within our framework.

\section{Conclusion and Limitations}\label{sec: conclusion}
This paper introduces a novel ensemble technique that significantly improves generalization by estimating the mode of the sampling distribution of the base learner via subsampling. In particular, our approach converts polynomially decaying generalization tails into exponential decay, thus providing order-of-magnitude improvements as opposed to constant factor improvements exhibited by variance reduction. Extensive numerical experiments in both machine learning and stochastic programming validate its effectiveness, especially for scenarios with heavy-tailed data and slow convergence rates. This work underscores the powerful potential of our new ensemble approach across a broad range of machine learning applications.

Regarding limitation, our method may increase model bias like other subsampling-based techniques such as subagging \citep{buhlmann2002analyzing}, making it best suited for applications with relatively low bias, e.g., when the base learner is sufficiently expressive. Moreover, the tail guarantee of our method requires the mode of the sampling distribution of the base learner to be a reasonably good model.

\bibliography{neurips_2025}
\bibliographystyle{plain}

%%%%%%%%%%%%%%%%%%%%%%%%%%%%%%%%%%%%%%%%%%%%%%%%%%%%%%%%%%%%
\newpage
\appendix

\begin{center}
\Large\textbf{Supplemental Materials}
\end{center}

The appendices are organized as follows.
Appendix \ref{sec: additional technical results} presents additional technical discussion for Theorem \ref{thm: general majority vote}.
Next, in Appendix \ref{app: proofs}, we document the proofs of the main theoretical results in our paper.
Specifically, we introduce some preliminary definitions and lemmas in Appendix \ref{app: preliminaries}.
Then, the formal statement and the proof of Theorem \ref{thm: general majority vote} can be found in Appendix \ref{app: proof_theorem1_proposition1}. The proof of Corollary \ref{cor: application of move to linear program example} is in Appendix \ref{sec: proof of linear program example}. The formal statement and the proof of Theorem \ref{thm: finite-sample bound for multiple predictions two phase splitting} can be found in Appendix \ref{sec: theories for multiple predictions}.
To improve clarity, we defer the proofs for all technical lemmas to Appendix \ref{sec: proof of technical lemmas}.
In Appendix \ref{app: linear_regression}, we provide another motivating example that supplements Example \ref{ex: linear_program}.
Finally, we provide additional numerical experiments in Appendix \ref{app: additional_experiment}.

\section{Implications of Theorem \ref{thm: general majority vote} for Strong Base Learners}\label{sec: additional technical results}
We provide a brief discussion of Theorem \ref{thm: general majority vote} applied to fast convergent base learners. 
Based on Theorem \ref{thm: general majority vote}, the way $\max_{\theta\in\Theta}p_k(\theta)$ and $\max_{\theta\in\Theta/\Theta^{\delta}}p_k(\theta)$ enter into \eqref{eq: general finite-sample bound for large gap_informal} reflects how the generalization performance of the base learning algorithm is inherited by our framework. To explain, large $\max_{\theta\in\Theta}p_k(\theta)$ and small $\max_{\theta\in\Theta/\Theta^{\delta}}p_k(\theta)$ correspond to better generalization of the base learning algorithm. This can be exploited by the bound \eqref{eq: general finite-sample bound for large gap_informal} with the presence of $\max\set{1-\max_{\theta\in\Theta}p_k(\theta),\;\max_{\theta\in\Theta/\Theta^{\delta}}p_k(\theta)}$, which is captured with our sharper concentration of U-statistics with binary kernels. In particular, for base learning algorithms with fast generalization convergence, say $1-\max_{\theta\in\Theta}p_k(\theta)=\mc{O}(e^{-k})$ and $\max_{\theta\in\Theta/\Theta^{\delta}}p_k(\theta)=\mc{O}(e^{-k})$ for simplicity, we have $C_1\max\set{1-\max_{\theta\in\Theta}p_k(\theta),\;\max_{\theta\in\Theta/\Theta^{\delta}}p_k(\theta)}=\mc{O}(e^{-k})$ and hence the first term in \eqref{eq: general finite-sample bound for large gap_informal} becomes $\mc{O}(e^{-n})$ which matches the error of the base learning algorithm applied directly to the full data set.

\section{Proofs for Main Theoretical Results}\label{app: proofs}
\subsection{Preliminaries}\label{app: preliminaries}
An important tool in the development of our theories is the U-statistic that naturally arises in subsampling without replacement. We first present the definition of U-statistic below and its concentration properties in Lemma \ref{lem:MGF dominance}.
The proof of Lemma \ref{lem:MGF dominance} can be found in Appendix \ref{subsec: proof of lemma lem:MGF dominance}.

\begin{definition}\label{def: U-statistic}
    Given the i.i.d. data set $\{z_1,\ldots,z_n\}\subset \mathcal{Z}$ and a (not necessarily symmetric) kernel of order $k\leq n$ is a function $\kappa:\mathcal{Z}^k\to \R$ such that $\expect{\abs{\kappa(z_1,\ldots,z_k)}}<\infty$, the U-statistic associated with the kernel $\kappa$ is 
    \begin{equation*}
        U(z_1,\ldots,z_n)=\frac{1}{n(n-1)\cdots (n-k+1)}\sum_{1\leq i_1,i_2,\cdots,i_k\leq n\ \mathrm{s.t.}\ i_s\neq i_t\ \forall 1\leq s<t\leq k}\kappa(z_{i_1},\ldots,z_{i_k}).
    \end{equation*}
\end{definition}

% \begin{theorem}[Exponentially fast convergence]
%     Under certain conditions, if the decision space $\Theta$ is finite and discrete, regardless of potential heavy-tailedness of the uncertain cost $l(\theta,z)$, we have for Algorithm \ref{bagging majority vote} that
%     \begin{equation*}
%         P(\hat{\theta}^{BAG}_{n,k}\notin \argmin_{\theta\in\Theta}E[l(\theta,z)])\leq \vert \Theta \vert(2p_k^*e^{1-2p_k^*})^{\frac{n}{2k}}
%     \end{equation*}
%     where
%     \begin{equation*}
%         p_k^*:=\min_{\theta\in\argmin_{\theta\in\Theta}E[l(\theta,z)]}\max_{\theta'\notin \argmin_{\theta\in\Theta}E[l(\theta,z)]} P\left(\frac{1}{k}\sum_{i=1}^kl(\theta,z_i) \geq \frac{1}{k}\sum_{i=1}^kl(\theta',z_i)\right).
%     \end{equation*}
% \end{theorem}
% \textit{Proof. }Chernoff bound for U-statistics with binary kernel ? + Union bound.

\begin{lemma}[MGF dominance of U-statistics from \cite{hoeffding1963probability}]\label{lem:MGF dominance}
    For any integer $0<k\leq n$ and any kernel $\kappa(z_1,\ldots,z_k)$, let $U(z_1,\ldots,z_n)$ be the corresponding U-statistic defined in Definition \ref{def: U-statistic}, and
    \begin{equation}\label{sectioning approximation}
        \bar{\kappa}(z_1,\ldots,z_n)=\frac{1}{\lfloor n/k \rfloor}\sum_{i=1}^{\lfloor n/k \rfloor}\kappa(z_{k(i-1)+1},\ldots,z_{ki})
    \end{equation}
    be the average of the kernel across the first $\lfloor n/k \rfloor k$ data.
    Then, for every $t\in \R$, it holds that
    \begin{equation*}
        \expect{\exp(tU)}\leq \expect{\exp(t\bar{\kappa})}.
    \end{equation*}
\end{lemma}

Next, we present our sharper concentration bound for U-statistics with binary kernels.
The proof of Lemma \ref{lem:chernoff bound for p_k} can be found in Appendix \ref{subsec: proof of lemma lem:chernoff bound for p_k}.
\begin{lemma}[Concentration bound for U-statistics with binary kernels]\label{lem:chernoff bound for p_k}
% For each $\theta\in\Theta$ it holds that
% \begin{equation}
% \begin{aligned}
% &\prob{\hat{p}_k(\theta)-p_k(\theta) > \epsilon}\leq \exp\prth{-\dfrac{n}{k} \cdot D_{\operatorname{KL}}\prth{p_k(\theta)+\epsilon \| p_k(\theta)}}\\
% &\prob{\hat{p}_k(\theta)-p_k(\theta) < -\epsilon}\leq \exp\prth{-\dfrac{n}{k} \cdot D_{\operatorname{KL}}\prth{p_k(\theta)-\epsilon \| p_k(\theta)}}.
% \end{aligned}
% \end{equation}
Let $\kappa(z_1,\dots,z_k;\omega)$ be a $\{0,1\}$-valued kernel of order $k\leq n$ that possibly depends on additional randomness $\omega$ that is independent of the data $\set{z_1,\ldots,z_n}$, $\kappa^*(z_1,\dots,z_k):=\expect{\kappa(z_1,\dots,z_k;\omega)\vert z_1,\dots,z_k}$, and $U(z_1,\dots,z_n)$ be the U-statistic associated with $\kappa^*$. Then, it holds that
\begin{equation*}
\begin{aligned}
&\prob{U-\expect{\kappa} \geq \epsilon}\leq \exp\prth{-\dfrac{n}{2k} \cdot D_{\operatorname{KL}}\prth{\expect{\kappa}+\epsilon \| \expect{\kappa}}},\\
&\prob{U-\expect{\kappa} \leq -\epsilon}\leq \exp\prth{-\dfrac{n}{2k} \cdot D_{\operatorname{KL}}\prth{\expect{\kappa}-\epsilon \| \expect{\kappa}}},
\end{aligned}
\end{equation*}
where $D_{\operatorname{KL}}(p\| q):= p\ln \frac{p}{q} + (1-p)\ln\frac{1-p}{1-q}$ is the KL-divergence between two Bernoulli random variables with parameters $p$ and $q$, respectively.
\end{lemma}

Below, Lemma \ref{lemma: DKL} gives lower bounds for KL divergences which help analyze the bounds in Lemma \ref{lem:chernoff bound for p_k}.
The proof of Lemma \ref{lemma: DKL} is deferred to Appendix \ref{subsec: proof of lemma lemma: DKL}.
\begin{lemma}\label{lemma: DKL}
Let $D_{\operatorname{KL}}(p\| q):= p\ln \frac{p}{q} + (1-p)\ln\frac{1-p}{1-q}$ be the KL-divergence between two Bernoulli random variables with parameters $p$ and $q$, respectively.
Then, it holds that
% \begin{equation}\label{KL lower bound: squared}
% \Dkl(p\|q)\geq \dfrac{(p-q)^2}{2\max_{r\in [\min(p,q),\max(p,q)]} r(1-r)},
% \end{equation}
% and that
\begin{equation}\label{KL lower bound: ratio}
    \Dkl(p\|q)\geq p\ln \frac{p}{q} + q-p.
\end{equation}
If $p\in[\gamma,1-\gamma]$ for some $\gamma\in(0,\frac{1}{2}]$, it also holds that
\begin{equation}\label{KL lower bound: p close to 1/2}
\Dkl(p\|q)\geq -\ln \prth{2(q(1-q))^{\gamma}}.
\end{equation}
\end{lemma}

% The probability in \eqref{resampled SAA solution distribution} can be expressed as a U-statistic 
% \begin{equation*}
% \begin{aligned}
% \hat{p}_k(\theta) 
% % &= \dfrac{1}{{n \choose k}}\sum_{1\leq i_1<\ldots<i_k\leq n}\mathbbm{1}\left(\theta=\argmin_{\theta\in\Theta}\frac{1}{k}\sum_{j=1}^kl(\theta,z_{i_j})\right)\\
% &=\frac{k!}{n(n-1)\cdots(n-k+1)}\sum_{1\leq i_1<\ldots<i_k\leq n}\mathbbm{1}\left(\theta=\argmin_{\theta'\in\Theta}\frac{1}{k}\sum_{j=1}^kl(\theta',z_{i_j})\right)
% \end{aligned}
% \end{equation*}
% with expected value
% \begin{equation*}
%     \expect{\hat{p}_k(\theta)}=p_k(\theta).
% \end{equation*}
% Using Lemma \ref{lem:MGF dominance} we can obtain the following Chernoff-type concentration bound:

To incorporate all the proposed algorithms in a unified theoretical framework, we consider a set-valued mapping
\begin{equation}\label{eq: set-valued learning algorithm}
    \mathbb{A}(z_1,\ldots,z_k;\omega):\mathcal{Z}^k\times \mathbf{\Omega}\to 2^{\Theta},
\end{equation}
where $\omega\in \mathbf{\Omega}$ denotes algorithmic randomness that is independent of the data $\set{z_1,\ldots,z_k}\in \mathcal{Z}^k$. Each of our proposed algorithms attempts to solve the probability-maximization problem
\begin{equation}\label{prob:maximizing event probability empirical 2}
    \max_{\theta\in\Theta}\ \hat{p}_k(\theta):=\reprob{\theta\in \mathbb{A}(z^*_1,\ldots,z^*_k;\omega)},
\end{equation}
for a certain choice of $\mathbb{A}$, where $\set{z^*_1,\ldots,z^*_k}$ is subsampled from the i.i.d. data $\set{z_1,\ldots,z_n}$ uniformly without replacement, and $\mathbb P_*$ denotes the probability with respect to the algorithmic randomness $\omega$ and the subsampling randomness conditioned on the data. Note that this problem is an empirical approximation of the problem
\begin{equation}\label{prob:maximizing event probability 2}
    \max_{\theta\in\Theta}\ p_k(\theta):=\prob{\theta\in \mathbb{A}(z_1,\ldots,z_k;\omega)}.
\end{equation}
The problem actually solved with a finite number of subsamples is
\begin{equation}\label{prob:maximizing event probability finite B}
    \max_{\theta\in\Theta}\ \bar{p}_k(\theta):=\frac{1}{B}\sum_{b=1}^B\mathbbm{1}(\theta\in \mathbb{A}(z_1^b,\ldots,z_k^b;\omega_b)).
\end{equation}

Specifically, Algorithm \ref{bagging majority vote: set estimator} uses
\begin{equation}\label{indicator for bagging with single prediction}
    \mathbb{A}(z^*_1,\ldots,z^*_k;\omega)=\set{\mathcal{A}(z^*_1,\ldots,z^*_k;\omega)},
\end{equation}
where $\mathcal{A}$ denotes the base learning algorithm, and Algorithm \ref{bagging majority vote: two phase} uses
\begin{equation}\label{indicator for bagging with multiple predictions}
    \mathbb{A}(z^*_1,\ldots,z^*_{k_2};\omega)=\set{\theta\in \mathcal{S}:\frac{1}{k_2}\sum_{i=1}^{k_2}l(\theta,z^*_i)\leq \min_{\theta'\in\mathcal{S}}\frac{1}{k_2}\sum_{i=1}^{k_2}l(\theta',z^*_i)+\epsilon},
\end{equation}
conditioned on the solution set $\mathcal{S}$ retrieved in Phase I. Note that no algorithmic randomness is involved in \eqref{indicator for bagging with multiple predictions} once the set $\mathcal{S}$ is given. 
Now, we introduce the following definitions.
\begin{definition}\label{def: S_delta_hat}
For any $\delta\in[0,1]$, let
\begin{equation}\label{nearly optima set}
    \mathcal{P}^{\delta}_k:=\{\theta\in\Theta:p_k(\theta)\geq \max_{\theta'\in\Theta}p_k(\theta')-\delta\}
\end{equation}
be the set of $\delta$-optimal solutions of problem \eqref{prob:maximizing event probability 2}. Let
\begin{equation*}
    \theta_k^{\max}\in \argmax_{\theta\in\Theta}p_k(\theta)
\end{equation*}
be a solution with maximum probability that is chosen in a unique manner if there are multiple such solutions. Let
\begin{equation}\label{nearly optima set estimated}
    \widehat{\mathcal{P}}^{\delta}_k:=\{\theta\in\Theta:\hat{p}_k(\theta)\geq \hat{p}_k(\theta_k^{\max})-\delta\}
\end{equation}
be the set of $\delta$-optimal solutions relative to $\theta_k^{\max}$ for problem \eqref{prob:maximizing event probability empirical 2}.
\end{definition}

\begin{definition}\label{def: nearly optimal sets and convergence probabilities}
Let
\begin{equation}\label{true nearly optima set}
    \Theta^{\delta}:=\left\{\theta\in\Theta:L(\theta)\leq \min_{\theta'\in\Theta}L(\theta')+\delta\right\}
\end{equation}
be the set of $\delta$-optimal solutions of problem \eqref{opt}. In particular, $\Theta^0$ represents the set of optimal solutions. Let
\begin{equation}\label{SAA nearly optima set}
    \widehat{\Theta}_k^{\delta}:=\left\{\theta\in\Theta:\frac{1}{k}\sum_{i=1}^kl(\theta,z_i)\leq \min_{\theta'\in\Theta}\frac{1}{k}\sum_{i=1}^kl(\theta',z_i)+\delta\right\}
\end{equation}
be the set of $\delta$-optimal solutions of the SAA with i.i.d. data $(z_1,\ldots,z_k)$.
% . For every $\epsilon\geq 0$ and $\delta\geq 0$ define
% \begin{equation}\label{SAA convergence inner prob}
% q_k^{\epsilon,\delta}:=\prob{\widehat{\Theta}_k^{\epsilon}\subseteq \Theta^{\delta}},
% \end{equation}
% and
% \begin{equation}\label{SAA convergence outer prob}
% r_k^{\epsilon}:=\prob{\Theta^0\subseteq \widehat{\Theta}_k^{\epsilon}}.
% \end{equation}
\end{definition}

\subsection{Proof of Theorem \ref{thm: general majority vote}}\label{app: proof_theorem1_proposition1}
 \begin{theorem}[Formal finite-sample bound for Algorithm \ref{bagging majority vote: set estimator}]\label{thm: general majority vote_formal}
    Consider discrete decision space $\Theta$.
    Let $p_k^{\max}:=\max_{\theta\in\Theta}p_k(\theta)$, where $p_k(\theta)$ is defined in \eqref{prob:maximizing selection probability},
%For any $\delta > 0$, denote $\mathcal{E}_{k,\delta}:=\mathbb P(L(\mathcal{A}(z_1,\ldots,z_k))>\min_{\theta\in\Theta}L(\theta)+\delta)$ as the excess risk tail of $\mathcal{A}$, 
and
\begin{equation}
    \eta_{k,\delta}:=p_k^{\max} - \max_{\theta\in\Theta/\Theta^{\delta}}p_k(\theta),
\end{equation}
where $\max_{\theta\in\Theta\backslash\Theta^{\delta}}p_k(\theta)$ evaluates to $0$ if $\Theta\backslash\Theta^{\delta}$ is empty. Then, for every $k\leq n$ and $\delta\geq 0$ such that $\eta_{k,\delta}>0$, the solution output by \move satisfies that
{\small\begin{equation}\label{finite-sample bound for general bagging solution}
\begin{aligned}
    &\ \prob{L(\hat{\theta}_n) > \min_{\theta\in\Theta}L(\theta)+\delta}\\
    \leq & \abs{\Theta}
    \Bigg[\exp\prth{-\dfrac{n}{2k} \cdot D_{\operatorname{KL}}\prth{p_k^{\max}-\dfrac{3\eta_{k,\delta}}{4} \Big\| p_k^{\max}-\eta_{k,\delta}}}+2\exp\prth{-\dfrac{n}{2k} \cdot D_{\operatorname{KL}}\prth{p_k^{\max}-\dfrac{\eta_{k,\delta}}{4} \Big\| p_k^{\max}}}\\
    &\qquad +\exp\prth{-\dfrac{B}{24}\cdot \dfrac{\eta_{k,\delta}^2}{\min\set{p_k^{\max},1-p_k^{\max}}+3\eta_{k,\delta}/4}}\\
    &\qquad +\mathbbm{1}\prth{p_k^{\max}+\dfrac{\eta_{k,\delta}}{4}\leq 1}\cdot\exp\Big(-\dfrac{n}{2k} \cdot D_{\operatorname{KL}}\prth{p_k^{\max}+\dfrac{\eta_{k,\delta}}{4} \Big\| p_k^{\max}}-\dfrac{B}{24}\cdot \dfrac{\eta_{k,\delta}^2}{1-p_k^{\max}+\eta_{k,\delta}/4}\Big)\Bigg].
\end{aligned}
\end{equation}}
In particular, if $\eta_{k,\delta}>4/5$, \eqref{finite-sample bound for general bagging solution} is further bounded by
\begin{equation}\label{eq: general finite-sample bound for large gap}
\abs{\Theta}\prth{3\min\set{e^{-2/5}, C_1\max\set{1-p_k^{\max},\;\max_{\theta\in\Theta/\Theta^{\delta}}p_k(\theta)}}^{\frac{n}{C_2k}} + e^{-B/C_3}},
\end{equation}
where $C_1,C_2,C_3>0$ are universal constants, $\abs{\Theta}$ denotes the cardinality of $\Theta$, and $D_{\operatorname{KL}}(p\| q):= p\ln \frac{p}{q} + (1-p)\ln\frac{1-p}{1-q}$ is the Kullback–Leibler divergence between two Bernoulli distributions with means $p$ and $q$.
\end{theorem}

We consider Algorithm \ref{bagging majority vote: set estimator 2}, a generalization of Algorithm \ref{bagging majority vote: set estimator} applied to the set-valued learning algorithm $\mathbb{A}$ in \eqref{eq: set-valued learning algorithm}.
This framework recovers Algorithm \ref{bagging majority vote: set estimator} as a special case under condition \eqref{indicator for bagging with single prediction}.
Again, we omit the algorithmic randomness $\omega$ in $\mathbb{A}$ for convenience.
For Algorithm \ref{bagging majority vote: set estimator 2}, we derive the following finite-sample guarantee.

\begin{algorithm}
\caption{Majority Vote Ensembling for Set-Valued Learning Algorithms}
\label{bagging majority vote: set estimator 2}
\begin{algorithmic}[1]
\STATE {Input: A set-valued learning algorithm $\mathbb{A}$, $n$ i.i.d. observations $\mathbf{z}_{1:n}=(z_1,\ldots,z_n)$, positive integers $k<n$, and ensemble size $B$.}

% \vspace{1ex}

\FOR{$b=1$ to $B$ }
\STATE {Randomly sample $\mathbf{z}_k^b=(z_1^b,\ldots,z_k^b)$ uniformly from $\mathbf{z}_{1:n}$ without replacement, and obtain $\Theta_k^b=\mathbb{A}(z_1^b,\ldots,z_k^b)$
}
% \STATE {Compute $\bar{\psi}^b_{AV}=\frac{1}{R}\sum_{r=1}^R\tilde{\psi}_r(\widehat{F}_1^b,\ldots,\widehat{F}_m^b)$}
\ENDFOR

% \vspace{1ex}

\STATE {Output $\hat{\theta}_n\in\argmax_{\theta\in\Theta}\sum_{b=1}^B\mathbbm{1}(\theta\in \Theta_k^b)$.}
\end{algorithmic}
\end{algorithm}

\begin{theorem}[Finite-sample bound for Algorithm \ref{bagging majority vote: set estimator 2}]\label{thm: general majority vote 2}
    Consider discrete decision space $\Theta$. 
    Let $p_k^{\max}:=\max_{\theta\in\Theta}p_k(\theta)$, where $p_k(\theta)$ is defined in \eqref{prob:maximizing event probability 2}. 
For any $\delta > 0$, denote
    \begin{equation}\label{SAA prob gap 2}
        \bar{\eta}_{k,\delta}:=p_k^{\max} - \max_{\theta\in\Theta\backslash\Theta^{\delta}}p_k(\theta),
        % \bar{\eta}_{k,\delta}:=\sup_{\delta'\geq 0}\frac{q_{k,\delta'}}{\abs{\Theta^{\delta'}}} + \mathcal{E}_{k,\delta}-1
    \end{equation}
    where $\max_{\theta\in\Theta\backslash\Theta^{\delta}}p_k(\theta)$ evaluates to $0$ if $\Theta\backslash\Theta^{\delta}$ is empty. Then, for every $k\leq n$ and $\delta\geq0$ such that $\bar{\eta}_{k,\delta}>0$, the solution output by Algorithm \ref{bagging majority vote: set estimator 2} satisfies that
    \begin{equation}\label{finite-sample bound for general bagging solution 2}
    \begin{aligned}
        &\prob{L(\hat{\theta}_n) > \min_{\theta\in\Theta}L(\theta)+\delta}\\
        \leq & \abs{\Theta}
        \Bigg[\exp\prth{-\dfrac{n}{2k} \cdot D_{\operatorname{KL}}\prth{p_k^{\max}-\dfrac{3\eta}{4} \Big\| p_k^{\max}-\eta}}+2\exp\prth{-\dfrac{n}{2k} \cdot D_{\operatorname{KL}}\prth{p_k^{\max}-\dfrac{\eta}{4} \Big\| p_k^{\max}}}\\
        &\hspace{5ex} +\exp\prth{-\dfrac{B}{24}\cdot \dfrac{\eta^2}{\min\set{p_k^{\max},1-p_k^{\max}}+3\eta/4}}\\
        &\hspace{5ex} +\mathbbm{1}\prth{p_k^{\max}+\dfrac{\eta}{4}\leq 1}\cdot\exp\Big(-\dfrac{n}{2k} \cdot D_{\operatorname{KL}}\prth{p_k^{\max}+\dfrac{\eta}{4} \Big\| p_k^{\max}}-\dfrac{B}{24}\cdot \dfrac{\eta^2}{1-p_k^{\max}+\eta/4}\Big)\Bigg]
    \end{aligned}
    \end{equation}
    for every $\eta\in(0,\bar{\eta}_{k,\delta}]$. In particular, if $\bar{\eta}_{k,\delta}>4/5$, \eqref{finite-sample bound for general bagging solution 2} is further bounded by
    \begin{equation}\label{eq: general finite-sample bound for large gap 2}
    \abs{\Theta}\prth{3\min\set{e^{-2/5}, C_1\max\set{1-p_k^{\max},\;\max_{\theta\in\Theta\backslash\Theta^{\delta}}p_k(\theta)}}^{\frac{n}{C_2k}} + \exp\prth{-\frac{B}{C_3}}},
    \end{equation}
    where $C_1,C_2,C_3>0$ are universal constants, and $D_{\operatorname{KL}}(p\| q):= p\ln \frac{p}{q} + (1-p)\ln\frac{1-p}{1-q}$ is the Kullback–Leibler divergence between two Bernoulli distributions with means $p$ and $q$.
\end{theorem}

\begin{proof}[Proof of Theorem \ref{thm: general majority vote 2}]
We first prove excess risk tail bounds for the problem \eqref{prob:maximizing event probability 2}, split into two lemmas, Lemmas \ref{lem:optima set approximation} and \ref{lem:optimality conditioned on data} below.
The proofs for these two lemmas can be found in Appendix \ref{subsec: proof of lemma lem:optima set approximation} and Appendix \ref{subsec: proof of lemma lem:optimality conditioned on data}, respectively.
\begin{lemma}\label{lem:optima set approximation}
Consider discrete decision space $\Theta$. Recall from Definition \ref{def: S_delta_hat} that $p_k^{\max}=p_k(\theta_k^{\max})$ holds for $\theta_k^{\max}$. For every $0\leq\epsilon\leq \delta\leq p_k^{\max}$, it holds that
\begin{equation*}%\label{eq: lowerbound_setbelong_1}
\begin{aligned}
\prob{\widehat{\mathcal{P}}^{\epsilon}_k\not\subseteq\mathcal{P}^{\delta}_k}
\leq  \abs{\Theta} 
&\left[\exp\prth{-\dfrac{n}{2k} \cdot D_{\operatorname{KL}}\prth{p_k^{\max}-\dfrac{\delta+\epsilon}{2} \Big\| p_k^{\max}-\delta}}\right. \\
+ &\left.\exp\prth{-\dfrac{n}{2k} \cdot D_{\operatorname{KL}}\prth{p_k^{\max}-\dfrac{\delta-\epsilon}{2} \Big\| p_k^{\max}}}\right].
\end{aligned}
\end{equation*}
% Furthermore, if $p_k^{\max} \leq  1/2$ and $p_k(\theta) \leq 1/2 - (\delta-\epsilon)/2$ for all $\theta\in \Theta\backslash \mathcal{P}^{\delta}_k$, we further have
% \begin{equation}\label{eq: lowerbound_setbelong_2}
% \begin{aligned}
% \prob{\widehat{\mathcal{P}}^{\epsilon}_k\subseteq\mathcal{P}^{\delta}_k}
% \geq 1
% &- |\Theta\backslash \mathcal{P}^{\delta}_k|\exp\prth{-\dfrac{n(\delta-\epsilon)^2}{8k\cdot p_k^{\max}}} \\
% &- \sum_{\theta\in \Theta\backslash \mathcal{P}^{\delta}_k} \exp\prth{-\dfrac{n(\delta-\epsilon)^2}{8k\cdot \brac{p_k(\theta) + (\delta-\epsilon)/2}}}.
% \end{aligned}
% \end{equation}
\end{lemma}

\begin{lemma}\label{lem:optimality conditioned on data}
    % We have that
    % \begin{equation}\label{eq: P_*p_bar_{\theta}}
    %     % \reprob{\bar{p}_k(\theta)= \max_{\theta'\in\Theta}\bar{p}_k(\theta')}\leq \exp\prth{-B\max\left(\sqrt{\hat{p}_k^{\max}}-\sqrt{\hat{p}_k(\theta)},\sqrt{1-\hat{p}_k(\theta)}-\sqrt{1-\hat{p}_k^{\max}}\right)^2}
    %     \reprob{\bar{p}_k(\theta)= \max_{\theta'\in\Theta}\bar{p}_k(\theta')}\leq \exp\prth{-B\prth{\sqrt{1-\hat{p}_k(\theta)}-\sqrt{1-\hat{p}_k^{\max}}}^2}
    % \end{equation}
    % for every $\theta\in\Theta$. 
    Consider discrete decision space $\Theta$. For every $\epsilon\in[0,1]$ it holds for the solution output by Algorithm \ref{bagging majority vote: set estimator 2} that
    \begin{equation*}
        % \reprob{\hat{\theta}_{n,k}^{BAG}\notin \widehat{\mathcal{P}}^{\epsilon}_k}\leq \abs{\Theta} \exp\prth{-B\max\prth{\sqrt{\hat{p}_k^{\max}}-\sqrt{\hat{p}_k^{\max}-\epsilon},\sqrt{1-\hat{p}_k^{\max}+\epsilon}-\sqrt{1-\hat{p}_k^{\max}}}^2},
        \reprob{\hat{\theta}_n\notin \widehat{\mathcal{P}}^{\epsilon}_k}\leq \abs{\Theta}\cdot \exp\prth{-\dfrac{B}{6}\cdot \dfrac{\epsilon^2}{\min\set{\hat{p}_k(\theta_k^{\max}),1-\hat{p}_k(\theta_k^{\max})}+\epsilon}},
    \end{equation*}
    where $\abs{\cdot}$ denotes the cardinality of a set and $\mathbb P_*$ denotes the probability with respect to both the resampling randomness conditioned on the observations and the algorithmic randomness.
\end{lemma}

We are now ready for the proof of Theorem \ref{thm: general majority vote 2}. We first note that, if $\bar{\eta}_{k,\delta}>0$,
    % we equivalently have that $\bar{\eta}_{k,\delta} = \max_{\theta\in\Theta^\delta}p_k(\theta) - \max_{\theta\in\Theta\backslash\Theta^{\delta}}p_k(\theta) > 0$.Hence, 
    it follows from Definition \ref{def: S_delta_hat} that
    \begin{equation*}
        \mathcal{P}_k^{\eta} \subseteq\Theta^{\delta}\text{\ for any\ }\eta\in (0,\;\bar{\eta}_{k,\delta}).
    \end{equation*}
    Therefore, for any $\eta\in (0,\;\bar{\eta}_{k,\delta})$, we can write that
    \begin{equation}\label{eq: xnk bag upperbound}
    \begin{aligned}
    \prob{\hat{\theta}_n\notin \Theta^{\delta}}\leq  \prob{\hat{\theta}_n\notin \mathcal{P}^{\eta}_k} &\leq& \prob{\left\{\hat{\theta}_n\notin \widehat{\mathcal{P}}^{\eta/2}_k\right\}\cup \left\{ \widehat{\mathcal{P}}^{\eta/2}_k\not\subseteq\mathcal{P}^{\eta}_k \right\}}\\
    &\leq& \prob{\hat{\theta}_n\notin \widehat{\mathcal{P}}^{\eta/2}_k}+\prob{ \widehat{\mathcal{P}}^{\eta/2}_k\not\subseteq\mathcal{P}^{\eta}_k }.
    \end{aligned}
    \end{equation}
    We first evaluate the second probability on the right-hand side of \eqref{eq: xnk bag upperbound}.
    Lemma \ref{lem:optima set approximation} gives that
    % with $\epsilon=\eta/2$ and $\delta=\eta$ gives that
    
    \begin{equation}\label{eq: Sk eta/2 notsubset Sketa}
    \begin{aligned}
\prob{\widehat{\mathcal{P}}^{\eta/2}_k\not\subseteq\mathcal{P}^{\eta}_k}
    \leq  \abs{\Theta} 
    \Bigg[&\exp\prth{-\dfrac{n}{2k} \cdot D_{\operatorname{KL}}\prth{p_k^{\max}-\dfrac{3\eta}{4} \Big\| p_k^{\max}-\eta}}\\ &\qquad+\exp\prth{-\dfrac{n}{2k} \cdot D_{\operatorname{KL}}\prth{p_k^{\max}-\dfrac{\eta}{4} \Big\| p_k^{\max}}}\Bigg].
    \end{aligned}
    \end{equation} 
    Next, by applying Lemma \ref{lem:optimality conditioned on data} with $\epsilon=\eta/2$, we can bound the first probability on the right-hand side of \eqref{eq: xnk bag upperbound} as 
    \begin{equation}\label{solution missing prob}
        \prob{\hat{\theta}_n\notin \widehat{\mathcal{P}}^{\eta/2}_k}\leq \abs{\Theta}\cdot \expect{\exp\prth{-\dfrac{B}{24}\cdot \dfrac{\eta^2}{\min\set{\hat{p}_k(\theta_k^{\max}),1-\hat{p}_k(\theta_k^{\max})}+\eta/2}}}.
    \end{equation}
    Conditioned on the value of $\hat{p}_k(\theta_k^{\max})$, we can further upper-bound the right-hand side of \eqref{solution missing prob} as follows
    \begin{eqnarray*}
        &&\expect{\exp\prth{-\dfrac{B}{24}\cdot \dfrac{\eta^2}{\min\set{\hat{p}_k(\theta_k^{\max}), 1-\hat{p}_k(\theta_k^{\max})}+\eta/2}}}\\
        &\leq &\prob{\hat{p}_k(\theta_k^{\max})\leq p_k^{\max}-\frac{\eta}{4}}\cdot \exp\prth{-\dfrac{B}{24}\cdot \dfrac{\eta^2}{p_k^{\max}+\eta/4}}+\\
        &&\prob{\abs{\hat{p}_k(\theta_k^{\max})-p_k^{\max}}<\frac{\eta}{4}}\cdot \exp\prth{-\dfrac{B}{24}\cdot \dfrac{\eta^2}{\min\set{p_k^{\max},1-p_k^{\max}}+3\eta/4}} +\\
        &&\prob{\hat{p}_k(\theta_k^{\max})\geq p_k^{\max}+\frac{\eta}{4}}\cdot \exp\prth{-\dfrac{B}{24}\cdot \dfrac{\eta^2}{1-p_k^{\max}+\eta/4}}\\
        &\leq &\prob{\hat{p}_k(\theta_k^{\max})\leq p_k^{\max}-\frac{\eta}{4}}+ \exp\prth{-\dfrac{B}{24}\cdot \dfrac{\eta^2}{\min\set{p_k^{\max},1-p_k^{\max}}+3\eta/4}}+\\
        &&\prob{\hat{p}_k(\theta_k^{\max})\geq p_k^{\max}+\frac{\eta}{4}}\cdot \exp\prth{-\dfrac{B}{24}\cdot \dfrac{\eta^2}{1-p_k^{\max}+\eta/4}}\\
        &\overset{(i)}{\leq} &\exp\prth{-\dfrac{n}{2k} \cdot D_{\operatorname{KL}}\prth{p_k^{\max}-\dfrac{\eta}{4} \Big\| p_k^{\max}}}+\\
        &&\exp\prth{-\dfrac{B}{24}\cdot \dfrac{\eta^2}{\min\set{p_k^{\max},1-p_k^{\max}}+3\eta/4}}+\\
        && \mathbbm{1}\prth{p_k^{\max}+\dfrac{\eta}{4}\leq 1}\cdot\exp\prth{-\dfrac{n}{2k} \cdot D_{\operatorname{KL}}\prth{p_k^{\max}+\dfrac{\eta}{4} \Big\| p_k^{\max}}}\cdot \exp\prth{-\dfrac{B}{24}\cdot \dfrac{\eta^2}{1-p_k^{\max}+\eta/4}},
    \end{eqnarray*}
    where inequality $(i)$ results from applying Lemma \ref{lem:chernoff bound for p_k} with $\hat{p}_k(\theta_k^{\max})$, the U-statistic estimate for $p_k^{\max}$.
    Together, the above equations imply that
    \begin{eqnarray*}
        && \prob{\hat{\theta}_n\notin \Theta^{\delta}}\\
        &\leq & \abs{\Theta}
        \Bigg[\exp\prth{-\dfrac{n}{2k} \cdot D_{\operatorname{KL}}\prth{p_k^{\max}-\dfrac{3\eta}{4} \Big\| p_k^{\max}-\eta}}\\
        &&+2\exp\prth{-\dfrac{n}{2k} \cdot D_{\operatorname{KL}}\prth{p_k^{\max}-\dfrac{\eta}{4} \Big\| p_k^{\max}}}\\
        &&+\exp\prth{-\dfrac{B}{24}\cdot \dfrac{\eta^2}{\min\set{p_k^{\max},1-p_k^{\max}}+3\eta/4}}\\
        &&+\mathbbm{1}\prth{p_k^{\max}+\dfrac{\eta}{4}\leq 1}\cdot\exp\prth{-\dfrac{n}{2k} \cdot D_{\operatorname{KL}}\prth{p_k^{\max}+\dfrac{\eta}{4} \Big\| p_k^{\max}}-\dfrac{B}{24}\cdot \dfrac{\eta^2}{1-p_k^{\max}+\eta/4}}\Bigg].
    \end{eqnarray*}
    Since the above probability bound is left-continuous in $\eta$ and $\eta$ can be arbitrarily chosen from $(0,\;\bar{\eta}_{k,\delta})$, the validity of the case $\eta=\bar{\eta}_{k,\delta}$ follows from pushing $\eta$ to the limit $\bar{\eta}_{k,\delta}$. This gives \eqref{finite-sample bound for general bagging solution 2}.

    To simplify the bound in the case $\bar{\eta}_{k,\delta}>4/5$. Consider the bound \eqref{finite-sample bound for general bagging solution 2} with $\eta=\bar{\eta}_{k,\delta}$. Since $p_k^{\max}\geq \bar{\eta}_{k,\delta}$ by the definition of $\bar{\eta}_{k,\delta}$, it must hold that $p_k^{\max}+\bar{\eta}_{k,\delta}/4>4/5+1/5=1$, therefore the last term in the finite-sample bound \eqref{finite-sample bound for general bagging solution 2} vanishes. To simplify the first two terms in the finite-sample bound, we note that
\begin{equation*}
\begin{aligned}
    p_k^{\max}-\dfrac{3\bar{\eta}_{k,\delta}}{4}&\leq 1-\dfrac{3}{4}\cdot \dfrac{4}{5}=\dfrac{2}{5},\\
    p_k^{\max}-\dfrac{3\bar{\eta}_{k,\delta}}{4}&\geq \bar{\eta}_{k,\delta}-\dfrac{3\bar{\eta}_{k,\delta}}{4} \geq \dfrac{1}{5},\\
    p_k^{\max}-\dfrac{\bar{\eta}_{k,\delta}}{4}&\leq 1-\dfrac{1}{4}\cdot \dfrac{4}{5}=\dfrac{4}{5},\\
    p_k^{\max}-\dfrac{\bar{\eta}_{k,\delta}}{4}&\geq \bar{\eta}_{k,\delta}-\dfrac{\bar{\eta}_{k,\delta}}{4} \geq \dfrac{3}{5},\\
\end{aligned}
\end{equation*}
and that $p_k^{\max}-\bar{\eta}_{k,\delta}\leq 1-\bar{\eta}_{k,\delta}\leq 1/5$, therefore by the bound \eqref{KL lower bound: p close to 1/2} from Lemma \ref{lemma: DKL}, we can bound the first two terms as
\begin{eqnarray*}
     &&\exp\prth{-\dfrac{n}{2k} \cdot D_{\operatorname{KL}}\prth{p_k^{\max}-\dfrac{3\bar{\eta}_{k,\delta}}{4} \Big\| p_k^{\max}-\bar{\eta}_{k,\delta}}}\\
     &\leq& \exp\prth{\dfrac{n}{2k}\ln \prth{2((p_k^{\max}-\bar{\eta}_{k,\delta})(1-p_k^{\max}+\bar{\eta}_{k,\delta}))^{1/5}}}\\
     &=&  \prth{2((p_k^{\max}-\bar{\eta}_{k,\delta})(1-p_k^{\max}+\bar{\eta}_{k,\delta}))^{1/5}}^{n/(2k)}\\
     &\leq& \prth{2(p_k^{\max}-\bar{\eta}_{k,\delta})^{1/5}}^{n/(2k)}\\
     &=& \prth{2^5(p_k^{\max}-\bar{\eta}_{k,\delta})}^{n/(10k)},
\end{eqnarray*}
and similarly
\begin{eqnarray*}
     \exp\prth{-\dfrac{n}{2k} \cdot D_{\operatorname{KL}}\prth{p_k^{\max}-\dfrac{\bar{\eta}_{k,\delta}}{4} \Big\| p_k^{\max}}}
     &\leq& \exp\prth{\dfrac{n}{2k}\ln \prth{2(p_k^{\max}(1-p_k^{\max}))^{1/5}}}\\
     &=&  \prth{2(p_k^{\max}(1-p_k^{\max}))^{1/5}}^{n/(2k)}\\
     &\leq& \prth{2(1-p_k^{\max})^{1/5}}^{n/(2k)}\\
     &=& \prth{2^5(1-p_k^{\max})}^{n/(10k)}.
\end{eqnarray*}
On the other hand, by Lemma \ref{lemma: DKL} both $D_{\operatorname{KL}}\prth{p_k^{\max}-3\bar{\eta}_{k,\delta}/4 \| p_k^{\max}-\bar{\eta}_{k,\delta}}$ and $D_{\operatorname{KL}}\prth{p_k^{\max}-\bar{\eta}_{k,\delta}/4 \| p_k^{\max}}$ are bounded below by $\bar{\eta}_{k,\delta}^2/8$, therefore
\begin{equation*}
    \exp\prth{-\dfrac{n}{2k} \cdot D_{\operatorname{KL}}\prth{p_k^{\max}-\dfrac{3\bar{\eta}_{k,\delta}}{4} \Big\| p_k^{\max}-\bar{\eta}_{k,\delta}}}\leq \exp\prth{-\dfrac{n}{2k}\cdot \dfrac{\bar{\eta}_{k,\delta}^2}{8}}\leq \exp\prth{-\dfrac{n}{25k}},
\end{equation*}
and the same holds for $\exp\prth{-n/(2k) \cdot D_{\operatorname{KL}}\prth{p_k^{\max}-\bar{\eta}_{k,\delta}/4 \| p_k^{\max}}}$. For the third term in the bound \eqref{finite-sample bound for general bagging solution 2} we have
\begin{equation*}
    \dfrac{\bar{\eta}_{k,\delta}^2}{\min\set{p_k^{\max}, 1-p_k^{\max}}+3\bar{\eta}_{k,\delta}/4}\geq \dfrac{(4/5)^2}{\min\set{1,1/5}+3/4}\geq\dfrac{16}{25},
\end{equation*}
and hence
\begin{equation*}
    \exp\prth{-\dfrac{B}{24}\cdot \dfrac{\bar{\eta}_{k,\delta}^2}{\min\set{p_k^{\max}, 1-p_k^{\max}}+3\bar{\eta}_{k,\delta}/4}}\leq \exp\prth{-\dfrac{B}{75/2}}.
\end{equation*}
The first desired bound then follows by setting $C_1,C_2,C_3$ to be the appropriate constants. This completes the proof of Theorem \ref{thm: general majority vote 2}.
\end{proof}

\begin{proof}[Proof of Theorem \ref{thm: general majority vote_formal}]
Algorithm \ref{bagging majority vote: set estimator} is a special case of Algorithm \ref{bagging majority vote: set estimator 2} with the learning algorithm \eqref{indicator for bagging with single prediction} that outputs a singleton, therefore the results of Theorem \ref{thm: general majority vote 2} automatically apply. In particular, $\eta_{k,\delta}= \bar{\eta}_{k,\delta}$ in the context of Theorem \ref{thm: general majority vote_formal}, and \eqref{finite-sample bound for general bagging solution} follows from setting $\eta$ to be $\eta_{k,\delta}$ in \eqref{finite-sample bound for general bagging solution 2}.
\end{proof}

\subsection{Proof of Corollary \ref{cor: application of move to linear program example}}\label{sec: proof of linear program example}
\begin{proof}[Proof of Corollary \ref{cor: application of move to linear program example}]
    By the continuity and symmetry of $z$ we have $q_k=\mbb{P}(\sum_{i=1}^kz_i>k) + \mbb{P}(\sum_{i=1}^kz_i\in(0,k))=1/2+\mbb{P}(\sum_{i=1}^kz_i\in(0,k))$. Since $z$ has a non-zero density everywhere, $\mbb{P}(\sum_{i=1}^kz_i\in(0,k))>0$, thus $q_k>1/2$ for every $k>0$. We note that the SAA of the linear program outputs either $0$ or $1$, therefore the space $[0,1]$ can be effectively viewed as the binary set $\{0,1\}$ and Theorem \ref{thm: general majority vote} is applicable with $\abs{\Theta}=2$. To apply Theorem \ref{thm: general majority vote}, it can be easily seen that $\max_{\theta\in\Theta}p_k(\theta)=\mbb{P}(\hat\theta=0)=q_k$ and that $\max_{\theta\in\Theta/\Theta^{\delta}}p_k(\theta)=\mbb{P}(\hat\theta=1)=1-q_k$ for $\delta<1$. This gives $\eta_{k,\delta}=2q_k-1>0$. Therefore the bound \eqref{finite-sample bound for general bagging solution_informal} holds for every $k>0$ and $\delta
    <1$. If $q_k>0.9$, we have $\eta_{k,\delta}>4/5$, and hence the bound \eqref{eq: finite-sample bound for data splitting two phase under large pmax_informal} holds. The particular form of the bound \eqref{eq: finite-sample bound for data splitting two phase under large pmax_informal} is then obtained by plugging in the values for $\max_{\theta\in\Theta}p_k(\theta)$, $\max_{\theta\in\Theta/\Theta^{\delta}}p_k(\theta)$ and $\abs{\Theta}$.
\end{proof}

\subsection{Proof of Theorem \ref{thm: finite-sample bound for multiple predictions two phase splitting}}\label{sec: theories for multiple predictions}
\begin{theorem}[Formal finite-sample bound for Algorithm \ref{bagging majority vote: two phase}]\label{thm: finite-sample bound for multiple predictions two phase splitting_formal}
    Let $\mathcal{E}_{k,\delta}:=\mathbb P(L(\mathcal{A}(z_1,\ldots,z_k))>\min_{\theta\in\Theta}L(\theta)+\delta)$ be the excess risk tail of $\mathcal{A}$.
    Consider Algorithm \ref{bagging majority vote: two phase} with data splitting, i.e., $\mathsf{ROVEs}$. Let $T_k(\cdot):=\mathbb P(\sup_{\theta\in\Theta}\lvert (1/k)\sum_{i=1}^kl(\theta,z_i)-L(\theta)\rvert> \cdot)$ be the tail function of the maximum deviation of the empirical objective estimate. For every $\delta >0$, if $\epsilon$ is chosen such that $\prob{\epsilon\in [\underline{\epsilon},\overline{\epsilon}]}=1$ for some $0< \underline{\epsilon}\leq \overline{\epsilon}<\delta$ and $T_{k_2}\prth{(\delta-\overline{\epsilon})/2}+T_{k_2}\prth{\underline{\epsilon}/2}<1/5$, then
    \begin{equation}\label{eq: finite-sample bound for data splitting two phase under large pmax}
    \begin{aligned}
        \mbb{P}\Big(L(\hat{\theta}_n) > \min_{\theta\in\Theta}L(\theta)+2\delta\Big)\leq&B_1\brac{3\min\set{e^{-2/5},C_1T_{k_2}\prth{\dfrac{\min\set{\underline{\epsilon}, \delta-\overline{\epsilon}}}{2}}}^{\frac{n}{2C_2k_2}}+e^{-{B_2}/{C_3}}}\\
        &+\min\set{e^{-(1-\mathcal{E}_{k_1,\delta})/C_4},C_5\mathcal{E}_{k_1,\delta}}^{\frac{n}{2C_6k_1}}+e^{-B_1(1-\mathcal{E}_{k_1,\delta})/C_7},
    \end{aligned}
    \end{equation}
    where $C_1,C_2,C_3$ are the same as those in Theorem \ref{thm: general majority vote_formal}, and $C_4,C_5,C_6,C_7$ are universal constants.
    % it holds that
    % \begin{equation}\label{bound for conditional convergence}
    %     \prob{\hat{\theta}_{n,k}^{BAG}\notin \Theta^{2\delta}\vert \mathcal{S},\epsilon}\leq \prob{L(\hat{\theta}_{n,k}^{BAG})>\min_{\theta\in\mathcal{S}} L(\theta)+\delta\vert \mathcal{S},\epsilon}.
    % \end{equation}

    Consider Algorithm \ref{bagging majority vote: two phase} without data splitting, i.e., $\mathsf{ROVE}$, and discrete space $\Theta$. Assume $\lim_{k\to\infty}T_k(\delta)= 0$ for all $\delta>0$. Then, for every fixed $\delta>0$, we have $\lim_{n\to\infty}\mathbb P(L(\hat{\theta}_n) > \min_{\theta\in\Theta}L(\theta)+2\delta)\to 0$, if $\limsup_{k\to \infty}\mathcal{E}_{k,\delta}<1$, $\prob{\epsilon>\delta/2}\to 0$, $k_1\text{ and }k_2\to\infty$, $n/k_1\text{ and }n/k_2\to\infty$, and $B_1,B_2\to\infty$ as $n\to\infty$.
    % and $\liminf_{n\to\infty}k\geq \underline{k}$, where $\underline{k}$ is a finite constant that only depends on the problem \eqref{opt}
\end{theorem}

We first present two lemmas to be used in the main proof. The following Lemma \ref{lem: near optimality of retrieved solution pool} characterizes the exponentially improving quality of the solution set retrieved in Phase I, where its proof can be found in Appendix \ref{subsec: proof of lemma lem: near optimality of retrieved solution pool}.
\begin{lemma}[Quality of retrieved solutions in Algorithm \ref{bagging majority vote: two phase}]\label{lem: near optimality of retrieved solution pool}
For every $k$ and $\delta\geq 0$, the set of retrieved solutions $\mathcal{S}$ from Phase I of Algorithm \ref{bagging majority vote: two phase} with $k_1=k$ and without data splitting satisfies that
    \begin{equation}\label{eq: retrieval tail bound}
        \prob{\mathcal{S}\cap \Theta^{\delta}=\emptyset}\leq \min\set{e^{-(1-\mathcal{E}_{k,\delta})/C_4},C_5\mathcal{E}_{k,\delta}}^{\frac{n}{C_6k}} + \exp\prth{-\dfrac{B_1}{C_7}(1-\mathcal{E}_{k,\delta})},
    \end{equation}
    where $C_4,C_5,C_6,C_7>0$ are universal constants. The same bound with $n$ replaced by $n/2$ holds true for Algorithm \ref{bagging majority vote: two phase} with data splitting.
\end{lemma}

Then, Lemma \ref{lem: gap translation for multiple predictions} gives bounds for the excess risk sensitivity $\bar{\eta}_{k,\delta}$ in the case of the set-valued learning algorithm \eqref{indicator for bagging with multiple predictions}.
The proof of Lemma \ref{lem: gap translation for multiple predictions} can be found in Appendix \ref{subsec: proof of lemma lem: gap translation for multiple predictions}.
\begin{lemma}[Bounds of $\bar{\eta}_{k,\delta}$ for the set-valued learning algorithm \eqref{indicator for bagging with multiple predictions}]\label{lem: gap translation for multiple predictions}
Consider discrete decision space $\Theta$. If the set-valued learning algorithm 
\begin{equation*}
    \mathbb{A}(z_1,\ldots,z_k;\omega):=\set{\theta\in \Theta:\frac{1}{k}\sum_{i=1}^kl(\theta,z_i)\leq \min_{\theta'\in\Theta}\frac{1}{k}\sum_{i=1}^kl(\theta',z_i)+\epsilon}
\end{equation*}
is used with $\epsilon\geq 0$, it holds that
\begin{align}
    &p_k^{\max}=\max_{\theta\in\Theta}p_k(\theta) \geq 1-T_k\prth{\dfrac{\epsilon}{2}},\label{eq: pkmax lower bound for multiple predictions}\\
    &\max_{\theta\in\Theta\backslash\Theta^{\delta}}p_k(\theta)\leq T_k\prth{\dfrac{\delta-\epsilon}{2}},\label{eq: pkmax - eta upper bound}
\end{align}
and hence
\begin{equation}
    \bar{\eta}_{k,\delta}\geq 1-T_k\prth{\dfrac{\epsilon}{2}}-T_k\prth{\dfrac{\delta-\epsilon}{2}},\label{eq: prob gap lower bound for multiple predictions}
\end{equation}
    % \begin{equation}\label{eq: pkmax lower bound for multiple predictions}
    %     p_k^{\max}=\max_{\theta\in\Theta}p_k(\theta) \geq 1-T_k\prth{\dfrac{\epsilon}{2}},
    % \end{equation}
where $T_k$ is the tail probability defined in Theorem \ref{thm: finite-sample bound for multiple predictions two phase splitting}.
% For any $\delta\geq 0$, whenever $1-T_k\prth{\epsilon/2}-T_k\prth{(\delta-\epsilon)/2}>0$, it holds that
% \begin{equation}\label{eq: prob gap lower bound for multiple predictions}
%         \bar{\eta}_{k,\delta}\geq 1-T_k\prth{\dfrac{\epsilon}{2}}-T_k\prth{\dfrac{\delta-\epsilon}{2}},
% \end{equation}
% and that
% \begin{equation}\label{eq: pkmax - eta upper bound}
%     p_k^{\max}-\bar{\eta}_{k,\delta}\leq T_k\prth{\dfrac{\delta-\epsilon}{2}}.
% \end{equation}
\end{lemma}

To prove Theorem \ref{thm: finite-sample bound for multiple predictions two phase splitting_formal}, we also introduce some notations. 
For every non-empty subset $\mathcal{W}\subseteq\Theta$, we use the following counterpart of Definition \ref{def: nearly optimal sets and convergence probabilities}. Let
\begin{equation}\label{true nearly optima set: restricted}
    \mathcal{W}^{\delta}:=\left\{\theta\in\mathcal{W}:L(\theta)\leq \min_{\theta'\in\mathcal{W}}L(\theta')+\delta\right\}
\end{equation}
be the set of $\delta$-optimal solutions in the restricted decision space $\mathcal{W}$, and
\begin{equation}\label{SAA nearly optima set: restricted}
    \widehat{\mathcal{W}}_k^{\delta}:=\left\{\theta\in\mathcal{W}:\frac{1}{k}\sum_{i=1}^kl(\theta,z_i)= \min_{\theta'\in\mathcal{W}}\frac{1}{k}\sum_{i=1}^kl(\theta',z_i)+\delta\right\}
\end{equation}
be the set of $\delta$-optimal solutions of the SAA with an i.i.d. data set of size $k$.
% For every $\epsilon\geq 0$ and $\delta\geq 0$ let
% \begin{equation}\label{SAA convergence inner prob: restricted}
% q_k^{\epsilon,\delta,\mathcal{W}}:=\prob{\widehat{\mathcal{W}}_k^{\epsilon}\subseteq \mathcal{W}^{\delta}},
% \end{equation}
% and
% \begin{equation}\label{SAA convergence outer prob: restricted}
% r_k^{\epsilon.\mathcal{W}}:=\prob{\mathcal{W}^0\subseteq \widehat{\mathcal{W}}_k^{\epsilon}},
% \end{equation}
% be the counterparts of $q_k^{\epsilon,\delta}$ and $r_k^{\epsilon}$ respectively.

\begin{proof}[Proof of Theorem \ref{thm: finite-sample bound for multiple predictions two phase splitting_formal} for $\mathsf{ROVEs}$]
        % By Lemma \ref{lem: gap translation for single prediction} and the condition of this theorem we know that $r_k^{\epsilon}+q_k^{\epsilon,\delta}-1\in(0,\bar{\eta}_{k,\delta}]$ and $\bar{\eta}_{k,\delta}>0$, therefore the first part of the theorem follows immediately from Theorem \ref{thm: general majority vote}. To derive the second part, we note that $\bar{\eta}_{k,\delta}\geq r_k^{\epsilon}+q_k^{\epsilon,\delta}-1>4/5$, therefore the bound \eqref{eq: general finite-sample bound for large gap} from Corollary \ref{coro: general majority vote under extreme p} applies. To get from the bound \eqref{eq: general finite-sample bound for large gap} to the desired bound of this theorem, note that $1-p_k^{\max}\leq 1-r_k^{\epsilon}$ by Lemma \ref{lem: gap translation for multiple predictions} and also that $p_k^{\max}-\bar{\eta}_{k,\delta}=\max_{\theta\in\Theta\backslash\Theta^{\delta}}p_k(\theta)\leq 1-q_k^{\epsilon,\delta}$ from the proof of Lemma \ref{lem: gap translation for multiple predictions}, and hence $\max\prth{1-p_k^{\max}, p_k^{\max}-\bar{\eta}_{k,\delta}}\leq 1-\min(r_k^{\epsilon},q_k^{\epsilon,\delta})$.
        Given the retrieved solution set $\mathcal{S}$ and the chosen $\epsilon$, the rest of Phase II of Algorithm \ref{bagging majority vote: two phase} exactly performs Algorithm \ref{bagging majority vote: set estimator 2} on the restricted problem $\min_{\theta\in\mathcal{S}}\expect{l(\theta,z)}$ to obtain $\hat{\theta}_n$ with the data $\bm{z}_{\lfloor n/2 \rfloor+1:n}$, the set-valued learning algorithm \eqref{indicator for bagging with multiple predictions}, the chosen $\epsilon$ value and $k=k_2,B=B_2$.
        
        % Applying Lemma \eqref{lem: gap translation for multiple predictions} to the restricted problem, we see that if $r_k^{\epsilon,S}+q_k^{\epsilon,\delta,S}-1>4/5$, the bound \eqref{eq: general finite-sample bound for large gap 2} from Theorem \ref{thm: general majority vote 2} entails that

        % therefore Theorem \ref{thm: finite-sample bound for multiple predictions} automatically applies as stated in the theorem.

    To show the upper bound for the unconditional convergence probability $\prob{\hat{\theta}_n\notin \Theta^{2\delta}}$, note that
    \begin{equation*}
        \set{\mathcal{S}\cap \Theta^{\delta}\neq \emptyset}\cap \set{L(\hat{\theta}_n)\leq \min_{\theta\in\mathcal{S}} L(\theta)+\delta} \subseteq  \set{\hat{\theta}_n\in \Theta^{2\delta}},
    \end{equation*}
    and hence by union bound we can write
    \begin{equation}\label{eq: tail decomp}
        \prob{\hat{\theta}_n\notin \Theta^{2\delta}}\leq \prob{\mathcal{S}\cap \Theta^{\delta}= \emptyset}+\prob{L(\hat{\theta}_n)> \min_{\theta\in\mathcal{S}} L(\theta)+\delta}.
    \end{equation}
    $\prob{\mathcal{S}\cap \Theta^{\delta}= \emptyset}$ has a bound from Lemma \ref{lem: near optimality of retrieved solution pool}. We focus on the second probability.

    For a fixed retrieved subset $\mathcal{S}\subseteq \Theta$, define the tail of the maximum deviation on $\mathcal{S}$
    \begin{equation*}
        T_k^{\mathcal{S}}(\cdot):=\prob{\sup_{\theta\in\mathcal{S}}\abs{\frac{1}{k}\sum_{i=1}^kl(\theta,z_i)-L(\theta)}>\cdot}.
    \end{equation*}
    % and the counterpart of \eqref{SAA prob gap 2} for the restricted space $\mathcal{S}$
    % \begin{equation}
    %     \bar{\eta}_{k,\delta}^{\mathcal{S}}:=\max_{\theta\in\mathcal{S}^{\delta}}p_k(\theta) - \max_{\theta\in\mathcal{S}\backslash\mathcal{S}^{\delta}}p_k(\theta),
    %     % \bar{\eta}_{k,\delta}:=\sup_{\delta'\geq 0}\frac{q_{k,\delta'}}{\abs{\Theta^{\delta'}}} + \mathcal{E}_{k,\delta}-1
    % \end{equation}
    It is straightforward that $T_k^{\mathcal{S}}(\cdot)\leq T_k(\cdot)$ where $T_k$ is the tail of the maximum deviation over the whole space $\Theta$. Since $\prob{\epsilon\in[\underline{\epsilon},\overline{\epsilon}]}=1$, we have
    \begin{equation*}
        1-T_{k_2}^{\mathcal{S}}\prth{\dfrac{\epsilon}{2}}-T_{k_2}^{\mathcal{S}}\prth{\dfrac{\delta-\epsilon}{2}}\geq 1-T_{k_2}^{\mathcal{S}}\prth{\dfrac{\underline{\epsilon}}{2}}-T_{k_2}^{\mathcal{S}}\prth{\dfrac{\delta-\overline{\epsilon}}{2}}.
    \end{equation*}
    % for the retrieved solution set $\mathcal{S}$ that
    % \begin{equation}\label{eq: uniform lower bound for q}
    %     q_k^{\epsilon,\delta,\mathcal{S}}\geq q_k^{\overline{\epsilon},\delta,\mathcal{S}}\geq 1-T_k\prth{\dfrac{\delta-\overline{\epsilon}}{2}},
    % \end{equation}
    % and
    % \begin{equation}\label{eq: uniform lower bound for r}
    %     r_k^{\epsilon,\mathcal{S}}\geq r_k^{\underline{\epsilon},\mathcal{S}}\geq 1-T_k\prth{\dfrac{\underline{\epsilon}}{2}}.
    % \end{equation}
    If $T_{k_2}\prth{(\delta-\overline{\epsilon})/2}+T_{k_2}\prth{\underline{\epsilon}/2}<1/5$, we have $T_{k_2}^{\mathcal{S}}\prth{(\delta-\overline{\epsilon})/2}+T_{k_2}^{\mathcal{S}}\prth{\underline{\epsilon}/2}<1/5$ and subsequently $1-T_{k_2}^{\mathcal{S}}\prth{(\delta-\epsilon)/2}-T_{k_2}^{\mathcal{S}}\prth{\epsilon/2}>4/5$, and hence $\bar{\eta}_{k_2,\eta}\geq 1-T_{k_2}^{\mathcal{S}}\prth{(\delta-\epsilon)/2}-T_{k_2}^{\mathcal{S}}\prth{\epsilon/2}>4/5$ by Lemma \ref{lem: gap translation for multiple predictions} for Phase II of $\mathsf{ROVEs}$ conditioned on $\mathcal{S}$ and $\epsilon$, therefore the bound \eqref{eq: general finite-sample bound for large gap 2} from Theorem \ref{thm: general majority vote 2} applies. Using the inequalities \eqref{eq: pkmax lower bound for multiple predictions} and \eqref{eq: pkmax - eta upper bound} to upper bound the $\min\set{1-p_k^{\max},p_k^{\max}-\bar{\eta}_{k,\delta}}$ term in \eqref{eq: general finite-sample bound for large gap 2} gives
    \begin{eqnarray*}
        &&\prob{L(\hat{\theta}_n)>\min_{\theta\in\mathcal{S}} L(\theta)+\delta\big\vert \mathcal{S},\epsilon}\\
        &\leq &\abs{\mathcal{S}}\prth{3\min\set{e^{-2/5}, C_1\max\set{T_{k_2}^{\mathcal{S}}\prth{\dfrac{\underline{\epsilon}}{2}},T_{k_2}^{\mathcal{S}}\prth{\dfrac{\delta-\overline{\epsilon}}{2}}}}^{\frac{n}{2C_2k_2}}+\exp\prth{-\dfrac{B_2}{C_3}}}\\
        &= &\abs{\mathcal{S}}\prth{3\min\set{e^{-2/5}, C_1T_{k_2}^{\mathcal{S}}\prth{\dfrac{\min\set{\underline{\epsilon}, \delta-\overline{\epsilon}}}{2}}}^{\frac{n}{2C_2k_2}}+\exp\prth{-\dfrac{B_2}{C_3}}}\\
        &\leq &\abs{\mathcal{S}}\prth{3\min\set{e^{-2/5}, C_1T_{k_2}\prth{\dfrac{\min\set{\underline{\epsilon}, \delta-\overline{\epsilon}}}{2}}}^{\frac{n}{2C_2k_2}}+\exp\prth{-\dfrac{B_2}{C_3}}}.
    \end{eqnarray*}
    Further relaxing $\abs{\mathcal{S}}$ to $B_1$ and taking full expectation on both sides give
    {\small
    \begin{equation*}
        \prob{L(\hat{\theta}_n)>\min_{\theta\in\mathcal{S}} L(\theta)+\delta}\leq B_1\prth{3\min\set{e^{-2/5}, C_1T_{k_2}\prth{\dfrac{\min\set{\underline{\epsilon}, \delta-\overline{\epsilon}}}{2}}}^{\frac{n}{2C_2k_2}}+\exp\prth{-\dfrac{B_2}{C_3}}}.
    \end{equation*}}
    \hspace{-3pt}This leads to the desired bound \eqref{eq: finite-sample bound for data splitting two phase under large pmax} after the above bound is plugged into \eqref{eq: tail decomp} and the bound \eqref{eq: retrieval tail bound} from Lemma \ref{lem: near optimality of retrieved solution pool} is applied with $k=k_1$.
\end{proof}

\begin{proof}[Proof of Theorem \ref{thm: finite-sample bound for multiple predictions two phase splitting_formal} for $\mathsf{ROVE}$]
% For every non-empty subset $\mathcal{W}\subseteq\Theta$, we use the following counterpart of Definition \ref{def: nearly optimal sets and convergence probabilities}. Let
% \begin{equation}\label{true nearly optima set: restricted}
%     \mathcal{W}^{\delta}:=\left\{\theta\in\mathcal{W}:L(\theta)\leq \min_{\theta'\in\mathcal{W}}L(\theta')+\delta\right\}
% \end{equation}
% be the set of $\delta$-optimal solutions in the restricted decision space $\mathcal{W}$, and
% \begin{equation}\label{SAA nearly optima set: restricted}
%     \widehat{\mathcal{W}}_k^{\delta}:=\left\{\theta\in\mathcal{W}:\frac{1}{k}\sum_{i=1}^kl(\theta,z_i)= \min_{\theta'\in\mathcal{W}}\frac{1}{k}\sum_{i=1}^kl(\theta',z_i)+\delta\right\}
% \end{equation}
% be the set of $\delta$-optimal solutions of the restricted SAA with an i.i.d. data set of size $k$. For every $\epsilon\geq 0$ and $\delta\geq 0$ let
% \begin{equation}\label{SAA convergence inner prob: restricted}
% q_k^{\epsilon,\delta,\mathcal{W}}:=\prob{\widehat{\mathcal{W}}_k^{\epsilon}\subseteq \mathcal{W}^{\delta}},
% \end{equation}
% and
% \begin{equation}\label{SAA convergence outer prob: restricted}
% r_k^{\epsilon.\mathcal{W}}:=\prob{\mathcal{W}^0\subseteq \widehat{\mathcal{W}}_k^{\epsilon}},
% \end{equation}
% to be the inner and outer convergence probabilities for the restricted SAA.
    For every non-empty subset $\mathcal{W}\subseteq\Theta$ and $k_2$, we consider the indicator
    \begin{equation*}
        \mathbbm{1}_{k_2}^{\theta,\mathcal{W}, \epsilon}(z_1,\ldots,z_{k_2}):=\mathbbm{1}\prth{\frac{1}{k_2}\sum_{i=1}^{k_2}l(\theta,z_i)\leq \min_{\theta'\in\mathcal{W}}\frac{1}{k_2}\sum_{i=1}^{k_2}l(\theta',z_i)+\epsilon}\quad \text{for }\theta\in\mathcal{W},\epsilon\in[0,\delta/2],
    \end{equation*}
    which indicates whether a solution $\theta\in\mathcal{W}$ is $\epsilon$-optimal for the SAA formed by $\set{z_1,\ldots,z_{k_2}}$. Here we add $\epsilon$ and $\mathcal{W}$ to the superscript to emphasize its dependence on them. The counterparts of the solution probabilities $p_k,\hat p_k,\bar p_k$ for $\mathbbm{1}_{k_2}^{\theta,\mathcal{W}, \epsilon}$ are
    \begin{equation*}
    \begin{aligned}
        p_{k_2}^{\mathcal{W},\epsilon}(\theta)&:=\expect{\mathbbm{1}_{k_2}^{\theta,\mathcal{W},\epsilon}(z_1,\ldots,z_{k_2})},\\
        \hat p_{k_2}^{\mathcal{W},\epsilon}(\theta)&:=\reexpect{\mathbbm{1}_{k_2}^{\theta,\mathcal{W},\epsilon}(z_1^*,\ldots,z_{k_2}^*)},\\
        \bar p_{k_2}^{\mathcal{W},\epsilon}(\theta)&:=\dfrac{1}{B_2}\sum_{b=1}^{B_2}\mathbbm{1}_{k_2}^{\theta,\mathcal{W},\epsilon}(z_1^b,\ldots,z_{k_2}^b).
    \end{aligned}
    \end{equation*}
    We need to show the uniform convergence of these probabilities for $\epsilon\in[0,\delta/2]$. To do so, we define a slighted modified version of $\mathbbm{1}_{k_2}^{\theta,\mathcal{W}, \epsilon}$
    \begin{equation*}
        \mathbbm{1}_{k_2}^{\theta,\mathcal{W}, \epsilon-}(z_1,\ldots,z_{k_2}):=\mathbbm{1}\prth{\frac{1}{k_2}\sum_{i=1}^{k_2}l(\theta,z_i)< \min_{\theta'\in\mathcal{W}}\frac{1}{k_2}\sum_{i=1}^{k_2}l(\theta',z_i)+\epsilon}\quad \text{for }\theta\in\mathcal{W},\epsilon\in[0,\delta/2],
    \end{equation*}
    which indicates a strict $\epsilon$-optimal solution, and let $p_{k_2}^{\mathcal{W},\epsilon-},\hat p_{k_2}^{\mathcal{W},\epsilon-},\bar p_{k_2}^{\mathcal{W},\epsilon-}$ be the corresponding counterparts of solution probabilities. For any integer $m>1$ we construct brackets of size at most $1/m$ to cover the family of indicator functions $\{\mathbbm{1}_{k_2}^{\theta,\mathcal{W}, \epsilon}:\epsilon\in[0,\delta/2]\}$, i.e., let $m'=\lfloor p_{k_2}^{\mathcal{W},\delta/2}(\theta)m\rfloor$ and
    \begin{equation*}
    \begin{aligned}
    \epsilon_0&:=0,\\
        \epsilon_i&:=\inf\set{\epsilon\in[0,\delta/2]:p_{k_2}^{\mathcal{W},\epsilon}(\theta)\geq i/m}\quad\text{for }1\leq i\leq m',\\
        \epsilon_{m'+1}&:=\dfrac{\delta}{2},
    \end{aligned}
    \end{equation*}
    where we assume that $\epsilon_i,i=0,\ldots,m'+1$ are strictly increasing without loss of generality (otherwise we can delete duplicated values). Then for any $\epsilon\in[\epsilon_i,\epsilon_{i+1})$, we have that
    \begin{eqnarray*}
        \bar p_{k_2}^{\mathcal{W},\epsilon}(\theta)- p_{k_2}^{\mathcal{W},\epsilon}(\theta)&\leq& \bar p_{k_2}^{\mathcal{W},\epsilon_{i+1}-}(\theta)-p_{k_2}^{\mathcal{W},\epsilon_i}(\theta)\\
        &\leq& \bar p_{k_2}^{\mathcal{W},\epsilon_{i+1}-}(\theta)-p_{k_2}^{\mathcal{W},\epsilon_{i+1}-}(\theta)+p_{k_2}^{\mathcal{W},\epsilon_{i+1}-}(\theta)-p_{k_2}^{\mathcal{W},\epsilon_i}(\theta)\\
        &\leq &\bar p_{k_2}^{\mathcal{W},\epsilon_{i+1}-}(\theta)-p_{k_2}^{\mathcal{W},\epsilon_{i+1}-}(\theta)+\dfrac{1}{m}
    \end{eqnarray*}
    and that
    \begin{eqnarray*}
        \bar p_{k_2}^{\mathcal{W},\epsilon}(\theta)- p_{k_2}^{\mathcal{W},\epsilon}(\theta)&\geq& \bar p_{k_2}^{\mathcal{W},\epsilon_i}(\theta)-p_{k_2}^{\mathcal{W},\epsilon_{i+1}-}(\theta)\\
        &\geq& \bar p_{k_2}^{\mathcal{W},\epsilon_i}(\theta)-p_{k_2}^{\mathcal{W},\epsilon_i}(\theta)+p_{k_2}^{\mathcal{W},\epsilon_i}(\theta)-p_{k_2}^{\mathcal{W},\epsilon_{i+1}-}(\theta)\\
        &\geq &\bar p_{k_2}^{\mathcal{W},\epsilon_i}(\theta)-p_{k_2}^{\mathcal{W},\epsilon_i}(\theta)-\dfrac{1}{m}.
    \end{eqnarray*}
    Therefore
    {\small
    \begin{equation}
        \sup_{\epsilon\in[0,\delta/2]}\abs{\bar p_{k_2}^{\mathcal{W},\epsilon}(\theta)- p_{k_2}^{\mathcal{W},\epsilon}(\theta)}
        \leq\max_{0\leq i\leq m'+1}\max\set{\abs{\bar p_{k_2}^{\mathcal{W},\epsilon_i}(\theta)-p_{k_2}^{\mathcal{W},\epsilon_i}(\theta)},\abs{\bar p_{k_2}^{\mathcal{W},\epsilon_i-}(\theta)-p_{k_2}^{\mathcal{W},\epsilon_i-}(\theta)}}+\dfrac{1}{m}.\label{eq: bound for U-process supremum}
    \end{equation}}
    \hspace{-3pt}To show that the random variable in \eqref{eq: bound for U-process supremum} converges to $0$ in probability, we note that the U-statistic has the minimum variance among all unbiased estimators, in particular the following simple sample average estimators based on the first $\lfloor n/k_2 \rfloor \cdot k_2$ data
    \begin{equation*}
    \begin{aligned}
        \tilde p_{k_2}^{\mathcal{W},\epsilon}(\theta)&:=\dfrac{1}{\lfloor n/k_2 \rfloor}\sum_{i=1}^{\lfloor n/k_2 \rfloor}\mathbbm{1}_{k_2}^{\theta,\mathcal{W},\epsilon}(z_{k_2(i-1)+1},\ldots,z_{k_2 i}),\\
        \tilde p_{k_2}^{\mathcal{W},\epsilon-}(\theta)&:=\dfrac{1}{\lfloor n/k_2 \rfloor}\sum_{i=1}^{\lfloor n/k_2 \rfloor}\mathbbm{1}_{k_2}^{\theta,\mathcal{W},\epsilon-}(z_{k_2(i-1)+1},\ldots,z_{k_2 i}).
    \end{aligned}
    \end{equation*}
    Therefore we can write
    \begin{eqnarray*}
        &&\expect{\prth{\max_{0\leq i\leq m'+1}\max\set{\abs{\bar p_{k_2}^{\mathcal{W},\epsilon_i}(\theta)-p_{k_2}^{\mathcal{W},\epsilon_i}(\theta)},\abs{\bar p_{k_2}^{\mathcal{W},\epsilon_i-}(\theta)-p_{k_2}^{\mathcal{W},\epsilon_i-}(\theta)}}}^2}\\
        &\leq&\sum_{0\leq i\leq m'+1}\prth{\expect{\prth{\bar p_{k_2}^{\mathcal{W},\epsilon_i}(\theta)-p_{k_2}^{\mathcal{W},\epsilon_i}(\theta)}^2}+\expect{\prth{\bar p_{k_2}^{\mathcal{W},\epsilon_i-}(\theta)-p_{k_2}^{\mathcal{W},\epsilon_i-}(\theta)}^2}}\\
        &\leq& \sum_{0\leq i\leq m'+1}\prth{\expect{\prth{\bar p_{k_2}^{\mathcal{W},\epsilon_i}(\theta)-\hat p_{k_2}^{\mathcal{W},\epsilon_i}(\theta)}^2}+\expect{\prth{\hat p_{k_2}^{\mathcal{W},\epsilon_i}(\theta)-p_{k_2}^{\mathcal{W},\epsilon_i}(\theta)}^2}}+\\
        &&\sum_{0\leq i\leq m'+1}\prth{\expect{\prth{\bar p_{k_2}^{\mathcal{W},\epsilon_i-}(\theta)-\hat p_{k_2}^{\mathcal{W},\epsilon_i-}(\theta)}^2}+\expect{\prth{\hat p_{k_2}^{\mathcal{W},\epsilon_i-}(\theta)-p_{k_2}^{\mathcal{W},\epsilon_i-}(\theta)}^2}}\\
        &&\text{\ \ since $\bar p_{k_2}^{\mathcal{W},\epsilon_i}(\theta)$ and $\bar p_{k_2}^{\mathcal{W},\epsilon_i-}(\theta)$ are conditionally unbiased for $\hat p_{k_2}^{\mathcal{W},\epsilon_i}(\theta)$ and $\hat p_{k_2}^{\mathcal{W},\epsilon_i-}(\theta)$}\\
        &\leq& \sum_{0\leq i\leq m'+1}\prth{\expect{\reexpect{\prth{\bar p_{k_2}^{\mathcal{W},\epsilon_i}(\theta)-\hat p_{k_2}^{\mathcal{W},\epsilon_i}(\theta)}^2}}+\expect{\prth{\tilde p_{k_2}^{\mathcal{W},\epsilon_i}(\theta)-p_{k_2}^{\mathcal{W},\epsilon_i}(\theta)}^2}}+\\
        &&\sum_{0\leq i\leq m'+1}\prth{\expect{\reexpect{\prth{\bar p_{k_2}^{\mathcal{W},\epsilon_i-}(\theta)-\hat p_{k_2}^{\mathcal{W},\epsilon_i-}(\theta)}^2}}+\expect{\prth{\tilde p_{k_2}^{\mathcal{W},\epsilon_i-}(\theta)-p_k^{\mathcal{W},\epsilon_i-}(\theta)}^2}}\\
        &\leq& (m'+2)\prth{\dfrac{2}{B_2}+\dfrac{2}{\lfloor n/k_2 \rfloor}}\leq (m+2)\prth{\dfrac{2}{B_2}+\dfrac{4}{n/k_2}}.
    \end{eqnarray*}
    By Minkowski inequality, the supremum satisfies
    \begin{equation*}
        \expect{\sup_{\epsilon\in[0,\delta/2]}\abs{\bar p_{k_2}^{\mathcal{W},\epsilon}(\theta)- p_{k_2}^{\mathcal{W},\epsilon}(\theta)}}\leq \sqrt{(m+2)\prth{\dfrac{2}{B_2}+\dfrac{4}{n/k_2}}}+\frac{1}{m}.
    \end{equation*}
    Choosing $m$ such that $m\to\infty$, $m/B_2\to 0$ and $mk_2/n\to 0$ leads to the convergence $\sup_{\epsilon\in[0,\delta/2]}\abs{\bar p_{k_2}^{\mathcal{W},\epsilon}(\theta)- p_{k_2}^{\mathcal{W},\epsilon}(\theta)}\to 0$ in probability. Since $\Theta$ has finite cardinality and has a finite number of subsets, it also holds that
    \begin{equation}\label{uniform weak convergence of bar p}
        \sup_{\mathcal{W}\subseteq\Theta,\theta\in\mathcal{W}, \epsilon\in[0,\delta/2]}\abs{\bar p_{k_2}^{\mathcal{W},\epsilon}(\theta)- p_{k_2}^{\mathcal{W},\epsilon}(\theta)}\to 0\text{\ in probability}.
    \end{equation}

    Recall the bound \eqref{bound of sensitivity second term for one model} from the proof of Lemma \ref{lem: gap translation for multiple predictions} in Appendix \ref{subsec: proof of lemma lem: gap translation for multiple predictions}. Here we have the similar bound $\max_{\theta\in\mathcal{W}\backslash\mathcal{W}^{\delta}}p_{k_2}^{\mathcal{W},\epsilon}(\theta)\leq \prob{\widehat{\mathcal{W}}_{k_2}^{\epsilon}\not\subseteq \mathcal{W}^{\delta}}$, and hence
    % continues to hold with $\Theta$ replaced by $\mathcal{W}$, i.e., $\max_{\theta\in\mathcal{W}\backslash\mathcal{W}^{\delta}}p_{k_2}^{\mathcal{W},\epsilon}(\theta)\leq 1-q_k^{\epsilon, \delta,\mathcal{W}}$. Since $q_k^{\epsilon, \delta,\mathcal{W}}$ decreases in $\epsilon$ we have $\inf_{\epsilon\in[0,\delta/2]}q_k^{\epsilon, \delta,\mathcal{W}}=q_k^{\delta/2, \delta,\mathcal{W}}$, we can write
% On the other hand, since $(1/k)\cdot\sum_{i=1}^kl(\theta,z_i)\to L(\theta)$ in probability for each $\theta\in\mathcal{W}$, it holds that $q_k^{\delta/2, \delta,\mathcal{W}}\to 1$ as $k\to\infty$ for a fixed $\delta>0$. Therefore
% \begin{equation*}
% \sup_{\epsilon\in[0,\delta/2]}\max_{\theta\in\mathcal{W}\backslash\mathcal{W}^{\delta}}p_k^{\mathcal{W},\epsilon}(\theta)\leq 1-\inf_{\epsilon\in[0,\delta/2]}q_k^{\epsilon, \delta,\mathcal{W}} \leq 1-q_k^{\delta/2, \delta,\mathcal{W}}=\prob{\widehat{\mathcal{W}}_k^{\delta/2}\not\subseteq \mathcal{W}^{\delta}}.
% \end{equation*}
\begin{equation*}
\sup_{\epsilon\in[0,\delta/2]}\max_{\theta\in\mathcal{W}\backslash\mathcal{W}^{\delta}}p_k^{\mathcal{W},\epsilon}(\theta)\leq \sup_{\epsilon\in[0,\delta/2]}\prob{\widehat{\mathcal{W}}_{k_2}^{\epsilon}\not\subseteq \mathcal{W}^{\delta}}=\prob{\widehat{\mathcal{W}}_{k_2}^{\delta/2}\not\subseteq \mathcal{W}^{\delta}}.
\end{equation*}
We bound the probability $\prob{\widehat{\mathcal{W}}_{k_2}^{\delta/2}\not\subseteq \mathcal{W}^{\delta}}$ more carefully. We let
\begin{equation*}
\begin{aligned}
    \Delta_o&:=\min\set{L(\theta')-L(\theta):\theta,\theta'\in\Theta,L(\theta')>L(\theta)}>0,\\
    \hat{L}_{k_2}(\theta)&:=\frac{1}{k_2}\sum_{i=1}^{k_2}l(\theta,z_i),
\end{aligned}
\end{equation*}
and have
\begin{eqnarray*}
    &&\set{\widehat{\mathcal{W}}_{k_2}^{\delta/2}\not\subseteq \mathcal{W}^{\delta}}\\
    &\subseteq &\bigcup_{\theta,\theta'\in\mathcal{W}\text{ s.t. }L(\theta')-L(\theta)>\delta}\set{\hat{L}_{k_2}(\theta')\leq \hat{L}_{k_2}(\theta)+\dfrac{\delta}{2}}\\
    &\subseteq& \bigcup_{\theta,\theta'\in\Theta\text{ s.t. }L(\theta')-L(\theta)>\delta} \set{\hat{L}_{k_2}(\theta')-L(\theta')+L(\theta')-L(\theta)\leq \hat{L}_{k_2}(\theta)-L(\theta)+\dfrac{\delta}{2}}\\
    &\subseteq& \bigcup_{\theta,\theta'\in\Theta\text{ s.t. }L(\theta')-L(\theta)>\delta} \set{\hat{L}_{k_2}(\theta')-L(\theta')+\max\set{\Delta,\delta}\leq \hat{L}_{k_2}(\theta)-L(\theta)+\dfrac{\delta}{2}}\\
    &&\text{\ \ by the definition of $\Delta_o$}\\
    &\subseteq& \bigcup_{\theta,\theta'\in\Theta} \set{\hat{L}_{k_2}(\theta')-L(\theta')+\max\set{\Delta_o-\dfrac{\delta}{2},\dfrac{\delta}{2}}\leq \hat{L}_{k_2}(\theta)-L(\theta)}\\
    &\subseteq &\bigcup_{\theta,\theta'\in\Theta} \set{\hat{L}_{k_2}(\theta')-L(\theta')\leq-\max\set{\dfrac{\Delta_o}{2}-\dfrac{\delta}{4},\dfrac{\delta}{4}}\text{ or }\hat{L}_{k_2}(\theta)-L(\theta)\geq\max\set{\dfrac{\Delta_o}{2}-\dfrac{\delta}{4},\dfrac{\delta}{4}}}\\
    &\subseteq&\bigcup_{\theta\in\Theta} \set{\abs{\hat{L}_{k_2}(\theta)-L(\theta)}\geq \max\set{\dfrac{\Delta_o}{2}-\dfrac{\delta}{4},\dfrac{\delta}{4}}}\\
    &\subseteq&\bigcup_{\theta\in\Theta} \set{\abs{\hat{L}_{k_2}(\theta)-L(\theta)}\geq \dfrac{\Delta_o}{4}}\\
    &\subseteq&\set{\sup_{\theta\in\Theta}\abs{\hat{L}_{k_2}(\theta)-L(\theta)}\geq \dfrac{\Delta_o}{4}},
\end{eqnarray*}
where the last line holds because $\max\set{\Delta_o/2-\delta/4,\delta/4}\geq \Delta_o/4$. This gives
\begin{equation*}
\sup_{\epsilon\in[0,\delta/2]}\max_{\theta\in\mathcal{W}\backslash\mathcal{W}^{\delta}}p_{k_2}^{\mathcal{W},\epsilon}(\theta)\leq T_{k_2}\prth{\dfrac{\Delta_o}{4}}\to 0\text{\ \ as\ }k_2\to\infty.
\end{equation*}
We also have the trivial bound $\inf_{\epsilon\in[0,\delta/2]}\max_{\theta\in\mathcal{W}}p_{k_2}^{\mathcal{W},\epsilon}(\theta)=\max_{\theta\in\mathcal{W}}p_{k_2}^{\mathcal{W},0}(\theta)\geq 1/\abs{\mathcal{W}}$, where the inequality comes from the fact that $\sum_{\theta\in\mathcal{W}}p_{k_2}^{\mathcal{W},0}(\theta)\geq 1$. Now choose a $\underline{k}<\infty$ such that 
\begin{equation*}
    T_{k_2}\prth{\dfrac{\Delta_o}{4}}\leq \frac{1}{2\abs{\Theta}}\text{\ \  for all $k_2\geq \underline{k}$},
\end{equation*}
and we have for all $k_2\geq \underline{k}$ and all non-empty $\mathcal{W}\subseteq\Theta$ that
\begin{equation}
\begin{aligned}
\inf_{\epsilon\in[0,\delta/2]}\prth{\max_{\theta\in\mathcal{W}}p_{k_2}^{\mathcal{W},\epsilon}(\theta)-\max_{\theta\in\mathcal{W}\backslash\mathcal{W}^{\delta}}p_{k_2}^{\mathcal{W},\epsilon}(\theta)}&\geq \inf_{\epsilon\in[0,\delta/2]}\max_{\theta\in\mathcal{W}}p_{k_2}^{\mathcal{W},\epsilon}(\theta)-\sup_{\epsilon\in[0,\delta/2]}\max_{\theta\in\mathcal{W}\backslash\mathcal{W}^{\delta}}p_{k_2}^{\mathcal{W},\epsilon}(\theta)\\
    &\geq \dfrac{1}{\abs{\mathcal{W}}}-\frac{1}{2\abs{\Theta}}\geq \dfrac{1}{2\abs{\Theta}}.
\end{aligned}
\end{equation}
Due to the uniform convergence \eqref{uniform weak convergence of bar p}, we have
{\small
\begin{equation*}
    \min_{\mathcal{W}\subseteq\Theta}\inf_{\epsilon\in[0,\delta/2]}\prth{\max_{\theta\in\mathcal{W}}\bar p_{k_2}^{\mathcal{W},\epsilon}(\theta)-\max_{\theta\in\mathcal{W}\backslash\mathcal{W}^{\delta}}\bar p_{k_2}^{\mathcal{W},\epsilon}(\theta)}\to\min_{\mathcal{W}\subseteq\Theta}\inf_{\epsilon\in[0,\delta/2]}\prth{\max_{\theta\in\mathcal{W}} p_{k_2}^{\mathcal{W},\epsilon}(\theta)-\max_{\theta\in\mathcal{W}\backslash\mathcal{W}^{\delta}} p_{k_2}^{\mathcal{W},\epsilon}(\theta)}
\end{equation*}}
\hspace{-3pt}in probability, and hence
\begin{equation}\label{convergence of supremum probability error}
\prob{\min_{\mathcal{W}\subseteq\Theta}\inf_{\epsilon\in[0,\delta/2]}\prth{\max_{\theta\in\mathcal{W}}\bar p_{k_2}^{\mathcal{W},\epsilon}(\theta)-\max_{\theta\in\mathcal{W}\backslash\mathcal{W}^{\delta}}\bar p_{k_2}^{\mathcal{W},\epsilon}(\theta)}\leq 0}\to 0.
\end{equation}

Finally, we combine all the pieces to get
\begin{eqnarray*}
    &&\set{\hat{\theta}_n\not\in \Theta^{2\delta}}\\
    &\subseteq& \set{\mathcal{S}\cap \Theta^{\delta}=\emptyset}\cup \set{\hat{\theta}_n\not\in \mathcal{S}^{\delta}}\\
    &\subseteq& \set{\mathcal{S}\cap \Theta^{\delta}=\emptyset}\cup  \set{\max_{\theta\in\mathcal{S}}\bar p_{k_2}^{\mathcal{S},\epsilon}(\theta)-\max_{\theta\in\mathcal{S}\backslash\mathcal{S}^{\delta}}\bar p_{k_2}^{\mathcal{S},\epsilon}(\theta)\leq 0}\\
    &\subseteq& \set{\mathcal{S}\cap \Theta^{\delta}=\emptyset}\cup \set{\epsilon>\dfrac{\delta}{2}}\cup \set{\inf_{\epsilon\in[0,\delta/2]}\prth{\max_{\theta\in\mathcal{S}}\bar p_{k_2}^{\mathcal{S},\epsilon}(\theta)-\max_{\theta\in\mathcal{S}\backslash\mathcal{S}^{\delta}}\bar p_{k_2}^{\mathcal{S},\epsilon}(\theta)}\leq 0}\\
    &\subseteq& \set{\mathcal{S}\cap \Theta^{\delta}=\emptyset}\cup \set{\epsilon>\dfrac{\delta}{2}}\cup \set{\min_{\mathcal{W}\subseteq\Theta}\inf_{\epsilon\in[0,\delta/2]}\prth{\max_{\theta\in\mathcal{W}}\bar p_{k_2}^{\mathcal{W},\epsilon}(\theta)-\max_{\theta\in\mathcal{W}\backslash\mathcal{W}^{\delta}}\bar p_{k_2}^{\mathcal{W},\epsilon}(\theta)}\leq 0}.
\end{eqnarray*}
By Lemma \ref{lem: near optimality of retrieved solution pool} we have $\prob{\mathcal{S}\cap \Theta^{\delta}=\emptyset}\to 0$ under the conditions that $\limsup_{k\to\infty} \mathcal{E}_{k,\delta}<1$ and $k_1,n/k_1,B_1\to\infty$. Together with the condition $\prob{\epsilon\geq \delta/2}\to 0$ and \eqref{convergence of supremum probability error}, we conclude $\prob{\hat{\theta}_n\not\in \Theta^{2\delta}}\to 0$ by the union bound.
\end{proof}

% \newpage
\section{Proofs for Technical Lemmas}\label{sec: proof of technical lemmas}

\subsection{Proof of Lemma \ref{lem:MGF dominance}}\label{subsec: proof of lemma lem:MGF dominance}
By symmetry, we have that
\begin{equation*}
    U(z_1,\ldots,z_n) = \frac{1}{n!}\sum_{\mathrm{bijection}\ \pi:[n]\rightarrow [n]}\bar{\kappa}(z_{\pi(1)},\ldots,z_{\pi(n)}),
\end{equation*}
where we denote $[n]:= \{1,\ldots,n\}$.
Then, by the convexity of the exponential function and Jensen's inequality, we have that
\begin{eqnarray*}
    \expect{\exp(tU)}&=& \expect{\exp\left(t\cdot\frac{1}{n!}\sum_{\mathrm{bijection}\ \pi:[n]\rightarrow [n]}\bar{\kappa}(z_{\pi(1)},\ldots,z_{\pi(n)})\right)}\\
    &\leq& \expect{\frac{1}{n!}\sum_{\mathrm{bijection}\ \pi:[n]\rightarrow [n]}\exp\left(t\cdot\bar{\kappa}(z_{\pi(1)},\ldots,z_{\pi(n)})\right)}\\
    &=&\expect{\exp\left(t\cdot\bar{\kappa}(z_1,\ldots,z_n)\right)}.
\end{eqnarray*}
This completes the proof.
\qed

\subsection{Proof of Lemma \ref{lem:chernoff bound for p_k}}\label{subsec: proof of lemma lem:chernoff bound for p_k}
We first consider the direction $U-\expect{\kappa} \geq \epsilon$. Let
\begin{equation*}
\tilde \kappa^* := \dfrac{1}{\hat n} \sum_{i = 1}^{\hat n} \kappa^*(z_{k(i-1)+1},\ldots,z_{ki}),
\end{equation*}
and
\begin{equation*}
\tilde \kappa := \dfrac{1}{\hat n} \sum_{i = 1}^{\hat n} \kappa(z_{k(i-1)+1},\ldots,z_{ki};\omega_i),
\end{equation*}
where we use the shorthand notation $\hat n := \lfloor\frac{n}{k}\rfloor$, and $\omega_i$'s are mutually independent and also independent from $\set{z_1,\ldots,z_n}$.
Then, since $\expect{\kappa}=\expect{\kappa^*}$, for all $t > 0$ it holds that
\begin{equation}\label{eq:hat_pk_geq}
\begin{aligned}
\prob{U-\expect{\kappa} \geq \epsilon} 
&= \prob{\exp\left(tU\right) \geq \exp\prth{t\prth{\expect{\kappa}+\epsilon}}}\\
&\overset{(i)}{\leq} \exp\prth{-t\prth{\expect{\kappa}+\epsilon}}\cdot \expect{\exp\prth{tU}}\\
&\overset{(ii)}{\leq} \exp\prth{-t\prth{\expect{\kappa}+\epsilon}}\cdot \expect{\exp\prth{t\tilde \kappa^*}}\\
&\overset{(iii)}{\leq} \exp\prth{-t\prth{\expect{\kappa}+\epsilon}}\cdot \expect{\exp\prth{t\tilde \kappa}},
% &= \exp\prth{-t\epsilon}\cdot \expect{\exp\prth{t\prth{\tilde p_k(\theta) - p_k(\theta)}}},
\end{aligned}
\end{equation}
where we apply the Markov inequality in $(i)$, step $(ii)$ is due to Lemma \ref{lem:MGF dominance}, and step $(iii)$ uses Jensen's inequality and the convexity of the exponential function.
Due to independence, $\tilde \kappa$ can be viewed as the sample average of $\hat n$ i.i.d. Bernoulli random variables, i.e., $\tilde \kappa \sim \frac{1}{\hat n}\sum_{i= 1}^{\hat n}\operatorname{Bernoulli}\prth{\expect{\kappa}}$.
Hence, we have that
\begin{equation}\label{eq:mgf_simplify_tilde}
\begin{aligned}
\expect{\exp\prth{t\tilde \kappa}} 
&= \expect{\exp\prth{\frac{t}{\hat n}\sum_{i=1}^{\hat n} \operatorname{Bernoulli}\prth{\expect{\kappa}}}} \\
&= \prth{\expect{\exp\prth{\frac{t}{\hat n} \operatorname{Bernoulli}\prth{\expect{\kappa}}}}}^{\hat n}\\
& = \brac{(1-\expect{\kappa}) +\expect{\kappa}\cdot \exp\prth{\dfrac{t}{\hat n}}}^{\hat n},
\end{aligned}
\end{equation}
where we use the moment-generating function of Bernoulli random variables in the last line.
Substituting \eqref{eq:mgf_simplify_tilde} into \eqref{eq:hat_pk_geq}, we have that
\begin{equation}\label{eq:hat_pk_geq_upperbound}
\prob{U-\expect{\kappa} \geq \epsilon}\leq \exp\prth{-t\prth{\expect{\kappa}+\epsilon}}\cdot \brac{(1-\expect{\kappa}) + \expect{\kappa}\cdot \exp\prth{\dfrac{t}{\hat n}}}^{\hat n} =: f(t).
\end{equation}
Now, we consider minimizing $f(t)$ for $t>0$. Let $g(t) = \log f(t)$, then it holds that
\begin{equation*}
g^\prime(t) = -(\expect{\kappa} + \epsilon) + \dfrac{\expect{\kappa}\cdot \exp\prth{\frac{t}{\hat n}}}{(1-\expect{\kappa}) + \expect{\kappa}\cdot \exp\prth{\frac{t}{\hat n}}}.
\end{equation*}
By setting $g^\prime(t) = 0$, it is easy to verify that the minimum point of $f(t)$, denoted by $t^\star$, satisfies that
\begin{equation}\label{eq: mgf_optimal}
\begin{aligned}
&\expect{\kappa}\cdot \exp\prth{\frac{t}{\hat n}}\cdot (1-\expect{\kappa} - \epsilon) = (1-\expect{\kappa})\cdot (\expect{\kappa}+\epsilon)\\
\Leftrightarrow\quad & \exp(t) = \brac{\dfrac{(1-\expect{\kappa})\cdot (\expect{\kappa}+\epsilon)}{\expect{\kappa}\cdot (1-\expect{\kappa}-\epsilon)}}^{\hat n}.
\end{aligned}
\end{equation}
Substituting \eqref{eq: mgf_optimal} into \eqref{eq:hat_pk_geq_upperbound} gives
\begin{eqnarray}
\notag \prob{U-\expect{\kappa} \geq \epsilon} 
&\leq&  \prth{\dfrac{1-\expect{\kappa}}{1-\expect{\kappa}-\epsilon}}^{\hat n}\cdot \brac{\dfrac{\expect{\kappa}\cdot \prth{1-\expect{\kappa}-\epsilon}}{\prth{1-\expect{\kappa}}\prth{\expect{\kappa}+\epsilon}}}^{\hat n \prth{\expect{\kappa}+\epsilon}}\\
\notag & = &\brac{\prth{\dfrac{1-\expect{\kappa}}{1-\expect{\kappa}-\epsilon}}^{1-\expect{\kappa}-\epsilon}\cdot \prth{\dfrac{\expect{\kappa}}{\expect{\kappa}+\epsilon}}^{\expect{\kappa}+\epsilon}}^{\hat n}\\
&= &\exp\prth{-\hat n \cdot D_{\operatorname{KL}}\prth{\expect{\kappa}+\epsilon \| \expect{\kappa}}}.\label{final prob bound of the > side}
\end{eqnarray}
Since $n/k\leq 2\hat n$, the first bound immediately follows from \eqref{final prob bound of the > side}.

Since $D_{\operatorname{KL}}(p\| q)=D_{\operatorname{KL}}(1-p\| 1-q)$, the bound for the reverse side $U-\expect{\kappa} \leq -\epsilon$ then follows by applying the first bound to the flipped binary kernel $1-\kappa$ and $1-U$. This completes the proof of Lemma \ref{lem:chernoff bound for p_k}.
\qed

\subsection{Proof of Lemma \ref{lemma: DKL}}\label{subsec: proof of lemma lemma: DKL}
% To prove inequality \eqref{KL lower bound: squared}, we consider the case $0\leq q < p \leq 1$ without loss of generality.
% We perform the second-order Taylor expansion of $\Dkl(p\|q)$ at $q$:
% \begin{equation}\label{eq: dkl_2ndtylor}
% \Dkl(p\| q ) = \Dkl(q\|q) + \dfrac{\partial \Dkl(q\| q )}{\partial p}\cdot (p-q) + \dfrac{\partial^2 \Dkl(r\| q )}{\partial p^2} \cdot \dfrac{(p-q)^2}{2},
% \end{equation}
% where $r\in [q,p]$. It is easy to compute that 
% \begin{equation*}
% \dfrac{\partial \Dkl(p\| q )}{\partial p} = \ln \dfrac{p}{q} - \ln\dfrac{1-p}{1-q},\qquad  \dfrac{\partial^2 \Dkl(p\| q )}{\partial p^2} = \dfrac{1}{p(1-p)}.
% \end{equation*}
% Since $\Dkl(q\| q) = \frac{\partial \Dkl(q\| q )}{\partial p} = 0$, it follows from \eqref{eq: dkl_2ndtylor} that
% \begin{equation}
%  \Dkl(p\| q )\geq \min_{r\in [q,p]}\set{\dfrac{\partial^2 \Dkl(r\| q )}{\partial p^2}}  \cdot \dfrac{(p-q)^2}{2} = \dfrac{(p-q)^2}{2\max_{r\in [q,p]} r(1-r)}.
% \end{equation}
% The other case $0\leq p < q \leq 1$ of \eqref{KL lower bound: squared} can be similarly derived. 
To show \eqref{KL lower bound: ratio}, some basic calculus shows that for any fixed $q$, the function $g(p):=(1-p)\ln\frac{1-p}{1-q}$ is convex in $p$, and we have that
\begin{equation*}
    g(q) = 0,g'(q) = -1.
\end{equation*}
Therefore $g(p)\geq g(q)+g'(q)(p-q)=q-p$, which implies \eqref{KL lower bound: ratio} immediately.

The lower bound \eqref{KL lower bound: p close to 1/2} follows from
\begin{eqnarray*}
    \Dkl(p\| q )&\geq&  -p\ln q-(1-p)\ln (1-q) + \min_{p\in [\gamma, 1-\gamma]}\{p\ln p + (1-p)\ln(1-p)\}\\
    &\geq& -\gamma \ln q-\gamma \ln (1-q) -\ln 2 = -\ln (2(q(1-q))^{\gamma}).
\end{eqnarray*}
This completes the proof of Lemma \ref{lemma: DKL}.
\qed

\subsection{Proof of Lemma \ref{lem:optima set approximation}}\label{subsec: proof of lemma lem:optima set approximation}
By Definition \ref{def: S_delta_hat}, we observe the following equivalence
\begin{equation*}
\set{\widehat{\mathcal{P}}^{\epsilon}_k\not\subseteq\mathcal{P}^{\delta}_k} = 
\bigcup_{\theta\in \Theta\backslash \mathcal{P}^{\delta}_k} \set{\theta \in \widehat{\mathcal{P}}^{\epsilon}_k} = \bigcup_{\theta\in \Theta\backslash \mathcal{P}^{\delta}_k} \set{\hp_k(\theta) \geq \hp_k\prth{\theta_k^{\max}} -\epsilon}.
\end{equation*}
Hence, by the union bound, it holds that
\begin{equation*}
\prob{\widehat{\mathcal{P}}^{\epsilon}_k\not\subseteq\mathcal{P}^{\delta}_k} \leq \sum_{\theta\in \Theta\backslash \mathcal{P}^{\delta}_k} \prob{\hp_k(\theta) \geq \hp_k\prth{\theta_k^{\max}} -\epsilon}.
\end{equation*}
We further bound the probability $\prob{\set{\hp_k(\theta) \geq \hp_k\prth{\theta_k^{\max}} -\epsilon}}$ as follows
\begin{eqnarray}\label{union bound for crossing}
&&\prob{\hp_k(\theta) \geq \hp_k\prth{\theta_k^{\max}} -\epsilon}\notag\\
&\leq&\prob{\set{\hp_k(\theta)\geq p_k(\theta_k^{\max})-\dfrac{\delta+\epsilon}{2}}\cap \set{\hp_k\prth{\theta_k^{\max}}\leq p_k(\theta_k^{\max})-\dfrac{\delta-\epsilon}{2}}}\\
&\leq& \prob{\hp_k(\theta)\geq p_k(\theta_k^{\max})-\dfrac{\delta+\epsilon}{2}} +\prob{\hp_k\prth{\theta_k^{\max}}\leq p_k(\theta_k^{\max})-\dfrac{\delta-\epsilon}{2}}.\notag
\end{eqnarray}
On one hand, the first probability in \eqref{union bound for crossing} is solely determined by and increasing in $p_k(\theta)=\expect{\hp_k(\theta)}$. On the other hand, we have $p_k(\theta)<p_k\prth{\theta_k^{\max}} - \delta$ for every $\theta\in  \Theta\backslash\mathcal{P}^{\delta}_k$ by the definition of $\mathcal{P}^{\delta}_k$. Therefore we can slightly abuse the notation to write
\begin{eqnarray*}
    \prob{\hp_k(\theta) \geq \hp_k\prth{\theta_k^{\max}} -\epsilon}&\leq& \prob{\hp_k(\theta)\geq p_k(\theta_k^{\max})-\dfrac{\delta+\epsilon}{2}\Big\vert p_k(\theta)=p_k(\theta_k^{\max})-\delta}\\
    &&\hspace{2ex}+\prob{\hp_k\prth{\theta_k^{\max}}\leq p_k(\theta_k^{\max})-\dfrac{\delta-\epsilon}{2}}\\
    &\leq& \prob{\hp_k(\theta)-p_k(\theta) \geq \dfrac{\delta-\epsilon}{2}\Big\vert p_k(\theta)=p_k(\theta_k^{\max})-\delta}\\
    &&\hspace{2ex}+\prob{\hp_k\prth{\theta_k^{\max}}-p_k(\theta_k^{\max})\leq -\dfrac{\delta-\epsilon}{2}}.
\end{eqnarray*}
Note that, with $\kappa(z_1,\ldots,z_k;\omega):=\mathbf{1}\prth{\theta\in \mathbb{A}(z_1,\ldots,z_k;\omega)}$, the probability $\hp_k(\theta)$ can be viewed as a U-statistic with the kernel $\kappa^*(z_1,\ldots,z_k):=\expect{\kappa(z_1,\ldots,z_k;\omega)\vert z_1,\ldots,z_k}$.
% \begin{equation*}
% \begin{aligned}
% \hat{p}_k(\theta) &=\frac{k!}{n(n-1)\cdots(n-k+1)}\sum_{1\leq i_1<\ldots<i_k\leq n}\mathbbm{1}(\theta\in \Theta_k(z_{i_1},\ldots,z_{i_k})).
% \end{aligned}
% \end{equation*}
A similar representation holds for $\hp_k\prth{\theta_k^{\max}}$ as well. Therefore, we can apply Lemma \ref{lem:chernoff bound for p_k} to conclude that
\begin{equation*}
\begin{aligned}
\prob{\widehat{\mathcal{P}}^{\epsilon}_k\not\subseteq\mathcal{P}^{\delta}_k}
&\leq \sum_{\theta\in \Theta\backslash \mathcal{P}^{\delta}_k} \prob{\hp_k(\theta) \geq \hp_k\prth{\theta_k^{\max}} -\epsilon}\\
&\leq \abs{\Theta\backslash \mathcal{P}^{\delta}_k}\left[\prob{\hp_k(\theta) - p_k(\theta) \geq \dfrac{\delta-\epsilon}{2}\Big\vert p_k(\theta)=p_k(\theta_k^{\max})-\delta}\right.\\
&\hspace{12ex}+ \left.\prob{p_k\prth{\theta_k^{\max}} - \hp_k\prth{\theta_k^{\max}}   \leq -\dfrac{\delta-\epsilon}{2}}\right]\\
&\leq \abs{\Theta} \left[\exp\prth{-\dfrac{n}{2k} \cdot D_{\operatorname{KL}}\prth{p_k\prth{\theta_k^{\max}}-\delta+\dfrac{\delta-\epsilon}{2} \Big\| p_k\prth{\theta_k^{\max}}-\delta}}\right. \\
&\hspace{1.9cm}+ \left.\exp\prth{-\dfrac{n}{2k} \cdot D_{\operatorname{KL}}\prth{p_k\prth{\theta_k^{\max}}-\dfrac{\delta-\epsilon}{2} \Big\| p_k\prth{\theta_k^{\max}}}}\right],
\end{aligned}
\end{equation*}
which completes the proof of Lemma \ref{lem:optima set approximation}.
% Inequality \eqref{eq: lowerbound_setbelong_2} can be further derived by applying Lemma \ref{lemma: DKL} to the right-hand side of \eqref{eq: lowerbound_setbelong_1}, where we note that $\max_{r\in [q,p]} r(1-r) \leq p$ if $0\leq q < p \leq 1/2$.
% This completes the proof of Lemma \ref{lem:optima set approximation}.
\qed

\subsection{Proof of Lemma \ref{lem:optimality conditioned on data}}\label{subsec: proof of lemma lem:optimality conditioned on data}
We observe that $\bar p_k(\theta)$ is an conditionally unbiased estimator for $\hat p_k(\theta)$, i.e., $\reexpect{\bar p_k(\theta)} = \hp_k(\theta)$. We can express the difference between $\bar p_k(\theta)$ and $\bar p_k(\theta_k^{\max})$ as the sample average
\begin{equation*}
\bar p_k(\theta)   - \bar p_k(\theta_k^{\max}) = \dfrac{1}{B} \sum_{b= 1}^B \brac{\mathbbm{1}(\theta\in \mathbb{A}(z_1^b,\ldots,z_k^b)) - \mathbbm{1}(\theta_k^{\max}\in \mathbb{A}(z_1^b,\ldots,z_k^b))},
\end{equation*}
whose expectation is equal to $\hp_k(\theta)   - \hp_k(\theta^{\max}_k)$. We denote by
\begin{equation*}
    \mathbbm{1}_{\theta}^*:=\mathbbm{1}(\theta\in \mathbb{A}(z_1^*,\ldots,z_k^*))\text{\  for }\theta\in\Theta
\end{equation*}
for convenience, where $(z_1^*,\ldots,z_k^*)$ represents a random subsample. Then by Bernstein's inequality, we have every $t\geq 0$ that
\begin{equation}\label{eq: bernstein bound to event diff}
\begin{aligned}
    \reprob{\bar{p}_k(\theta)- \bar{p}_k(\hat{\theta}_k^{\max})-(\hat p_k(\theta)-\hat p_k(\theta_k^{\max}))\geq t}
    \leq &\exp\prth{-B\cdot \dfrac{t^2}{2\mathrm{Var}_*(\mathbbm{1}_{\theta}^* - \mathbbm{1}_{\theta_k^{\max}}^*)+4/3\cdot t}}.
\end{aligned}
\end{equation}
Since
\begin{eqnarray*}
    \mathrm{Var}_*(\mathbbm{1}_{\theta}^* - \mathbbm{1}_{\theta_k^{\max}}^*)&\leq& \reexpect{(\mathbbm{1}_{\theta}^* - \mathbbm{1}_{\theta_k^{\max}}^*)^2}\leq \hat p_k(\theta)+\hat p_k(\theta_k^{\max})\leq 2\hat p_k(\theta_k^{\max}),
\end{eqnarray*}
and
\begin{eqnarray*}
    \mathrm{Var}_*(\mathbbm{1}_{\theta}^* - \mathbbm{1}_{\theta_k^{\max}}^*)&\leq&\mathrm{Var}_*(1-\mathbbm{1}_{\theta}^*- 1+\mathbbm{1}_{\theta_k^{\max}}^*)\\
    &\leq&\reexpect{(1-\mathbbm{1}_{\theta}^* -1+ \mathbbm{1}_{\theta_k^{\max}}^*)^2}\\
    &\leq& 1-\hat p_k(\theta)+1-\hat p_k(\theta_k^{\max})\leq 2(1-\hat p_k(\theta)),
\end{eqnarray*}
we have $\mathrm{Var}_*(\mathbbm{1}_{\theta}^* - \mathbbm{1}_{\theta_k^{\max}}^*)\leq 2\min\set{\hat p_k(\theta_k^{\max}),1-\hat p_k(\theta)}$. Substituting this bound to \eqref{eq: bernstein bound to event diff} and taking $t=\hat p_k(\theta_k^{\max})-\hat p_k(\theta)$ lead to
\begin{eqnarray*}
    &&\reprob{\bar{p}_k(\theta)- \bar{p}_k(\hat{\theta}_k^{\max})\geq 0}\\
    &\leq& \exp\prth{-B\cdot \dfrac{(\hat p_k(\theta_k^{\max})-\hat p_k(\theta))^2}{4\min\set{\hat p_k(\theta_k^{\max}),1-\hat p_k(\theta)}+4/3\cdot (\hat p_k(\theta_k^{\max})-\hat p_k(\theta))}}\\
    &\leq& \exp\prth{-B\cdot \dfrac{(\hat p_k(\theta_k^{\max})-\hat p_k(\theta))^2}{4\min\set{\hat p_k(\theta_k^{\max}),1-\hat p_k(\theta_k^{\max})}+16/3\cdot (\hat p_k(\theta_k^{\max})-\hat p_k(\theta))}}\\
    &\leq& \exp\prth{-\dfrac{B}{6}\cdot \dfrac{(\hat p_k(\theta_k^{\max})-\hat p_k(\theta))^2}{\min\set{\hat p_k(\theta_k^{\max}),1-\hat p_k(\theta_k^{\max})}+\hat p_k(\theta_k^{\max})-\hat p_k(\theta)}}.
\end{eqnarray*}
% \begin{equation*}
%     \reprob{\bar{p}_k(\theta)\geq \bar{p}_k(\hat{\theta}_k^{\max})}\leq \exp\prth{-B\max\prth{\sqrt{\hat{p}_k^{\max}}-\sqrt{\hat{p}_k(\theta)},\sqrt{1-\hat{p}_k(\theta)}-\sqrt{1-\hat{p}_k^{\max}}}^2}.
%     % \label{conditional maximizing probability bound 1-p}
%     % &\leq&\exp\prth{-B\prth{\sqrt{1-\hat{p}_k(\theta)}-\sqrt{1-\hat{p}_k^{\max}}}^2}.\label{conditional maximizing probability bound 1-p}
% \end{equation*}
Therefore, we have that
\begin{equation*}
\begin{aligned}
\mbb{P}_*\prth{\hat{\theta}_n\notin \widehat{\mathcal{P}}^{\epsilon}_k} 
&= \mbb{P}_*\prth{\bigcup_{\theta\in \Theta\backslash \widehat{\mathcal{P}}^{\epsilon}_k}\set{\bar p_k(\theta)   = \max_{\theta'\in\Theta}\bar{p}_k(\theta')} }\\
&\leq \sum_{\theta\in \Theta\backslash \widehat{\mathcal{P}}^{\epsilon}_k} \mbb{P}_*\prth{\bar p_k(\theta)   = \max_{\theta'\in\Theta}\bar{p}_k(\theta')}\\
&\leq \sum_{\theta\in \Theta\backslash \widehat{\mathcal{P}}^{\epsilon}_k} \reprob{\bar{p}_k(\theta)\geq \bar{p}_k(\theta_k^{\max})}\\
&\leq \sum_{\theta\in \Theta\backslash \widehat{\mathcal{P}}^{\epsilon}_k} \exp\prth{-\dfrac{B}{6}\cdot \dfrac{(\hat p_k(\theta_k^{\max})-\hat p_k(\theta))^2}{\min\set{\hat p_k(\theta_k^{\max}),1-\hat p_k(\theta_k^{\max})}+\hat p_k(\theta_k^{\max})-\hat p_k(\theta)}}.\\
\end{aligned}
\end{equation*}
Note that the function $x^2/(\min\set{\hat p_k(\theta_k^{\max}),1-\hat p_k(\theta_k^{\max})}+x)$ in $x\in[0,1]$ is monotonically increasing and that $\hat p_k(\theta_k^{\max})-\hat p_k(\theta)>\epsilon$ for all $\theta\in \Theta\backslash \widehat{\mathcal{P}}^{\epsilon}_k$. Therefore, we can further bound the probability as
\begin{eqnarray*}
    \mbb{P}_*\prth{\hat{\theta}_n\notin \widehat{\mathcal{P}}^{\epsilon}_k} &\leq& \abs{\Theta\backslash \widehat{\mathcal{P}}^{\epsilon}_k}\cdot \exp\prth{-\dfrac{B}{6}\cdot \dfrac{\epsilon^2}{\min\set{\hat p_k^{\max},1-\hat p_k^{\max}}+\epsilon}}.
\end{eqnarray*}
Noting that $\abs{\Theta\backslash \widehat{\mathcal{P}}^{\epsilon}_k}\leq \abs{\Theta}$ completes the proof of Lemma \ref{lem:optimality conditioned on data}.
\qed

\subsection{Proof of Lemma \ref{lem: near optimality of retrieved solution pool}}\label{subsec: proof of lemma lem: near optimality of retrieved solution pool}
     % Denote by $\hat{\theta}_k^*:=\hat{\theta}_k(z_1^*,\ldots,z_k^*)$ an optimal solution output by the training algorithm on a resampled data set of size $k$, and by $\hat{\theta}_k:=\hat{\theta}_k(z_1,\ldots,z_k)$ an optimal solution output by the training algorithm based on an i.i.d. data set of size $k$. 
     Let $(z_1^*,\ldots,z_k^*)$ be a random subsample and $\mbb{P}_*$ be the probability with respect to the subsampling randomness conditioned on the data and the algorithmic randomness. Consider the two probabilities
    \begin{equation*}
            \prob{\mathcal{A}(z_1,\ldots,z_k)\in\Theta^{\delta}},\ \reprob{\mathcal{A}(z_1^*,\ldots,z_k^*)\in\Theta^{\delta}}.
    \end{equation*}
    We have $1-\mathcal{E}_{k,\delta}=\prob{\mathcal{A}(z_1,\ldots,z_k)\in\Theta^{\delta}}$,
    % \begin{equation*}
    %     \mathcal{E}_{k,\delta}:=\prob{\hat{\theta}_k\in\Theta^{\delta}}\geq \prob{\widehat{\Theta}_k^0\subseteq \Theta^{\delta}} = q_k^{0,\delta}
    % \end{equation*}
    % by definition of $q_k^{0,\delta}$. 
    and the conditional probability
    \begin{equation*}
        \prob{\mathcal{S}\cap \Theta^{\delta}=\emptyset\Big\rvert \reprob{\mathcal{A}(z_1^*,\ldots,z_k^*)\in\Theta^{\delta}}}=\prth{1-\reprob{\mathcal{A}(z_1^*,\ldots,z_k^*)\in\Theta^{\delta}}}^{B_1}.
    \end{equation*}
    Therefore we can write
    \begin{eqnarray}
        \notag\prob{\mathcal{S}\cap \Theta^{\delta}=\emptyset}&=&\expect{\prth{1-\reprob{\mathcal{A}(z_1^*,\ldots,z_k^*)\in\Theta^{\delta}}}^{B_1}}\\
        &\leq& \prob{\reprob{\mathcal{A}(z_1^*,\ldots,z_k^*)\in\Theta^{\delta}}< \frac{1-\mathcal{E}_{k,\delta}}{e}}+\prth{1-\frac{1-\mathcal{E}_{k,\delta}}{e}}^{B_1},\label{bound decomp}
    \end{eqnarray}
    where $e$ is the base of the natural logarithm. Applying Lemma \ref{lem:chernoff bound for p_k} with $\kappa(z_1,\ldots,z_k;\omega):=\mathbbm{1}\prth{\mathcal{A}(z_1,\ldots,z_k;\omega)\in \Theta^{\delta}}$ gives
    \begin{equation*}
        \prob{\reprob{\mathcal{A}(z_1^*,\ldots,z_k^*)\in\Theta^{\delta}} < \frac{1-\mathcal{E}_{k,\delta}}{e}}\leq \exp\prth{-\frac{n}{2k}\cdot\Dkl\prth{\frac{1-\mathcal{E}_{k,\delta}}{e}\Big\| 1-\mathcal{E}_{k,\delta}}}.
    \end{equation*}
    Further applying the bound \eqref{KL lower bound: ratio} from Lemma \ref{lemma: DKL} to the KL divergence on the right-hand side leads to
    \begin{equation*}
        \Dkl\prth{\frac{1-\mathcal{E}_{k,\delta}}{e}\Big\| 1-\mathcal{E}_{k,\delta}}\geq \frac{1-\mathcal{E}_{k,\delta}}{e}\ln\frac{1}{e}+1-\mathcal{E}_{k,\delta}-\frac{1-\mathcal{E}_{k,\delta}}{e}= \prth{1-\frac{2}{e}}(1-\mathcal{E}_{k,\delta}),
    \end{equation*}
    and
    \begin{eqnarray*}
        &&\Dkl\prth{\frac{1-\mathcal{E}_{k,\delta}}{e}\Big\| 1-\mathcal{E}_{k,\delta}}\\
        &=&\Dkl\prth{1-\frac{1-\mathcal{E}_{k,\delta}}{e}\Big\| \mathcal{E}_{k,\delta}}\\
        &\geq &\prth{1-\frac{1-\mathcal{E}_{k,\delta}}{e}}\ln\frac{1-(1-\mathcal{E}_{k,\delta})/e}{\mathcal{E}_{k,\delta}}-(1-\mathcal{E}_{k,\delta})+\frac{1-\mathcal{E}_{k,\delta}}{e}\text{\ \ by bound \eqref{KL lower bound: ratio}}\\
        &\geq & \prth{1-\frac{1-\mathcal{E}_{k,\delta}}{e}}\ln\prth{1-\frac{1-\mathcal{E}_{k,\delta}}{e}}- \prth{1-\frac{1}{e}}\ln \mathcal{E}_{k,\delta}-1+\frac{1}{e}\\
        &\geq &\prth{1-\frac{1}{e}}\ln\prth{1-\frac{1}{e}}- \prth{1-\frac{1}{e}}\ln \mathcal{E}_{k,\delta}-1+\frac{1}{e}\\
        &=&\prth{1-\frac{1}{e}}\ln\frac{e-1}{e^2\mathcal{E}_{k,\delta}}.
    \end{eqnarray*}
    Combining the two bounds for the KL divergence we have
$$
\begin{aligned}
&\ \prob{\reprob{\mathcal{A}(z_1^*,\ldots,z_k^*)\in\Theta^{\delta}} < \frac{1-\mathcal{E}_{k,\delta}}{e}}\\
\leq&\ \min\set{\exp\prth{-\frac{n}{2k}\cdot\prth{1-\frac{2}{e}}(1-\mathcal{E}_{k,\delta})},\prth{\frac{e^2\mathcal{E}_{k,\delta}}{e-1}}^{(1-1/e)\frac{n}{2k}}}.
\end{aligned}
$$ 
Note that the second term on the right-hand side of \eqref{bound decomp} satisfies that $\prth{1-(1-\mathcal{E}_{k,\delta})/e}^{B_1}\leq \exp\prth{-B_1(1-\mathcal{E}_{k,\delta})/e}$.
    Thus, we derive that
    \begin{equation*}
    \begin{aligned}
        &\ \prob{\mathcal{S}\cap \Theta^{\delta}=\emptyset}\\
        \leq&\ \min\set{\exp\prth{-\frac{n}{2k}\cdot\prth{1-\frac{2}{e}}(1-\mathcal{E}_{k,\delta})},\prth{\frac{e^2\mathcal{E}_{k,\delta}}{e-1}}^{(1-1/e)\frac{n}{2k}}} + \exp\prth{-\frac{B_1(1-\mathcal{E}_{k,\delta})}{e}}\\
        \leq&\ \min\set{\exp\prth{-\dfrac{1-2/e}{1-1/e}\cdot (1-\mathcal{E}_{k,\delta})},\frac{e^2\mathcal{E}_{k,\delta}}{e-1}}^{(1-1/e)\frac{n}{2k}} + \exp\prth{-\frac{B_1(1-\mathcal{E}_{k,\delta})}{e}}.
    \end{aligned}
    \end{equation*}
The conclusion then follows by setting $C_4,C_5,C_6,C_7$ to be the appropriate constants.
\qed

\subsection{Proof of Lemma \ref{lem: gap translation for multiple predictions}}\label{subsec: proof of lemma lem: gap translation for multiple predictions}
    Let $\hat{L}_k(\theta):=\frac{1}{k}\sum_{i=1}^kl(\theta,z_i)$. Let $\theta^*$ be an optimal solution of \eqref{opt}. We have
    \begin{equation*}
        \max_{\theta\in\Theta}p_k(\theta)\geq p_k(\theta^*) =\prob{\theta^*\in \widehat{\Theta}_k^{\epsilon}} \geq \prob{\Theta^0\subseteq \widehat{\Theta}_k^{\epsilon}}.
    \end{equation*}
    To bound the probability on the right hand side, we write
    \begin{eqnarray*}
        \set{\Theta^0\not\subseteq\widehat{\Theta}_k^{\epsilon}}&\subseteq& \bigcup_{\theta\in\Theta^0,\theta'\in\Theta}\set{\hat{L}_k(\theta)> \hat{L}_k(\theta')+\epsilon}\\
        &=& \bigcup_{\theta\in\Theta^0,\theta'\in\Theta}\set{\hat{L}_k(\theta)-L(\theta)> \hat{L}_k(\theta')-L(\theta')+L(\theta')-L(\theta)+\epsilon}\\
        &\subseteq& \bigcup_{\theta\in\Theta^0,\theta'\in\Theta}\set{\hat{L}_k(\theta)-L(\theta)> \hat{L}_k(\theta')-L(\theta')+\epsilon}\\
        &\subseteq& \bigcup_{\theta\in\Theta^0,\theta'\in\Theta}\set{\hat{L}_k(\theta)-L(\theta)> \dfrac{\epsilon}{2}\text{ or }\hat{L}_k(\theta')-L(\theta')<-\dfrac{\epsilon}{2}}\\
        &\subseteq& \bigcup_{\theta\in\Theta}\set{\abs{\hat{L}_k(\theta)-L(\theta)}> \dfrac{\epsilon}{2}}\\
        &=& \set{\max_{\theta\in\Theta}\abs{\hat{L}_k(\theta)-L(\theta)}> \dfrac{\epsilon}{2}},
    \end{eqnarray*}
    therefore
    \begin{equation}\label{bound of sensitivity first term}
        \max_{\theta\in\Theta}p_k(\theta)\geq \prob{\max_{\theta\in\Theta}\abs{\hat{L}_k(\theta)-L(\theta)}\leq \dfrac{\epsilon}{2}}\geq 1-T_k\prth{\dfrac{\epsilon}{2}}.
    \end{equation}
    This proves \eqref{eq: pkmax lower bound for multiple predictions}. To bound the other term $\max_{\theta\in\Theta\backslash\Theta^{\delta}}p_k(\theta)$, for any $\theta\in\Theta\backslash\Theta^{\delta}$ it holds that
    \begin{equation}\label{bound of sensitivity second term for one model}
        p_k(\theta)=\prob{\theta\in \widehat{\Theta}_k^{\epsilon}}\leq \prob{\widehat{\Theta}_k^{\epsilon}\not\subseteq \Theta^{\delta}},
    \end{equation}
    and hence $\max_{\theta\in\Theta\backslash\Theta^{\delta}}p_k(\theta)\leq \prob{\widehat{\Theta}_k^{\epsilon}\not\subseteq \Theta^{\delta}}$. To bound the latter, we have
    \begin{eqnarray*}
        \set{\widehat{\Theta}_k^{\epsilon}\not \subseteq\Theta^{\delta}}&\subseteq&\bigcup_{\theta,\theta'\in\Theta\text{ s.t. }L(\theta')-L(\theta)>\delta}\set{\hat{L}_k(\theta')\leq \hat{L}_k(\theta)+\epsilon}\\
        &=&\bigcup_{\theta,\theta'\in\Theta\text{ s.t. }L(\theta')-L(\theta)>\delta}\set{\hat{L}_k(\theta')-L(\theta')+L(\theta')-L(\theta)\leq \hat{L}_k(\theta)-L(\theta)+\epsilon}\\
        &\subseteq&\bigcup_{\theta,\theta'\in\Theta\text{ s.t. }L(\theta')-L(\theta)>\delta}\set{\hat{L}_k(\theta')-L(\theta')+\delta< \hat{L}_k(\theta)-L(\theta)+\epsilon}\\
        &\subseteq&\bigcup_{\theta,\theta'\in\Theta\text{ s.t. }L(\theta')-L(\theta)>\delta}\set{\hat{L}_k(\theta')-L(\theta')<-\dfrac{\delta-\epsilon}{2}\text{ or }\hat{L}_k(\theta)-L(\theta)>\dfrac{\delta-\epsilon}{2}}\\
        &\subseteq&\bigcup_{\theta\in\Theta}\set{\abs{\hat{L}_k(\theta)-L(\theta)}>\dfrac{\delta-\epsilon}{2}}\\
        &=&\set{\max_{\theta\in\Theta}\abs{\hat{L}_k(\theta)-L(\theta)}>\dfrac{\delta-\epsilon}{2}},
    \end{eqnarray*}
    therefore
    \begin{equation}\label{bound of sensitivity second term}
        \max_{\theta\in\Theta\backslash\Theta^{\delta}}p_k(\theta)\leq \prob{\max_{\theta\in\Theta}\abs{\hat{L}_k(\theta)-L(\theta)}> \dfrac{\delta-\epsilon}{2}}\leq T_k\prth{\dfrac{\delta-\epsilon}{2}}.
    \end{equation}
    This immediately gives \eqref{eq: pkmax - eta upper bound}. \eqref{eq: prob gap lower bound for multiple predictions} is obvious given \eqref{eq: pkmax lower bound for multiple predictions} and \eqref{eq: pkmax - eta upper bound}.
    % Combining \eqref{bound of sensitivity first term} and \eqref{bound of sensitivity second term} gives
    % \begin{equation}\label{bound of sensitivity overall}
    %     \max_{\theta\in\Theta}p_k(\theta)-\max_{\theta\in\Theta\backslash\Theta^{\delta}}p_k(\theta)\geq 1-T_k\prth{\dfrac{\epsilon}{2}}-T_k\prth{\dfrac{\delta-\epsilon}{2}}.
    % \end{equation}
    % Whenever $1-T_k\prth{\epsilon/2}-T_k\prth{(\delta-\epsilon)/2}>0$, we have $\max_{\theta\in\Theta}p_k(\theta)>\max_{\theta\in\Theta\backslash\Theta^{\delta}}p_k(\theta)$ and thus $\max_{\theta\in\Theta}p_k(\theta)=\max_{\theta\in\Theta^{\delta}}p_k(\theta)$, i.e., $\bar{\eta}_{k,\delta}=\max_{\theta\in\Theta}p_k(\theta)-\max_{\theta\in\Theta\backslash\Theta^{\delta}}p_k(\theta)\geq r_k^{\epsilon}-(1-q_k^{\epsilon,\delta})$. This together with \eqref{bound of sensitivity overall} proves \eqref{eq: prob gap lower bound for multiple predictions}. To show \eqref{eq: pkmax - eta upper bound}, note that $p_k^{\max}-\bar{\eta}_{k,\delta}=\max_{\theta\in\Theta^{\delta}}p_k(\theta)-\bar{\eta}_{k,\delta}=\max_{\theta\in\Theta\backslash\Theta^{\delta}}p_k(\theta)$ and hence \eqref{eq: pkmax - eta upper bound} follows immediately from \eqref{bound of sensitivity second term}.
\qed

\section{Improving Tail Decay in Linear Regression}\label{app: linear_regression}
In this section, we present another example demonstrating that our algorithm is capable of turning a polynomial-tailed base learner into exponential.
We consider the linear regression problem
\begin{equation}\label{eq: linear_regression_ex}
\min_{\theta\in [-1,1]}\mathbb{E}\bigl[(x\theta - y)^2\bigr],
\end{equation}
where the data $\{(x_i,y_i)\}_{i=1}^n$ are i.i.d. samples such that $x_i \in \{-1,1\}$, and $y_i=x_i\theta^*+\epsilon_i$. 

\begin{assumption}\label{ass:poly_tail}
The unknown true coefficient of problem \eqref{eq: linear_regression_ex} is $\theta^* = 0$.
The random variables $\{\epsilon_i\}_{i=1}^n$ are i.i.d. distributed with zero mean and symmetric with respect to $0$.
The second and forth moments of $\epsilon_i$ are finite, denoted as $\sigma^2=\mathbb{E}\left[\epsilon_i^2\right]$ and $\mu_4=\mathbb{E}\left[\epsilon_i^4\right]$.
Moreover, there exist constants $C>0$ and $\alpha>0$ such that $P(\epsilon_i>t)>C(t+1)^{-\alpha}$ for all $t>0$, i.e., $\epsilon_i$ has a polynomial tail.
\end{assumption}

Under the setting described by \eqref{eq: linear_regression_ex} and Assumption \ref{ass:poly_tail}, the least-squares estimator of $\theta$ is given by
$$
\theta^{LS}_n=\mathcal{P}_{[-1,1]}\left(\frac{\sum_{i=1}^nx_iy_i}{\sum_{i=1}^nx_i^2}\right)=\mathcal{P}_{[-1,1]}\left(\frac{\sum_{i=1}^nx_i\epsilon_i}{\sum_{i=1}^nx_i^2}+\theta^*\right)=\mathcal{P}_{[-1,1]}\left(\frac{\sum_{i=1}^nx_i\epsilon_i}{n}\right),
$$
where $\mathcal{P}_{[-1,1]}(\cdot)$ denotes the projection operator onto the interval $[-1,1]$. 
Since the true coefficient $\theta^* = 0$, for any estimator $\hat \theta$ that takes values between $[-1,1]$, its excess risk is equal to $(\hat \theta)^2$.
For instance, the excess risk of $\theta^{LS}_n$ can be expressed as
$$
\prth{\theta^{LS}_n}^2=\min\left\{\left(\frac{\sum_{i=1}^nx_i\epsilon_i}{n}\right)^2,1\right\}.
$$

\begin{theorem}\label{thm: linear regression ex}
Under Assumption \ref{ass:poly_tail}, the followings hold true. 
% \donghao{Since we will potentially put this theorem in the main body, I intentionally write informal statement for the second bullet.}
\begin{itemize}
    \item The excess risk of the least-squares estimator $\theta^{LS}_n$ exhibits a polynomial tail in $n$. Specifically, for every $\delta\in (0,1)$ and $n>1$, it holds that 
    \begin{equation}\label{eq: poly_tail}
    \mathbb P\left((\theta^{LS}_n)^2>\delta\right)>C(n\sqrt{\delta}+1)^{-\alpha}.
    \end{equation}
    \item Under our ensemble method, the excess risk of the output estimator $\hat\theta_n$ has an exponentially decreasing tail.
\end{itemize}
\end{theorem}

\subsection{Proof of Theorem \ref{thm: linear regression ex}}
We first show the polynomial tail of excess risk for the least-squares estimator $\theta_n^{L S}$.
For $k \leq n$, let $\bar \epsilon_{k}:=1/k\cdot {\sum_{i=1}^{k}\epsilon_i}$ be the sample average of the first $k$ noise terms.
Then, it holds that $\bar \epsilon_{k+1}=\prth{k\bar \epsilon_{k}+\epsilon_{k+1}}/(k+1)$.

For every $\delta\in (0,1)$ and $n>1$, we have that
\begin{equation}\label{eq: linear_regression_polytail1}
\mathbb P\left((\theta^{LS}_n)^2>\delta\right)=\mathbb P\left(\left(\frac{\sum_{i=1}^nx_i\epsilon_i}{n}\right)^2>\delta\right)
=\mathbb P\left(\left(\frac{\sum_{i=1}^n\epsilon_i}{n}\right)^2>\delta\right)
=2\mathbb P\big(\bar \epsilon_n>\sqrt{\delta}\big),
\end{equation}
where we used the symmetry of \(\epsilon_i\) and that \(x_i\in\{-1,1\}\).
Then, using the recursive relation between $\bar \epsilon_n$ and $\bar \epsilon_{n-1}$, we can further show that
\begin{equation}\label{eq: linear_regression_polytail2}
\begin{aligned}
\mathbb P\big(\bar \epsilon_n>\sqrt{\delta}\big) 
&\geq \prob{\frac{(n-1)\bar \epsilon_{n-1}}{n}\geq 0 \text{ and }\frac{\epsilon_n}{n}>\sqrt{\delta}}\\
&= \prob{\frac{(n-1)\bar \epsilon_{n-1}}{n}\geq 0}\cdot \prob{\frac{\epsilon_n}{n}>\sqrt{\delta}}\\
&> \prob{\bar\epsilon_{n-1} \geq 0}\cdot C(n\sqrt{\delta}+1)^{-\alpha},
\end{aligned}
\end{equation}
where the second line is due to the independence between $\bar\epsilon_{n-1}$ and $\epsilon_n$ in the second line, and the last line uses Assumption \ref{ass:poly_tail}.
Since each $\epsilon_i$ is symmetric, we have $\prob{\bar\epsilon_{n-1} \geq 0} = 1/2$.
Hence, the proof of \eqref{eq: poly_tail} is completed by combining \eqref{eq: linear_regression_polytail1} and \eqref{eq: linear_regression_polytail2}.

Now, we proceed to show the exponential tail of excess risk for our ensemble method (Algorithm \ref{bagging majority vote: two phase}), where the proof is based on Theorem \ref{thm: finite-sample bound for multiple predictions two phase splitting_formal}, i.e., the formal version of Theorem \ref{thm: finite-sample bound for multiple predictions two phase splitting}.
To apply the bound \eqref{eq: finite-sample bound for data splitting two phase under large pmax} from Theorem \ref{thm: finite-sample bound for multiple predictions two phase splitting_formal}, we need to derive upper bounds for the following two quantities: the empirical process tail $T_k(\cdot)$, and the excess risk tail of the base learner, i.e., the least-squares estimator $\theta^{LS}_k$ with $k$ samples.

For any $t > 0$, we can show that the empirical process tail satisfies that
\begin{equation}\label{eq: empirical_process_tail_LR1}
\begin{aligned}
T_k(t)&=\mathbb P\prth{\sup_{\theta\in [-1,1]} \abs{\frac{1}{k}\sum_{i=1}^k(x_i\theta-y_i)^2 - \mathbb E[(x\theta-y)^2]}>t}\\
&=\mathbb P\prth{\sup_{\theta\in [-1,1]} \abs{\frac{1}{k}\sum_{i=1}^k(x_i\theta-\epsilon_i)^2 - (\theta^2 + \sigma^2)}>t}\\
&=\mathbb P\prth{\sup_{\theta\in [-1,1]} \abs{ \frac{1}{k}\sum_{i=1}^k\epsilon_i^2-\sigma^2 - 2\theta\cdot \frac{1}{k}\sum_{i=1}^k x_i\epsilon_i}>t }\\
&\leq \mathbb P\prth{\abs{-\sigma^2 + \frac{1}{k}\sum_{i=1}^k\epsilon_i^2 - \frac{2}{k}\sum_{i=1}^k x_i\epsilon_i}>t }+\mathbb P\prth{\abs{-\sigma^2 +\frac{1}{k}\sum_{i=1}^k\epsilon_i^2+ \frac{2}{k}\sum_{i=1}^k x_i\epsilon_i}>t },
\end{aligned}
\end{equation}
where the last inequality uses the union bound and the observation that the maximum of the absolute value term is achieved either at $\theta = 1$ or $\theta = -1$.
Using the symmetry of $\epsilon_i$ and the fact that $x_i\in \set{-1,1}$, we further derive from \eqref{eq: empirical_process_tail_LR1} that
\begin{equation}\label{eq: empirical_process_tail_LR2}
\begin{aligned}
T_k(t)&\leq 2\mathbb P\prth{\abs{ \frac{1}{k}\sum_{i=1}^k\epsilon_i^2-\sigma^2 - \frac{2}{k}\sum_{i=1}^k\epsilon_i}>t }\\
&\leq 2\mathbb P\prth{\abs{ \frac{1}{k}\sum_{i=1}^k\epsilon_i^2-\sigma^2} >\frac{t}{2}}+2\mathbb P\prth{ \abs{\frac{2}{k}\sum_{i=1}^k\epsilon_i}>\frac{t}{2} }\\
&\leq \frac{8\mu_4}{kt^2}+\frac{32\sigma^2}{kt^2},
\end{aligned}
\end{equation}
where we apply the union bound again in the second line and use the Markov's inequality in the last line.

Now, we assess the excess risk tail of the least-squares estimator $\theta^{LS}_k$.
For $\delta < 1$, similar as \eqref{eq: linear_regression_polytail1}, we can apply the Markov's inequality to show that
\begin{equation}\label{eq: excess_risk_tail_LR}
\mathcal E_{k,\delta} = \mathbb P\prth{(\theta^{LS}_k)^2>\delta} = \mathbb P\left(\left(\frac{\sum_{i=1}^k\epsilon_i}{k}\right)^2>\delta\right)\leq \frac{\sigma^2}{k\delta}.
\end{equation}

By instantiating \eqref{eq: finite-sample bound for data splitting two phase under large pmax} in Theorem~\ref{thm: finite-sample bound for multiple predictions two phase splitting_formal} with the tail bounds on \(T_k(t)\) and \(\mathcal{E}_{k,\delta}\) given by \eqref{eq: empirical_process_tail_LR2} and \eqref{eq: excess_risk_tail_LR}, we can finally obtain the following tail bound on the excess risk of our estimator $\hat \theta_n$:
\begin{equation}\label{eq: exponential_tail_LR}
\begin{aligned}
\mathbb P\prth{ (\hat \theta_n)^2>\delta }
&\leq B_1\prth{3\min\set{e^{-2/5},C_1\frac{32(\mu_4+4\sigma^2)}{k_2\min\set{\underline{\epsilon},\delta-\overline{\epsilon}}^2}}^{\frac{n}{2C_2k_2}}+e^{-B_2/C_3}}\\\\
&\quad\ +\min\set{ e^{-\big(1-\frac{\sigma^2}{k_1\delta}\big)\big/C_4},C_5\frac{\sigma^2}{k_1\delta}}^{\frac{n}{2C_6k_1}}+e^{-B_1\big(1-\frac{\sigma^2}{k_1\delta}\big)\big/C_7},
\end{aligned}
\end{equation}
for every $k_1,k_2 \leq n$ and $\delta\in (0,1)$ such that $\delta>\overline{\epsilon}$, $\frac{32(\mu_4+4\sigma^2)}{k_2(\delta-\overline{\epsilon})^2}+\frac{32(\mu_4+4\sigma^2)}{k_2\underline{\epsilon}^2}<\frac{1}{5}$, and $\frac{\sigma^2}{k_1\delta}<1$.
Note that these conditions guarantee that the upper bound in \eqref{eq: exponential_tail_LR} is meaningful and that $T_{k_2}((\delta-\overline{\epsilon})/2) + T_{k_2}(\underline{\epsilon}/2)<1/5$, as required by Theorem \ref{thm: finite-sample bound for multiple predictions two phase splitting}.
Therefore, we conclude that the excess risk for the output solution $\hat \theta_n$ of our ensemble method has an exponential tail, which completes the proof of Theorem \ref{thm: linear regression ex}.
\qed

% \newpage
\section{Additional Numerical Experiments}\label{app: additional_experiment}
This section supplements Section \ref{section: experiments}. We first provide details for the architecture of the neural networks in Section \ref{app: network architecture}, and the considered stochastic programs in Section \ref{app: prob_intro}. Section \ref{subsec: exp_profile} presents a comprehensive profiling of hyperparameters of our methods, and Section \ref{app: additional_figures} provides additional experimental results that evaluate our algorithms from various perspectives.

\subsection{MLP Architecture}\label{app: network architecture}
The input layer of our MLPs has the same number of neurons as the input dimension, and the output layer is a single neuron that gives the final prediction. All activations are ReLU. The architecture of hidden layers is as follows under different numbers of hidden layers $H$:
\begin{itemize}
    \item $H=2$: Each hidden layer has 50 neurons.
    \item $H=4$: There are 50, 300, 300, 50 neurons from the first to the fourth hidden layer.
    \item $H=6$: There are 50, 300, 500, 500 300, 50 neurons from the first to the sixth hidden layer.
    \item $H=8$: There are 50, 300, 500, 800, 800 500 300, 50 neurons from the first to the eighth hidden layer.
\end{itemize}
% For experiments on real data, the MLP with 4 hidden layers has 100, 300, 300, 100 neurons from the first to the fourth hidden layer.

\subsection{Stochastic Programming Problems}\label{app: prob_intro}
\textbf{Resource Allocation \citeAPX{kleywegt2002sample}.$\quad$} 
The decision maker wants to choose a subset of $m$ projects. 
A quantity $q$ of low-cost resource is available to be allocated, and any additional resource can be obtained at an incremental unit cost $c$. 
Each project $i$ has an expected reward $r_i$.
The amount of resource required by each project $i$ is a random variable, denoted by $W_i$.
% The associated optimization problem can be formulated as:
We can formulate the problem as
\begin{equation}
\max _{\theta \in\{0,1\}^m}\left\{\sum_{i=1}^m r_i \theta_i-c \mathbb{E}\left[\sum_{i=1}^m W_i \theta_i -q\right]^{+}\right\}.
\end{equation}
In the experiment, we consider the three-product scenario, i.e., $m=3$, and assume that the random variable $W_i$ follows the Pareto distribution.
% where the associated SAA problem can be formulated as the folloing integer linear program:
% \begin{equation}
% \begin{array}{lll}
% \max _{\theta, z} & \sum_{i=1}^m r_i \theta_i-\frac{c}{n} \sum_{j=1}^n z_j & \\[5pt]
% \text {subject to } & z_j \geq \sum_{i=1}^m W_i^j \theta_i-q, & j=1, \ldots, n, \\[5pt]
% & \theta_i \in\{0,1\}, & i=1, \ldots, m, \\[5pt]
% & z_j \geq 0, & j=1, \ldots, n.
% \end{array}
% \end{equation}

\textbf{Supply Chain Network Design \citeAPX[Chapter 1.5]{shapiro2021lectures}.$\quad$}
Consider a network of suppliers, processing facilities, and customers, where the goal is to optimize the overall supply chain efficiency.
The supply chain design problem can be formulated as a two-stage stochastic optimization problem
\begin{equation}
\min_{\theta \in \{0,1\}^{\abs{P}}} \sum_{p\in P}c_p \theta_p+\mathbb{E}[Q(\theta, z)],
\end{equation}
where $P$ is the set of processing facilities, $c_p$ is the cost of opening facility $p$, and $z$ is the vector of (random) parameters, i.e., $(h, q, d, s, R, M)$ in \eqref{eq: network_second_stage}.
Function $Q(\theta, z)$ represents the total processing and transportation cost, and it is equal to the optimal objective value of the following second-stage problem:
\begin{equation}\label{eq: network_second_stage}
\begin{array}{cl}
\min_{y \geq 0, z \geq 0} & q^{\top} y+h^{\top} z \\[5pt]
\text { s.t. } & N y=0, \\[5pt]
& C y+z \geq d, \\[5pt]
& S y \leq s, \\[5pt]
& R y \leq M \theta,
\end{array}
\end{equation}
where $N, C, S$ are appropriate matrices that describe the network flow constraints.
More details about this example can be found in \citeAPX[Chapter 1.5]{shapiro2021lectures}.
In our experiment, we consider the scenario of 3 suppliers, 2 facilities, 3 consumers, and 5 products.
We choose supply $s$ and demand $d$ as random variables that follow the Pareto distribution.
% s, p, c, g = 3, 2, 3, 5

% \textbf{Model selection}\hspace{5pt}
% In machine learning, model selection is typically done by testing candidate models on a validation set with a possibly limited amount of data.
% It can be viewed as a stochastic optimization problem when treating candidate models as the decision space, where the stochasticity comes from the randomness of the underlying data.
% Specifically, we consider a LASSO problem \cite{tibshirani1996regression}, i.e., $\min_{\beta}\|y-X \beta\|_2^2+\lambda\|\beta\|_0$, where $y\in \mbb{R}$, $X\in \mbb{R}^{3000}$, and the noises are Pareto distributed.
% The goal is to select an appropriate penalization weight $\lambda$ using the validation set. 
% For all candidate $\lambda$, the same training set of size $1000$ is used to learn the associated vector $\beta(\lambda)$, ensuring a certain level of noise in candidate solutions.
% Then, the validation is performed using out-of-sample data.

\textbf{Maximum Weight Matching and Stochastic Linear Program.$\quad$}
% A linear program with deterministic constraints can be regarded as a discrete optimization problem, where the decision space comprises the vertices of the associated polyhedron.
We explore both the maximum weight matching problem and the linear program that arises from it. 
Let $G = (V,E)$ be a general graph, where each edge $e\in E$ is associated with a (possibly) random weight $w_e$. For each node $v \in V$, denote $E(v)$ as the set of edges incident to $v$.
Based on this setup, we consider the following linear program
\begin{equation}\label{eq: LP_maximum_weight_matching}
\begin{array}{lll}
\max_{\theta\in [0,1]^{\abs{E}}} & \expect{\sum_{e\in E} w_e\theta_e} & \\[5pt]
\text {subject to } & \sum_{e\in E(v)} a_e \theta_e \leq 1,\quad &\forall v\in V,
\end{array}
\end{equation}
where $a_e$ is some positive coefficient.
When $a_e = 1$ for all $e\in E$ and $\theta$ is restricted to the discrete set $\{0,1\}^{\abs{E}}$, \eqref{eq: LP_maximum_weight_matching} is equivalent to the maximum weight matching problem.
For the maximum weight matching, we consider a complete bipartite graph with 5 nodes on each side (the dimension is 25).
The weights of nine edges are Pareto distributed and the remaining are prespecified constants.
For the linear programming problem, we consider a 28-dimensional instance (the underlying graph is an 8-node complete graph), where all $w_e$ follows the Pareto distribution.

\textbf{Mean-Variance Portfolio Optimization.$\quad$}
Consider constructing a portfolio based on $m$ assets. Each asset $i$ has a rate of return $r_i$ that is random with mean $\mu_i$. The goal is to minimize the variance of the portfolio while ensuring that the expected rate of return surpasses a target level $b$.
The problem can be formulated as
\begin{equation}
\begin{array}{ll}
\min_{\theta} &  \expect{(\sum_{i=1}^m(r_i-\mu_i)\theta_i)^2} \\[5pt]
\text {subject to } & \sum_{i=1}^m \mu_i \theta_i \geq b, \\[5pt]
& \sum_{i=1}^m \theta_i=1,\\[5pt]
&\theta_i\geq 0,\quad \forall i=1,\ldots,m,
\end{array}
\end{equation}
where $\theta$ is the decision variable and each $\mu_i$ is assumed known.
In the experiment, we consider a scenario with 10 assets, i.e., $m=10$, where each rate of return $r_i$ is a linear combination of the rates of return of 100 underlying assets in the form $r_i=\Tilde{r}_{10(i-1)+1}/2 + \sum_{j=1}^{100}\Tilde{r}_j/200$. Each of these underlying assets has a Pareto rate of return $\Tilde{r}_j,j=1,\ldots,100$.
% Consider constructing a portfolio based on $m$ assets with a budget $b$. Each asset $i$ has a unit value $v_i$ and a return $z_i v_i$, where $z_i$ is a random variable with mean $\mu_i$.
% The goal is to minimize the Conditional Value-at-Risk (CVaR) at the 95\% confidence level while ensuring that the expected total return surpasses a target level $R$.
% The problem can be formulated as
% \begin{equation}
% \begin{array}{ll}
% \min_{\theta\in \mbb{N}^m, c\in \mbb{R}} &  c+ \frac{1}{0.05}\expect{\prth{-\sum_{i=1}^m z_i v_i \theta_i -c }_+} \\[5pt]
% \text {subject to } & \sum_{i=1}^m v_i \theta_i \leq b, \\[5pt]
% & \sum_{i=1}^m z_i v_i \theta_i \geq R,
% \end{array}
% \end{equation}
% where $\theta$ is the decision variable and $c$ is an auxiliary variable used to compute the CVaR.
% In the experiment, we consider a scenario with 6 assets, i.e., $m=6$, where the random multiplicity $z_i$ is Pareto distributed.

\subsection{Hyperparameter Profiling}\label{subsec: exp_profile}
We test the effect of different hyperparameters in our ensemble methods, including subsample sizes $k,k_1,k_2$, ensemble sizes $B, B_1, B_2$, and threshold $\epsilon$.
Throughout this profiling stage, we use the sample average approximation (SAA) as the base algorithm.
To profile the effect of subsample sizes and ensemble sizes, we consider the resource allocation problem.

\paragraph{Subsample Size.}
We explored scenarios where $k$ (equivalently $k_1$ and $k_2$) is both dependent on and independent of the total sample size $n$ (see Figures \ref{subfig: profile_alg1_k_2}, \ref{subfig: profile_alg1_k_1}, and \ref{subfig: profile_alg1_k_3}). 
The results suggest that a constant $k$ generally suffices, although the optimal $k$ varies by problem instance. 
For example, Figures \ref{subfig: profile_alg1_k_1} and \ref{subfig: profile_alg1_k_3} show that $k=2$ yields the best performance; increasing $k$ degrades results. 
Conversely, in Figure \ref{subfig: profile_alg1_k_2}, $k=2$ proves inadequate, with larger $k$ delivering good results.
The underlying reason is that the effective performance of \move requires $\theta^* \in \argmax_{\theta \in \Theta} p_k(\theta)$.
In the former, this is achieved with only two samples, enabling \move to identify $\theta^*$ with a subsample size of $2$.
For the latter, a higher number of samples is required to meet this condition, explaining the suboptimal performance at \(k=2\).
In Figure \ref{fig: pkx_SSKP}, we simulate $p_k(\theta)$ for the two cases, which further explains the influence of the subsample size.

\paragraph{Ensemble Size.}
In Figure \ref{fig: profiling_B}, we illustrate the performance of \move and \rove under different $B, B_1,B_2$, where we set $k=k_1=k_2=10$ and $\epsilon =0.005$.
From the figure, we find that the performance of our ensemble methods is improving in the ensemble sizes.

\paragraph{Threshold $\epsilon$.}
The optimal choice of $\epsilon$ in \rove and \roves is problem-dependent and related to the number of (near) optimal solutions.
This dependence is illustrated by the performance of \rove shown in Figures \ref{subfig: profile_alg3_epsilon_multiple} and \ref{subfig: profile_alg3_epsilon_single}.
Hence, we propose an adaptive strategy defined as follows: Let $g(\epsilon):= 1/B_2\cdot \sum_{b=1}^{B_2}\mathbbm{1}(\hat{\theta}_n(\epsilon)\in \widehat{\Theta}^{\epsilon,b}_{k_2})$, where we use $\hat{\theta}_n(\epsilon)$ to emphasize the dependency of $\hat{\theta}_n$ on $\epsilon$.
Then, we select $\epsilon^*:= \min \set{\epsilon : g(\epsilon)\geq 1/2}$.
By definition, $g(\epsilon)$ is the proportion of times that $\hat{\theta}_n(\epsilon)$ is included in the ``near optimum set'' $\widehat{\Theta}^{\epsilon,b}_{k_2}$.
The choice of $\epsilon^*$ makes it more likely for the true optimal solution to be included in the ``near optimum set'', instead of being ruled out by suboptimal solutions.
Practically, $\epsilon^*$ can be efficiently determined using a binary search as an intermediate step between Phases I and II.
To prevent data leakage, we compute $\epsilon^*$ using $\mathbf{z}_{1:\lfloor\frac{n}{2}\rfloor}$ (Phase I data) for \roves.
From Figure \ref{fig: profiling_main_paper}, we observe that this adaptive strategy exhibits decent performance for all scenarios.
Similar behaviors can also be observed for \roves in Figure \ref{fig: profiling_epsilon_alg4}.

\paragraph{Recommended Configurations.}
Based on the profiling results, we summarize the recommended configurations used in all other experiments presented in the paper (unless specified otherwise):
\vspace{-5pt}
\begin{itemize}
    \item For discrete space $\Theta$, use $k=\max(10,n/200),B=200$ for \move, and $k_1=k_2=\max(10,n/200),B_1=20,B_2=200$ for \rove and \roves.
    \item For continuous space $\Theta$, use $k_1=\max(30,n/2),k_2=\max(30,n/200),B_1=50,B_2=200$ for \rove and \roves.
    \item The $\epsilon$ in \rove and \roves is selected such that $\max_{\theta\in \mathcal{S}}(1/B_2)\sum_{b=1}^{B_2}\mathbbm{1}(\theta\in \widehat{\Theta}^{\epsilon,b}_{k_2}) \approx 1/2$.
\end{itemize}

\begin{figure}[!htbp]
\centering
\hspace{-8pt}
\begin{subfigure}{0.33\textwidth}
\includegraphics[width = \linewidth]{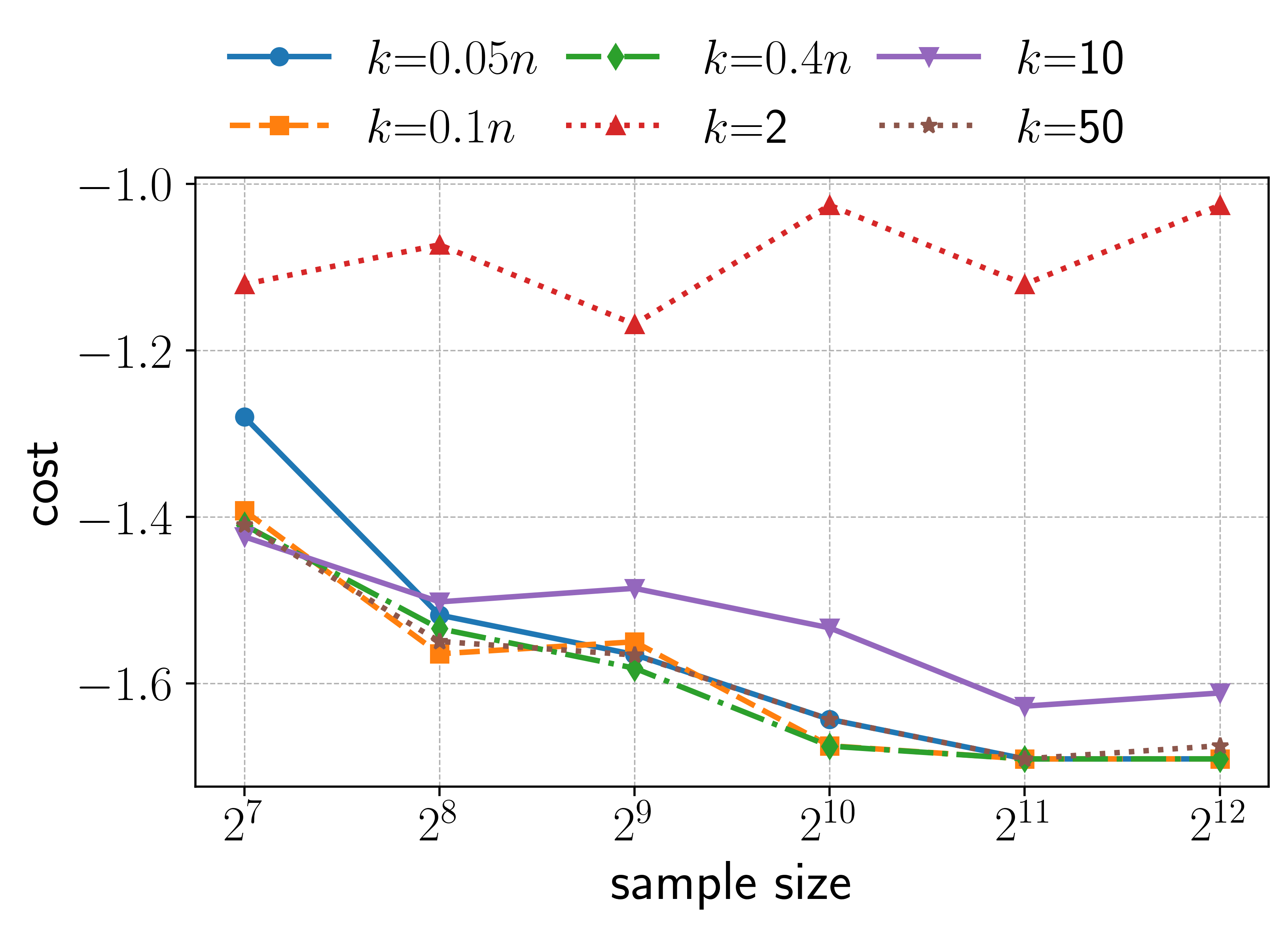}
\caption{Profiling for $k$ (\move).\label{subfig: profile_alg1_k_2}}
\end{subfigure}%
\begin{subfigure}{0.335\textwidth}
\includegraphics[width = \linewidth]{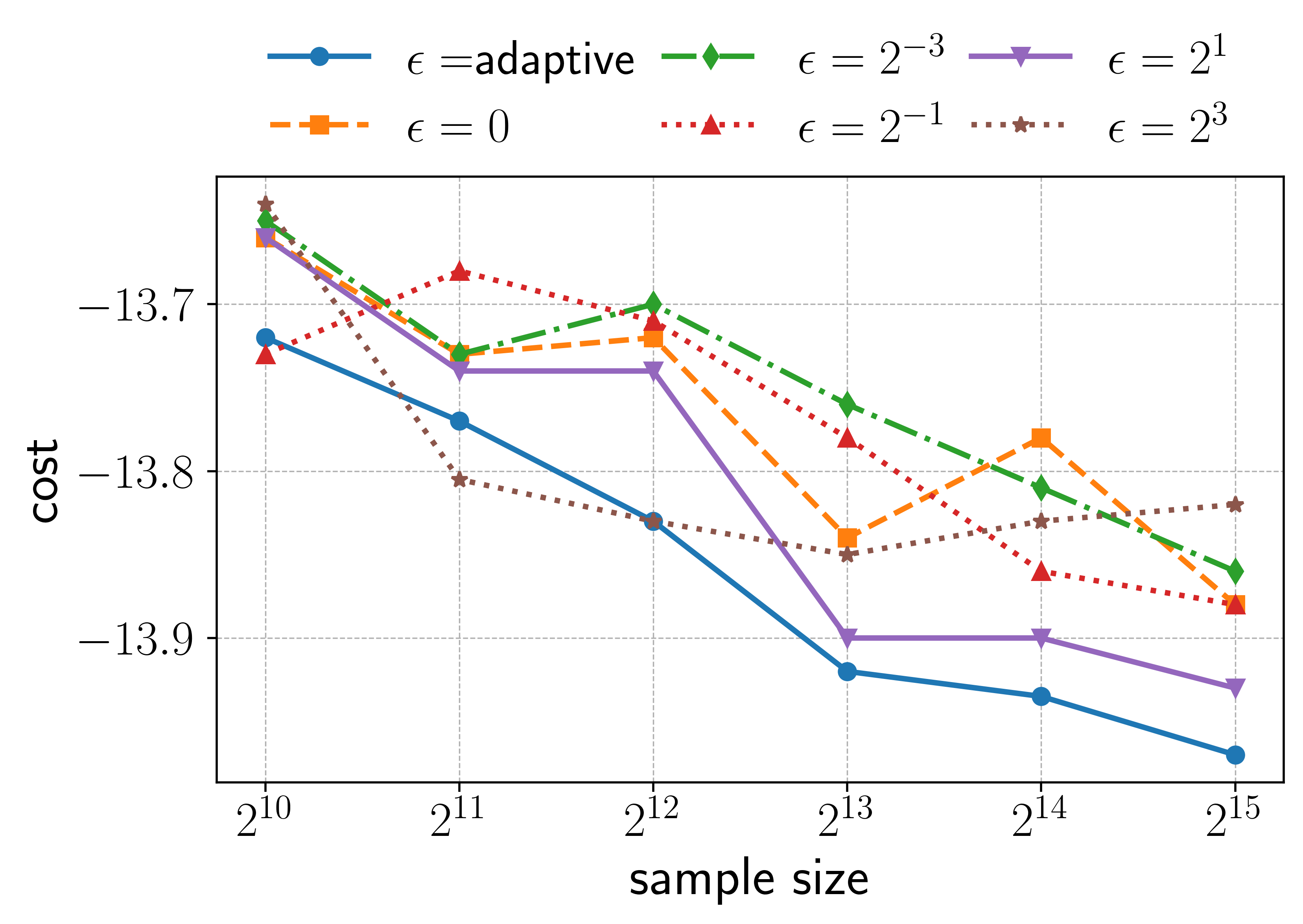}
\caption{Profiling for $\epsilon$ (multiple optima).\label{subfig: profile_alg3_epsilon_multiple}}
\end{subfigure}%
\begin{subfigure}{0.335\textwidth}
\includegraphics[width = \linewidth]{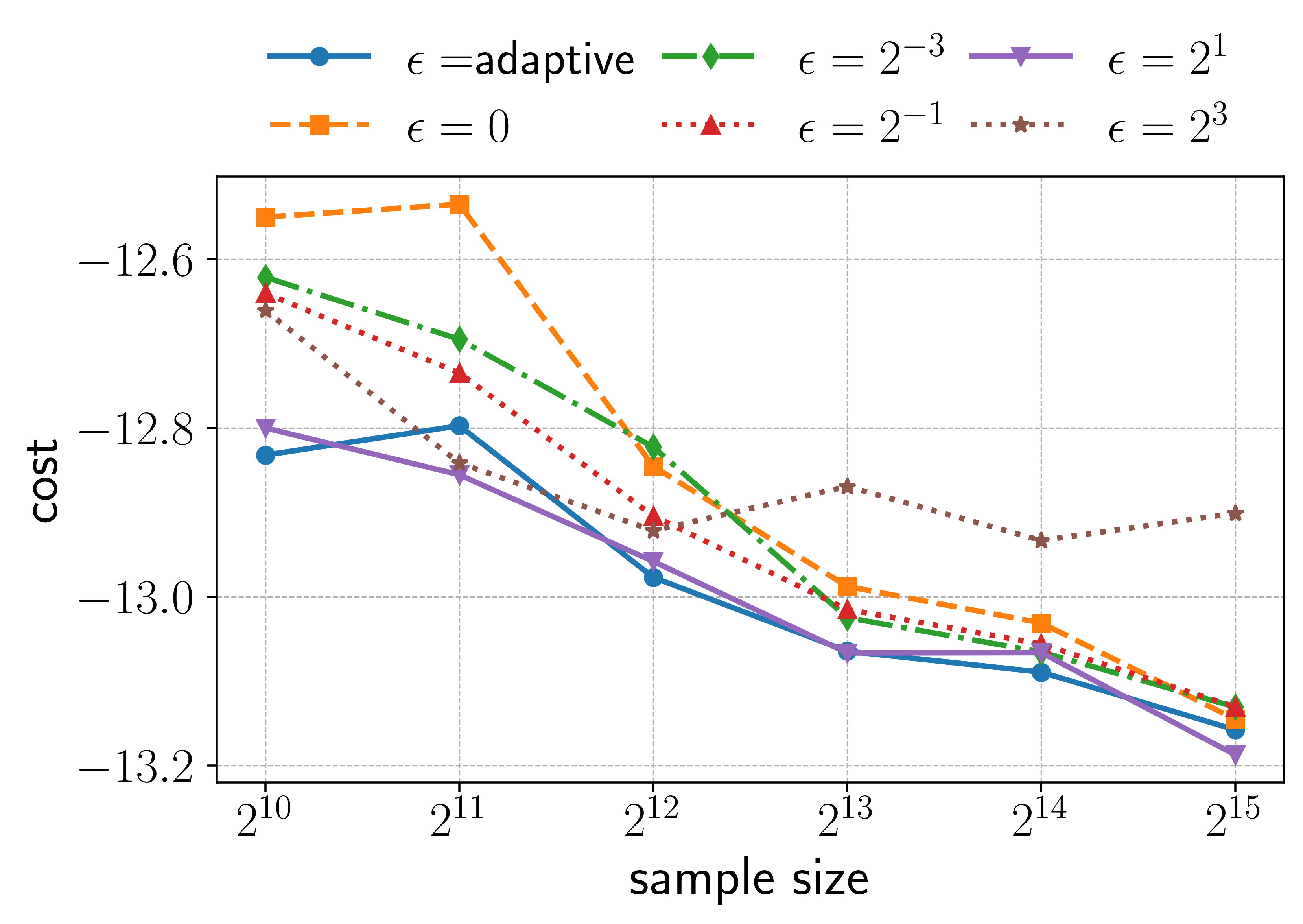}
\caption{Profiling for $\epsilon$ (unique optima).\label{subfig: profile_alg3_epsilon_single}}
\end{subfigure}
\caption{Profiling for subsample size $k$ and threshold $\epsilon$.
(a): Resource allocation problem, where $B = 200$;
(b) and (c): Linear program, where $k_1=k_2 =\max(10,0.005n)$, $B_1= 20$, and $B_2 = 200$.
% Performance of ROVE using SAA as base model in three instances of linear programs under different thresholds $\epsilon$.
% Hyperparameters: 
\label{fig: profiling_main_paper}}
\end{figure}

\begin{figure}[!htbp]
\centering
\hspace{-12pt}
\begin{subfigure}{0.34\textwidth}
\includegraphics[width = \linewidth]{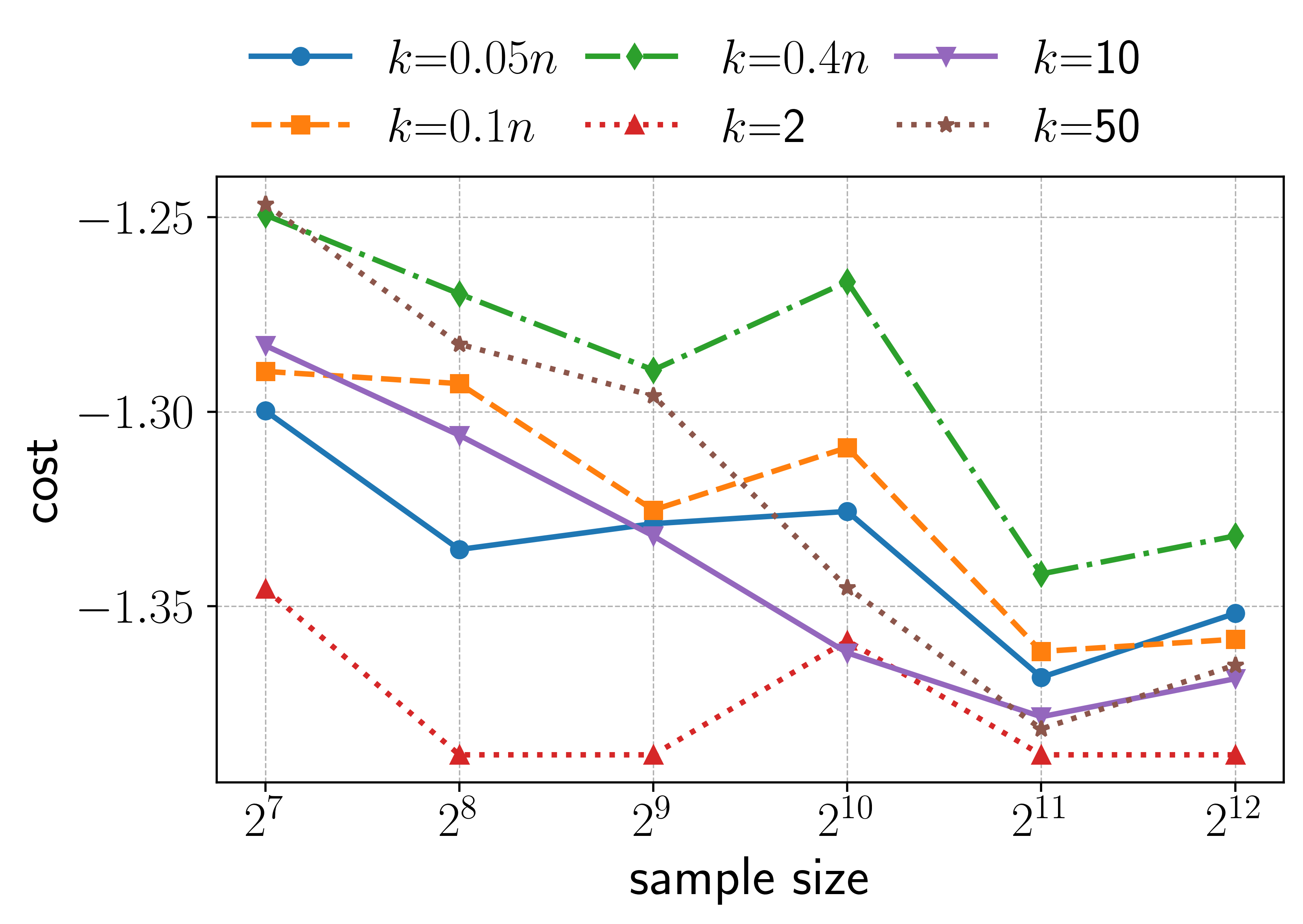}
\caption{Profiling for $k$ (instance 1).\label{subfig: profile_alg1_k_1}}
\end{subfigure}%
% \begin{subfigure}{0.34\textwidth}
% \includegraphics[width = \linewidth]{img/profiling_alg1_k_2.png}
% \caption{Instance 2.\label{subfig: profile_alg1_k_2}}
% \end{subfigure}%
\begin{subfigure}{0.34\textwidth}
\includegraphics[width = \linewidth]{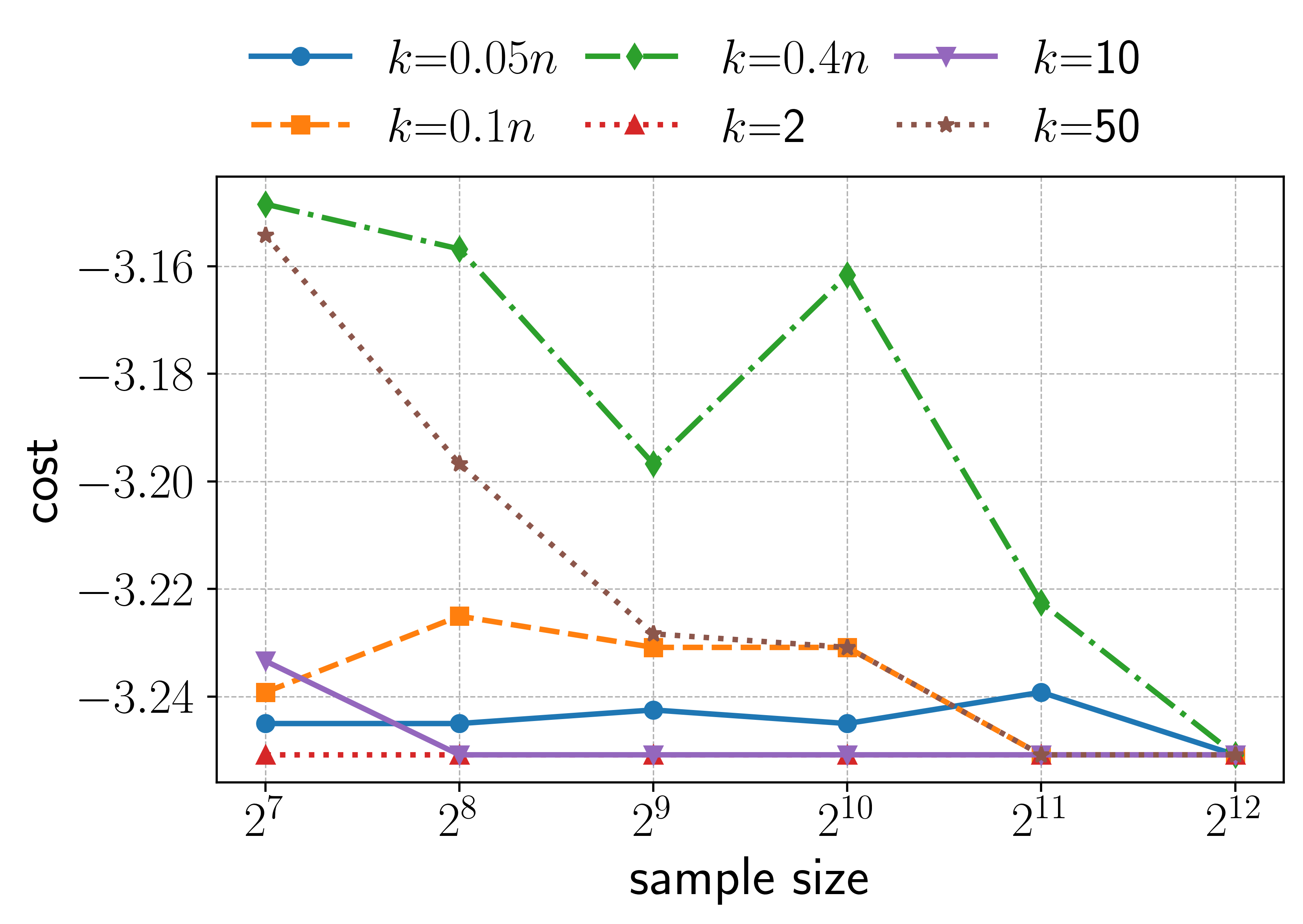}
\caption{Profiling for $k$ (instance 2).\label{subfig: profile_alg1_k_3}}
\end{subfigure}%
\begin{subfigure}{0.34\textwidth}
\includegraphics[width = \linewidth]{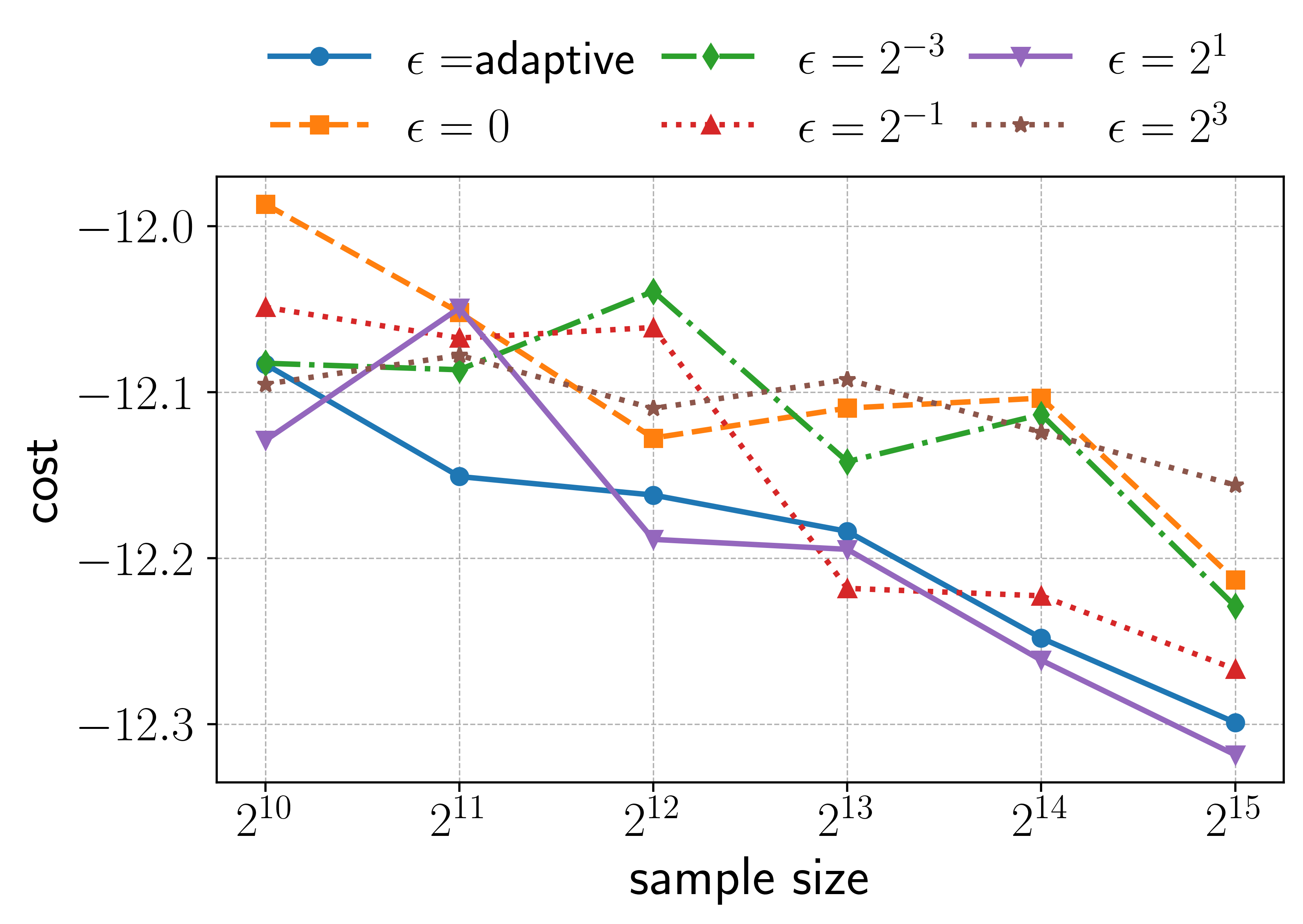}
\caption{Profiling for $\epsilon$ (near optima).\label{subfig: profile_alg3_epsilon_near}}
\end{subfigure}%
\caption{
Profiling results for subsample size $k$ and threshold $\epsilon$.
(a) and (b): Resource allocation problem using \move, where $B = 200$;
(c): Linear program with multiple near optima using \rove, where $k_1 = k_2=\max(10,0.005n)$, $B_1= 20$, and $B_2 = 200$.
\label{fig: profiling_additional}}
\end{figure}

\begin{figure}[!htbp]
\centering
\begin{subfigure}{0.4\textwidth}
\includegraphics[width = \linewidth]{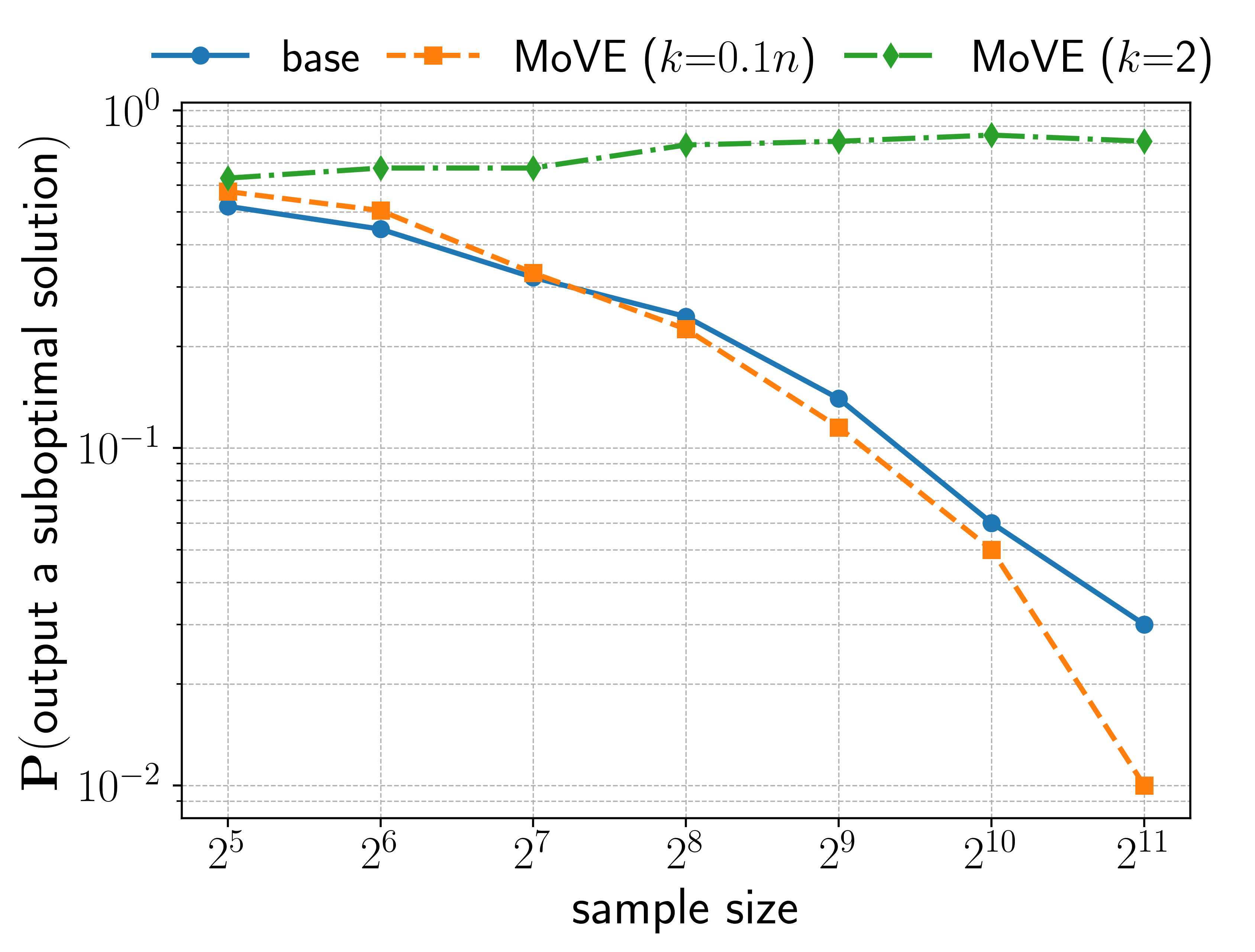}
\caption{Figure \ref{subfig: profile_alg1_k_2}: suboptimal probability.}\label{subfig: sskp_prob_suboptimal_2}
\end{subfigure}
\hspace{1cm}
\begin{subfigure}{0.4\textwidth}
\includegraphics[width = \linewidth]{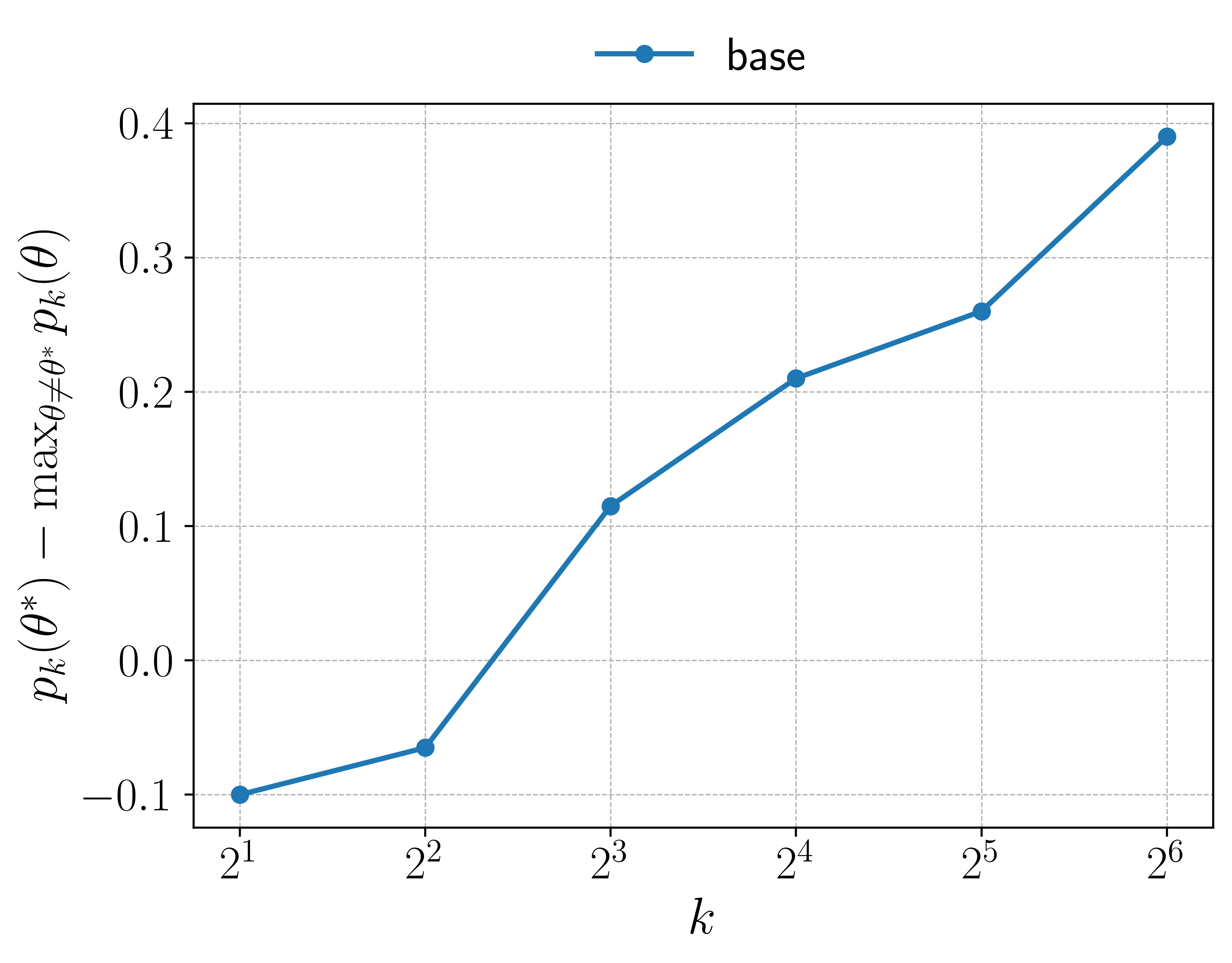}
\caption{Figure \ref{subfig: profile_alg1_k_2}: $p_k(\theta^*) - \max_{\theta\neq \theta^*}p_k(\theta)$.}\label{subfig: sskp_prob_diff_2}
\end{subfigure}
\begin{subfigure}{0.4\textwidth}
\includegraphics[width = \linewidth]{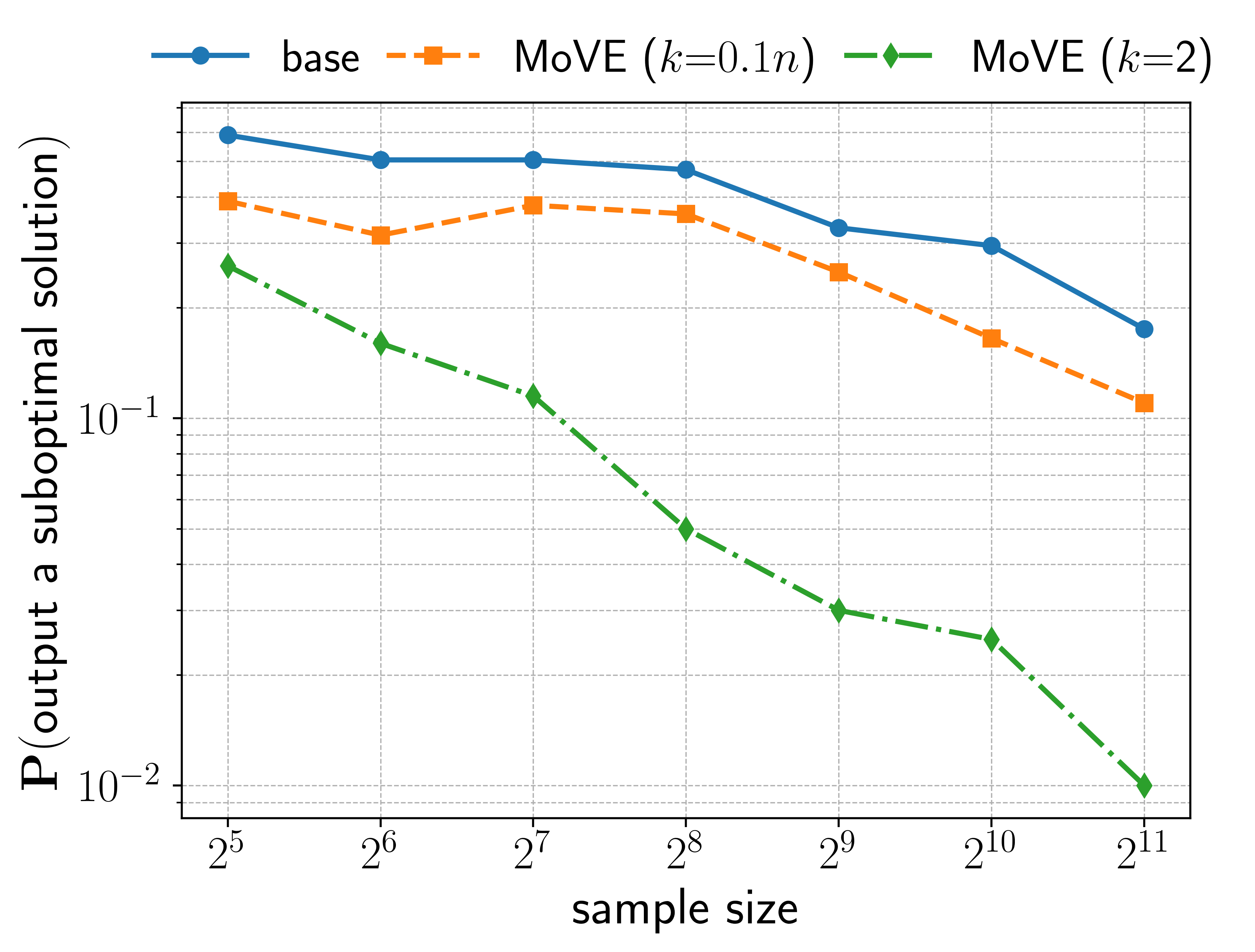}
\caption{Figure \ref{subfig: profile_alg1_k_1}: suboptimal probability.}\label{subfig: sskp_prob_suboptimal_1}
\end{subfigure}
\hspace{1cm}
\begin{subfigure}{0.4\textwidth}
\includegraphics[width = \linewidth]{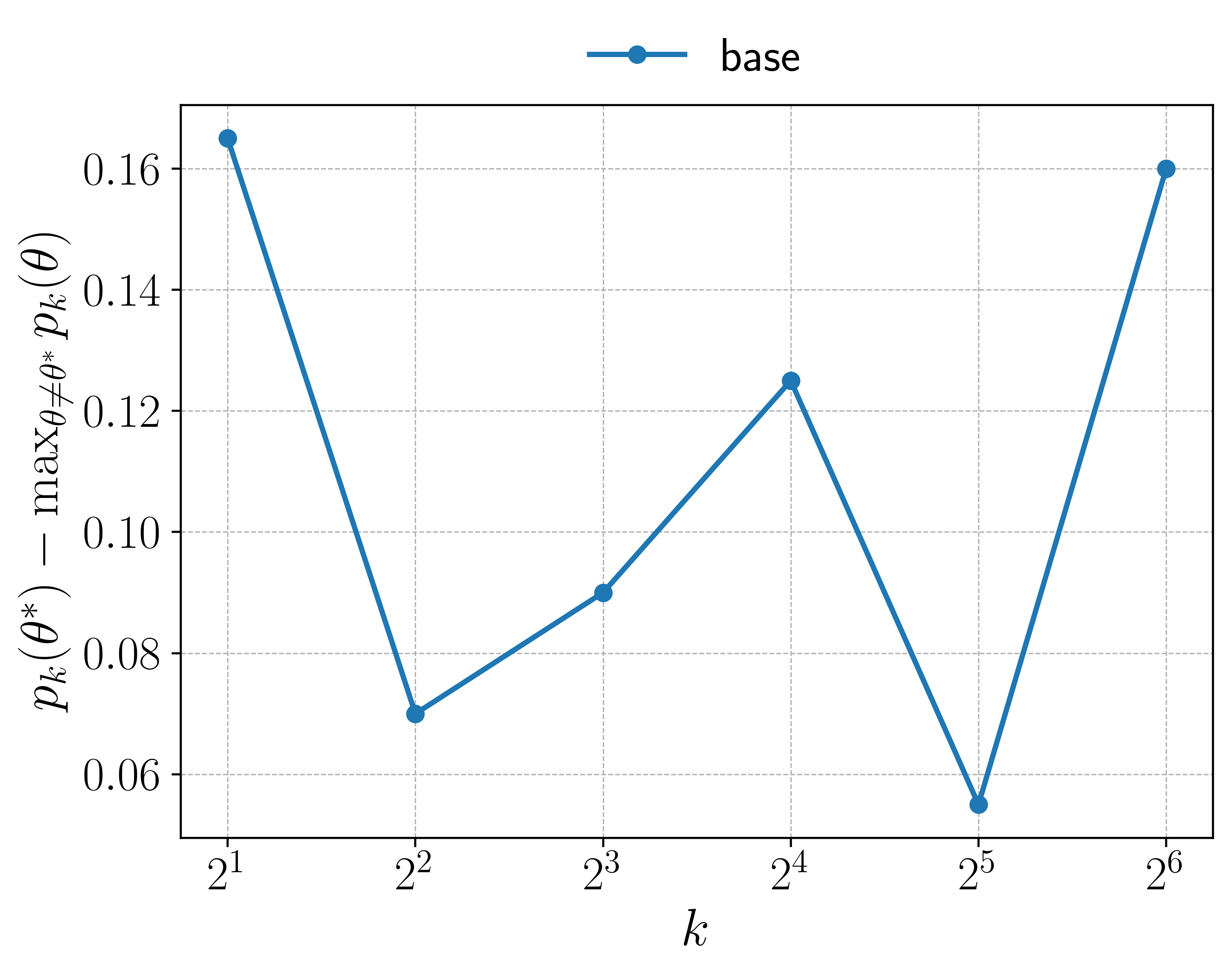}
\caption{Figure \ref{subfig: profile_alg1_k_1}: $p_k(\theta^*) - \max_{\theta\neq \theta^*}p_k(\theta)$.}\label{subfig: sskp_prob_diff_1}
\end{subfigure}
\caption{Performance of \move ($B=200$) in resource allocation, corresponding to the two instances in Figures \ref{subfig: profile_alg1_k_2} and \ref{subfig: profile_alg1_k_1}. 
Subfigures (b) and (d) explain the behaviors of \move with different subsample sizes $k$: In (b), we find that $p_k(\theta^*) - \max_{\theta\neq \theta^*}p_k(\theta)<0$ for $k\leq 4$, which results in the poor performance of \move with $k=2$ in Figure \ref{subfig: profile_alg1_k_2};
In (d), we have $p_2(\theta^*) - \max_{\theta\neq \theta^*}p_2(\theta)\approx 0.165$, thereby enabling \move to distinguish the optimal solution only using subsamples of size two, which results in the good performance of \move with $k=2$ in Figure \ref{subfig: profile_alg1_k_1}.
\label{fig: pkx_SSKP}}
\end{figure}

\begin{figure}[!htbp]
\centering
\hspace{-12pt}
\begin{subfigure}{0.34\textwidth}
\includegraphics[width = \linewidth]{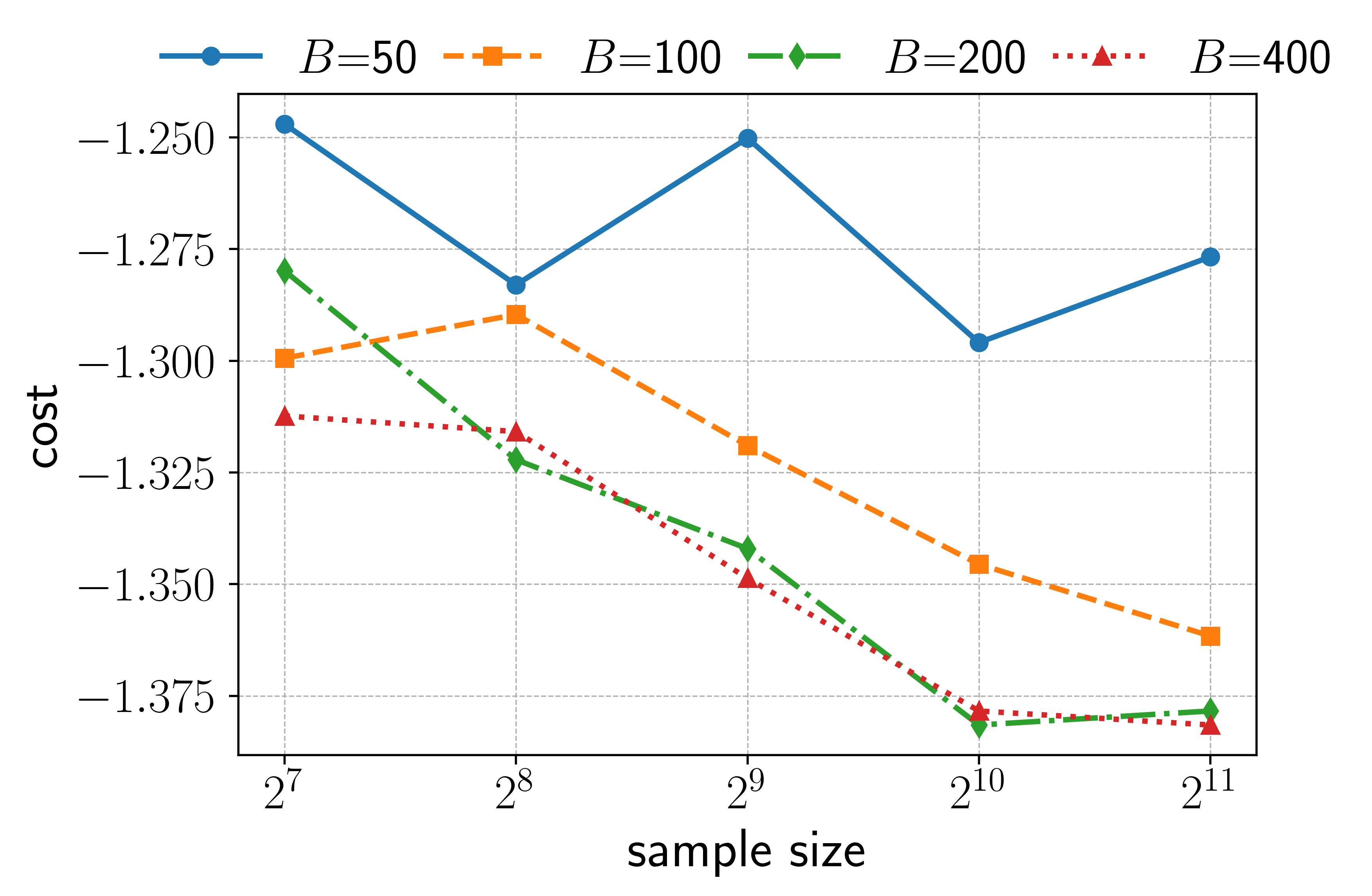}
\caption{\move.\label{subfig: profile_alg1_B}}
\end{subfigure}%
\begin{subfigure}{0.34\textwidth}
\includegraphics[width = \linewidth]{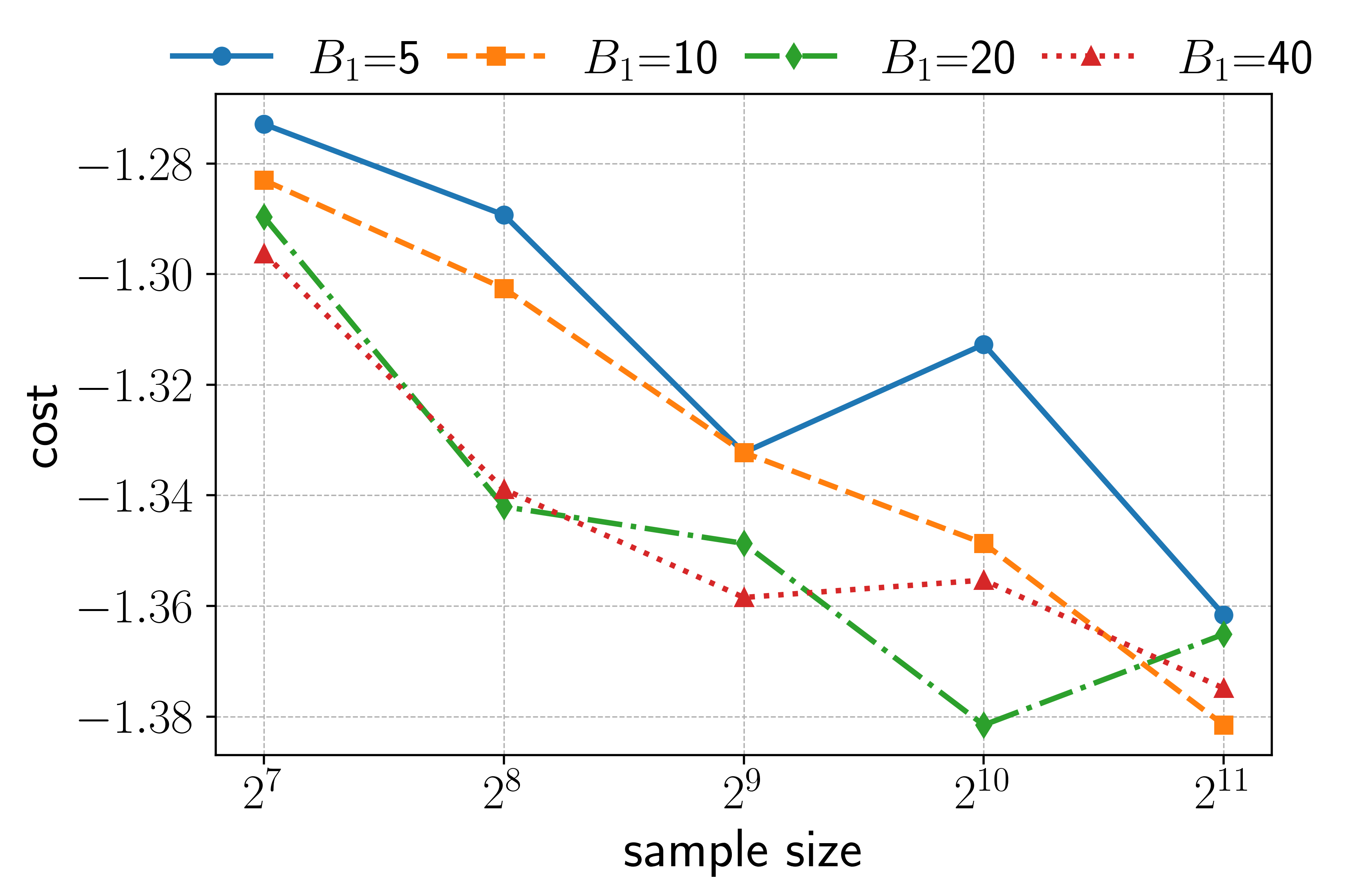}
\caption{\rove with $B_2 = 200$.\label{subfig: profile_alg3_B1}}
\end{subfigure}%
\begin{subfigure}{0.34\textwidth}
\includegraphics[width = \linewidth]{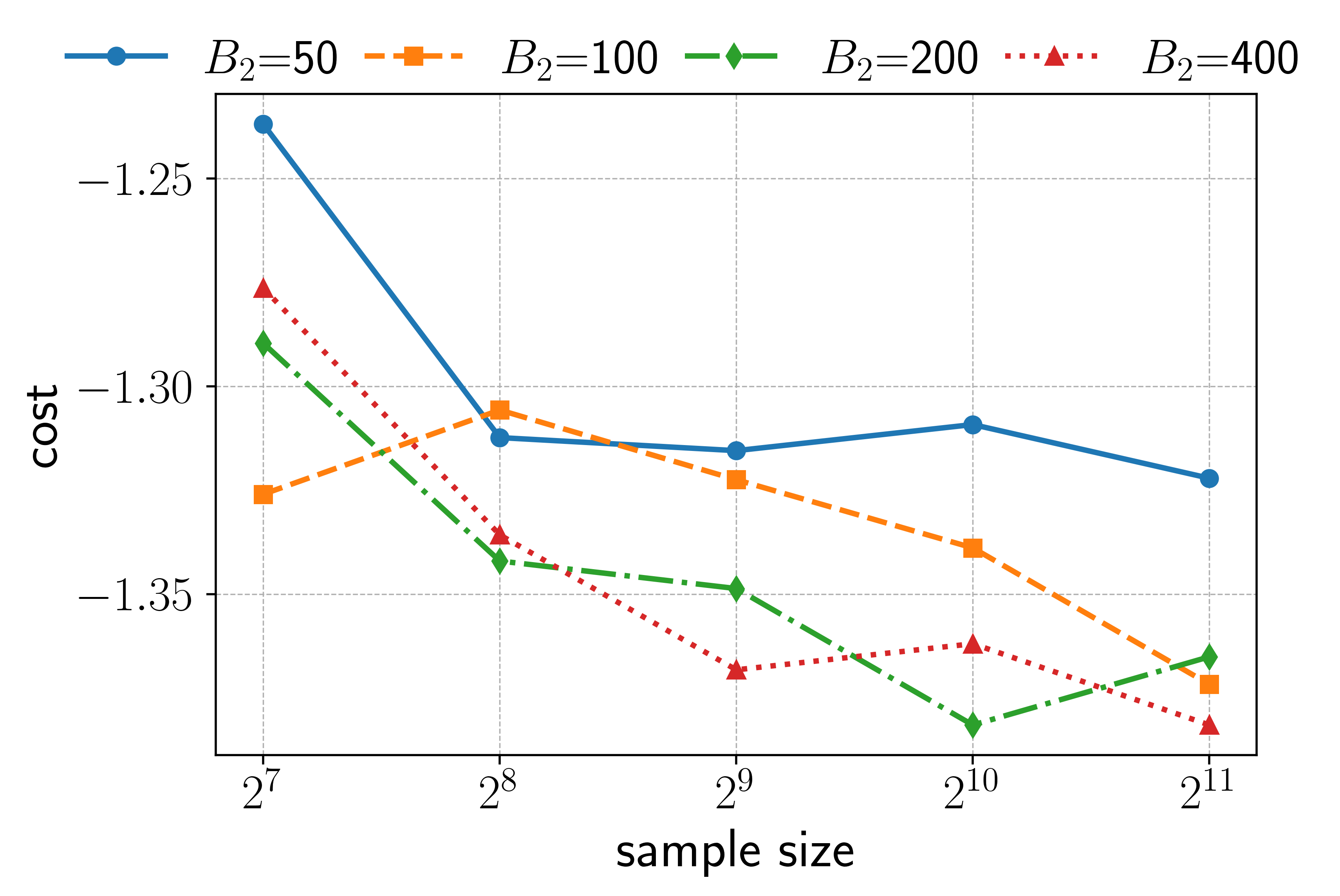}
\caption{\rove with $B_1 = 20$.\label{subfig: profile_alg3_B2}}
\end{subfigure}
\caption{Profiling for ensemble sizes $B, B_1, B_2$ in resource allocation. Subsample size is chosen as $k=k_1=k_2= 10$.
% We find that the performance of bagging is increasing in the ensemble sizes.
% Algorithm \ref{bagging majority vote: two phase} uses $k=10$ and $\epsilon = 0.005$.
\label{fig: profiling_B}}
\end{figure}

\begin{figure}[!htbp]
\centering
\hspace{-12pt}
\begin{subfigure}{0.34\textwidth}
\includegraphics[width = \linewidth]{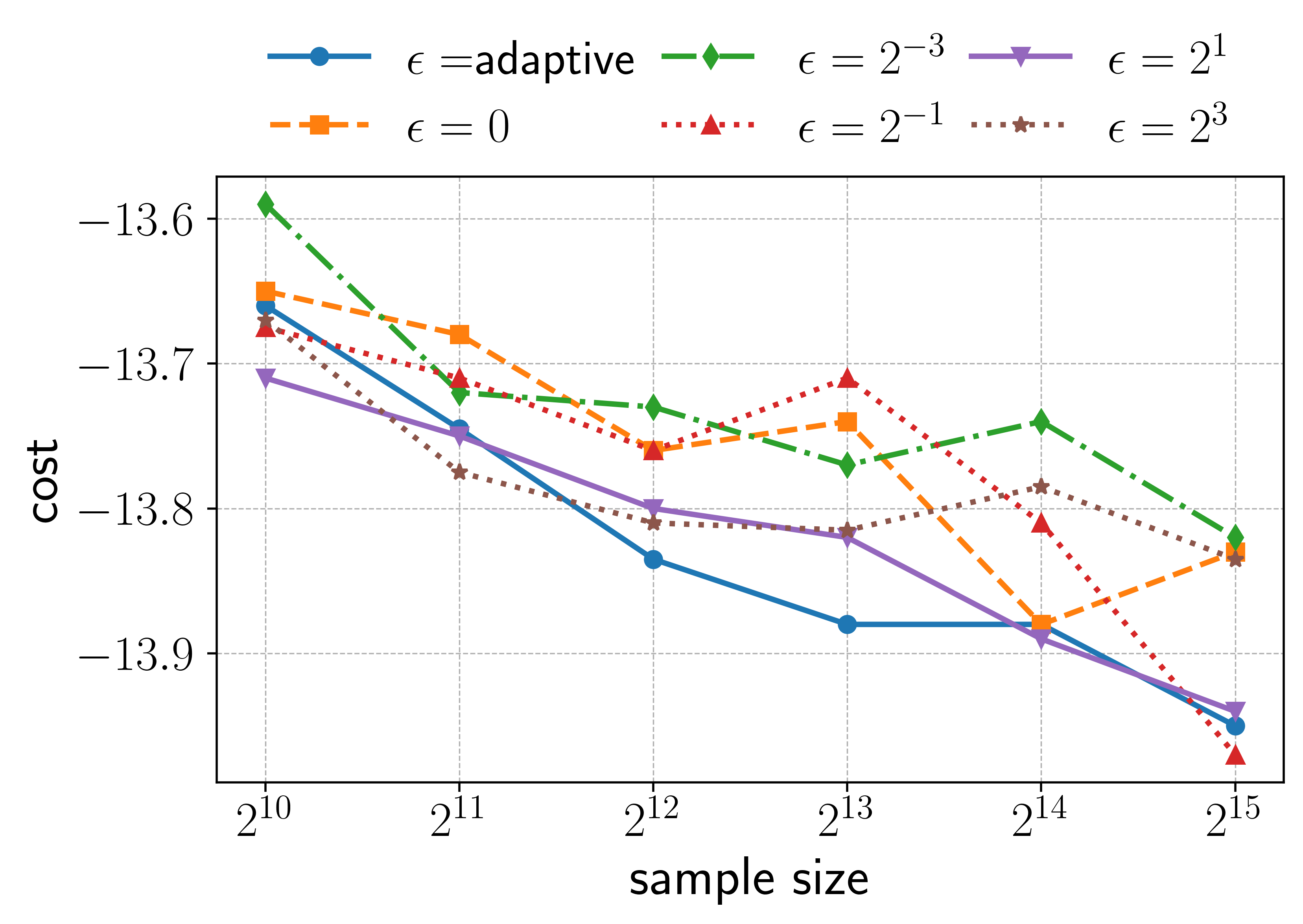}
\caption{Multiple optima.\label{subfig: profile_alg4_epsilon_multiple}}
\end{subfigure}%
\begin{subfigure}{0.34\textwidth}
\includegraphics[width = \linewidth]{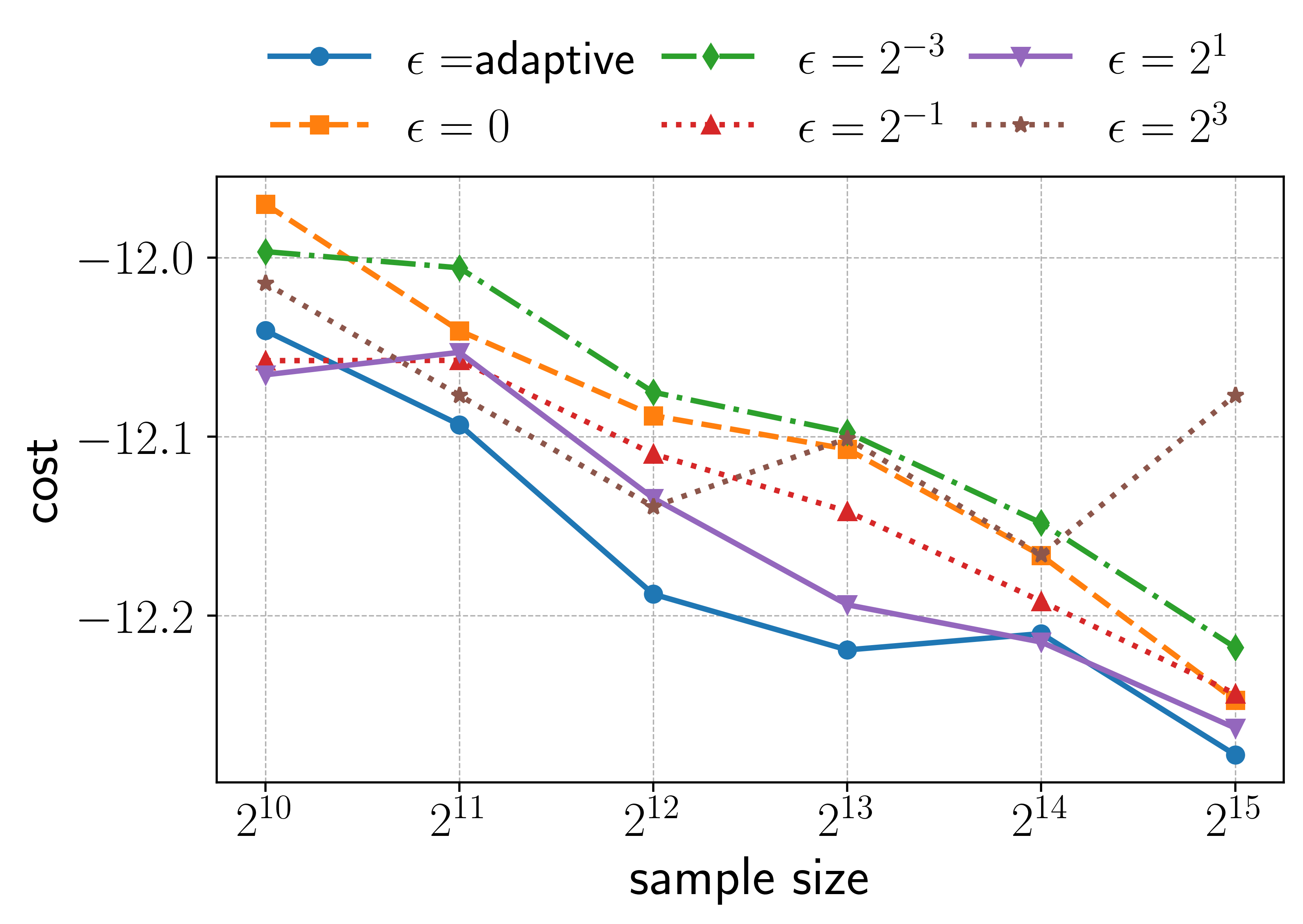}
\caption{Multiple near optimum.\label{subfig: profile_alg4_epsilon_near}}
\end{subfigure}%
\begin{subfigure}{0.34\textwidth}
\includegraphics[width = \linewidth]{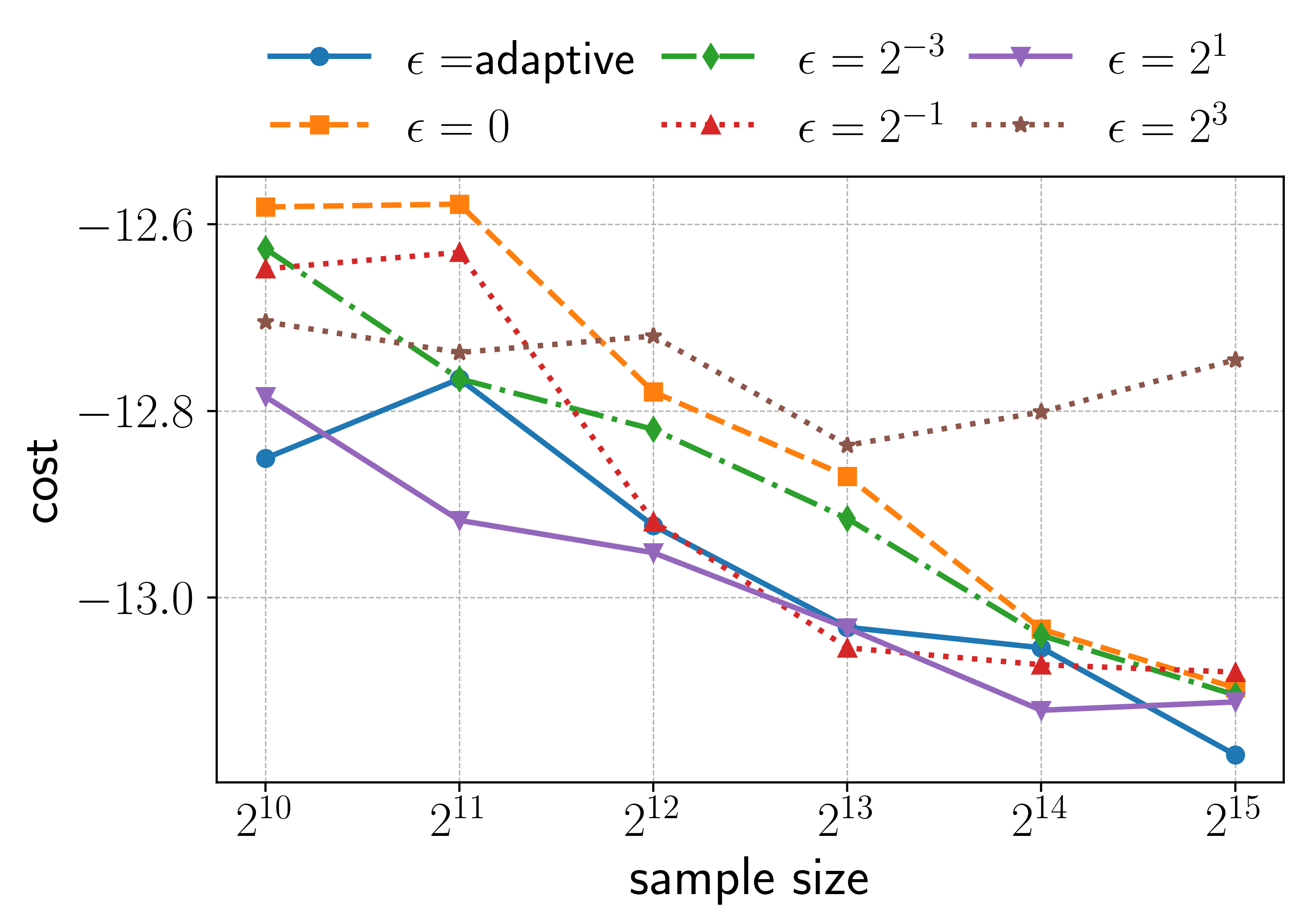}
\caption{Single optimum.\label{subfig: profile_alg4_epsilon_single}}
\end{subfigure}
\caption{Performance of \roves in three instances of linear programs under different thresholds $\epsilon$.
The setting is identical to that of Figures \ref{subfig: profile_alg3_epsilon_multiple}, \ref{subfig: profile_alg3_epsilon_single}, and \ref{subfig: profile_alg3_epsilon_near} for \rove.
Hyperparameters: $k_1=k_2=\max(10,0.005n)$, $B_1= 20$, and $B_2 = 200$. 
Compared with profiling results for \rove, we observe that the value of $\epsilon$ has similar impacts on the performance of \roves. Moreover, the proposed adaptive strategy also behaves well for \roves.
\label{fig: profiling_epsilon_alg4}}
\end{figure}

\newpage
\subsection{Additional Experimental Results}\label{app: additional_figures}
Here, we present additional figures that supplement the experiments and discussions in Section \ref{section: experiments}.
Recall that \move refers to Algorithm \ref{bagging majority vote: set estimator}, \rove refers to Algorithm \ref{bagging majority vote: two phase} without data splitting, and \roves refers to Algorithm \ref{bagging majority vote: two phase} with data splitting.
We briefly introduce each figure below and refer the reader to the figure caption for detailed discussions. Figures \ref{fig: alg_comparison_MLP appendix}-\ref{fig: bagging_dro} all follow the recommended configuration listed in Section \ref{section: experiments}.
\begin{itemize}[leftmargin=*]
    % \item In Figure \ref{fig: profiling_additional}, we provide additional profiling results for $k$ (equivalently $k_1$ and $k_2$) and $\epsilon$ as supplement to Figure \ref{fig: profiling_main_paper}.
    % \item In Figure \ref{fig: profiling_B}, we show the profiling results for subsample numbers, i.e., $B, B_1, B_2$.
    % \item In Figure \ref{fig: profiling_epsilon_alg4}, we show the performance of ROVEs during the profiling stage of threshold $\epsilon$, i.e., the counterpart of Figures \ref{subfig: profile_alg3_epsilon_multiple}, \ref{subfig: profile_alg3_epsilon_single}, and \ref{subfig: profile_alg3_epsilon_near} for ROVE.
    % \item In Figure \ref{fig: pkx_SSKP}, we simulate the probability $p_k(\theta)$ for SAA and display the probability that SAA or MoVE outputs a suboptimal solution. This helps explain the results in the profiling stage of subsample size $k$ in Figures \ref{subfig: profile_alg1_k_2} and \ref{subfig: profile_alg1_k_1}.
    \item Figure \ref{fig: alg_comparison_MLP appendix} supplements the results in Figure \ref{fig: alg_comparison_MLP} with MLPs with $H=2,4$ hidden layers.
    \item Figure \ref{fig: tree methods comparison appendix} supplements the results in Figure \ref{fig: tree methods comparison main paper} with a different synthetic example than in Section \ref{subsec:neural network experiment}.
    \item Figure \ref{fig: MLP L=4 public data plots appendix} supplements the results in Figure \ref{fig: MLP L=4 public data plots main paper} with three other real datasets: \textit{Wine Quality} \citeAPX{wine_quality_186} \textit{Online News} \citeAPX{online_news_popularity_332}, \textit{Appliances Energy} \citeAPX{appliances_energy_prediction_374}. \rove and the base algorithm perform comparably on these three datasets, potentially because they are lighter tailed than those in Figure \ref{fig: MLP L=4 public data plots main paper}.
    \item Figure \ref{fig: MLP d=30 L=4 plots appendix} shows results for MLP regression on a slightly different example than in Section \ref{subsec:neural network experiment}.
    \item Figures \ref{fig: linear regression plots appendix} and \ref{fig: ridge regression plots appendix} show results for regression with least squares regression and Ridge regression as the base learning algorithms respectively.
    \item Figure \ref{fig: comparison_light_tail} shows results on the stochastic linear program example with light-tailed uncertainties.
    \item Figure \ref{fig: comparison_network appendix} contains additional results on the supply chain network design example for different choices of hyperparameters and a different problem instance with strong correlation between solutions.
    \item In Figure \ref{fig: bagging_dro}, we apply our ensemble methods to resource allocation and maximum weight matching using DRO with Wasserstein metric as the base algorithm. This result, together with Figure \ref{fig: alg_comparison_SAA} where the base algorithm is SAA, demonstrates that the benefit of our ensemble methods is agnostic to the underlying base algorithm.
    \item In Figure \ref{fig: eta_comparison}, we simulate the generalization sensitivity $\bar{\eta}_{k,\delta}$, defined in \eqref{SAA prob gap 2}, which explains the superior performance of \rove and \roves in the presence of multiple optimal solutions.
    % \item In Figure \ref{fig: saa_bagging_vs_dro}, we compare the bagging approaches (using SAA as the base model) with Wasserstein metric-based DRO.
    \item In Figure \ref{fig: tail_heaviness}, we demonstrate how the tail heaviness of the problem affects the algorithm performance.
    The figure shows that the performance gap between \rove, \roves, and the base algorithm becomes increasingly significant as the tail of the uncertainty becomes heavier. This supports the effectiveness of \rove and \roves in handling heavy-tailed uncertainty, where the base algorithm's performance suffers. Note that here \move behaves similarly as the base learner due to optima multiplicity.
    % \item In Figure \ref{fig: running_time}, we show the running time of different methods in the network design problem against the total number of samples. We find that when the subsample number and size are fixed, bagging approaches have a significant computational advantage over the base model SAA.
\end{itemize}

\begin{figure}[!htbp]
\centering
\hspace{-13pt}
\begin{subfigure}{0.33\textwidth}
\includegraphics[width = \linewidth]{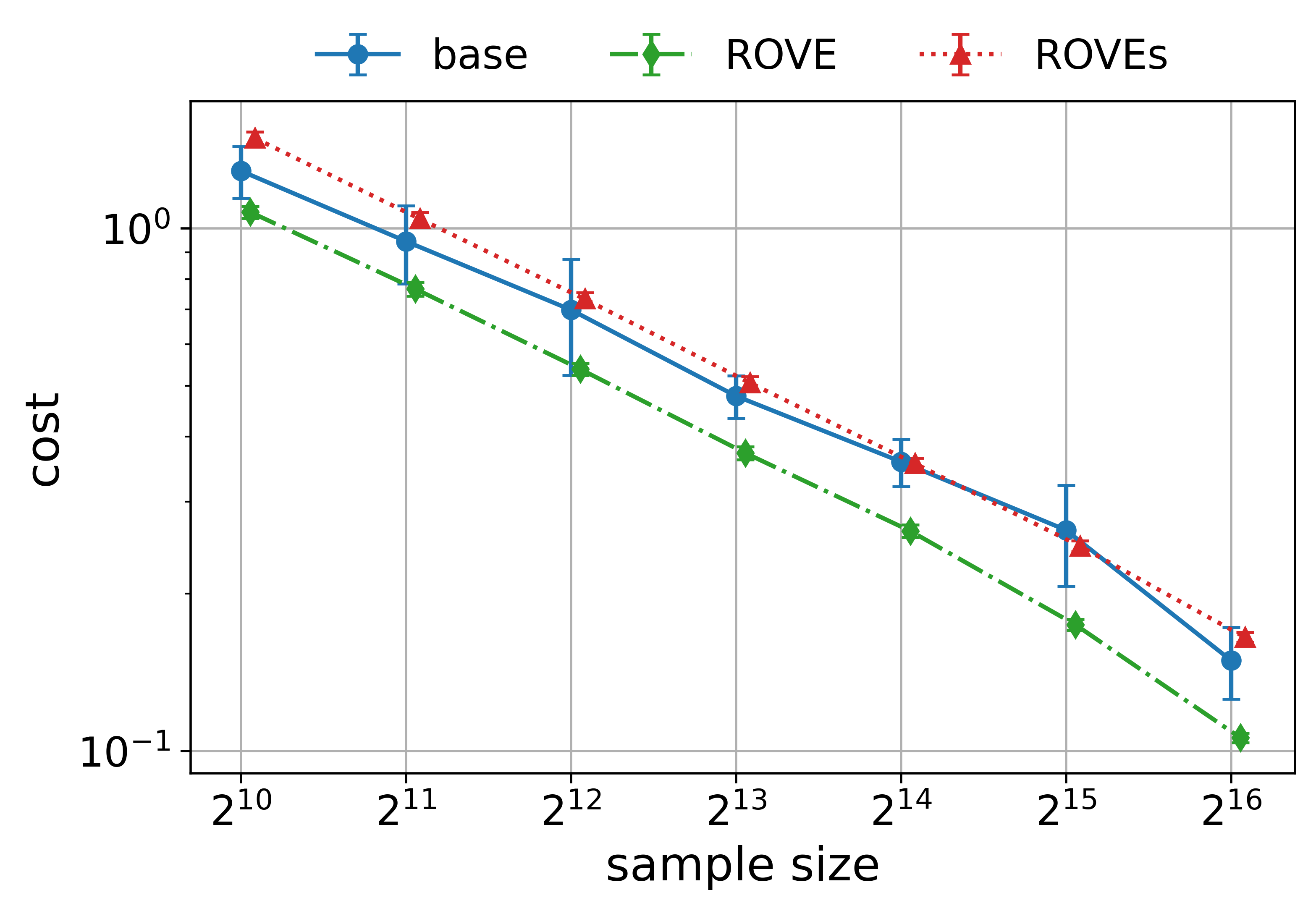}
\caption{Pareto noise, $H=2$.\label{subfig: MLP avg d=50 L=2}}
\end{subfigure}
\begin{subfigure}{0.33\textwidth}
\includegraphics[width = \linewidth]{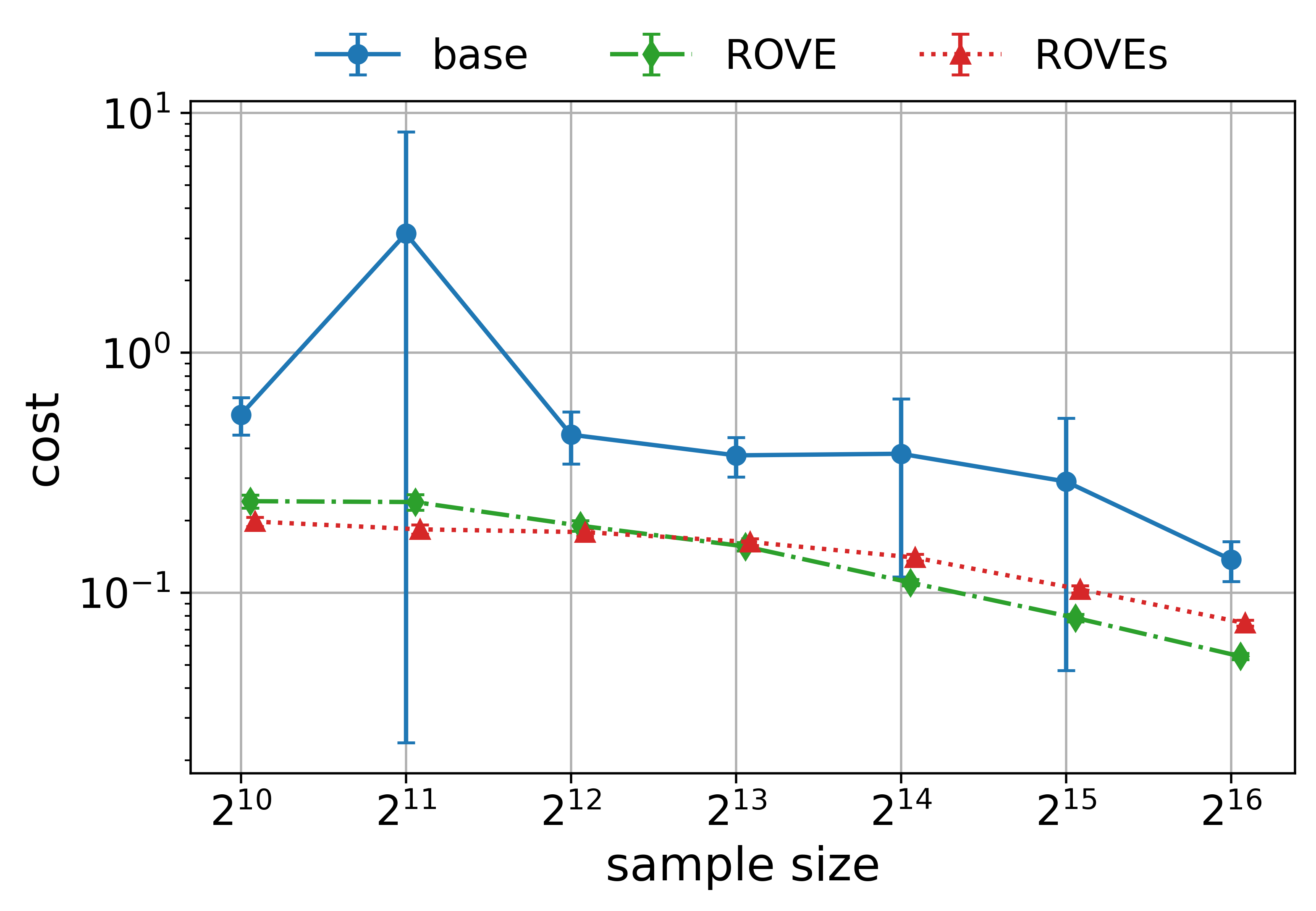}
\caption{Pareto noise, $H=6$.\label{subfig: MLP avg d=50 L=6}}
\end{subfigure}
\begin{subfigure}{0.33\textwidth}
\includegraphics[width = \linewidth]{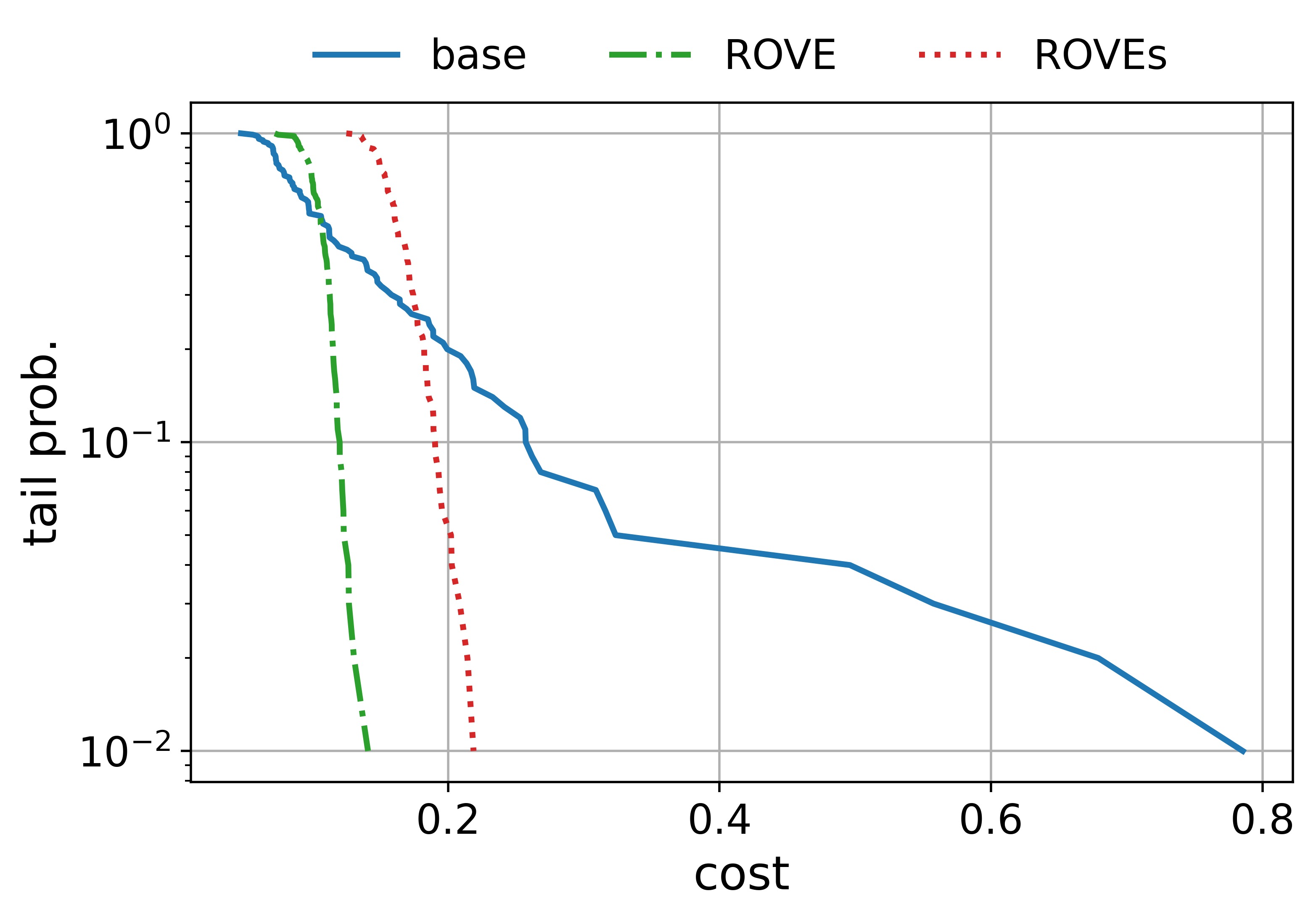}
\caption{Pareto noise, $H=2,n=2^{16}$.\label{subfig: MLP tail d=50 L=2 n=2^16}}
\end{subfigure}\\
\hspace{-13pt}
\begin{subfigure}{0.33\textwidth}
\includegraphics[width = \linewidth]{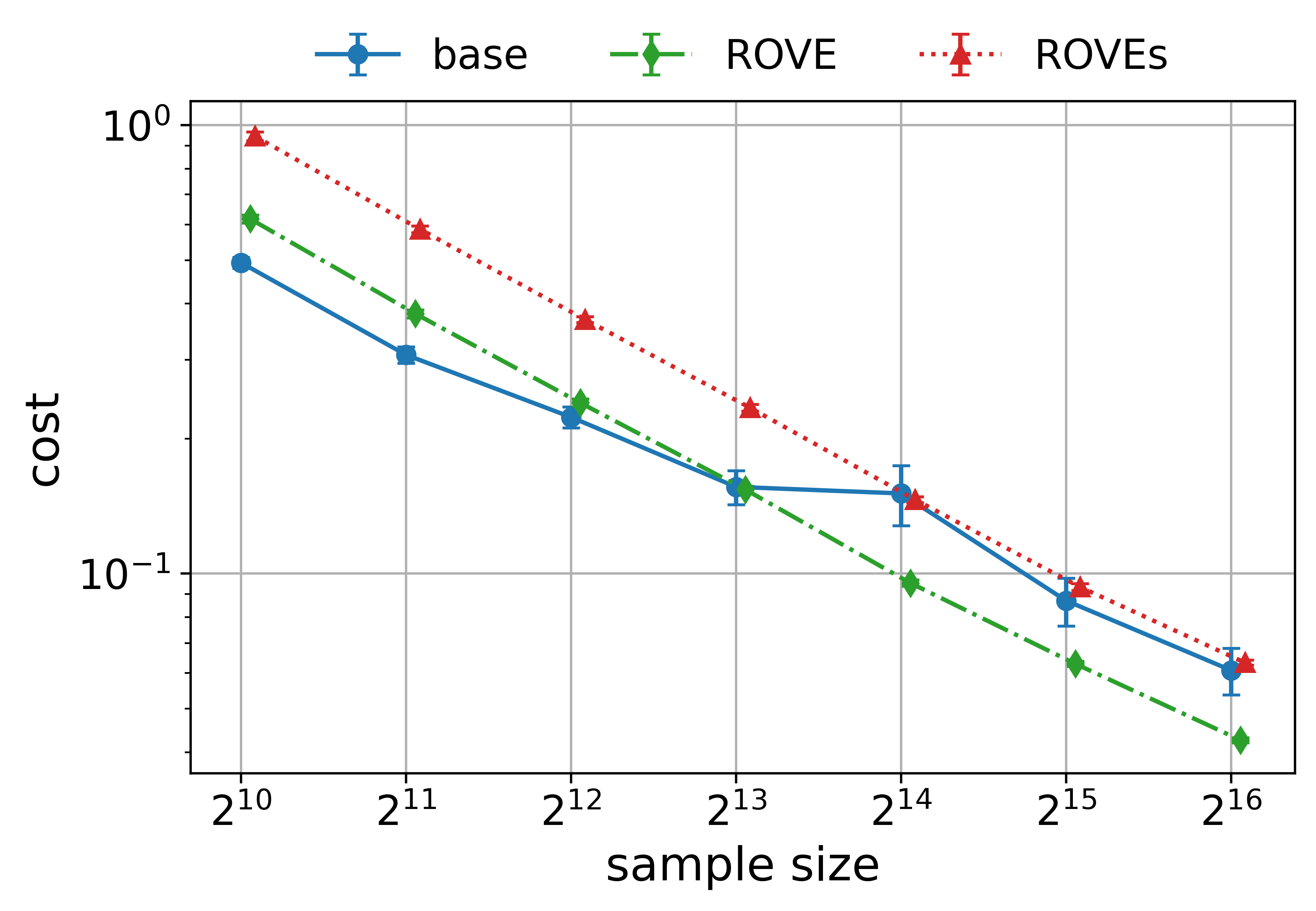}
\caption{Gaussian noise, $H=2$.\label{subfig: MLP avg d=50 L=2 light}}
\end{subfigure}
\begin{subfigure}{0.33\textwidth}
\includegraphics[width = \linewidth]{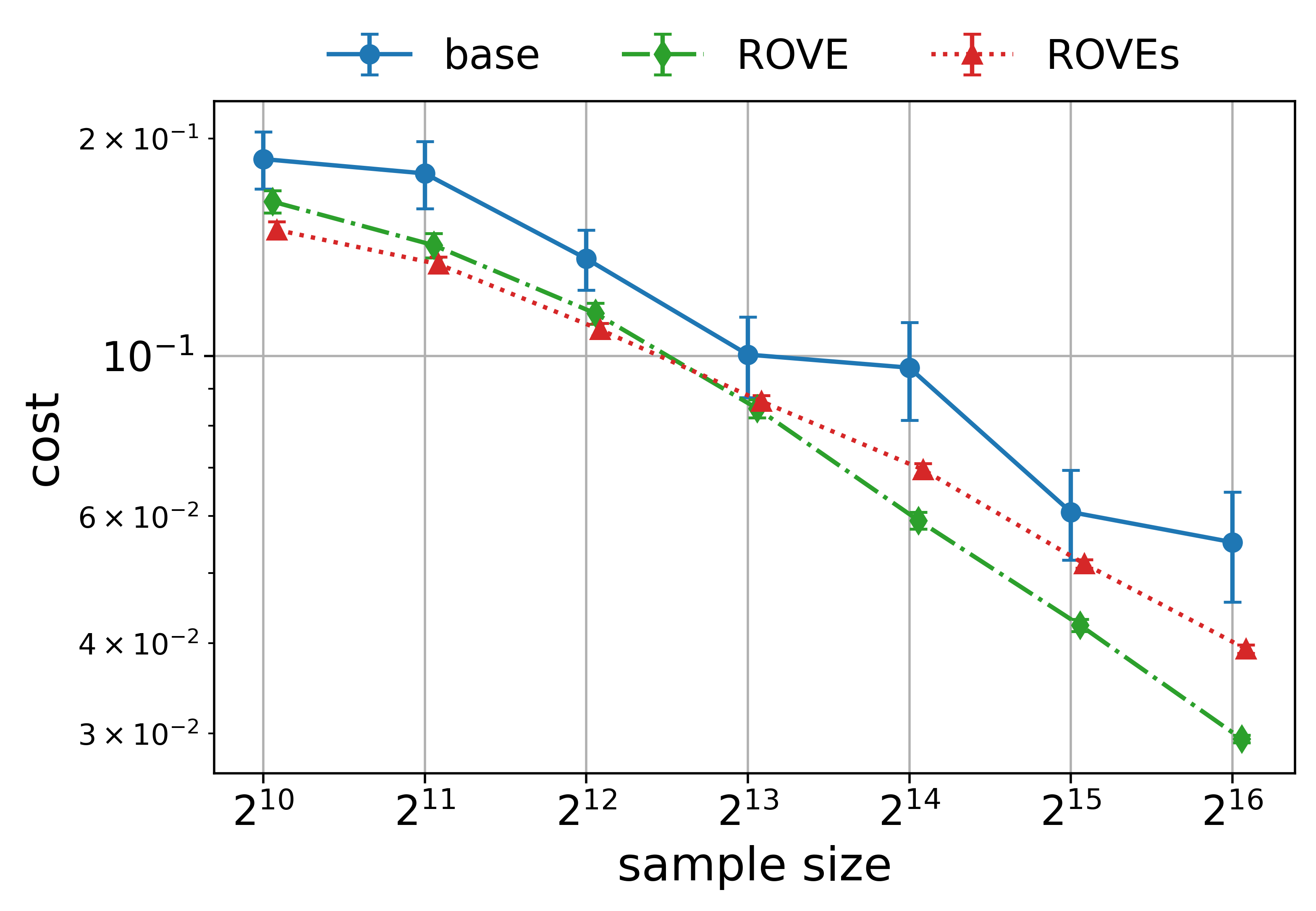}
\caption{Gaussian noise, $H=6$.\label{subfig: MLP avg d=50 L=6 light}}
\end{subfigure}
\begin{subfigure}{0.33\textwidth}
\includegraphics[width = \linewidth]{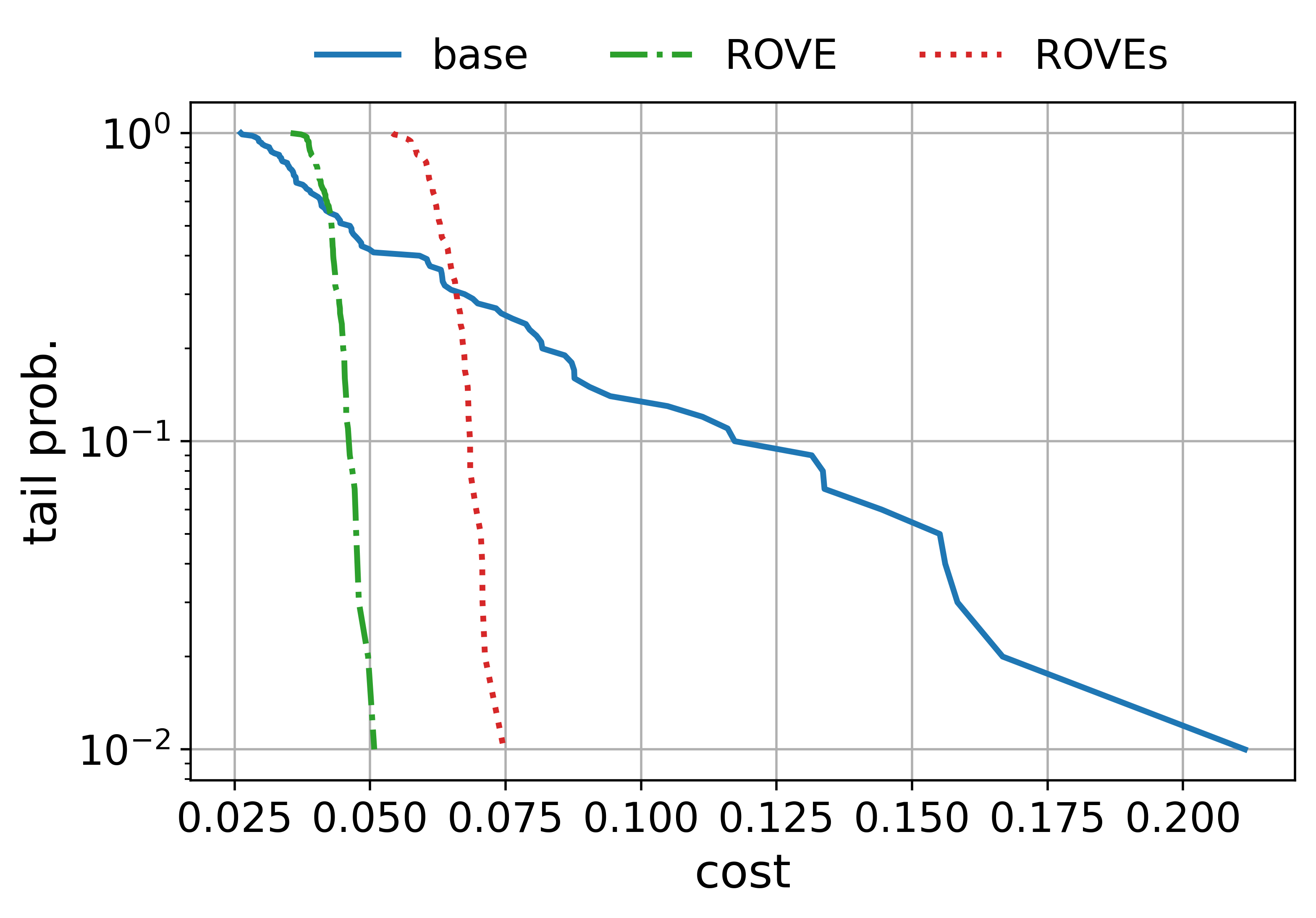}
\caption{Gaussian noise, $H=2,n=2^{16}$.\label{subfig: MLP tail d=50 L=2 n=2^16 light}}
\end{subfigure}
\caption{Results of neural networks with the same setup described in Section \ref{subsec:neural network experiment}. (a)(b)(d)(e): Expected out-of-sample costs (MSE) with $95\%$ confidence intervals under different noise distributions and varying numbers of hidden layers ($H$). (c) and (f): Tail probabilities of out-of-sample costs. In (a), \roves slightly underperforms the base learner probably due to the weak expressiveness and hence high bias of the MLP with $2$ hidden layers.
\label{fig: alg_comparison_MLP appendix}}
\end{figure}

\begin{figure*}[!htbp]
\centering
\hspace{-12pt}
\begin{subfigure}{0.49\textwidth}
\includegraphics[width = \linewidth]{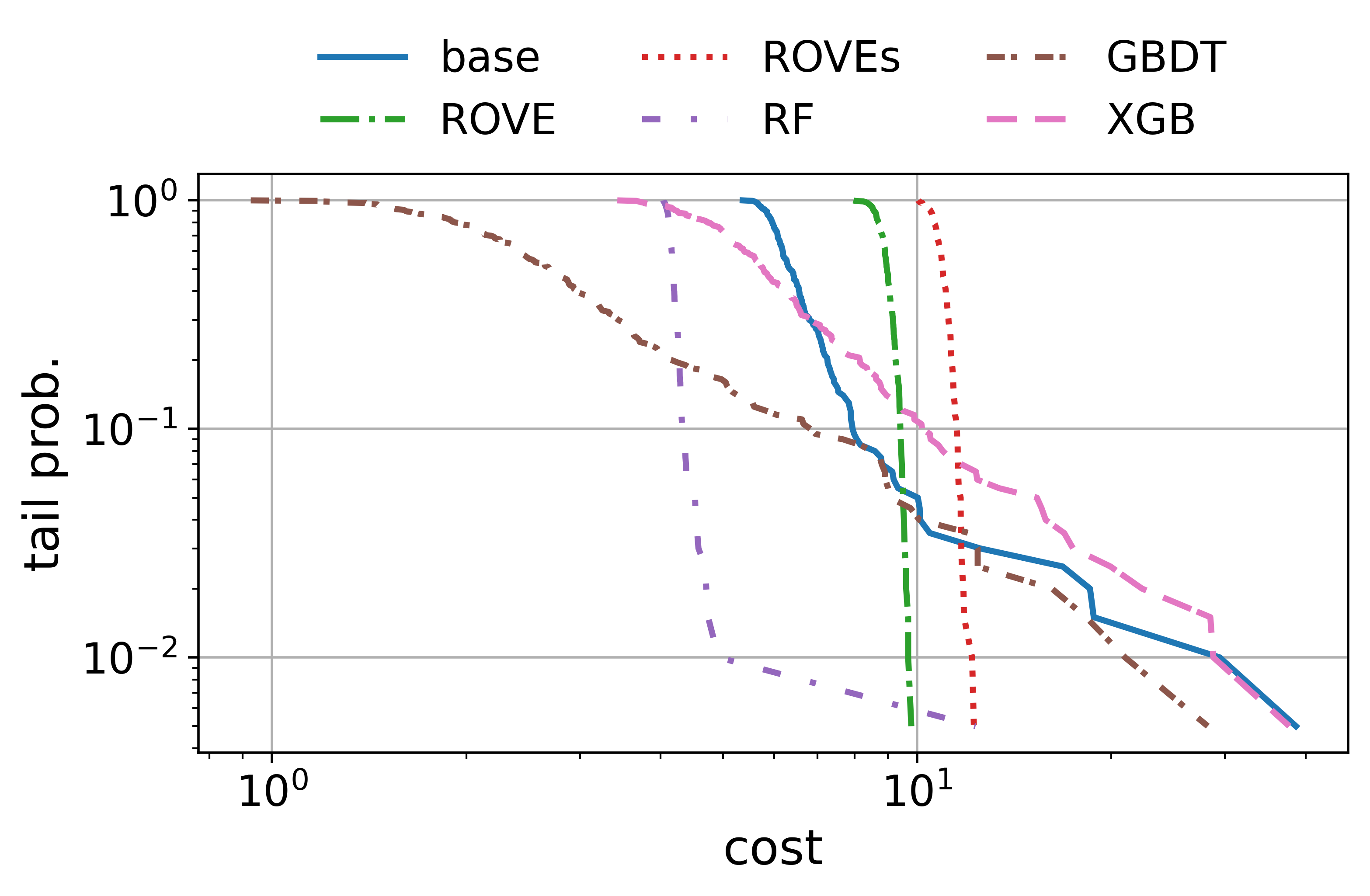}
\caption{Pareto shape = $2.0,n=2^{15}$.\label{subfig: sine tree methods tail shape 2.0}}
\end{subfigure}
\hspace{-6pt}
\begin{subfigure}{0.49\textwidth}
\includegraphics[width = \linewidth]{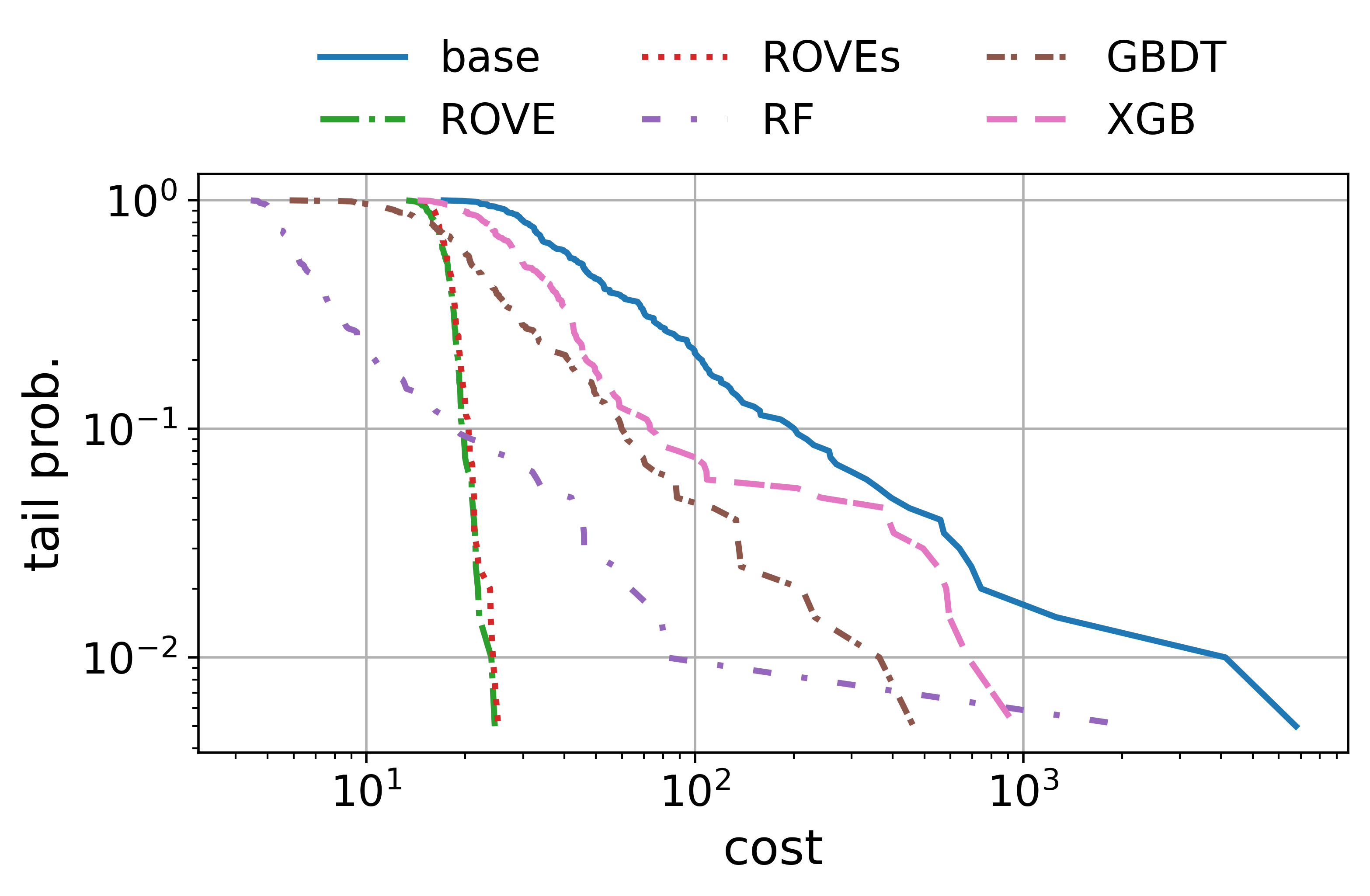}
\caption{Pareto shape = $1.5,n=2^{15}$.\label{subfig: sine tree methods tail shape 1.5}}
\end{subfigure}
\caption{Results of decision trees in terms of tail probabilities of out-of-sample costs (MSE). The data generation is $Y=10\sum_{i=0}^{9}2^{-i}\sin(2\pi X_i)+\epsilon_1-\epsilon_2$, where each $X_i$ is independently and uniformly drawn from $[0,1]$, and $\epsilon_1,\epsilon_2$ are independent Pareto variables of the same parameters. Hyperparameters: $k_1=\max(30,n/10),k_2=\max(30,n/200),B_1=50,B_2=200$.\label{fig: tree methods comparison appendix}}
\end{figure*}

\begin{figure*}[!htbp]
\centering
\hspace{-12pt}
\begin{subfigure}{0.33\textwidth}
\includegraphics[width = \linewidth]{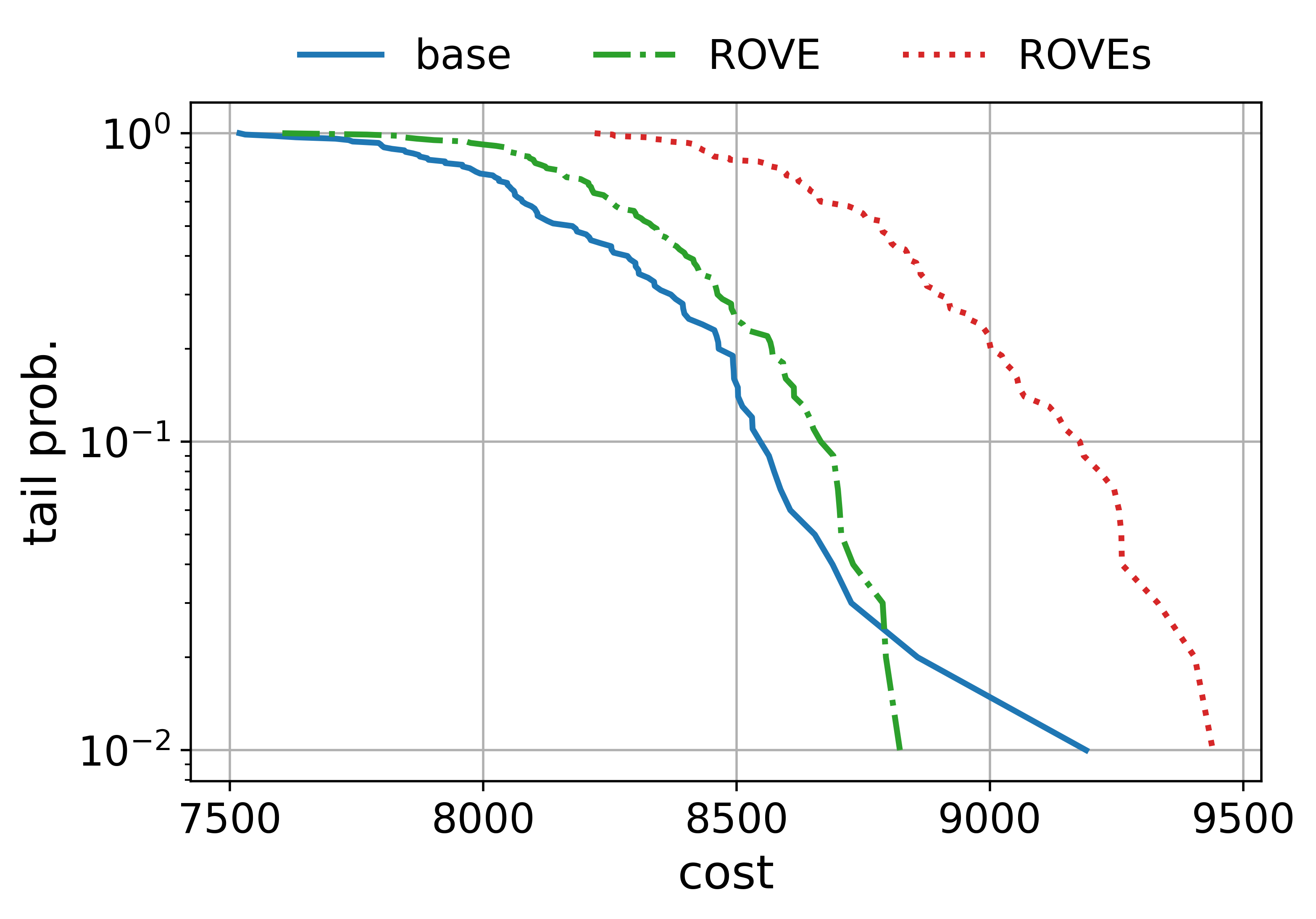}
\caption{\textit{Appliances Energy}.\label{subfig: appliances energy tail}}
\end{subfigure}
\hspace{-6pt}
\begin{subfigure}{0.33\textwidth}
\includegraphics[width = \linewidth]{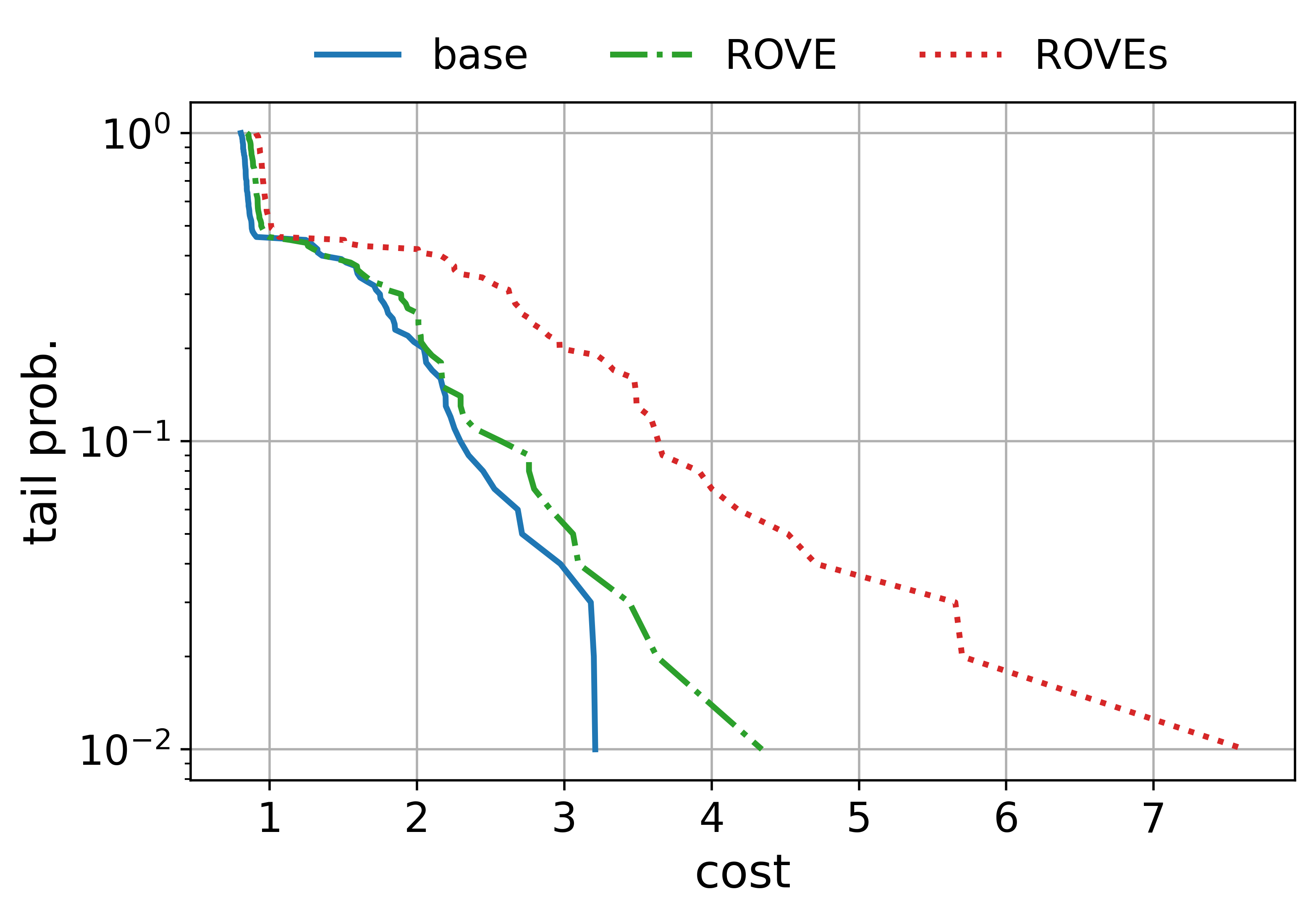}
\caption{\textit{Online News}.\label{subfig: online news tail}}
\end{subfigure}
\hspace{-6pt}
\begin{subfigure}{0.33\textwidth}
\includegraphics[width = \linewidth]{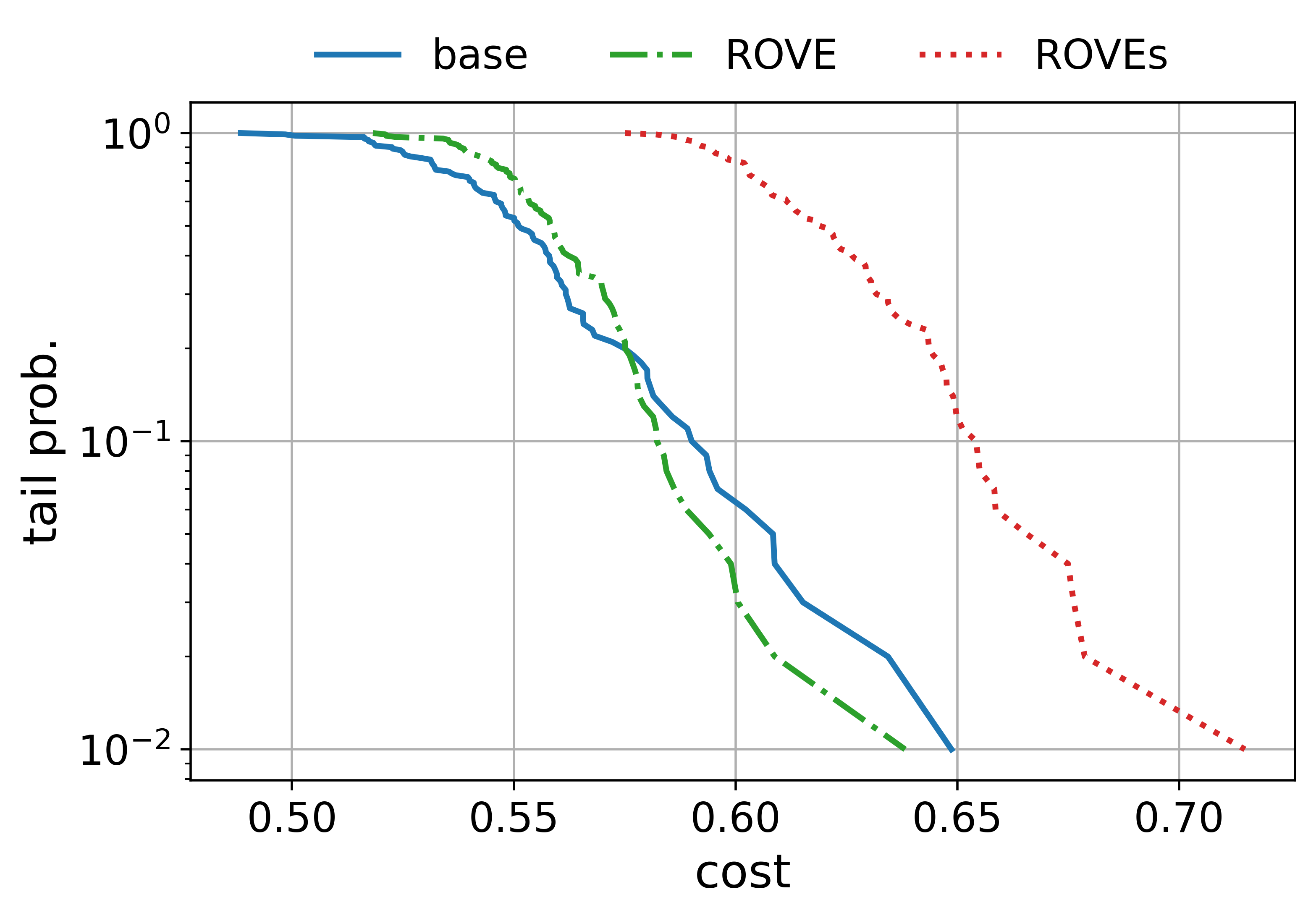}
\caption{\textit{Wine Quality}.\label{subfig: wine quality tail}}
\end{subfigure}
\caption{Results of neural networks with $4$ hidden layers on another three real datasets, in terms of tail probabilities of out-of-sample costs (MSE).\label{fig: MLP L=4 public data plots appendix}}
\end{figure*}

\begin{figure}[!htbp]
\centering
% \hspace{-8pt}
\begin{subfigure}{0.49\textwidth}
\includegraphics[width = \linewidth]{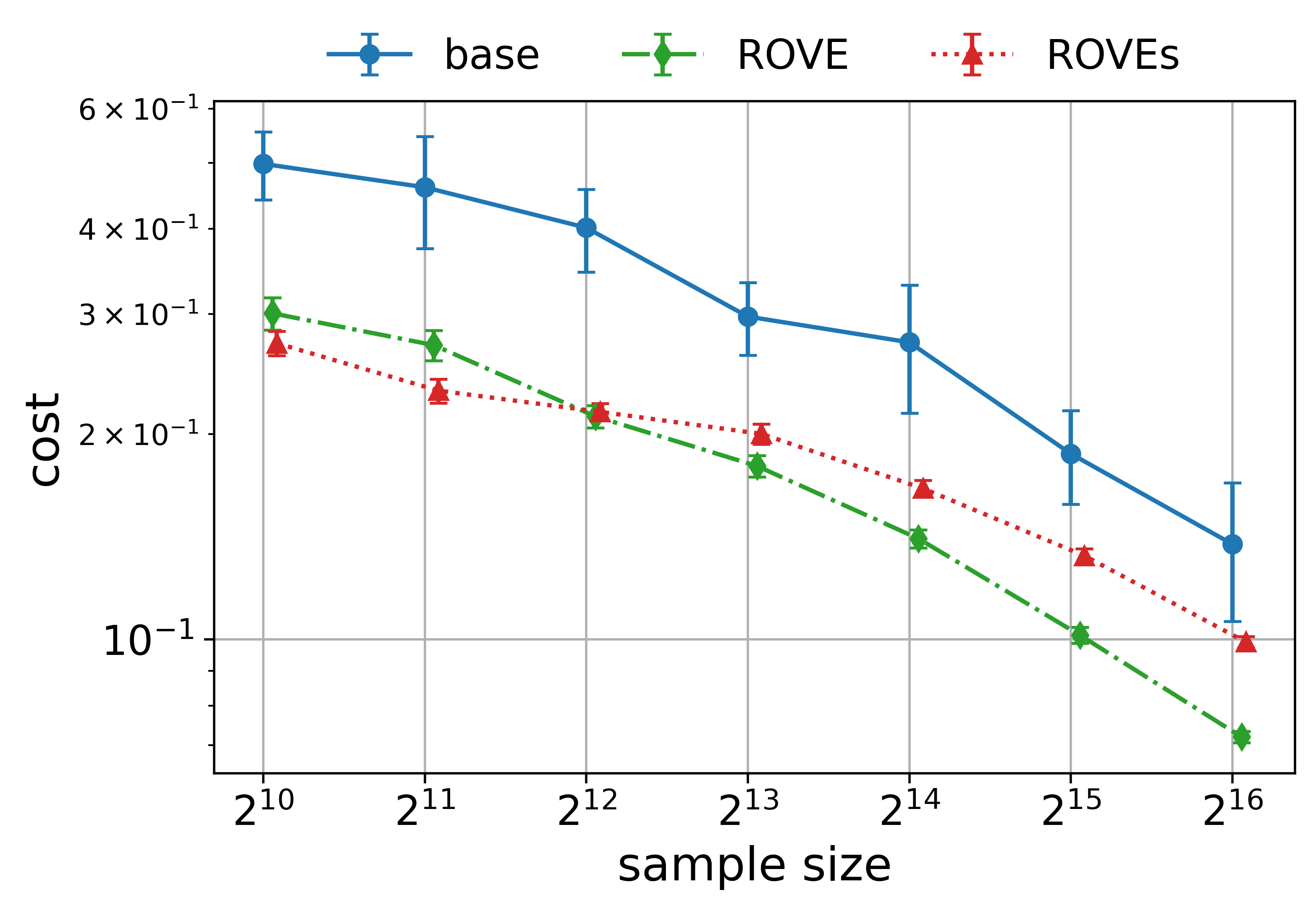}
\caption{Average performance with $95\%$ CIs.\label{subfig: MLP avg d=30 L=4}}
\end{subfigure}
\begin{subfigure}{0.49\textwidth}
\includegraphics[width = \linewidth]{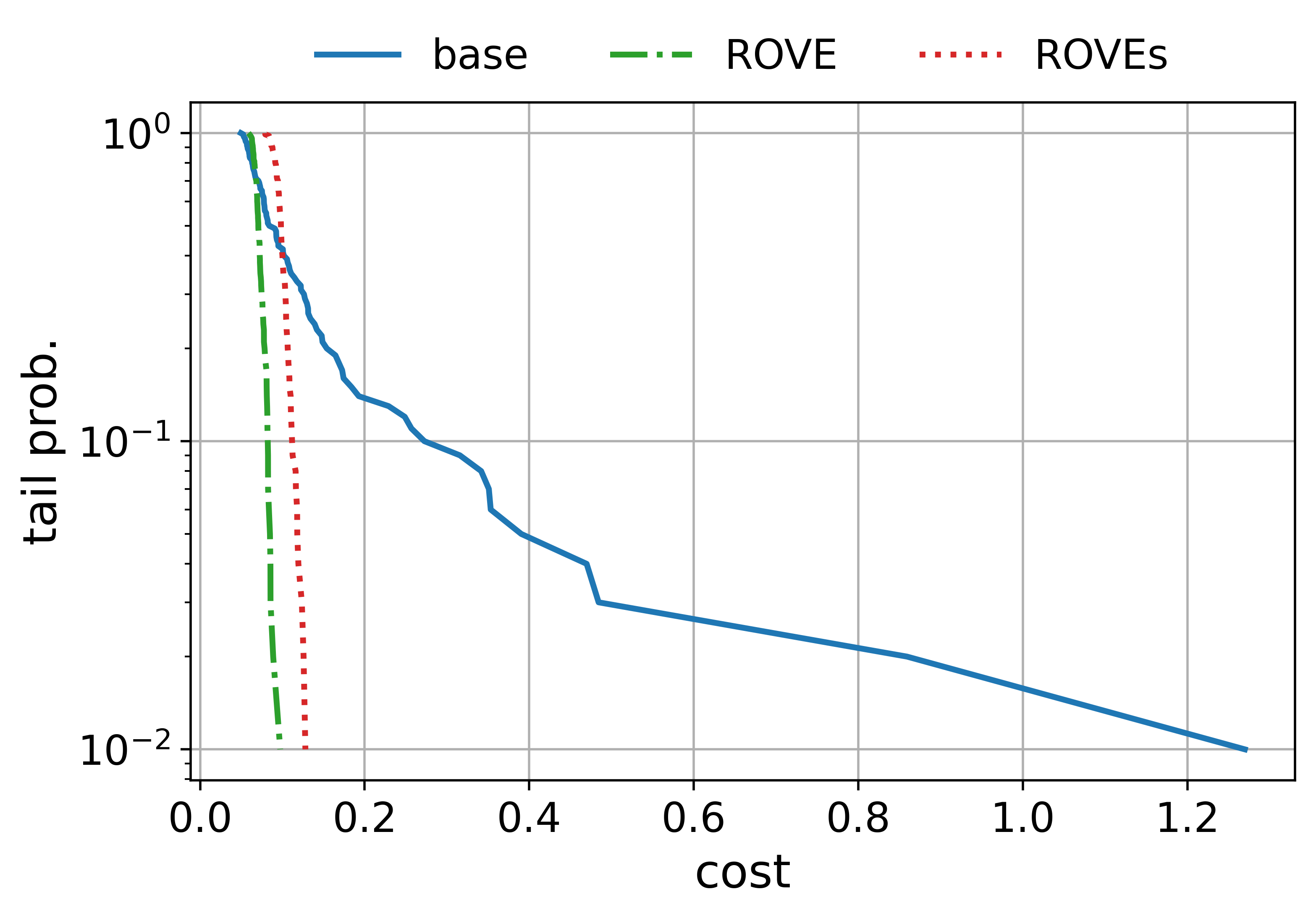}
\caption{Tail performance for $n=2^{16}$.\label{subfig: MLP tail d=30 L=4 n=2^16}}
\end{subfigure}
\caption{Results with an MLP of $H=4$ hidden layers. The setup is the same as in Section \ref{subsec:neural network experiment} except that the dimension of $X$ is now $30$ and the data generation becomes $Y = (1/30)\cdot \sum_{j=1}^{30} \log(X_j + 1) + \varepsilon$, where each \(X_j\) is drawn independently from \(\mathrm{Unif}(0, 2 + 198(j-1)/29)\).\label{fig: MLP d=30 L=4 plots appendix}}
\end{figure}

\begin{figure}[!htbp]
\centering
% \hspace{-8pt}
\begin{subfigure}{0.49\textwidth}
\includegraphics[width = \linewidth]{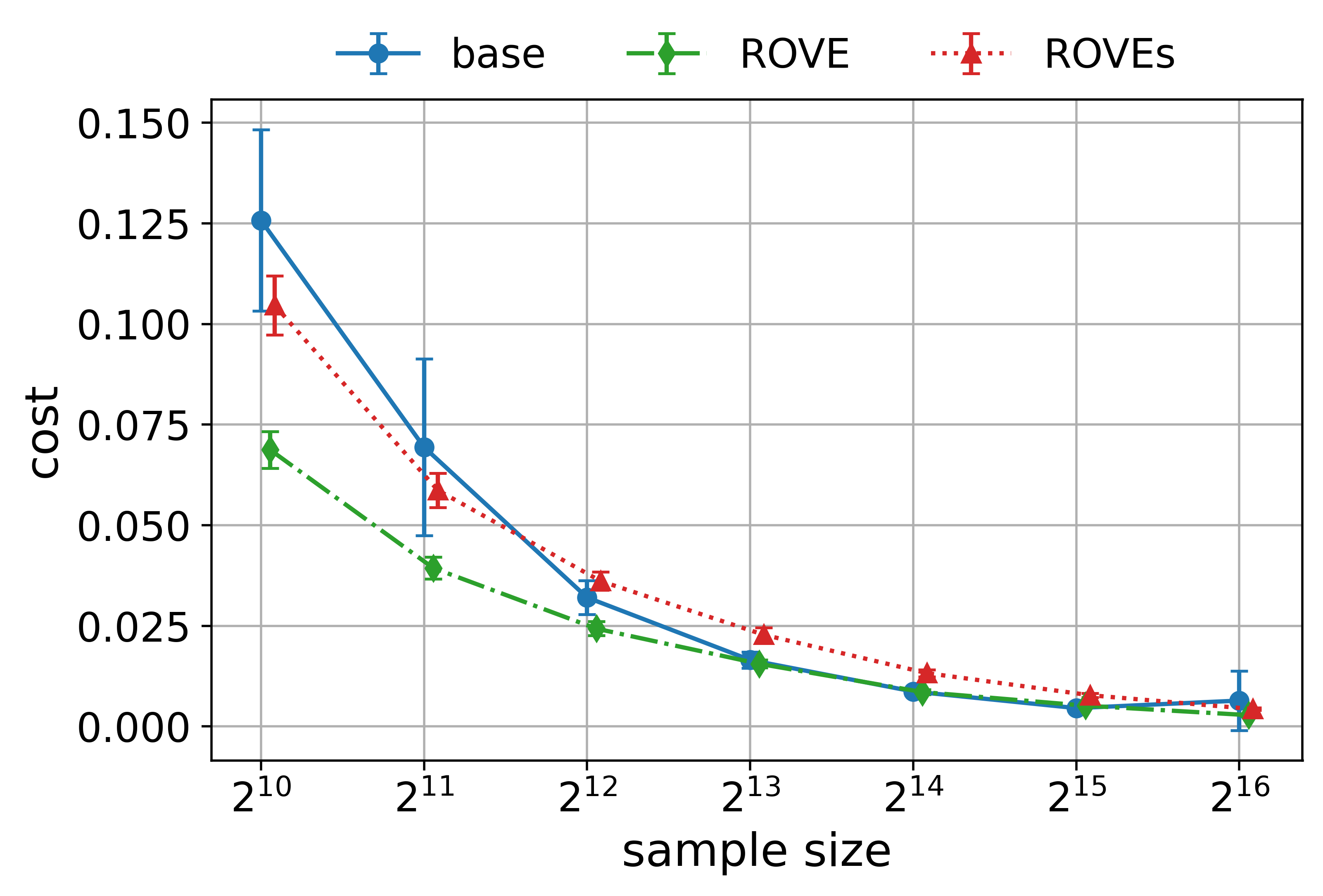}
\caption{$d=10$.\label{subfig: linear regression avg d=10}}
\end{subfigure}
\begin{subfigure}{0.49\textwidth}
\includegraphics[width = \linewidth]{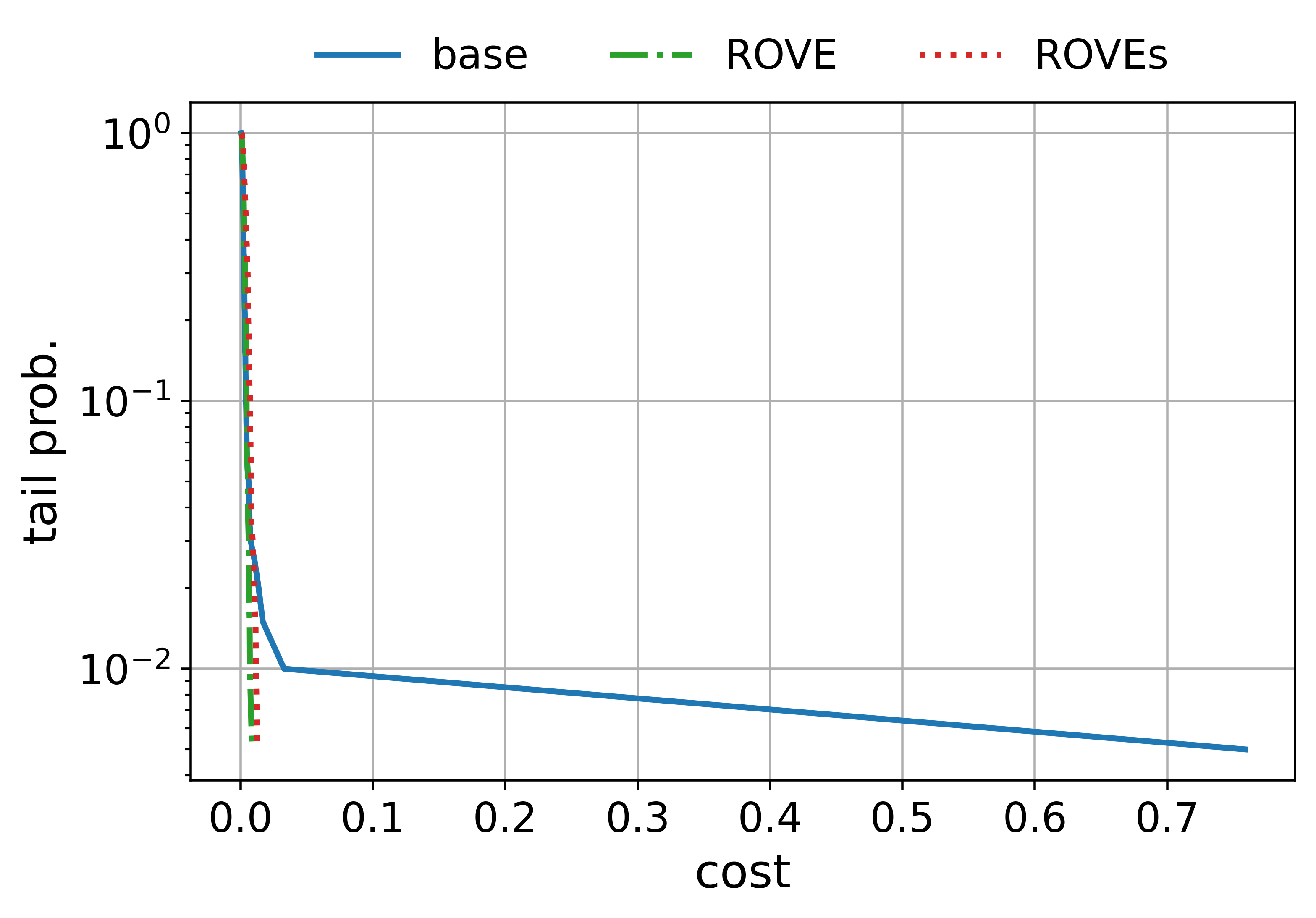}
\caption{Tail performance. $d=10,n=2^{16}$.\label{subfig: linear regression tail d=10 n=2^16}}
\end{subfigure}\\
\begin{subfigure}{0.49\textwidth}
\includegraphics[width = \linewidth]{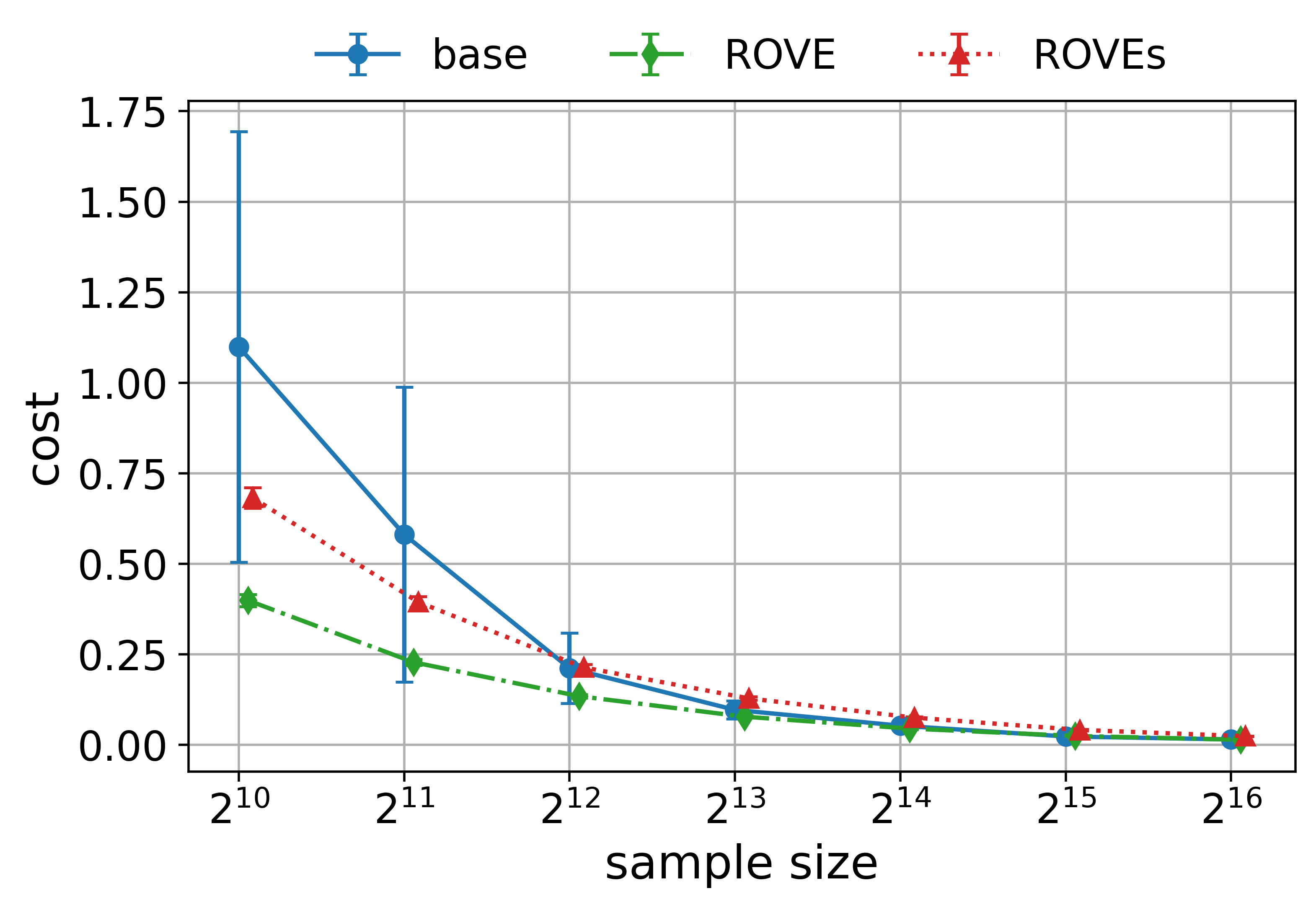}
\caption{$d=50$.\label{subfig: linear regression avg d=50}}
\end{subfigure}
\begin{subfigure}{0.49\textwidth}
\includegraphics[width = \linewidth]{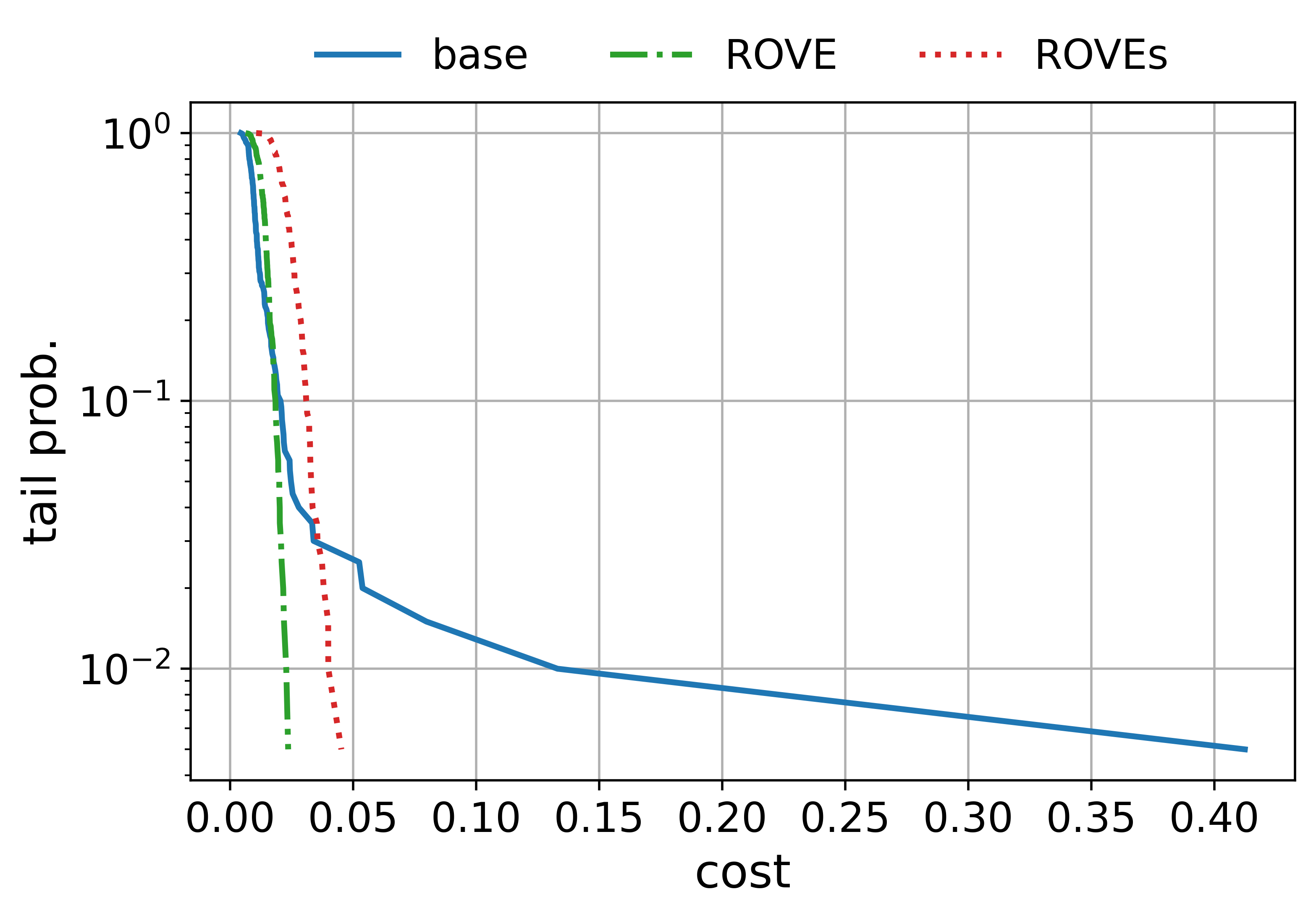}
\caption{Tail performance. $d=50,n=2^{16}$.\label{subfig: linear regression tail d=50 n=2^16}}
\end{subfigure}
\caption{Linear regression with least squares regression as the base learning algorithm. Given the input dimension $d$, the data generation is $Y=\sum_{i=1}^d(-10 + 20(i-1)/(d-1))X_i+\varepsilon_1-\varepsilon_2$ where each $X_i$ is independent $\mathrm{Unif}(0,2+18(i-1)/(d-1))$ and each $\varepsilon_j,j=1,2$ is $\mathrm{Pareto}(2.1)$ and independent of $X$. (a) and (c): Expected out-of-sample error with $95\%$ confidence interval. (b) and (d): Tail probabilities of out-of-sample errors.\label{fig: linear regression plots appendix}}
\end{figure}

\begin{figure}[!htbp]
\centering
% \hspace{-8pt}
\begin{subfigure}{0.49\textwidth}
\includegraphics[width = \linewidth]{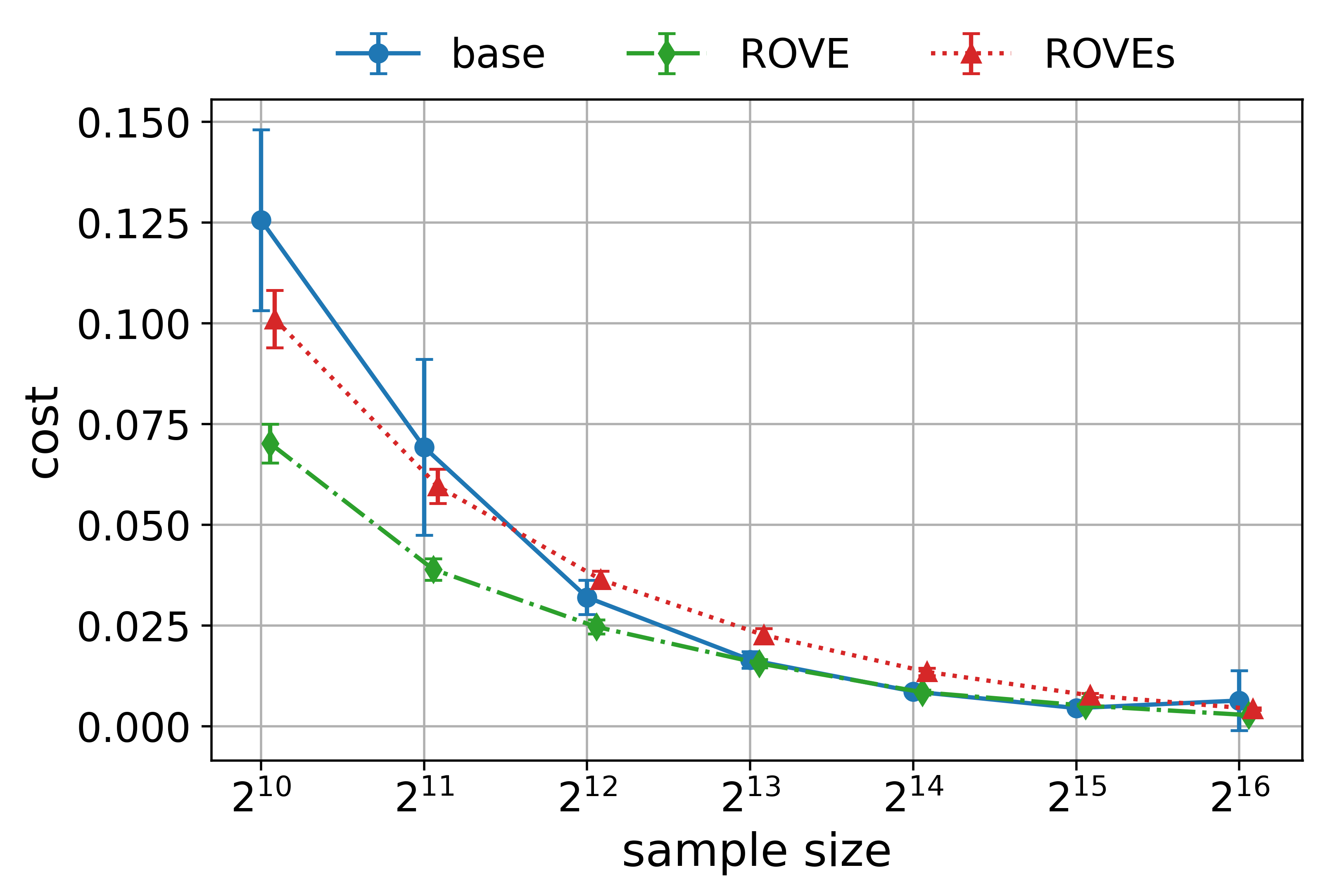}
\caption{$d=10$.\label{subfig: ridge regression avg d=10}}
\end{subfigure}
\begin{subfigure}{0.49\textwidth}
\includegraphics[width = \linewidth]{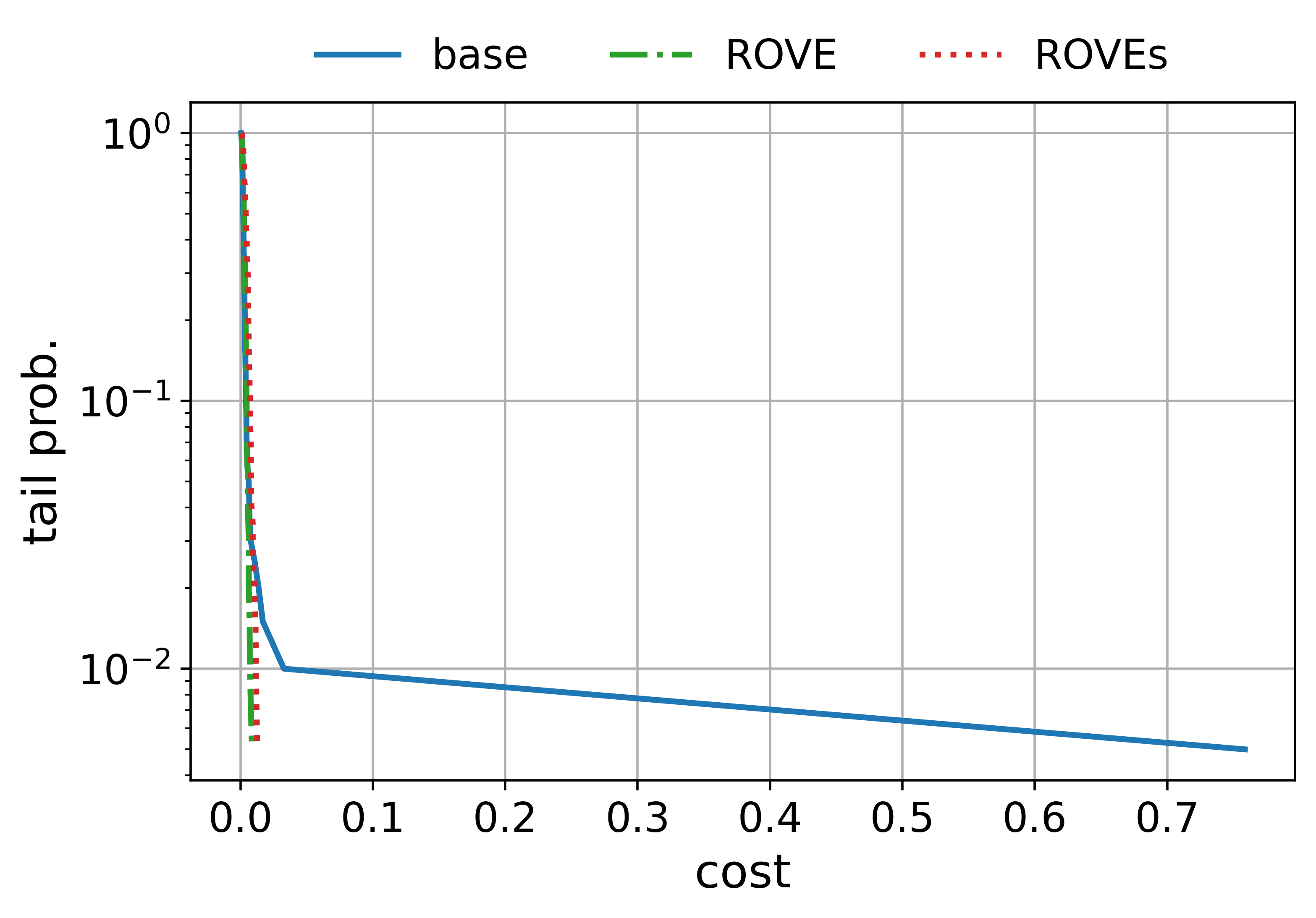}
\caption{Tail performance. $d=10,n=2^{16}$.\label{subfig: ridge regression tail d=10 n=2^16}}
\end{subfigure}\\
\begin{subfigure}{0.49\textwidth}
\includegraphics[width = \linewidth]{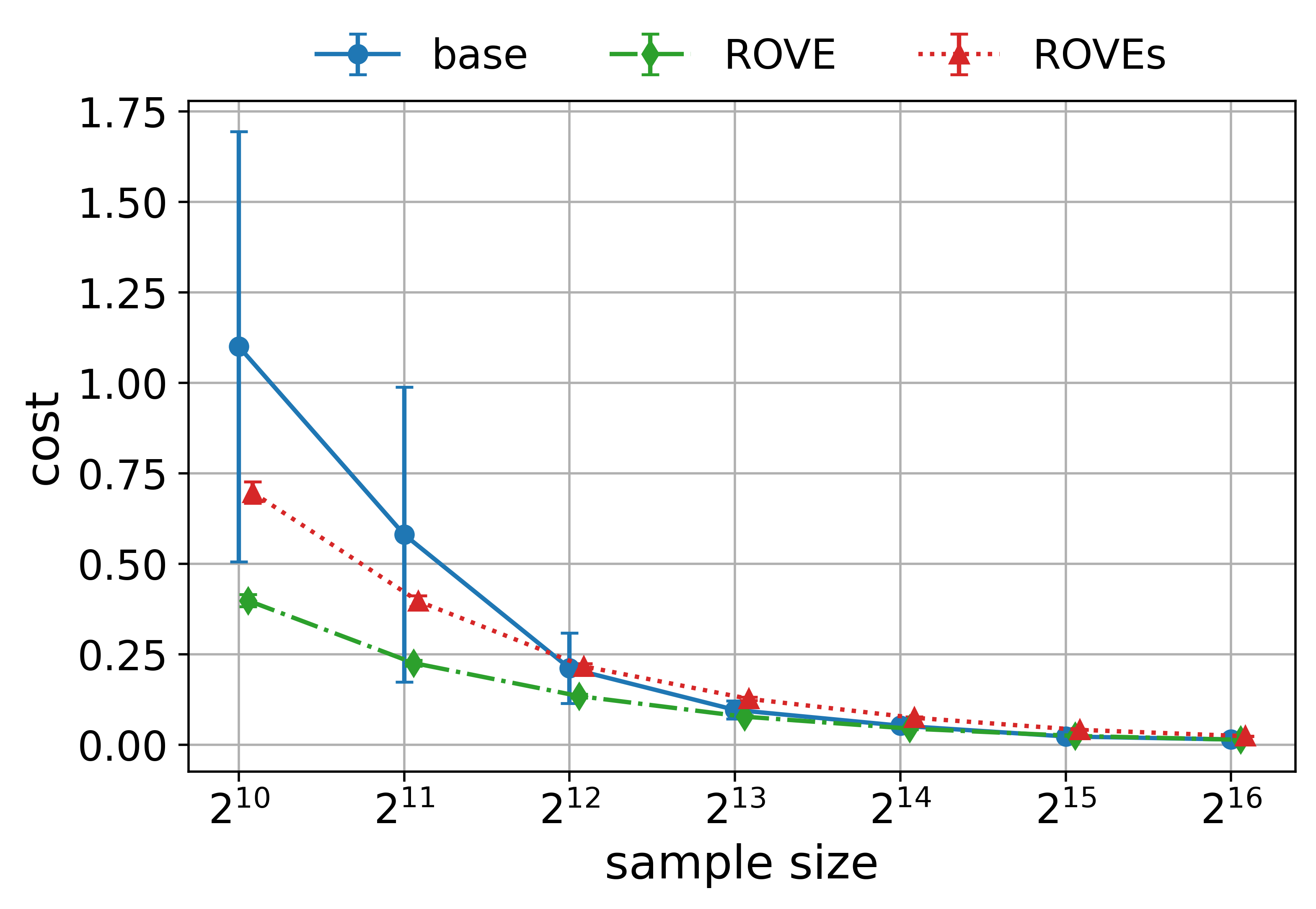}
\caption{$d=50$.\label{subfig: ridge regression avg d=50}}
\end{subfigure}
\begin{subfigure}{0.49\textwidth}
\includegraphics[width = \linewidth]{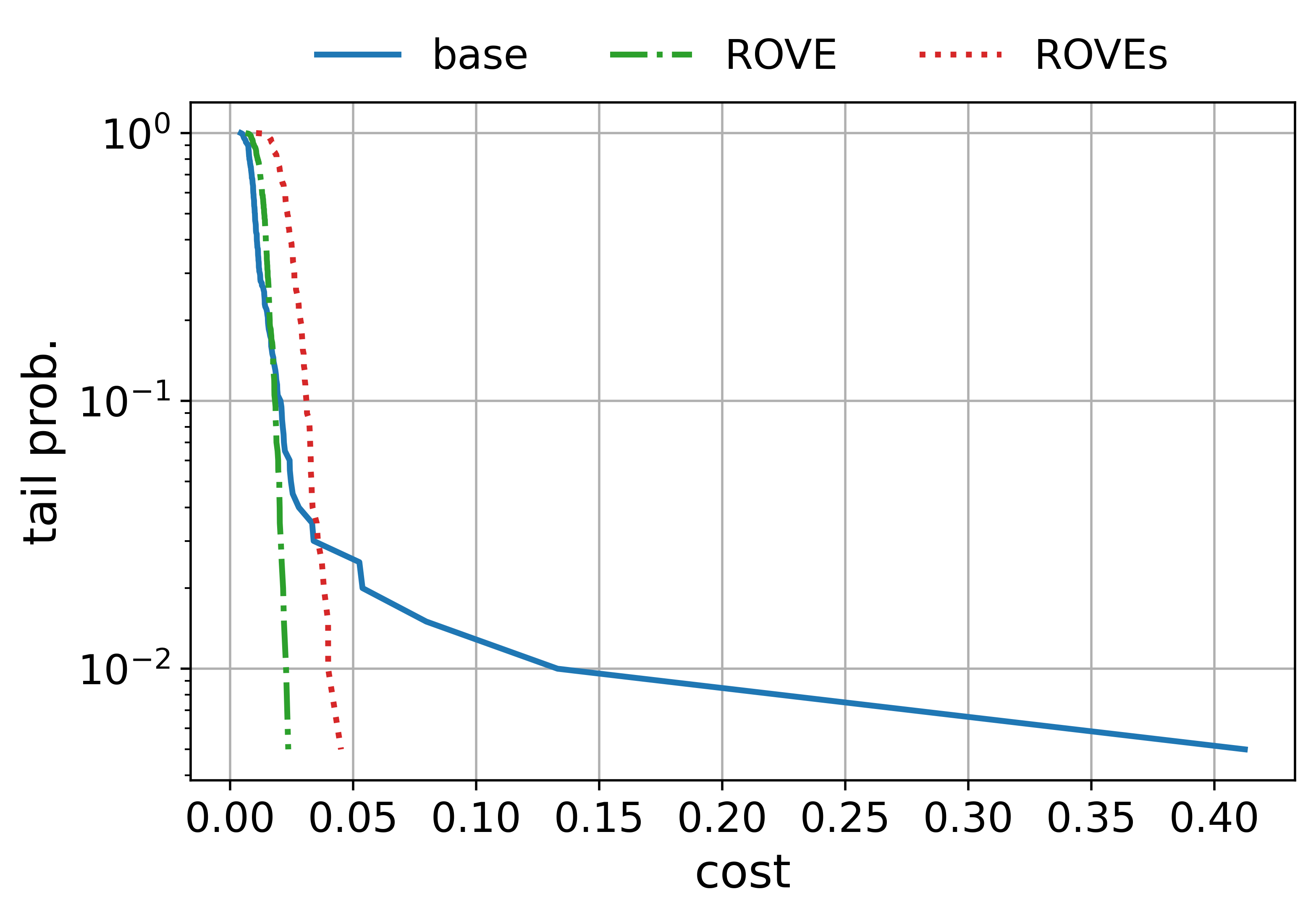}
\caption{Tail performance. $d=50,n=2^{16}$.\label{subfig: ridge regression tail d=50 n=2^16}}
\end{subfigure}
\caption{Linear regression with Ridge regression as the base learning algorithm. The same data generation as in Figure \ref{fig: linear regression plots appendix}. (a) and (c): Expected out-of-sample error with $95\%$ confidence interval. (b) and (d): Tail probabilities of out-of-sample errors.\label{fig: ridge regression plots appendix}}
\end{figure}

\begin{figure}[!htbp]
    \centering
\hspace{-12pt}
\begin{subfigure}{0.34\textwidth}
\includegraphics[width = \linewidth]{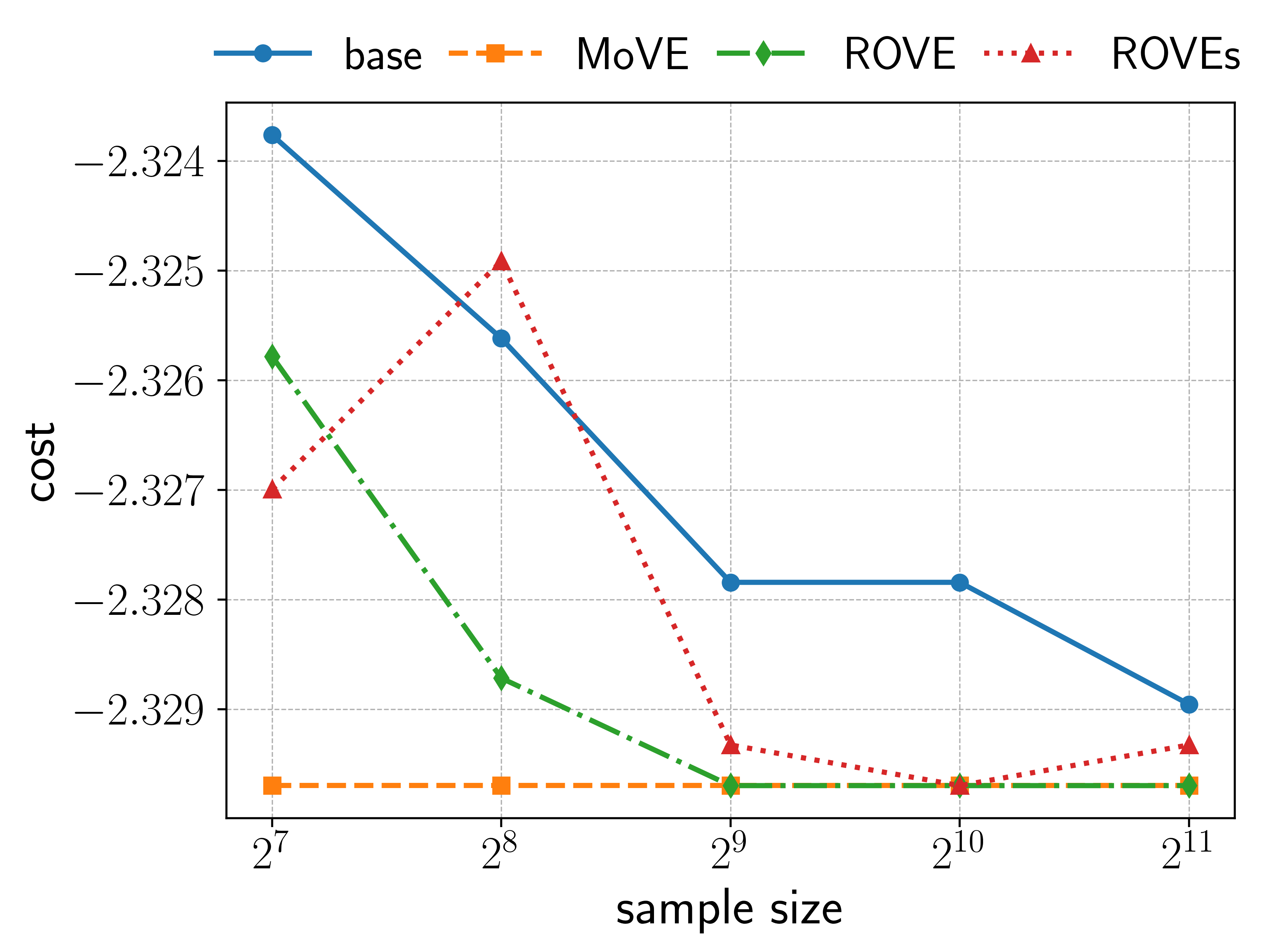}
\caption{Instance 1.\label{subfig: LP_light_1}}
\end{subfigure}%
\begin{subfigure}{0.34\textwidth}
\includegraphics[width = \linewidth]{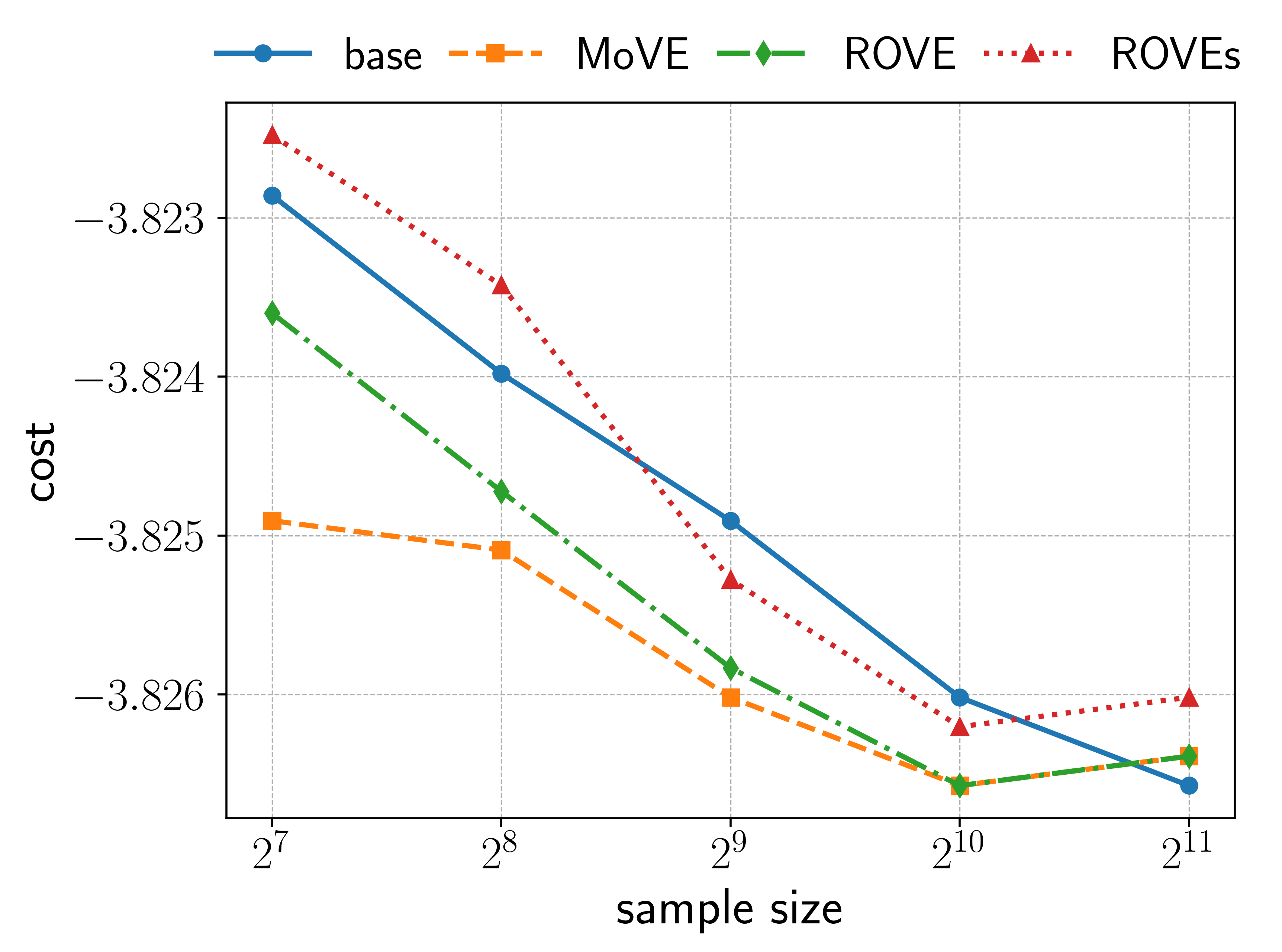}
\caption{Instance 2.\label{subfig: LP_light_2}}
\end{subfigure}%
\begin{subfigure}{0.34\textwidth}
\includegraphics[width = \linewidth]{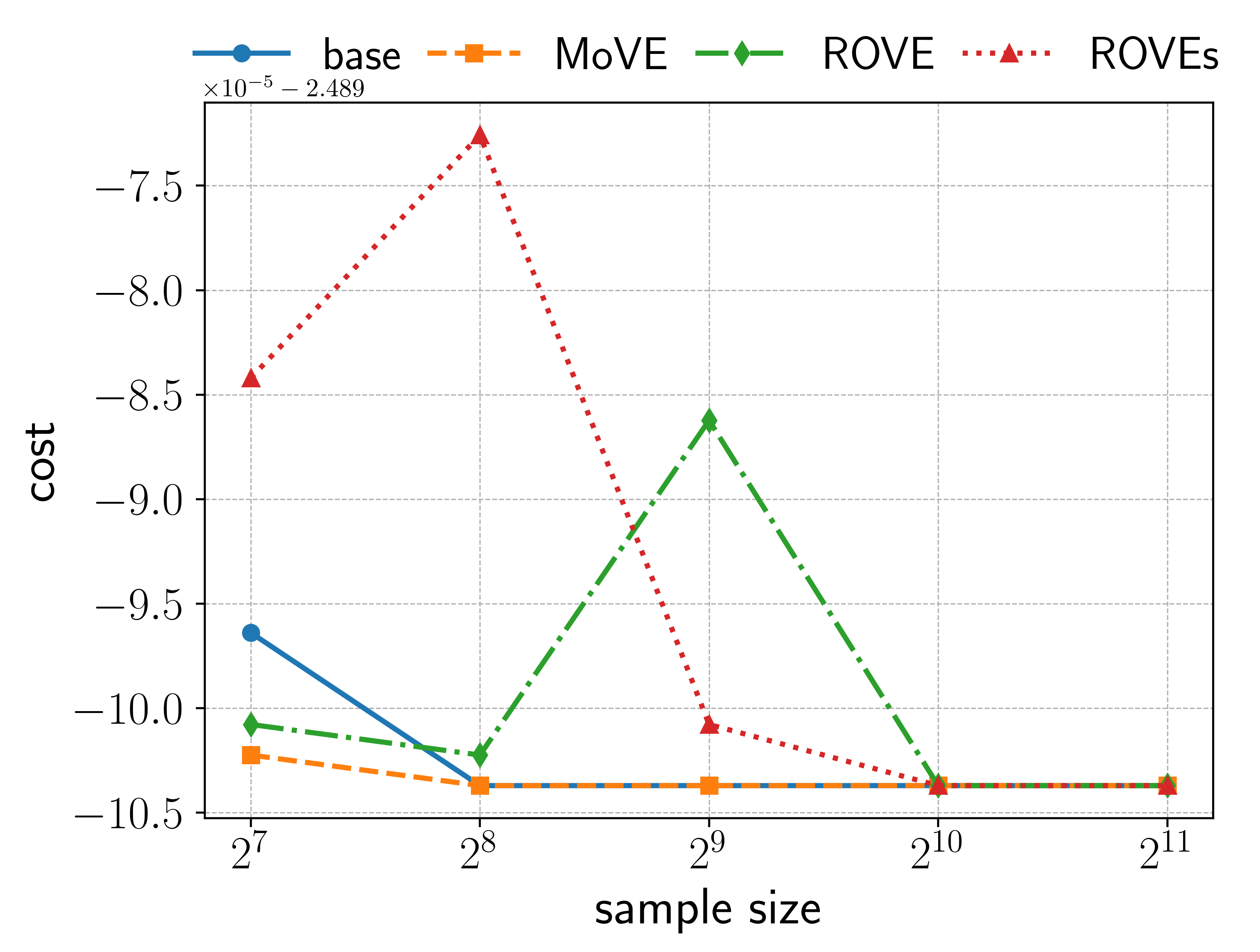}
\caption{Instance 3.\label{subfig: LP_light_3}}
\end{subfigure}
    \caption{Results for linear programs with light-tailed objectives. The base algorithm is SAA.}
    \label{fig: comparison_light_tail}
\end{figure}

\begin{figure}[!htbp]
    \centering
\begin{subfigure}{0.47\textwidth}
\includegraphics[width = \linewidth]{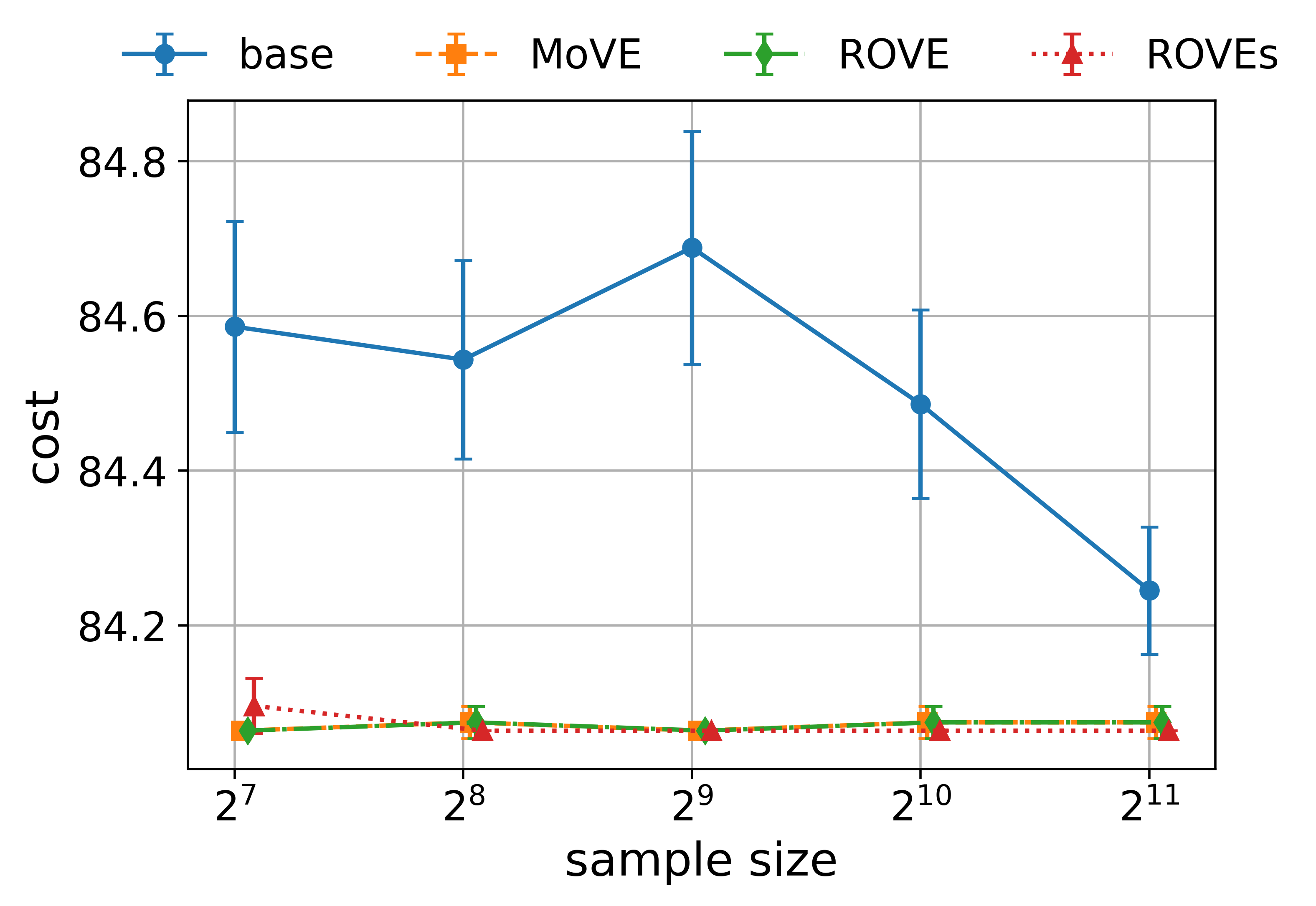}
\caption{$k=k_1=k_2=\max(10,n/10)$.\label{subfig: network design heavy tail k=(10,0.1)}}
\end{subfigure}%
\hfill
\begin{subfigure}{0.47\textwidth}
\includegraphics[width = \linewidth]{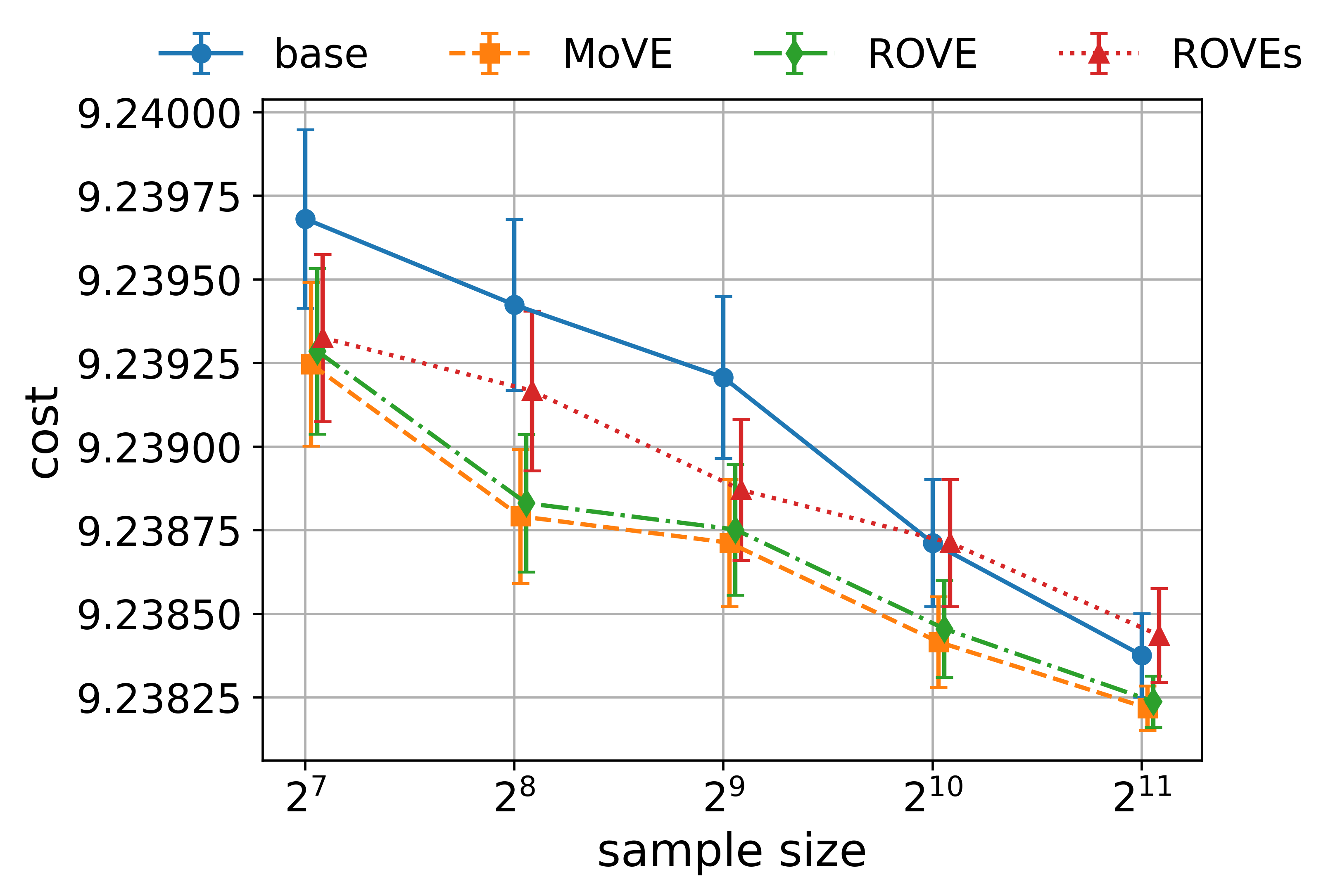}
\caption{A different instance with strong correlation.\label{subfig: network design strong correlation}}
\end{subfigure}%
    \caption{Results for supply chain network design. (a): The same problem instance as in Section \ref{subsec: alg_comparison} under a different hyperparameter choice: $k=\max(10,n/10),B=200$ for \move and $k_1=k_2=\max(10,n/10),B_1=20,B_2=200$ for \rove and \roves. (b): The same setup as in Section \ref{subsec: alg_comparison} but on a different problem instance for which the objectives under different solutions are strongly correlated. The strong correlation cancels out most of the heavy-tailed noise between solutions, making the base algorithm less susceptible to these noises, thus our ensemble methods appear less effective.}
    \label{fig: comparison_network appendix}
\end{figure}

\begin{figure}[!htbp]
    \centering
    \begin{subfigure}{0.47\textwidth}
    \includegraphics[width = \linewidth]{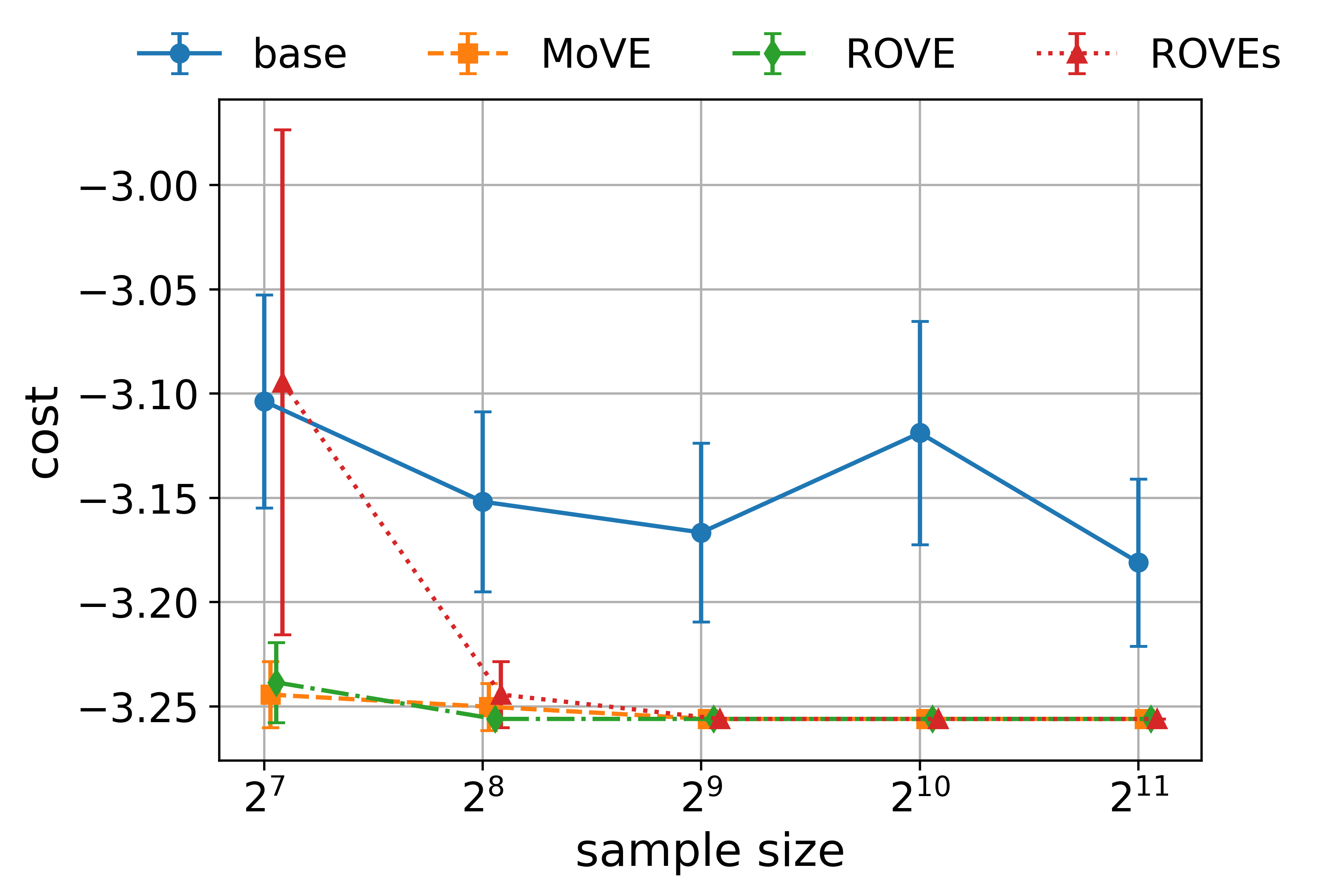}
    \caption{Resource allocation.\label{subfig: dro_bagging_sskp}}
    \end{subfigure}
    \hspace{10pt}
    \begin{subfigure}{0.47\textwidth}
    \includegraphics[width = \linewidth]{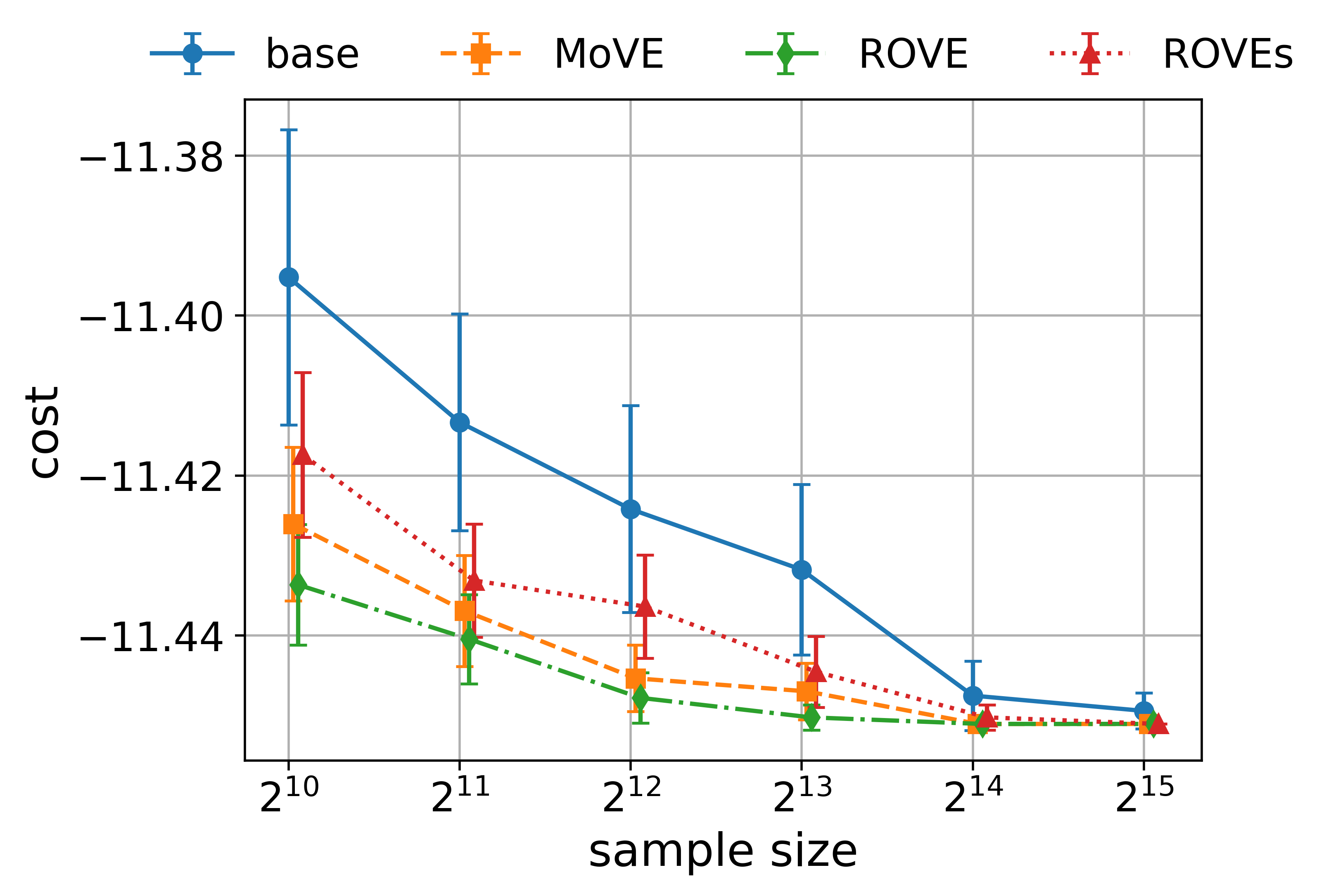}
    \caption{Maximum weight matching.\label{subfig: dro_bagging_matching}}
    \end{subfigure}
    \caption{Results for resource allocation and maximum weight matching when the base algorithm is DRO using 1-Wasserstein metric with the $l_{\infty}$ norm.
    \label{fig: bagging_dro}}
\end{figure}

\begin{figure}[!htbp]
    \centering
    \includegraphics[width = 0.7\linewidth]{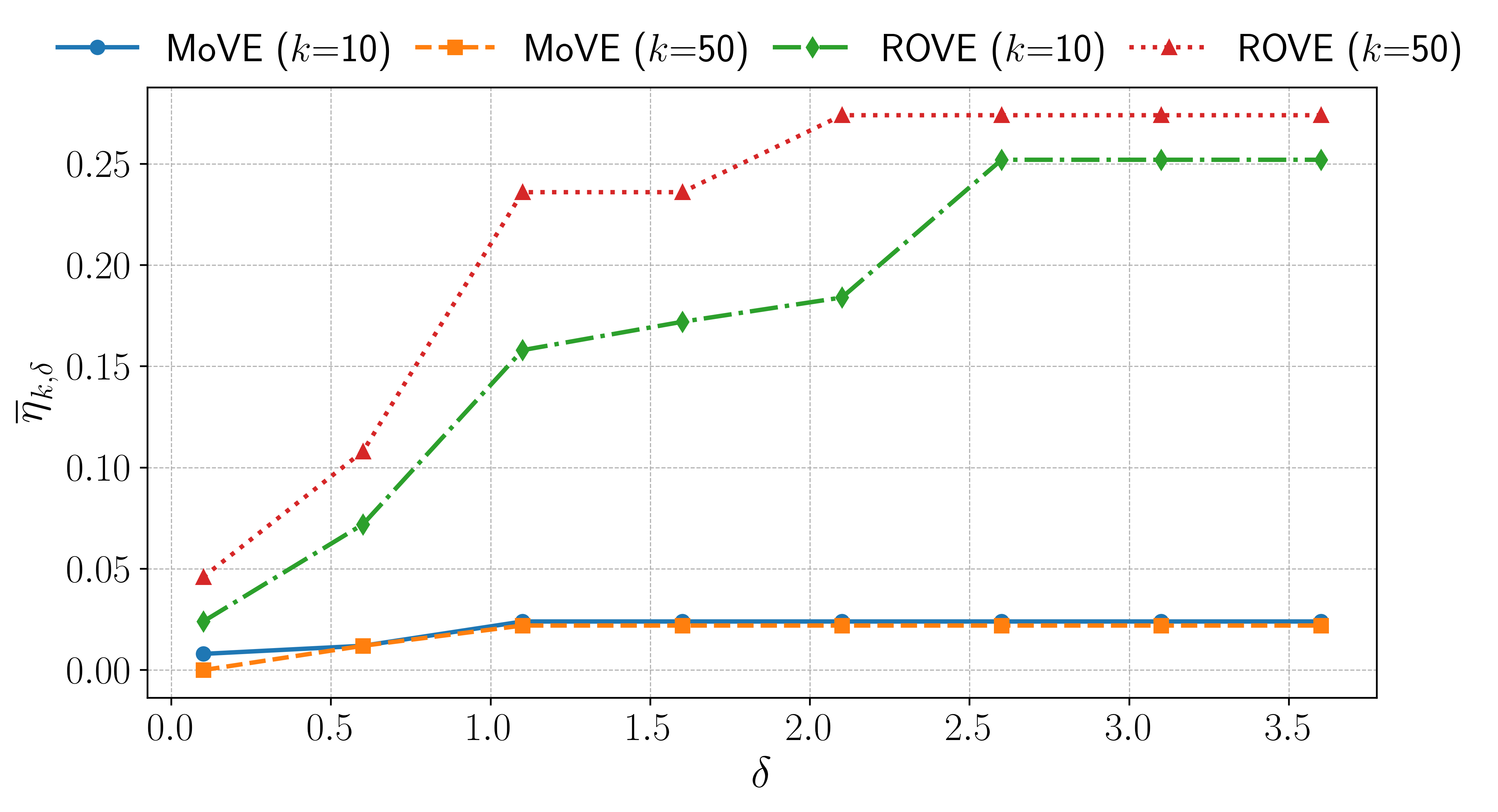}
    \caption{Comparison of $\bar{\eta}_{k,\delta}$ for \move and \rove in a linear program with multiple optima (corresponds to the instance in Figure \ref{subfig: saa_comparison_LP}). Threshold $\epsilon$ is chosen as $\epsilon = 4$ when $k=k_1=k_2=10$ and $\epsilon = 2.5$ when $k=k_1=k_2=50$, according to the adaptive strategy.
    Note that $\bar{\eta}_{k,\delta}=\max_{\theta\in \Theta}p_k(\theta) - \max_{\theta\in \Theta\backslash \Theta^{\delta}} p_k(\theta)$ by \eqref{SAA prob gap 2}, which measures the generalization sensitivity. 
    For \move, we have $p_k(\theta) = \mathbb P(\hat{\theta}_k^{SAA} =\theta)$; and for \rove, we have $p_k(\theta) = \mbb{P}(\theta\in  \widehat{\Theta}_k^{\epsilon})$, where $\widehat{\Theta}_k^{\epsilon}$ is the $\epsilon$-optimal set of SAA defined in \eqref{SAA nearly optima set}. From the figure, we can observe that the issue brought by the presence of multiple optimal solutions can be alleviated using the two-phase strategy in \rove.
    \label{fig: eta_comparison}}
\end{figure}

\begin{figure}[!htbp]
\centering
\includegraphics[width = 0.47\linewidth]{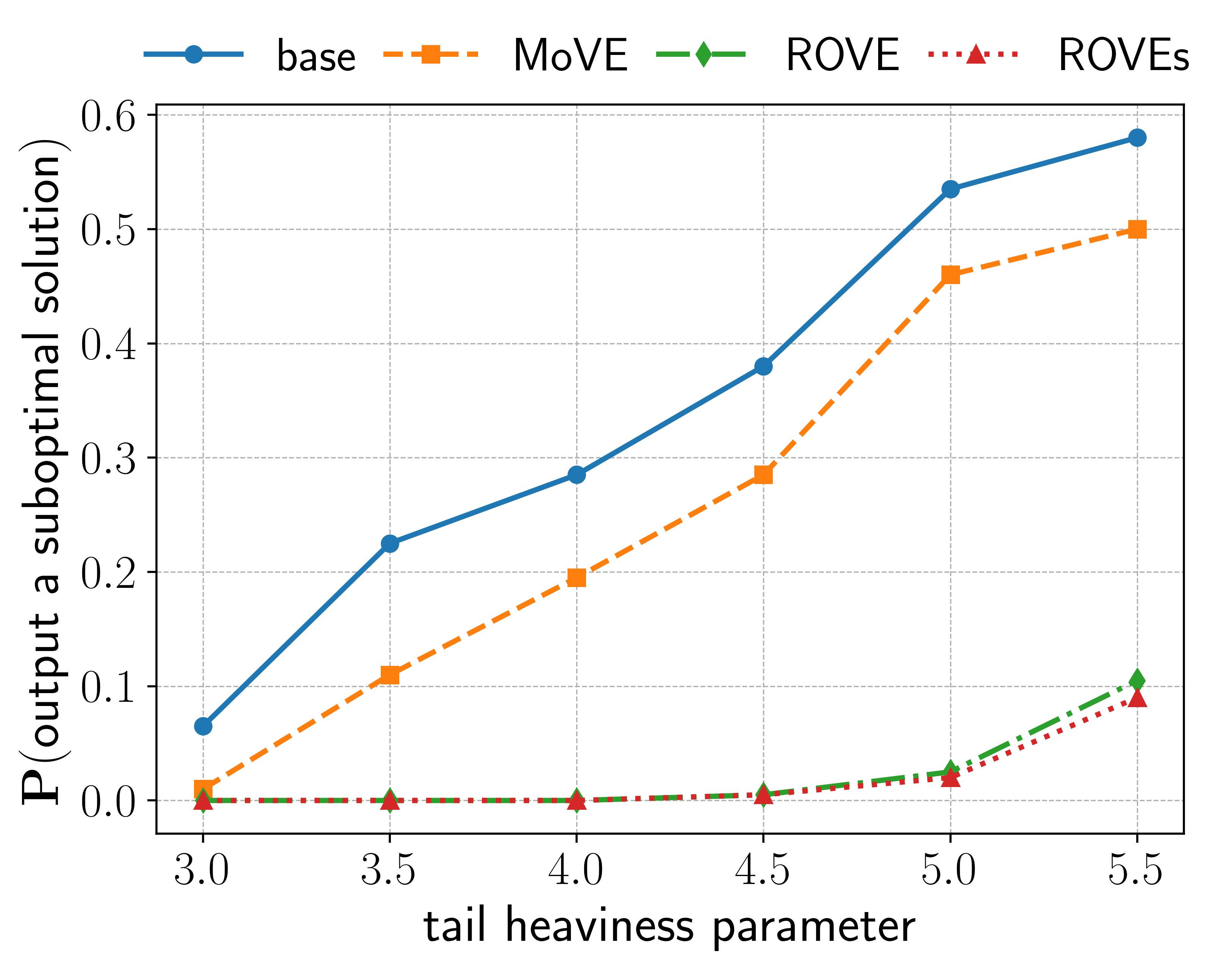}
\caption{Influence of tail heaviness in the stochastic linear program with multiple optima with $n=10^6$.
Hyperparameters: $k=50,B=2000$ for \move, $k_1=k_2=50,B_1=200,B_2 = 5000$ for \rove and \roves.
The tail heaviness parameter corresponds to the mean of the Pareto random coefficient.\label{fig: tail_heaviness}}
\end{figure}

% In Figure \ref{fig: eta_comparison}, we plot the generalization sensitivity $\eta_{k,\delta}$ for different bagging approaches, further explaining the superiority of ROVE and ROVEs when multiple optimal occur.

% \begin{figure}[htbp]
%     \centering
%     \begin{subfigure}{0.47\textwidth}
%     \includegraphics[width = \linewidth]{img/DRO_comparison_SSKP.png}
%     \caption{Resource allocation.\label{subfig: dro_comparison_sskp}}
%     \end{subfigure}
%     \hspace{10pt}
%     \begin{subfigure}{0.47\textwidth}
%     \includegraphics[width = \linewidth]{img/DRO_comparison_matching.png}
%     \caption{Maximum weight matching.\label{subfig: dro_comparison_matching}}
%     \end{subfigure}
%     \caption{Comparison between bagging approaches (using SAA as the base model) and DRO with Wasserstein metric, where $r$ denotes the radius of the ambiguity set. 
%     (a) corresponds to the instance in Figure \ref{subfig: saa_comparison_sskp} and (b) corresponds to the instance in Figure \ref{subfig: saa_comparison_matching}.
%     Hyperparameters: $k = \max(10,0.005n)$, $B = 200$, $B_1 = 20$, and $B_2 = 200$.
%     We observe that though DRO performs slightly better than SAA under some choices of the ambiguity set radius, the improvement is marginal compared to the bagging approaches.
%      }
%     \label{fig: saa_bagging_vs_dro}
% \end{figure}

\clearpage
\bibliographystyleAPX{plain}
\bibliographyAPX{neurips_2025}

\end{document}